\documentclass[10pt]{article}
\usepackage[all,2cell]{xy} \UseAllTwocells \SilentMatrices
\usepackage{latexsym,amsfonts,amssymb}
\usepackage{amsmath,amsthm,amscd}
\usepackage{hyperref,psfrag}
\usepackage{diagcat}
\usepackage{color}
\usepackage{etoolbox}
\usepackage[dvips]{epsfig}
\usepackage{psfrag}
\usepackage{graphicx}
\usepackage{a4wide}
\usepackage{geometry}

\usepackage{epigraph,wrapfig}

\normalfont\upshape

\usepackage{fancyhdr}
\theoremstyle{definition}
\newtheorem{thm}{Theorem}[section]

\newtheorem{lem}[thm]{Lemma}
\newtheorem{rem}[thm]{Remark}
\newtheorem{prop}[thm]{Proposition}
\newtheorem{defn}[thm]{Definition}

\newtheorem*{thm*}{Theorem}
\numberwithin{equation}{section}

\usepackage{bbm}

\def\Z{{\mathbbm Z}}

\def\1{{\mathbbm{1}}}

\newcommand{\Hom}{{\rm Hom}}

\renewcommand{\to}{\rightarrow}

\def\dif{\partial}

\def\lra{{\longrightarrow}}
   
\def\Id{\mathrm{Id}}
\def\mc{\mathcal}
\def\mf{\mathfrak}
\def\shuffle{\,\raise 1pt\hbox{$\scriptscriptstyle\cup{\mskip
               -4mu}\cup$}\,}
\newcommand{\refequal}[1]{\xy {\ar@{=}^{#1}
(-1,0)*{};(1,0)*{}};
\endxy}

\newrgbcolor{newred}{1.00 0.40 0.40}
\newstrandstyle{red}{type=string,color=newred,style=solid}
\newstrandstyle{Red}{type=ribbon,color=newred,style=solid}
\newstrandstyle{Green}{type=ribbon,style=solid,ribboncolor=green,color=green}

\title{A braid group action on a $p$-DG homotopy category}
\author{You Qi, Joshua Sussan, Yasuyoshi Yonezawa}

\date{\today}

\begin{document}
%
% ==============================================================================

\maketitle

\begin{abstract}
We construct a categorical braid group action on a homotopy category of $p$-DG modules of a deformed Webster algebra.
\end{abstract}

\setcounter{tocdepth}{2} \tableofcontents

\section{Introduction}
An approach to categorification of quantum groups, their representations, and quantum invariants at a prime $p$ root of unity was outlined in \cite{Hopforoots} and further developed in \cite{QYHopf}.  These works suggest that one should look for $p$-differentials on structures categorifying objects at generic values of the quantum parameter. There has been some progress in this program for quantum $\mathfrak{sl}_2$ \cite{EQ2, EQ1, KQ, KQS, QiSussan, QiSussan2, QiSussan3}.

In particular, a categorification of the braid group action on the Burau representation at a prime root of unity was constructed in \cite{QiSussan}.  The authors considered a $p$-DG structure on the algebra $A_n^!$ which describes a singular block of category $\mathcal{O}(\mathfrak{gl}_n)$ corresponding to the Young subgroup $S_1 \times S_{n-1}$.  Ignoring the $p$-DG structure, this categorical action is a consequence of Koszul duality and Khovanov and Seidel's action of the braid group on the homotopy category of modules over the zigzag algebra $A_n$.  Whereas the projective modules play the role of the Temperley-Lieb algebra in \cite{KS}, the simple objects form an exceptional sequence of objects on the other side of Koszul duality.  As a result, one must find projective resolutions of the simple objects of $A_n^!$ in order to construct the braid group action directly.
In the context of the $p$-differential, one must find cofibrant replacements of the simple objects.  This was the main technical step in \cite{QiSussan}.

Partially motivated by the construction in \cite{QiSussan}, a deformation $W=W(n,1)$ of $A_n^!$ was considered in \cite{KhovSussan} and the authors showed that there is a categorical braid group action on the homotopy category of $W$-modules.   This result was extended in \cite{KLSY} to a categorical braid group action on the homotopy category of a deformation of more general Webster algebras for $\mathfrak{sl}_2$, for which $W$ is a special case (hence the notation).

In this note, we return to the simplified setting of $W$.  There is a $p$-DG structure on this algebra (see \cite{Y}).  We show that there are braiding complexes in a homotopy category of $p$-DG $W$-modules using key ideas from \cite{KRWitt}.   Using some results from $p$-DG theory we extend the main result of \cite{KhovSussan} to the $p$-DG setting.

\begin{thm*}
There is a categorical action of the braid group on $n$ strands on the relative $p$-DG homotopy category of $W$.
\end{thm*}

Deforming the algebra $A_n^!$ allows us to replace the $p$-DG derived category in \cite{QiSussan} with the relative $p$-DG homotopy category here. This theorem should be compared to \cite[Theorem 5.14]{QiSussan}, albeit in the weaker context of the relative homotopy category. However, the authors believe that the results will become useful towards building a $p$-DG link homology theory, as proposed in \cite{QiSussan4}. These further questions will be addressed in subsequent works of the authors.

In \cite{KLSY} a braid group action on the homotopy category of deformed Webster algebras was constructed by exhibiting an action of the Khovanov-Lauda-Rouquier $2$-category \cite{KL3, Rou2}, and then using symmetric Howe duality.  We expect the main result here holds on the level of generality of \cite{KLSY}.  One would need to consider a $p$-DG version of the Khovanov-Lauda-Rouquier category and show the braiding of Rickard complexes \cite{CKbraid} holds in the presence of the $p$-differential.

In recent work \cite{Webdeformed}, the main result of \cite{KLSY} was proved by relating the deformed algebra to various objects in Lie theory and geometry.  It would be interesting to import the $p$-DG structure to the categories of Gelfand-Tsetlin modules and perverse sheaves studied there.

\subsection{Acknowledgements.}
Y.~Q.~is partially supported by the NSF grant DMS-1947532.
J.~S.~is partially supported by the NSF grant DMS-1807161 and PSC CUNY Award 63047-00 51.
Y.~Q.~and~Y.~Y.~were partially supported by the Research Institute for Mathematical Sciences, an International Joint Usage/Research Center located in Kyoto University.

%%%%%%%%%%%%%%%%%%%%%%%%%%%%%%%%%%%%%%%%%%%%%%%%%%%%%%%%%%%%%%%%%%%%%%%%%%%%%%%%%%%%%%%%%%%%%%%%%%%%%%%%%%%%%%%%%%%%%%
%%%%%%%%%%%%%%%%%%%%%%%%%%%%%%%%%%%%%%%%%%%%%%%%%%%%%%%%%%%%%%%%%%%%%%%%%%%%%%%%%%%%%%%%%%%%%%%%%%%%%%%%%%%%%%%%%%%%%%
\section{Some exact functors on \texorpdfstring{$p$}{p}-complexes}\label{secpdgtheory}

In this section, we gather some necessary background material for later use. Further details of this section can be found in \cite[Section 2.1]{KQ} and \cite[Section 2.1]{QiSussan4}.

\subsection{Extension functors}
Let $\Bbbk$ be a field of finite characteristic $p>0$ and $A$ be a $\Bbbk$-algebra. We think of $A$ as a graded algebra by setting all elements of $A$ to be of degree zero. A \emph{$p$-complex} of $A$-modules $U=\oplus_{i\in \Z}U^i$ is then a graded $A$-module, equipped with a degree-two endomorphism $\dif_U$ satisfying $\dif_U^p=0$. A morphism of $p$-complexes $f: U\lra V$ consists of $A$-module maps $f_i:U^i \lra V^i$, $i\in \Z$, that commute with the $p$-differentials on $U$ and $V$. A morphism $f: U\lra V$ of $p$-complexes is called \emph{null-homotopic} if there is a collection of $A$-linear maps $h_i: U^i\lra V^{i-2p+2}$ such that, for any $i\in \Z$,
\begin{equation}
    f_i= \sum_{k=0}^{p-1}\dif_U^{p-1-k} h_{i+2k} \dif_V^{k}.
\end{equation}
When $p=2$, these notions reduce to the usual notion of (co)chain complexes of $A$-modules over characteristic two, with the differential degree doubled. Furthermore, just as the usual homotopy category of chain complexes is triangulated, the \emph{homotopy category of $p$-complexes}, obtained from the (abelian) category of $p$-complexes modulo the class of null-homotopic morphisms, is also triangulated.
Let us denote this $p$-homotopy category by $\mc{C}(A,\dif)$.

In this subsection, we recall a functor relating the usual homotopy category $\mc{C}(A,d)$ of $A$ with its $p$-homotopy category $\mc{C}(A,\dif)$.

To do this, recall that a usual chain complex of $A$-modules $M$ consists of a collection of $A$-modules and homorphisms $d_i:M^i\lra M^{i+1}$ called coboundary maps
\[
\xymatrix{
\cdots \ar[r]^-{d_{i-2}}  & M^{i-1} \ar[r]^-{d_{i-1}} & M^{i} \ar[r]^-{d_{i}} & M^{i+1} \ar[r]^-{d_{i+1}} & M^{i+2} \ar[r]^-{d_{i+2}} & \cdots
} ,
\]
satisfying $d_i\circ d_{i-1}=0$ for all $i\in \Z$. A morphism of chain complexes $f:M\lra N$ consists of a sequence $f_i:M^i\lra N^i$ that commute with the differentials. A map is null-homotopic if there is a sequence of $A$-module maps $f_i:M^i\lra N^i$, $i\in \Z$, of $A$-modules, as depicted in the diagram below,
\begin{equation*}
 \xymatrix{
 \cdots\ar[r]^{d_{i-2}} & M^{i-1} \ar[dl]|-{h_{i-1}} \ar[r]^{d_{i-1}} \ar[d]|-{f_{i-1}} & M^{i} \ar[dl]|-{h_{i}} \ar[r]^{d_{i}} \ar[d]|-{f_{i}} & M^{i+1} \ar[dl]|-{h_{i+1}} \ar[r]^{d_{i+1}} \ar[d]|-{f_{i+1}} & M^{i+2}\ar[r]^{d_{i+2}} \ar[dl]|-{h_{i+2}} \ar[d]|-{f_{i+2}} & \cdots \ar[dl]|-{h_{i+3}}\\
 \cdots \ar[r]_{d_{i-2}} & N^{i-1} \ar[r]_{d_{i+1}} & N^{i} \ar[r]_{d_{i}} & N^{i+1} \ar[r]_{d_{i+1}} & N^{i+2}\ar[r]_{d_{i+2}} & \cdots
 }   
\end{equation*}
which satisfy $f_{i}=d_{i+1}\circ h_i+h_{i+1}\circ d_i$ for all $i\in \Z$. The homotopy category $\mc{C}(A,d)$, by construction, is the quotient of chain complexes over $A$ by the ideal of null-homotopic morphisms.

 We define the \emph{$p$-extension functor}
\begin{equation}
   \mc{P}: \mc{C}(A,d)\lra \mc{C}(A ,\dif) 
\end{equation}
as follows. Given a complex of $A$-modules, we repeat every term sitting in odd homological degrees $(p-1)$ times while keeping terms in even homological degrees unaltered. More explicitly, for a given complex
\[
\xymatrix{
\cdots \ar[r]^-{d_{2k-2}}  & M^{2k-1} \ar[r]^-{d_{2k-1}} & M^{2k} \ar[r]^-{d_{2k}} & M^{2k+1} \ar[r]^-{d_{2k+1}} & M^{2k+2} \ar[r]^-{d_{2k+2}} & \cdots
} ,
\]
the extended complex looks like
\[
\xymatrix{\cdots \ar[r]^-{d_{2k-2}} & M^{2k-1} \ar@{=}[r]& \cdots \ar@{=}[r] & M^{2k-1}\ar[r]^-{d_{2k-1}} \ar[r] &
M^{2k}\ar `r[rd] `_l `[llld] _-{\phi_{2k}} `[d] [lld]
& \\
& & M^{2k+1}\ar@{=}[r]&\cdots \ar@{=}[r]& M^{2k+1} \ar[r]^-{d_{2k+1}}& M^{2k+2} \ar[r]^{d_{2k+2}}& \cdots}.
\]
Likewise, for a chain-map
\begin{equation*}
\begin{gathered}
 \xymatrix{
 \cdots\ar[r]^-{d_{2k-3}} & M^{2k-2}  \ar[r]^{d_{2k-2}} \ar[d]|-{f_{2k-2}} & M^{2k-1}  \ar[r]^{d_{2k-1}} \ar[d]|-{f_{2k-1}} & M^{2k}  \ar[r]^{d_{2k}} \ar[d]|-{f_{2k}} &  \cdots\\
 \cdots \ar[r]_-{d_{2k-3}} & N^{2k-2} \ar[r]_{d_{2k-2}} & N^{2k-1} \ar[r]_{d_{2k-1}} & N^{2k} \ar[r]_{d_{2k}}  & \cdots
 }
 \end{gathered}
 \ ,
\end{equation*}
the obtained morphism of $p$-DG $A$-modules is given by
\[
\xymatrix{
 \cdots\ar[r]^-{d_{2k-3}} & M^{2k-2}  \ar[r]^{d_{2k-2}} \ar[d]|-{f_{2k-2}} & M^{2k-1}  \ar[d]|-{f_{2k-1}} \ar@{=}[r] & \cdots \ar@{=}[r] &  M^{2k-1}   \ar[r]^{d_{2k-1}} \ar[d]|-{f_{2k-1}} & M^{2k}  \ar[r]^{d_{2k}} \ar[d]|-{f_{2k}} &  \cdots\\
 \cdots \ar[r]_-{d_{2k-3}} & N^{2k-2} \ar[r]_{d_{2k-2}} & N^{2k-1} \ar@{=}[r] & \cdots \ar@{=}[r] & N^{2k-1}  \ar[r]_{d_{2k-1}} & N^{2k} \ar[r]_{d_{2k}}  & \cdots
 }   \ .
\]
This is clearly a functor from the abelian category of cochain complexes over $A$ into the category of $p$-complexes of $A$-modules, which we call $\widehat{\mc{P}}$.

Next, one may check that $\widehat{\mc{P}}$ preserves ideals of null-homotopic morphisms. Indeed, suppose $f=dh+hd$ is a null-homotopic morphism in $\mc{C}(A,d)$. We first extend, for any $i\in \Z$  and $h_i: M_i \lra N_{i-1}$, to a collection of maps 
$$\hat{\mc{P}}(h)_i: \hat{\mc{P}}(M)_i \lra \hat{\mc{P}}(N)_{i-2p+2}$$
as follows.
On unrepeated terms, $\hat{\mc{P}}(h)$ sends $M^{2k}$ to the copy of $N^{2k-1}$ sitting as the leftmost term in the repeated $N^{2k-1}$'s, while on the repeated terms, it only sends the rightmost $M^{2k+1}$ to the unrepeated $N^{2k}$ and acts by zero on the other repeated $M^{2k+1}$'s. Schematically, this has the effect:
\[
\xymatrix{
 \cdots\ar@{=}[r] & M^{2k-3} \ar[r]^-{d_{2k-3}} & M^{2k-2}  \ar[r]^{d_{2k-2}} & M^{2k-1}  \ar[dlll]|-{0}  \ar@{=}[r] & \cdots \ar@{=}[r] &  M^{2k-1}   \ar[r]^{d_{2k-1}} \ar[dlll]|-{h_{2k-1}} & M^{2k} \ar[r]^{d_{2k}} \ar[dlll]|-{h_{2k}} &  M^{2k+1} \ar[dlll]|-{0} \ar@{=}[r] &\cdots\\
\cdots \ar@{=}[r]&  N^{2k-3} \ar[r]_-{d_{2k-3}} & N^{2k-2} \ar[r]_{d_{2k-2}} & N^{2k-1} \ar@{=}[r] & \cdots \ar@{=}[r] & N^{2k-1}  \ar[r]_{d_{2k-1}} & N^{2k} \ar[r]_{d_{2k}}  & N^{2k+1} \ar@{=}[r]&\cdots
 }   
 \ .
\]

\begin{lem} \label{lemnullhom}
The functor $\widehat{\mc{P}}$ sends null-homotopic morphisms in $\mc{C}(A,d)$ to null-homotopic morphisms in $\mc{C}(A,\dif)$.
\end{lem}
\begin{proof}
The proof now is an easy exercise. See \cite[Lemma 2.2]{QiSussan4} for more details.
\end{proof}

This lemma implies that $\widehat{\mc{P}}$ descends to a functor
\begin{equation}\label{equation-functor-P}
    \mc{P}: \mc{C}(A,d)\lra \mc{C}(A,\dif).
\end{equation}
which we call the \emph{$p$-extension functor}. A key property of this functor that we will use is the following.

\begin{prop}
\cite[Proposition 2.3]{QiSussan4}
The functor $\mc{P}$ is exact.
\end{prop}

\subsection{Relative homotopy categories}
For any graded or ungraded algebra $B$ over $\Bbbk$, denote by $d_0$ the zero super differential and by $\dif_0$ the zero $p$-differential on $B$, while letting $B$ sit in homological degree zero. When $B$ is graded, the homological grading is independent of the internal grading of $B$.
For a graded module $M$ over a graded algebra $B$, we let $M \{n\}$ denote the module $M$, where the internal grading has been shifted up by $n$.

Suppose $(A,\dif_A)$ is a $p$-DG algebra, i.e., a graded algebra equipped with a differential $\dif_A$ of degree two, satisfying
\begin{equation}
\dif_A^p(a)\equiv 0 , \quad \quad \dif_A(ab)=\dif_A(a)b+a\dif_A(b),
\end{equation}
for all $a,b\in A$. In other words, $A$ is an algebra object in the module category of the graded Hopf algebra $H=\Bbbk[\dif]/(\dif^p)$, where the primitive degree-two generator $\dif\in H$ acts on $A$ by the differential $\dif_A$. 

Then, we may form the \emph{smash product algebra} $A\# H$ in this case. As a $\Bbbk$-vector space, $A\# H$ is isomorphic to $A\otimes H$, subject to the multiplication rule determined by 
\begin{equation}
(a\otimes \dif)(b\otimes \dif)=ab\otimes \dif^2+ a\dif_A(b)\otimes \dif.
\end{equation}
Notice that, by construction, $A\otimes 1$ and $1\otimes H$ sit in $A\# H$ as subalgebras. 

A module over $A\# H$ is also called a \emph{$p$-DG module}. We will consider the usual cochain complex or $p$-complexes of $p$-DG modules. In such a complex, the usual differential or the $p$-differential are required to respect the $H$-actions. Likewise, the null-homotopies in these cases intertwine $A\# H$-actions.

There is an exact forgetful functor between the usual homotopy categories of chain complexes of graded $A\# H$-modules 
\[
\mc{F}_d: \mc{C}(A\# H,d_0)\lra \mc{C}(A,d_0).
\]
An object $K^\bullet$ in $\mc{C}(A\# H,d_0)$ lies inside the kernel of the functor if and only if, when forgetting the $H$-module structure on each term of $K^\bullet$, the complex of graded $A$-modules $\mc{F}_d(K_\bullet)$ is null-homotopic. The null-homotopy map on $\mc{F}_d(K^\bullet)$, though, is not required to intertwine $H$-actions.
 
Likewise, there is an exact forgetful functor 
\[
\mc{F}_\dif: \mc{C}(A\# H,\dif_0)\lra \mc{C}(A,\dif_0).
\]
Similarly, an object $K^\bullet$ in $\mc{C}(A\# H,\dif_0)$ lies inside the kernel of the functor if and only if, when forgetting the $H$-module structure on each term of $K^\bullet$, the $p$-complex of $A$-modules $\mc{F}(K^\bullet)$ is null-homotopic. The null-homotopy map on $\mc{F}(K^\bullet)$, though, is not required to intertwine $H$-actions.

\begin{defn}\label{def-relative-homotopy-category}
Given a $p$-DG algebra $(A,\dif_A)$, the \emph{relative homotopy category} is the Verdier quotient 
$$\mc{C}^\dif(A,d_0):=\dfrac{\mc{C}(A\# H,d_0)}{\mathrm{Ker}(\mc{F}_d)}.$$
Likewise, the \emph{relative $p$-homotopy category} is the Verdier quotient 
$$\mc{C}^\dif(A,\dif_0):=\dfrac{\mc{C}(A\# H,\dif_0)}{\mathrm{Ker}(\mc{F}_\dif)}.$$
\end{defn}
The superscripts in the definitions are to remind the reader of the $H$-module structures on the objects.

The categories $\mc{C}^\dif(A,d_0)$ and  $\mc{C}^\dif(A,\dif_0)$ are triangulated with the triangulated structures inherited from those of $\mc{C}(A\#H,d_0)$ and $\mc{C}(A\# H,\dif_0)$.
For instance, distinguished triangles in the quotient category are declared to be those that are isomorphic to distinguished triangles in $\mc{C}(A\# H, d_0)$ and $\mc{C}(A\# H,\dif_0)$ respectively. Recall that the latter distinguished triangles arise from short exact sequences of ($p$-)cochain complexes over $A\# H$ that are termwise split exact \cite[Lemma 4.3]{QYHopf}.

By construction, there is a factorization of the forgetful functor
\[
\begin{gathered}
\xymatrix{ \mc{C}(A\# H,d_0) \ar[rr]^{\mc{F}_d} \ar[dr] && \mc{C}(A,d_0)\\
& \mc{C}^\dif(A,d_0)\ar[ur]&
} 
\end{gathered}
\ ,
\quad
\begin{gathered}
\xymatrix{ \mc{C}(A\# H,\dif_0) \ar[rr]^{\mc{F}_\dif} \ar[dr] && \mc{C}(A,\dif_0)\\
& \mc{C}^\dif(A,\dif_0)\ar[ur]&
}
\end{gathered} \ .
\]

\begin{prop} \cite[Proposition 2.13]{QiSussan4}  \label{relextot}
The $p$-extension functor $\mc{P}: \mc{C}(A\# H, d_0)\lra \mc{C}(A\# H, \dif_0)$ descends to an exact functor, still denoted $\mc{P}$, between the relative homotopy categories:
$$\mc{P}: \mc{C}^\dif (A, d_0)\lra \mc{C}^\dif(A, \dif_0) \ .$$
\end{prop}

\section{Deformed Webster algebras}
\subsection{The \texorpdfstring{$p$}{p}-DG algebra}
We begin by recalling the definition of a particular deformed Webster algebra $W(n,1)$. 
More general versions of these algebras $W(\mathbf{s},n)$ can be found in \cite{KhovSussan, KLSY}.
The $p$-DG structures on these algebras were introduced in \cite{Y}.
The algebras are deformations of the algebras introduced by Webster in \cite{Webcombined}.

\begin{defn}
Let $n\geq 0$ be an integer.
Let $\mathrm{Seq}(1^n,1)$ be the set of all sequences $\mathbf{i}=(i_1,...,i_{n+1})$ where $n$ of the entries are $1$ and the other entry is the symbol $\mf{b}$.
Therefore, we have $\vert \mathrm{Seq}(1^n,1)\vert=n+1$.
Denote by $\mathbf{i}_j$ the $j$-th entry of $\mathbf{i}$.
Each transposition $\sigma_j \in S_{n+1}$ naturally acts on the set of sequences.

$W=W(n,1)$ is the graded algebra over the ground field $\Bbbk$ generated by $e(\mathbf{i})$, where $\mathbf{i} \in \mathrm{Seq}(1^n,1)$, $x_j$, where $1\leq j\leq n$, $y$ and $\psi_j$, where $1\leq j \leq n$, satisfying the relations below. 

\begin{minipage}{0.3\textwidth}
\begin{align}
& \sum_{\mathbf{i}\in \mathrm{Seq}(1^n,1)}e(\mathbf{i}) = 1
\\
& e(\mathbf{i})e(\mathbf{j}) = \delta_{\mathbf{i},\mathbf{j}}e(\mathbf{i})
\\
& \psi_j e(\mathbf{i}) = e(\sigma_j(\mathbf{i}))\psi_j
\end{align}
\end{minipage}
\begin{minipage}{0.7\textwidth}
\begin{align}
\label{kill-red2}
& \psi_j e(\mathbf{i}) = 0 \quad \text{if $\mathbf{i}_j=\mathbf{i}_{j+1}=1$}
\\
& \psi_j\psi_\ell = \psi_\ell\psi_j \quad \text{if $|j - \ell| > 1$}
\\
\label{central}
& x_j \text{ and } y \text{ are central}\\
\label{r2-like2}
&\psi_j^2 e(\mathbf{i})=
(x_j-y) e(\mathbf{i})\quad
\text{if  $(\mathbf{i}_j,\mathbf{i}_{j+1})=(1,\mathfrak{b}),(\mathfrak{b},1)$}
%\\
%\label{cyclotomic}
%& e(\mathbf{i}) = 0 \text{ if $\mathbf{i}_1=\mathfrak{b}$}.
\end{align}
\end{minipage}

The degrees of the generators are
\[
\deg(e(\mathbf{i}))=0,\quad
\deg(x_j)=2,\quad
\deg(y)=2,\quad
\deg(\psi_j)=1
\ .
\]
\end{defn}

In some contexts, it is natural to impose the so-called cyclotomic relation
\begin{equation} \label{cyclotomic}
e(\mathbf{i}) = 0 \quad\text{if}\quad\mathbf{i}_1=\mathfrak{b}.
\end{equation}
Quotienting $W$ by the cyclotomic relation yields an algebra denoted by $\overline{W}$.

We now recall the diagrammatic description of the deformation of the Webster algebra $W=W(n,1)$.
Consider collections of smooth arcs in the plane connecting $ n $ red points and $1$ black point on one horizontal line with $n$ red points and $1$ black point on another horizontal line.
The red points correspond to the $1$'s in the sequence $\mathbf{i}$ of $e(\mathbf{i})$ and the black point corresponds to $\mathfrak{b}$ in the sequence $\mathbf{i}$.
The $ n $ red points and $1$ black point on the line appear in the order in which they appear in $\mathbf{i}$ of $e(\mathbf{i})$.
The arcs are colored in a manner consistent with their boundary points.
Arcs are assumed to have no critical points (in other words no cups or caps).
Arcs are allowed to intersect (as long as they are both not solid red), but no triple intersections are allowed.
Arcs can carry dots.  
Two diagrams that are related by an isotopy that does not change the combinatorial types of the diagrams or the relative position of crossings are taken to be equal. 
We give $ W(n,1) $ the structure of an algebra by concatenating diagrams vertically as long as the colors of the endpoints match. If they do not, the product of two diagrams is taken to be zero.  For two diagrams $D_1$ and $D_2$, their product $D_1 D_2$ is realized by stacking $D_1$ on top of $D_2$.

The generator $e(\mathbf{i})$ is represented by vertical strands comprised of $n$ red strands and one black strand.
For instance, in the case $\mathbf{i}=(1,1,\mathfrak{b},1)$, the generator $e(\mathbf{i})$ is represented by 
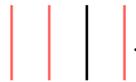
\[
\begin{DGCpicture}
\DGCstrand[red](0,0)(0,1)
\DGCstrand[red](.5,0)(.5,1)
\DGCstrand(1,0)(1,1)
\DGCstrand[red](1.5,0)(1.5,1)
\end{DGCpicture}
\ .
\]

The generator $y$ is represented by a black dot and $x_j $ is represented by a red dot on the $j$th red strand:
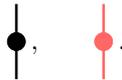
\[
\begin{DGCpicture}
\DGCstrand(0,0)(0,1)
\DGCdot{0.5}
\end{DGCpicture}
\ , \quad \quad
\begin{DGCpicture}
\DGCstrand[red](0,0)(0,1)
\DGCdot{0.5}
\end{DGCpicture}
\ .
\]

The generator $\psi_j e(\mathbf{i})$ is represented by a black-red crossing on the left of the diagram below if $\mathbf{i}_j=1$ and $\mathbf{i}_{j+1}=\mathfrak{b}$ or a black-red crossing on the right if $\mathbf{i}_j=\mathfrak{b}$ and $\mathbf{i}_{j+1}=1$:
\[
\begin{DGCpicture}
\DGCstrand(1,0)(0,1)
\DGCstrand[red](0,0)(1,1)
\end{DGCpicture}
\ , \quad \quad
\begin{DGCpicture}
\DGCstrand(0,0)(1,1)
\DGCstrand[red](1,0)(0,1)
\end{DGCpicture}
\ .
\]
Note that a red-red crossing does not appear due to relation \eqref{kill-red2}.

Far away generators commute.  
The relations \eqref{central} and \eqref{r2-like2} involving red-black strands are
\begin{gather} 
\label{Wrelations}
\begin{DGCpicture}
\DGCstrand(0,0)(1,1)
\DGCdot{0.25}
\DGCstrand[red](1,0)(0,1)
\end{DGCpicture}
~=~
\begin{DGCpicture}
\DGCstrand(0,0)(1,1)
\DGCdot{0.75}
\DGCstrand[red](1,0)(0,1)
\end{DGCpicture}
\ , \quad \quad
\begin{DGCpicture}
\DGCstrand(1,0)(0,1)
\DGCdot{0.25}
\DGCstrand[red](0,0)(1,1)
\end{DGCpicture}
~=~
\begin{DGCpicture}
\DGCstrand(1,0)(0,1)
\DGCdot{0.75}
\DGCstrand[red](0,0)(1,1)
\end{DGCpicture}
, \\
\begin{DGCpicture}
\DGCstrand[red](0,0)(1,1)
\DGCdot{0.25}
\DGCstrand(1,0)(0,1)
\end{DGCpicture}
~=~
\begin{DGCpicture}
\DGCstrand[red](0,0)(1,1)
\DGCdot{0.75}
\DGCstrand(1,0)(0,1)
\end{DGCpicture}
\ , \quad \quad
\begin{DGCpicture}
\DGCstrand[red](1,0)(0,1)
\DGCdot{0.25}
\DGCstrand(0,0)(1,1)
\end{DGCpicture}
~=~
\begin{DGCpicture}
\DGCstrand[red](1,0)(0,1)
\DGCdot{0.75}
\DGCstrand(0,0)(1,1)
\end{DGCpicture}
, \\
\begin{DGCpicture}[scale=0.55] \label{R2rel}
\DGCstrand(1,0)(0,1)(1,2)
\DGCstrand[red](0,0)(1,1)(0,2)
\end{DGCpicture}
~=~
\begin{DGCpicture}[scale=0.55]
\DGCstrand(1,0)(1,2)
\DGCdot{1}[r]{}
\DGCstrand[red](0,0)(0,2)
\end{DGCpicture}
~-~
\begin{DGCpicture}[scale=0.55]
\DGCstrand(1,0)(1,2)
\DGCstrand[red](0,0)(0,2)
\DGCdot{1}[r]{}
\end{DGCpicture}
\ , \quad \quad \quad
\begin{DGCpicture}[scale=0.55]
\DGCstrand(0,0)(1,1)(0,2)
\DGCstrand[red](1,0)(0,1)(1,2)
\end{DGCpicture}
~=~
\begin{DGCpicture}[scale=0.55]
\DGCstrand(0,0)(0,2)
\DGCdot{1}[l]{}
\DGCstrand[red](1,0)(1,2)
\end{DGCpicture}
~-~
\begin{DGCpicture}[scale=0.55]
\DGCstrand(0,0)(0,2)
\DGCstrand[red](1,0)(1,2)
\DGCdot{1}[l]{}
\end{DGCpicture}
\ .
\end{gather}
The cyclotomic relation \eqref{cyclotomic} translates to: a black strand, appearing on the far left of any diagram, annihilates the entire picture:
\begin{equation}\label{eqn-cyclotomic}
\begin{DGCpicture}
\DGCstrand(1,0)(1,1)
\DGCcoupon*(1.25,0.25)(1.75,0.75){$\cdots$}
\end{DGCpicture}
~=~0.
\end{equation}
%When we omit the cyclotomic relation ~\eqref{cyclotomic} we denote the corresponding algebra by $\widetilde{W}=\widetilde{W}(n,1)$.  
Note that setting red dots to be zero, we recover Webster's algebra, and so we may think of the polynomial algebra generated by red dots as a polynomial deformation space of the Webster algebra.

For $i=0,\ldots,n$, there is a sequence $(1^i,\mathfrak{b},1^{n-i})$.
Denote by $e_i$ the idempotent $e(1^i,\mathfrak{b},1^{n-i})$;
\begin{equation}\label{idempotents}
e_i=e(1^i,\mathfrak{b},1^{n-i})=
\begin{DGCpicture}
\DGCstrand[red](0,0)(0,1)[$^{1}$`{\ }]
\DGCcoupon*(.25,0.25)(.75,0.75){$\cdots$}
\DGCstrand[red](1,0)(1,1)[$^{i}$`{\ }]
\DGCstrand(1.5,0)(1.5,1)
\DGCstrand[red](2,0)(2,1)[$^{i+1}$`{\ }]
\DGCcoupon*(2.25,0.25)(2.75,0.75){$\cdots$}
\DGCstrand[red](3,0)(3,1)[$^{n}$`{\ }]
\end{DGCpicture}
\ .
\end{equation}

Over a base field $\Bbbk$ of finite characteristic $p>0$, a $p$-differential graded ($p$-DG) algebra structure was introduced on a generalization of $W$ in \cite{Y}.   

\begin{defn}
The $p$-derivation $\partial:W\to W$ of degree $2$ satisfying the Leibniz rule:
$$
\partial(ab) = \partial(a)b + a\partial(b)
$$
for any $a, b \in W$ is
defined on the generators of the algebra $W$ by 
\[
\dif(e_i)=0, \quad \dif(x_i)=x_i^2,\quad \dif(y)=y^2,\quad \dif(\psi_j)=x_j \psi_j e_{j-1}+y \psi_j e_j
\]
and extended by the Leibniz rule to the entire algebra. 
\end{defn}
An easy exercise shows that $\dif^p \equiv 0$.

In the diagrammatic description, we have
\begin{subequations}
\begin{gather}
\dif \left(~
\begin{DGCpicture}
\DGCstrand[red](0,0)(0,1)
\DGCdot{0.5}
\end{DGCpicture} 
~\right)=
\begin{DGCpicture}
\DGCstrand[red](0,0)(0,1)
\DGCdot{0.5}[ur]{$_2$}
\end{DGCpicture} 
\ , \quad \quad \quad
\dif\left(~
\begin{DGCpicture}
\DGCstrand(0,0)(0,1)
\DGCdot{0.5}
\end{DGCpicture}
~\right)=
\begin{DGCpicture}
\DGCstrand(0,0)(0,1)
\DGCdot{0.5}[ur]{$_2$}
\end{DGCpicture}
\ ,
\end{gather}
\begin{gather}
\label{diffonWeb}
\dif_{}\left(~
\begin{DGCpicture}
\DGCstrand(0,0)(1,1)
\DGCstrand[red](1,0)(0,1)
\end{DGCpicture}
~\right)=
\begin{DGCpicture}
\DGCstrand[red](1,0)(0,1)
\DGCdot{0.75}
\DGCstrand(0,0)(1,1)
\end{DGCpicture}, 
\quad \quad
\dif_{}\left(~
\begin{DGCpicture}
\DGCstrand(1,0)(0,1)
\DGCstrand[red](0,0)(1,1)
\end{DGCpicture}
~\right)=
\begin{DGCpicture}
\DGCstrand(1,0)(0,1)
\DGCdot{0.75}
\DGCstrand[red](0,0)(1,1)
\end{DGCpicture}
\ .
\end{gather}
\end{subequations}

\subsection{A basis}
A basis for the cyclotomic deformed Webster $\overline{W}(n,1)$ was given in \cite{KhovSussan} (see also \cite{Webdeformed} and \cite{SWSchur}).  We slightly modify this basis and a representation of the algebra for the case where the cyclotomic condition is omitted.

\begin{prop}
\label{repprop}
Let 
\[
R_n=\Bbbk[x_1,\ldots,x_n], \quad \quad V_{n,i}=R_n[y_i], \quad \quad
V_n=\bigoplus_{i=0}^n V_{n,i}.
\]

There is an action of $W(n,1)$ on $V_n$ determined by
\begin{eqnarray*}
e_i&:&f(\mathbf{x},y_j)\in V_{n,j}\mapsto\left\{
\begin{array}{ll}
f(\mathbf{x},y_i)\in V_{n,i}&\text{if $j=i$}\\
0&\text{if $j\not=i$}
\end{array}\right.
\\
x_k^{a_k} e_i&:&f(\mathbf{x},y_i)\in V_{n,i}\mapsto 
x_k^{a_k}f(\mathbf{x},y_i)\in V_{n,i},
\\
y^{a_k} e_i&:&f(\mathbf{x},y_i)\in V_{n,i}\mapsto 
y_i^{a_k}f(\mathbf{x},y_i)\in V_{n,i},
\\
\psi_i e_i&:&f(\mathbf{x},y_i)\in V_{n,i}\mapsto f(\mathbf{x},y_{i-1})\in V_{n,i-1}
\\
\psi_{i+1} e_i&:&f(\mathbf{x},y_i)\in V_{n,i}\mapsto (y_{i+1}-x_{i+1})f(\mathbf{x},y_{i+1})\in V_{n,i+1}
\ .
\end{eqnarray*}

In the diagrammatic description, we have
\begin{eqnarray*}
\begin{DGCpicture}
\DGCstrand[red](0,0)(0,1)[$^{1}$`{\ }]
\DGCdot{0.5}
\DGCcoupon*(-.5,0.25)(0,0.75){$^{a_1}$}
\DGCcoupon*(.25,0.25)(.75,0.75){$\cdots$}
\DGCstrand[red](1,0)(1,1)[$^{i}$`{\ }]
\DGCdot{0.5}
\DGCcoupon*(1,0.25)(1.5,0.75){$^{a_i}$}
\DGCstrand(2,0)(2,1)
\DGCdot{0.5}
\DGCcoupon*(1.5,0.5)(2,0.75){$^{b}$}
\DGCstrand[red](3,0)(3,1)[$^{i+1}$`{\ }]
\DGCdot{0.5}
\DGCcoupon*(2.2,0.25)(2.9,0.75){$^{a_{i+1}}$}
\DGCcoupon*(3.25,0.25)(3.75,0.75){$\cdots$}
\DGCstrand[red](4,0)(4,1)[$^{n}$`{\ }]
\DGCdot{0.5}
\DGCcoupon*(4.1,0.25)(4.5,0.75){$^{a_n}$}
\end{DGCpicture}
&\colon&
f(\mathbf{x},y_i) \in V_{n,i} 
~\mapsto
x_1^{a_1} \cdots x_n^{a_n} y_i^b f(\mathbf{x},y_i)\in V_{n,i} 
\\
\begin{DGCpicture}
\DGCstrand[red](0,0)(0,1)[$^{1}$`{\ }]
\DGCcoupon*(.25,0.25)(.75,0.75){$\cdots$}
\DGCstrand[red](1,0)(1,1)[$^{i-1}$`{\ }]
\DGCstrand[red](1.5,0)(2.5,1)[$^{i}$`{\ }]
\DGCstrand(2.5,0)(1.5,1)
\DGCstrand[red](3,0)(3,1)[$^{i+1}$`{\ }]
\DGCcoupon*(3.25,0.25)(3.75,0.75){$\cdots$}
\DGCstrand[red](4,0)(4,1)[$^{n}$`{\ }]
\end{DGCpicture}
&\colon&
f(\mathbf{x},y_i) \in V_{n,i} 
~\mapsto
f(\mathbf{x},y_{i-1}) \in V_{n,i-1} 
\\
\begin{DGCpicture}
\DGCstrand[red](0,0)(0,1)[$^{1}$`{\ }]
\DGCcoupon*(.25,0.25)(.75,0.75){$\cdots$}
\DGCstrand[red](1,0)(1,1)[$^{i}$`{\ }]
\DGCstrand(1.5,0)(2.5,1)
\DGCstrand[red](2.5,0)(1.5,1)[$^{i+1}$`{\ }]
\DGCstrand[red](3,0)(3,1)[$^{i+2}$`{\ }]
\DGCcoupon*(3.25,0.25)(3.75,0.75){$\cdots$}
\DGCstrand[red](4,0)(4,1)[$^{n}$`{\ }]
\end{DGCpicture}
&\colon&
f(\mathbf{x},y_i) \in V_{n,i} 
~\mapsto
(y_{i+1}-x_{i+1})f(\mathbf{x},y_{i+1}) \in V_{n,i+1} .
\end{eqnarray*}
\end{prop}

\begin{proof}
This is a straightforward check.
\end{proof}

Next we will define elements which form a basis of $W(n,1)$.  

Let ${\bf a}=(a_1,\ldots,a_n)\in \mathbb{Z}_{\geq 0}^n$ and $b\in \mathbb{Z}_{\geq 0}$.
For $0\leq i\leq j\leq n$, define

\begin{eqnarray*}
NE_{i,j}({\bf a},b)&:=&\prod_{k=1}^n x_k^{a_k}y^{b}\psi_{j}\psi_{j-1}\cdots \psi_{i+1}e_i
\\
&=&
\left\{
\begin{array}{ll}
\begin{DGCpicture}
\DGCstrand[red](-2,0)(-2,1)[$^{1}$`{\ }]
\DGCdot{0.75}
\DGCcoupon*(-2.5,.5)(-2,1){$^{a_1}$}
\DGCcoupon*(-1.75,0.25)(-1.25,0.75){$\cdots$}
\DGCstrand[red](-1,0)(-1,1)[$^{i}$`{\ }]
\DGCdot{0.75}
\DGCcoupon*(-1,.5)(-.5,1){$^{a_{i}}$}
\DGCstrand(0,0)(0,1)
\DGCdot{0.75}
\DGCcoupon*(-.5,.5)(0,1){$^b$}
\DGCstrand[red](1,0)(1,1)[$^{i+1}$`{\ }]
\DGCdot{0.75}
\DGCcoupon*(1,.5)(2,1){$^{a_{i+1}}$}
\DGCcoupon*(1.25,0.25)(1.75,.75){$\cdots$}
\DGCstrand[red](2,0)(2,1)[$^{n}$`{\ }]
\DGCdot{0.75}
\DGCcoupon*(2,.55)(2.5,1.05){$^{a_n}$}
\end{DGCpicture}
&\text{if $i=j$,}
\\
\begin{DGCpicture}
\DGCstrand[red](-2,0)(-2,1)[$^{1}$`{\ }]
\DGCdot{0.75}
\DGCcoupon*(-2.5,.5)(-2,1){$^{a_1}$}
\DGCcoupon*(-1.75,0.25)(-1.25,0.75){$\cdots$}
\DGCstrand[red](-1,0)(-1,1)[$^{i}$`{\ }]
\DGCdot{0.75}
\DGCcoupon*(-1,.5)(-.5,1){$^{a_{i}}$}
\DGCstrand(0,0)(3,1)
\DGCdot{0.75}
\DGCcoupon*(2.5,.75)(3,1.25){$^b$}
\DGCstrand[red](1,0)(1,1)[$^{i+1}$`{\ }]
\DGCdot{0.75}
\DGCcoupon*(.1,.5)(.75,1){$^{a_{i+1}}$}
\DGCcoupon*(1.25,0)(1.75,.5){$\cdots$}
\DGCstrand[red](2,0)(2,1)[$^{j}$`{\ }]
\DGCdot{0.75}
\DGCcoupon*(2,.55)(2.5,1.05){$^{a_j}$}
\DGCstrand[red](4,0)(4,1)[$^{j+1}$`{\ }]
\DGCdot{0.75}
\DGCcoupon*(3.15,.5)(3.85,1){$^{a_{j+1}}$}
\DGCcoupon*(4.25,0.25)(4.75,0.75){$\cdots$}
\DGCstrand[red](5,0)(5,1)[$^{n}$`{\ }]
\DGCdot{0.75}
\DGCcoupon*(5.15,.5)(5.55,1){$^{a_n}$}
\end{DGCpicture}
&\text{if $i<j$ .}
\end{array}
\right.
\end{eqnarray*}

For $0\leq i< j\leq n$, define
\begin{eqnarray*}
SE_{i,j}({\bf a},b)&:=&\prod_{k=1}^n x_k^{a_k}y^{b}\psi_{i+1}\psi_{i}\cdots \psi_{j}e_i
\\
&=&
\begin{DGCpicture}
\DGCstrand[red](-2,0)(-2,1)[$^{1}$`{\ }]
\DGCdot{0.75}
\DGCcoupon*(-2.5,.5)(-2,1){$^{a_1}$}
\DGCcoupon*(-1.75,0.25)(-1.25,0.75){$\cdots$}
\DGCstrand[red](-1,0)(-1,1)[$^{i}$`{\ }]
\DGCdot{0.75}
\DGCcoupon*(-1,.5)(-.5,1){$^{a_{i}}$}
\DGCstrand(3,0)(0,1)
\DGCdot{0.25}
\DGCcoupon*(2.5,.25)(3,.75){$^b$}
\DGCstrand[red](1,0)(1,1)[$^{i+1}$`{\ }]
\DGCdot{0.75}
\DGCcoupon*(.2,.7)(.85,1.25){$^{a_{i+1}}$}
\DGCcoupon*(1.25,0)(1.75,.5){$\cdots$}
\DGCstrand[red](2,0)(2,1)[$^{j}$`{\ }]
\DGCdot{0.75}
\DGCcoupon*(2,.55)(2.5,1.05){$^{a_j}$}
\DGCstrand[red](4,0)(4,1)[$^{j+1}$`{\ }]
\DGCdot{0.75}
\DGCcoupon*(3.15,.5)(3.85,1){$^{a_{j+1}}$}
\DGCcoupon*(4.25,0.25)(4.75,0.75){$\cdots$}
\DGCstrand[red](5,0)(5,1)[$^{n}$`{\ }]
\DGCdot{0.75}
\DGCcoupon*(5.15,.5)(5.55,1){$^{a_n}$}
\end{DGCpicture}
\ .
\end{eqnarray*}

The proof of the next proposition is similar to the proof of \cite[Proposition 2, Corollary 1]{KhovSussan}.  See also \cite[Proposition 4.9]{SWSchur} and \cite[Definition 2.7]{Webdeformed}.
\begin{prop}
We have the following facts about $W(n,1)$ and its representation $V_n$.
\begin{enumerate}
\item The action of $W(n,1)$ on $V_n$ is faithful.
\item $W(n,1)$ has a basis
\begin{equation*}
\{NE_{i,j}({\bf a},b), SE_{i',j'}({\bf a'},b') | 0 \leq i \leq j \leq n,
0 \leq i' < j' \leq n, {\bf a},{\bf a'}\in \mathbb{Z}_{\geq 0}^n, b, b' \in \mathbb{Z}_{\geq 0}
\} \ .
\end{equation*}
\end{enumerate}
\end{prop}

\section{Braid invariant}

\subsection{Bimodules}
We recall the $(W,W)$-bimodules $W_i$ for $i=1,\ldots,n-1$, and $W_{i,i+1}$ for $i=1,\ldots,n-2$ introduced in \cite{KhovSussan}.
These bimodules were generalized in \cite{KLSY}.
While all diagrams are supposed to be braid-like, we will occasionally, for visual reasons, have some pictures with cups and caps.  These pictures could easily be converted to braid-like diagrams.

In order to define the bimodules $W_i$ and $W_{i,i+1}$, we introduce the algebra $W((1^n)_{i,k},1)$ which is also a particular deformed Webster algebra and has been shown in \cite[Section 3.4]{KhovSussan} to be isomorphic to a subalgebra of $W$.

Let $(1^n)_{i,k}$ be the sequence whose $n-k$ entries are $1$ and $i$-th entry from left is $k$,
\[
(1^n)_{i,k}=(1^{i-1},k,1^{n-i-k+1}),
\]
and let $\mathrm{Seq}((1^n)_{i,k},1)$ be the set of all sequences $\mathbf{i}=(i_1,...,i_{n-k+2})$ composed of the sequence $(1^n)_{i,k}$ into which the symbol $\mf{b}$ is inserted.

For instance, the set $\mathrm{Seq}((1^4)_{2,2},1)$ is
\[
\{
(\mf{b},1,2,1),
(1,\mf{b},2,1),
(1,2,\mf{b},1),
(1,2,1,\mf{b})
\}.
\]
\begin{defn}
$W((1^n)_{i,k},1)$ is the graded algebra over the ground field $\Bbbk$ generated by 
\begin{itemize}
\item $e(\mathbf{i})$, where $\mathbf{i} \in \mathrm{Seq}((1^n)_{i,k},1)$,
\item $x_j$, where $1\leq j< i$ or $i < j\leq n-k+1$,
\item $y$,
\item $\psi_j$, where $1\leq j \leq n-k+1$,
\item $E(d)$, where $1\leq d \leq k$
\end{itemize}
satisfying the relations below.\\
\begin{minipage}{0.3\textwidth}
\begin{align}
& \sum_{\mathbf{i}\in \mathrm{Seq}((1^n)_{i,k},1)}e(\mathbf{i}) = 1
\\
& e(\mathbf{i})e(\mathbf{j}) = \delta_{\mathbf{i},\mathbf{j}}e(\mathbf{i})
\\
& \psi_j e(\mathbf{i}) = e(\sigma_j(\mathbf{i}))\psi_j
\end{align}
\end{minipage}
\begin{minipage}{0.7\textwidth}
\begin{align}
\label{kill-red}
& \psi_j e(\mathbf{i}) = 0 \quad \text{if $(\mathbf{i}_j,\mathbf{i}_{j+1})=(1,1),(1,k),(k,1)$}
\\
& \psi_j\psi_\ell = \psi_\ell\psi_j \quad \text{if $|j - \ell| > 1$}
\\
& x_j, y \text{ and } E(d) \text{ are central}
%\\
%\label{cyclotomic-2}
%& e(\mathbf{i}) = 0 \quad\text{ if $\mathbf{i}_1=\mathbf{b}$}
\end{align}
\end{minipage}
\begin{minipage}{\textwidth}
\begin{align}
\label{r2-like}
&\psi_j^2 e(\mathbf{i})=\left\{
	\begin{array}{ll}
		(y-x_j) e(\mathbf{i})
		&\text{if  $(\mathbf{i}_j,\mathbf{i}_{j+1})=(1,\mathfrak{b}),(\mathfrak{b},1)$},
		\\
		\displaystyle\sum_{a=0}^k(-1)^a E(a)y^{k-a} e(\mathbf{i})
		&\text{if  $(\mathbf{i}_j,\mathbf{i}_{j+1})=(k,\mathfrak{b}),(\mathfrak{b},k)$}.
	\end{array}\right.
\end{align}
\end{minipage}
The degrees of the generators are
\[
\deg(e(\mathbf{i}))=0,\quad
\deg(x_j)=2,\quad
\deg(y)=2,\quad
\deg(E(d))=2d,\quad
\deg(\psi_j e(\mathbf{i}))=
a \quad \text{if} \quad (\mathbf{i}_j,\mathbf{i}_{j+1})=(\mathfrak{b},a),(a,\mathfrak{b}),
\]
where $a$ is $1$ or $k$.
\end{defn}

We have a translation from $(1^n)_{i,k}$ to $(1^n)$ which is obtained by replacing $k$ of the sequence $(1^n)_{i,k}$ by $(\underbrace{1, \ldots, 1}_k)$. This translation naturally induces the map $\phi:\mathrm{Seq}((1^n)_{i,k},1)\to \mathrm{Seq}((1^n),1)$.

We define the inclusion map $\Phi$ from $W((1^n)_{i,k},1)$ to $W((1^n),1)$ by mapping idempotents $e(\mathbf{i})$ for $\mathbf{i}\in\mathrm{Seq}((1^n)_{i,k},n)$ by
\begin{equation*}
\Phi(e(\mathbf{i})) = e(\phi(\mathbf{i})),
\end{equation*}
mapping generators $x_j$ by
\begin{equation*}
\Phi(x_j ) =
\begin{cases}
x_j & \text{ if } 1\leq j < i , \\
x_{j+k-1} & \text{ if } i < j \leq n-k+1,
\end{cases}
\end{equation*}
mapping generators $\psi_{j}$ by 
\begin{equation*}
\Phi(\psi_j )=\begin{cases}
\psi_j(e_{j-1}+e_j)&j<i,\\
\psi_j\psi_{j+1}\cdots \psi_{j+k-1}e_{j+k-1}+\psi_{j+k-1}\psi_{j+k-2}\cdots \psi_{j}e_{j-1}&j=i,\\
\psi_{j+k-1}(e_{j+k-2}+e_{j+k-1})&j>i
\end{cases}
\end{equation*}
(where we recall the definition of $e_i$ is given in \eqref{idempotents}), and mapping generators $E(d)$ by
\begin{equation*}
\Phi(E(d)) =
\sum_{\substack{d_1+d_2+\cdots +d_k=d\\ d_1,...,d_k\geq 0}}x_{i}^{d_1}x_{i+1}^{d_2}\cdots x_{i+k-1}^{d_k}.
\end{equation*}
This inclusion map gives rise to a left action of $W((1^n)_{i,2},1)$ on $\hat{e}_{i}W((1^n),1)$ and a right action on $W((1^n),1)\hat{e}_{i}$, where $\hat{e}_{i}$ is $\sum_{j\not=i}e_j$. 
\begin{defn}
We define the bimodule $W_i$ as the tensor product of two deformed Webster algebras $W((1^n),1)$ over $W((1^n)_{i,2},1)$.
\[
W_i:=W((1^n),1)\hat{e}_{i}\otimes_{W((1^n)_{i,2},1)} \hat{e}_{i}W((1^n),1).
\]
\end{defn}

The following proposition follows directly from definitions.
\begin{prop}\label{bimod-rel}
In the bimodule $W_i$, we have the following equalities.
\begin{align}
&e_k\otimes e_\ell=0\quad\text{if } k\not=\ell,\quad
e_j \otimes e_j= e_j \otimes 1= 1 \otimes e_j,\quad y\otimes 1=1\otimes y,\\
\label{dot-slide-splitter}
&x_\ell\otimes e_j=1\otimes x_\ell e_j \quad\text{if } \ell\not=i,i+1,\quad
(x_i+x_{i+1})\otimes e_j=1\otimes (x_i +x_{i+1})e_j,\quad
x_ix_{i+1}\otimes e_j=1\otimes x_i x_{i+1}e_j,\\
\label{crossing-slide-splitter}
&\psi_{\ell} \otimes e_j=1\otimes \psi_{\ell}e_j \quad\text{if }\ell\not=i,i+1,\quad
\psi_{i+1}\psi_{i} \otimes e_j=1\otimes \psi_{i+1}\psi_{i}e_j, \quad
\psi_{i}\psi_{i+1}\otimes e_j=1\otimes \psi_{i}\psi_{i+1}e_j,
\end{align}
where $j\not=i$.
\end{prop}

A graphical description of the bimodules $W_i$ for $i=1,\ldots,n-1$ was introduced in \cite{KhovSussan}.
We consider collections of smooth arcs in the plane connecting $ n $ red points and $1$ black point on one horizontal line with $n$ red points and $1$ black point on another horizontal line.
The $i$th and $(i+1)$st red dots on one horizontal line must be connected to the
 $i$th and $(i+1)$st red dots on the other horizontal line by a diagram which has a thick red strand in the middle, which is given in \eqref{thickgenerator}.
The arcs are colored in a manner consistent with their boundary points and are assumed to have no critical points (in other words no cups or caps).
They are allowed to intersect, but no triple intersections are allowed.
Arcs are allowed to carry dots.  
Two diagrams that are related by an isotopy that does not change the combinatorial types of the diagrams or the relative position of crossings are taken to be equal.
The elements of the vector space $ W_i $ are formal linear combinations of these diagrams modulo the local relations for $W$ along with the relations given in
\eqref{symmetricrelations} and \eqref{blackthickred1} which  correspond to relations in Proposition \ref{bimod-rel}.

The elements $e_j\otimes e_j$, where $0\leq j<i$ or $i<j\leq n$, are represented by diagrams
\begin{equation}
\label{thickgenerator}
e_j\otimes e_j=\begin{DGCpicture}
\DGCstrand[red](-4,0)(-4,1.5)[$^{1}$`{\ }]
\DGCcoupon*(-3,0)(-4,1.5){$\cdots$}
\DGCstrand[red](-3,0)(-3,1.5)[$^{j}$`{\ }]
\DGCstrand(-2.5,0)(-2.5,1.5)
\DGCstrand[red](-2,0)(-2,1.5)[$^{j+1}$`{\ }]
\DGCcoupon*(-1,0)(-2,1.5){$\cdots$}
\DGCstrand[red](-1,0)(-1,1.5)[$^{i-1}$`{\ }]
\DGCstrand[red](-.5,0)(0,.5)[$^{i}$`{\ }]
\DGCstrand[red](0.5,0)(0,.5)[$^{i+1\,\,}$`{\ }]
\DGCstrand[Red](0,.5)(0,1)
\DGCstrand[red](0,1)(-0.5,1.5)
\DGCstrand[red](0,1)(0.5,1.5)
\DGCstrand[red](1,0)(1,1.5)[$^{\,\,i+2}$`{\ }]
\DGCcoupon*(1,0)(2,1.5){$\cdots$}
\DGCstrand[red](2,0)(2,1.5)[$^{n}$`{\ }]
\end{DGCpicture}
\in W_i
\ .
\end{equation}
%where $j+1$ is the position of the black strand from left.
The second and third equations of \eqref{dot-slide-splitter} in terms of diagrams are
\begin{equation}
\label{symmetricrelations}
\begin{DGCpicture}
\DGCstrand[red](0,0)(.5,.5)[$^{i}$`{\ }]
\DGCstrand[red](1,0)(.5,.5)[$^{i+1}$`{\ }]
\DGCstrand[Red](.5,.5)(.5,1)
\DGCstrand[red](.5,1)(0,1.5)
\DGCdot{E}
\DGCstrand[red](.5,1)(1,1.5)
\end{DGCpicture}
~+~
\begin{DGCpicture}
\DGCstrand[red](0,0)(.5,.5)[$^{i}$`{\ }]
\DGCstrand[red](1,0)(.5,.5)[$^{i+1}$`{\ }]
\DGCstrand[Red](.5,.5)(.5,1)
\DGCstrand[red](.5,1)(0,1.5)
\DGCstrand[red](.5,1)(1,1.5)
\DGCdot{E}
\end{DGCpicture}
~=~
\begin{DGCpicture}
\DGCstrand[red](0,0)(.5,.5)[$^{i}$`{\ }]
\DGCdot{B}
\DGCstrand[red](1,0)(.5,.5)[$^{i+1}$`{\ }]
\DGCstrand[Red](.5,.5)(.5,1)
\DGCstrand[red](.5,1)(0,1.5)
\DGCstrand[red](.5,1)(1,1.5)
\end{DGCpicture}
~+~
\begin{DGCpicture}
\DGCstrand[red](0,0)(.5,.5)[$^{i}$`{\ }]
\DGCstrand[red](1,0)(.5,.5)[$^{i+1}$`{\ }]
\DGCdot{B}
\DGCstrand[Red](.5,.5)(.5,1)
\DGCstrand[red](.5,1)(0,1.5)
\DGCstrand[red](.5,1)(1,1.5)
\end{DGCpicture}
\ , \quad \quad \quad
\begin{DGCpicture}
\DGCstrand[red](0,0)(.5,.5)[$^{i}$`{\ }]
\DGCstrand[red](1,0)(.5,.5)[$^{i+1}$`{\ }]
\DGCstrand[Red](.5,.5)(.5,1)
\DGCstrand[red](.5,1)(0,1.5)
\DGCdot{E}
\DGCstrand[red](.5,1)(1,1.5)
\DGCdot{E}
\end{DGCpicture}
~=~
\begin{DGCpicture}
\DGCstrand[red](0,0)(.5,.5)[$^{i}$`{\ }]
\DGCdot{B}
\DGCstrand[red](1,0)(.5,.5)[$^{i+1}$`{\ }]
\DGCdot{B}
\DGCstrand[Red](.5,.5)(.5,1)
\DGCstrand[red](.5,1)(0,1.5)
\DGCstrand[red](.5,1)(1,1.5)
\end{DGCpicture}
\ .
\end{equation}
The second and third equations of \eqref{crossing-slide-splitter} in terms of diagrams are
\begin{equation}
\label{blackthickred1}
\begin{DGCpicture}
\DGCstrand(0,0)(0,1.25)(2,1.75)
\DGCstrand[red](1,0)(1.5,0.5)[$^{i}$`{\ }]
\DGCstrand[red](2,0)(1.5,0.5)[$^{i+1}$`{\ }]
\DGCstrand[Red](1.5,0.5)(1.5,0.75)
\DGCstrand[red](1.5,0.75)(2,1.25)(1,1.75)
\DGCstrand[red](1.5,0.75)(1,1.25)
\DGCstrand[red](1,1.25)(0,1.75)
\end{DGCpicture}
~=~
\begin{DGCpicture}
\DGCstrand(0,0)(2,0.5)(2,1.75)
\DGCstrand[red](1,0)(0,0.5)(0.5,1)[$^{i}$`{\ }]
\DGCstrand[red](2,0)(1,0.5)[$^{i+1}$`{\ }]
\DGCstrand[red](1,0.5)(0.5,1)
\DGCstrand[Red](0.5,1)(0.5,1.25)
\DGCstrand[red](0.5,1.25)(0,1.75)
\DGCstrand[red](0.5,1.25)(1,1.75)
\end{DGCpicture}
\ , \quad \quad \quad
\begin{DGCpicture}
\DGCstrand(3,0)(3,1.25)(1,1.75)
\DGCstrand[red](1,0)(1.5,0.5)[$^{i}$`{\ }]
\DGCstrand[red](2,0)(1.5,0.5)[$^{i+1}$`{\ }]
\DGCstrand[Red](1.5,0.5)(1.5,0.75)
\DGCstrand[red](1.5,0.75)(2,1.25)
\DGCstrand[red](2,1.25)(3,1.75)
\DGCstrand[red](1.5,0.75)(1,1.25)(2,1.75)
\end{DGCpicture}
~=~
\begin{DGCpicture}
\DGCstrand(1,0)(-1,0.5)(-1,1.75)
\DGCstrand[red](-1,0)(0,0.5)[$^{i}$`{\ }]
\DGCstrand[red](0,0.5)(0.5,1)
\DGCstrand[red](0,0)(1,0.5)(0.5,1)[$^{i+1}$`{\ }]
\DGCstrand[Red](0.5,1)(0.5,1.25)
\DGCstrand[red](0.5,1.25)(0,1.75)
\DGCstrand[red](0.5,1.25)(1,1.75)
\end{DGCpicture}
\ .
\end{equation}

The bimodule $W_i$ naturally inherits a $p$-DG structure from $W$ as follows: 
\begin{equation*}
\dif_{W_i}:=\dif_W\otimes \Id + \Id \otimes \dif_W
\ .
\end{equation*}

One may twist the $p$-DG structure on $W_i$ to obtain a new $p$-DG bimodule $W_i^{-e_1}$.  As bimodules, $W_i = W_i^{-e_1}$ but the $p$-DG structure on the generator is twisted as follows:
\begin{equation}
\dif(1\otimes e_j)=
-(x_i+x_{i+1})\otimes e_j
\end{equation}
In terms of a diagrammatic description, this twisted differential is
\begin{equation}
\label{defofWitwisted}
\dif_{}\left(~
\begin{DGCpicture}
\DGCstrand[red](0,0)(.5,.5)
\DGCstrand[red](1,0)(.5,.5)
\DGCstrand[Red](.5,.5)(.5,1)
\DGCstrand[red](.5,1)(0,1.5)
\DGCstrand[red](.5,1)(1,1.5)
\end{DGCpicture}
~\right)
~=~
~-~
\begin{DGCpicture}
\DGCstrand[red](0,0)(.5,.5)
\DGCstrand[red](1,0)(.5,.5)
\DGCstrand[Red](.5,.5)(.5,1)
\DGCstrand[red](.5,1)(0,1.5)
\DGCdot{E}
\DGCstrand[red](.5,1)(1,1.5)
\end{DGCpicture}
~-~
\begin{DGCpicture}
\DGCstrand[red](0,0)(.5,.5)
\DGCstrand[red](1,0)(.5,.5)
\DGCstrand[Red](.5,.5)(.5,1)
\DGCstrand[red](.5,1)(0,1.5)
\DGCstrand[red](.5,1)(1,1.5)
\DGCdot{E}
\end{DGCpicture}
\ .
\end{equation}
We use the notation $W_i^{-e_1}$ because the differential is twisted by the first elementary symmetric function in the dots corresponding to the bimodule generator connected to the thick red strand.

We will also need the bimodules $W_{i,i+1}$ for $i=1,\ldots,n-2$.
\begin{defn}
We define the bimodule $W_{i,i+1}$ as the tensor product of two deformed Webster algebras $W((1^n),1)$ over $W((1^n)_{i,3},1)$.
\[
W_{i,i+1}:=W((1^n),1)\hat{e}_{i,i+1}\otimes_{W((1^n)_{i,3},1)} \hat{e}_{i,i+1}W((1^n),1),
\]
where $\displaystyle\hat{e}_{i,i+1}=\sum_{j\not=i,i+1}e_j$.
\end{defn}

The following proposition follows directly from definitions.
\begin{prop}\label{bimod-rel3}
In the bimodule $W_{i,i+1}$, we have the following equalities.
\begin{align}
&e_k\otimes e_\ell=0\quad\text{if } k\not=\ell,\quad
e_j \otimes e_j= e_j \otimes 1= 1 \otimes e_j,
\quad x_\ell\otimes e_j=1\otimes x_\ell e_j \quad\text{if } \ell\not=i,i+1,i+2,\\
&\psi_{\ell} \otimes e_j=1\otimes \psi_{\ell}e_j \quad\text{if }\ell\not=i,i+1,i+2,\quad y\otimes 1=1\otimes y,\\
\label{dot-slide-splitter3}
&
\sum_{\substack{a+b+c=d\\a,b,c\geq 0}}
x_i^a x_{i+1}^b x_{i+2}^c\otimes e_j=1\otimes\sum_{\substack{a+b+c=d\\a,b,c\geq 0}} x_i^a x_{i+1}^b x_{i+2}^c e_j \quad\text{ where } d=1,2,3,
\\
\label{crossing-slide-splitter3}
&\psi_{i+2}\psi_{i+1}\psi_{i} \otimes e_j=1\otimes \psi_{i+2}\psi_{i+1}\psi_{i}e_j, \quad
\psi_{i}\psi_{i+1}\psi_{i+2}\otimes e_j=1\otimes \psi_{i}\psi_{i+1}\psi_{i+2}e_j,
\end{align}
where $j\not=i,i+1$.
\end{prop}

For a diagrammatic description of $W_{i,i+1}$, we consider collections of smooth arcs in the plane connecting $ n $ red points and $1$ black point on one horizontal line with $n$ red points and $1$ black point on another horizontal line.
The $i$th, $(i+1)$st and $(i+2)$nd red dots on one horizontal line must be connected to the
 $i$th, $(i+1)$st and $(i+2)$nd red dots on the other horizontal line by a diagram which has a thick red strand in the middle, which is given in ~\eqref{thickthickgenerator}.
The arcs are colored in a manner consistent with their boundary points and are assumed to have no critical points (in other words no cups or caps).
They are allowed to intersect, but no triple intersections are allowed.
Arcs are allowed to carry dots.  
Two diagrams that are related by an isotopy that does not change the combinatorial types of the diagrams or the relative position of crossings are taken to be equal.
The elements of the vector space $ W_{i,i+1} $ are formal linear combinations of these diagrams modulo the local relations for $W$ along with the relations given in
~\eqref{thicksymmetricrelations1}, ~\eqref{thicksymmetricrelations2}, ~\eqref{thicksymmetricrelations3} and ~\eqref{thickblackthickred1}.

The elements $e_j\otimes e_j$, where $0\leq j<i$ or $i+1<j\leq n$, are represented by
\begin{equation}
\label{thickthickgenerator}
e_j\otimes e_j=\begin{DGCpicture}
\DGCstrand[red](-4,0)(-4,1.5)[$^{1}$`{\ }]
\DGCcoupon*(-3,0)(-4,1.5){$\cdots$}
\DGCstrand[red](-3,0)(-3,1.5)[$^{j}$`{\ }]
\DGCstrand(-2.5,0)(-2.5,1.5)
\DGCstrand[red](-2,0)(-2,1.5)[$^{j+1}$`{\ }]
\DGCcoupon*(-1,0)(-2,1.5){$\cdots$}
\DGCstrand[red](-1,0)(-1,1.5)[$^{i-1}$`{\ }]
\DGCstrand[red](-.5,0)(0,.5)[$^{i}$`{\ }]
\DGCstrand[red](0,0)(0,.5)
\DGCstrand[red](0.5,0)(0,.5)
\DGCstrand[Red](0,.5)(0,1)
\DGCstrand[red](0,1)(-0.5,1.5)
\DGCstrand[red](0,1)(0,1.5)
\DGCstrand[red](0,1)(0.5,1.5)
\DGCstrand[red](1,0)(1,1.5)[$^{\,\,i+3}$`{\ }]
\DGCcoupon*(1,0)(2,1.5){$\cdots$}
\DGCstrand[red](2,0)(2,1.5)[$^{n}$`{\ }]
\end{DGCpicture}
\quad\in W_{i,i+1}
\ .
\end{equation}
%where $j+1$ is the position of the black strand from left.

The generating equations in \eqref{dot-slide-splitter3} in terms of diagrams are
\begin{equation}
\label{thicksymmetricrelations1}
\begin{DGCpicture}
\DGCstrand[red](0,0)(.5,.5)
\DGCstrand[red](1,0)(.5,.5)
\DGCstrand[Red](.5,.5)(.5,1)
\DGCstrand[red](.5,1)(0,1.5)
\DGCdot{E}
\DGCstrand[red](.5,1)(1,1.5)
\DGCstrand[red](.5,0)(.5,.5)
\DGCstrand[red](.5,1)(.5,1.5)
\end{DGCpicture}
~+~
\begin{DGCpicture}
\DGCstrand[red](0,0)(.5,.5)
\DGCstrand[red](1,0)(.5,.5)
\DGCstrand[Red](.5,.5)(.5,1)
\DGCstrand[red](.5,1)(0,1.5)
\DGCstrand[red](.5,1)(1,1.5)
\DGCstrand[red](.5,0)(.5,.5)
\DGCstrand[red](.5,1)(.5,1.5)
\DGCdot{E}
\end{DGCpicture}
~+~
\begin{DGCpicture}
\DGCstrand[red](0,0)(.5,.5)
\DGCstrand[red](1,0)(.5,.5)
\DGCstrand[Red](.5,.5)(.5,1)
\DGCstrand[red](.5,1)(0,1.5)
\DGCstrand[red](.5,1)(1,1.5)
\DGCdot{E}
\DGCstrand[red](.5,0)(.5,.5)
\DGCstrand[red](.5,1)(.5,1.5)
\end{DGCpicture}
~=~
\begin{DGCpicture}
\DGCstrand[red](0,0)(.5,.5)
\DGCdot{B}
\DGCstrand[red](1,0)(.5,.5)
\DGCstrand[Red](.5,.5)(.5,1)
\DGCstrand[red](.5,1)(0,1.5)
\DGCstrand[red](.5,1)(1,1.5)
\DGCstrand[red](.5,0)(.5,.5)
\DGCstrand[red](.5,1)(.5,1.5)
\end{DGCpicture}
~+~
\begin{DGCpicture}
\DGCstrand[red](0,0)(.5,.5)
\DGCstrand[red](1,0)(.5,.5)
\DGCstrand[Red](.5,.5)(.5,1)
\DGCstrand[red](.5,1)(0,1.5)
\DGCstrand[red](.5,1)(1,1.5)
\DGCstrand[red](.5,0)(.5,.5)
\DGCdot{B}
\DGCstrand[red](.5,1)(.5,1.5)
\end{DGCpicture}
~+~
\begin{DGCpicture}
\DGCstrand[red](0,0)(.5,.5)
\DGCstrand[red](1,0)(.5,.5)
\DGCdot{B}
\DGCstrand[Red](.5,.5)(.5,1)
\DGCstrand[red](.5,1)(0,1.5)
\DGCstrand[red](.5,1)(1,1.5)
\DGCstrand[red](.5,0)(.5,.5)
\DGCstrand[red](.5,1)(.5,1.5)
\end{DGCpicture}
\ ,
\end{equation}

\begin{equation}
\label{thicksymmetricrelations2}
\begin{DGCpicture}
\DGCstrand[red](0,0)(.5,.5)
\DGCstrand[red](1,0)(.5,.5)
\DGCstrand[Red](.5,.5)(.5,1)
\DGCstrand[red](.5,1)(0,1.5)
\DGCdot{E}
\DGCstrand[red](.5,1)(1,1.5)
\DGCstrand[red](.5,0)(.5,.5)
\DGCstrand[red](.5,1)(.5,1.5)
\DGCdot{E}
\end{DGCpicture}
~+~
\begin{DGCpicture}
\DGCstrand[red](0,0)(.5,.5)
\DGCstrand[red](1,0)(.5,.5)
\DGCstrand[Red](.5,.5)(.5,1)
\DGCstrand[red](.5,1)(0,1.5)
\DGCdot{E}
\DGCstrand[red](.5,1)(1,1.5)
\DGCdot{E}
\DGCstrand[red](.5,0)(.5,.5)
\DGCstrand[red](.5,1)(.5,1.5)
\end{DGCpicture}
~+~
\begin{DGCpicture}
\DGCstrand[red](0,0)(.5,.5)
\DGCstrand[red](1,0)(.5,.5)
\DGCstrand[Red](.5,.5)(.5,1)
\DGCstrand[red](.5,1)(0,1.5)
\DGCstrand[red](.5,1)(1,1.5)
\DGCdot{E}
\DGCstrand[red](.5,0)(.5,.5)
\DGCstrand[red](.5,1)(.5,1.5)
\DGCdot{E}
\end{DGCpicture}
~=~
\begin{DGCpicture}
\DGCstrand[red](0,0)(.5,.5)
\DGCdot{B}
\DGCstrand[red](1,0)(.5,.5)
\DGCstrand[Red](.5,.5)(.5,1)
\DGCstrand[red](.5,1)(0,1.5)
\DGCstrand[red](.5,1)(1,1.5)
\DGCstrand[red](.5,0)(.5,.5)
\DGCdot{B}
\DGCstrand[red](.5,1)(.5,1.5)
\end{DGCpicture}
~+~
\begin{DGCpicture}
\DGCstrand[red](0,0)(.5,.5)
\DGCdot{B}
\DGCstrand[red](1,0)(.5,.5)
\DGCdot{B}
\DGCstrand[Red](.5,.5)(.5,1)
\DGCstrand[red](.5,1)(0,1.5)
\DGCstrand[red](.5,1)(1,1.5)
\DGCstrand[red](.5,0)(.5,.5)
\DGCstrand[red](.5,1)(.5,1.5)
\end{DGCpicture}
~+~
\begin{DGCpicture}
\DGCstrand[red](0,0)(.5,.5)
\DGCstrand[red](1,0)(.5,.5)
\DGCdot{B}
\DGCstrand[Red](.5,.5)(.5,1)
\DGCstrand[red](.5,1)(0,1.5)
\DGCstrand[red](.5,1)(1,1.5)
\DGCstrand[red](.5,0)(.5,.5)
\DGCdot{B}
\DGCstrand[red](.5,1)(.5,1.5)
\end{DGCpicture}
\ ,
\end{equation}

\begin{equation}
\label{thicksymmetricrelations3}
\begin{DGCpicture}
\DGCstrand[red](0,0)(.5,.5)
\DGCstrand[red](1,0)(.5,.5)
\DGCstrand[Red](.5,.5)(.5,1)
\DGCstrand[red](.5,1)(0,1.5)
\DGCdot{E}
\DGCstrand[red](.5,1)(1,1.5)
\DGCdot{E}
\DGCstrand[red](.5,0)(.5,.5)
\DGCstrand[red](.5,1)(.5,1.5)
\DGCdot{E}
\end{DGCpicture}
~=~
\begin{DGCpicture}
\DGCstrand[red](0,0)(.5,.5)
\DGCdot{B}
\DGCstrand[red](1,0)(.5,.5)
\DGCdot{B}
\DGCstrand[Red](.5,.5)(.5,1)
\DGCstrand[red](.5,1)(0,1.5)
\DGCstrand[red](.5,1)(1,1.5)
\DGCstrand[red](.5,0)(.5,.5)
\DGCdot{B}
\DGCstrand[red](.5,1)(.5,1.5)
\end{DGCpicture}
\ .
\end{equation}

The equations in \eqref{crossing-slide-splitter3} in terms of diagrams are
\begin{equation}
\label{thickblackthickred1}
\begin{DGCpicture}
\DGCstrand(0,0)(0,1.25)(2,1.75)
\DGCstrand[red](1,0)(1.5,0.5)[$^{i}$`{\ }]
\DGCstrand[red](1.5,0)(1.5,0.5)[$^{i+1}$`{\ }]
\DGCstrand[red](2,0)(1.5,0.5)[$^{i+2}$`{\ }]
\DGCstrand[Red](1.5,0.5)(1.5,0.75)
\DGCstrand[red](1.5,0.75)(2,1.25)(1,1.75)
\DGCstrand[red](1.5,0.75)(1.5,1)(0.5,1.625)(0.5,1.75)
\DGCstrand[red](1.5,0.75)(0,1.5)(0,1.75)
\end{DGCpicture}
~=~
\begin{DGCpicture}
\DGCstrand(0,0)(2,0.5)(2,1.75)
\DGCstrand[red](1,0)(0,0.5)(0.5,1)[$^{i}$`{\ }]
\DGCstrand[red](1.5,0)(1.5,0.125)[$^{i+1}$`{\ }]
\DGCstrand[red](1.5,0.125)(0.5,0.75)(0.5,1)
\DGCstrand[red](2,0)(2,0.25)[$^{i+2}$`{\ }]
\DGCstrand[red](2,0.25)(0.5,1)
\DGCstrand[Red](0.5,1)(0.5,1.25)
\DGCstrand[red](0.5,1.25)(0,1.75)
\DGCstrand[red](0.5,1.25)(0.5,1.75)
\DGCstrand[red](0.5,1.25)(1,1.75)
\end{DGCpicture}
\ , \quad \quad \quad
\begin{DGCpicture}
\DGCstrand(3,0)(3,1.25)(1,1.75)
\DGCstrand[red](1,0)(1.5,0.5)[$^{i}$`{\ }]
\DGCstrand[red](1.5,0)(1.5,0.5)[$^{i+1}$`{\ }]
\DGCstrand[red](2,0)(1.5,0.5)[$^{i+2}$`{\ }]
\DGCstrand[Red](1.5,0.5)(1.5,0.75)
\DGCstrand[red](1.5,0.75)(3,1.5)(3,1.75)
\DGCstrand[red](1.5,0.75)(1.5,1)(2.5,1.625)(2.5,1.75)
\DGCstrand[red](1.5,0.75)(1,1.25)(2,1.75)
\end{DGCpicture}
~=~
\begin{DGCpicture}
\DGCstrand(1,0)(-1,0.5)(-1,1.75)
\DGCstrand[red](-1,0)(-1,0.25)[$^{i}$`{\ }]
\DGCstrand[red](-1,0.25)(0.5,1)
\DGCstrand[red](-0.5,0)(-0.5,0.125)[$^{i+1}$`{\ }]
\DGCstrand[red](-0.5,0.125)(0.5,0.75)(0.5,1)
\DGCstrand[red](0,0)(1,0.5)(0.5,1)[$^{i+2}$`{\ }]
\DGCstrand[Red](0.5,1)(0.5,1.25)
\DGCstrand[red](0.5,1.25)(0,1.75)
\DGCstrand[red](0.5,1.25)(0.5,1.75)
\DGCstrand[red](0.5,1.25)(1,1.75)
\end{DGCpicture}
\ .
\end{equation}

The bimodule $W_{i,i+1}$ naturally inherits a $p$-DG structure from $W$ where one sets the differential on the generator \eqref{thickthickgenerator} to be zero.  

\begin{rem}
We will sometimes denote the elements in the equalities of
in \eqref{thickblackthickred1} as a crossing between a black strand and a thick red strand.  We take a similar convention for diagrams in \eqref{blackthickred1}.
\end{rem}

%\begin{DGCpicture}
%\DGCstrand[red](0,0)(0,1)[$^{1}$`{\ }]
%\DGCdot{0.5}
%\DGCcoupon*(-.5,0.25)(0,0.75){$a_1$}

\subsection{Bases for bimodules}
We construct bases of the bimodules just introduced. 

Let
\begin{align*}
\aleph_1({\bf a},b,c)
&~=~ \prod_{k=1}^n x_k^{a_k} y^b\psi_i\otimes\psi_i x_i^c e_i\\
&=\begin{DGCpicture}
\DGCstrand[red](-1,0)(-1,1.5)[$^{1}$`{\ }]
\DGCdot{E}[u]{$a_1$}
\DGCcoupon*(-.75,0.5)(-.25,1){$\cdots$}
\DGCstrand(.5,0)(0,.75)(.5,1.5)
\DGCdot{E}[u]{$b$}
\DGCstrand[red](0,0)(.5,.5)[$^{i}$`{\ }]
\DGCdot{B}[l]{$c$}
\DGCstrand[red](1,0)(.5,.5)[$^{i+1}$`{\ }]
\DGCstrand[Red](.5,.5)(.5,1)
\DGCstrand[red](.5,1)(0,1.5)
\DGCdot{E}[u]{$a_i$}
\DGCstrand[red](.5,1)(1,1.5)
\DGCdot{E}[u]{$a_{i+1}$}
\DGCcoupon*(1.25,0.5)(1.75,1){$\cdots$}
\DGCstrand[red](2,0)(2,1.5)[$^{n}$`{\ }]
\DGCdot{E}[u]{$a_n$}
\end{DGCpicture}
\end{align*}

\begin{align*}
\aleph_2({\bf a})
&~=~\prod_{k=1}^n x_k^{a_k} \psi_{i+1}\otimes\psi_{i+1} e_i
\\
&=\begin{DGCpicture}
\DGCstrand[red](-1,0)(-1,1.5)[$^{1}$`{\ }]
\DGCdot{E}[u]{$a_1$}
\DGCcoupon*(-.75,0.5)(-.25,1){$\cdots$}
\DGCstrand(.5,0)(1,.75)(.5,1.5)
%\DGCdot{E}[u]{$b$}
\DGCstrand[red](0,0)(.5,.5)[$^{i}$`{\ }]
%\DGCdot{B}[l]{$c$}
\DGCstrand[red](1,0)(.5,.5)[$^{i+1}$`{\ }]
\DGCstrand[Red](.5,.5)(.5,1)
\DGCstrand[red](.5,1)(0,1.5)
\DGCdot{E}[u]{$a_i$}
\DGCstrand[red](.5,1)(1,1.5)
\DGCdot{E}[u]{$a_{i+1}$}
\DGCcoupon*(1.25,0.5)(1.75,1){$\cdots$}
\DGCstrand[red](2,0)(2,1.5)[$^{n}$`{\ }]
\DGCdot{E}[u]{$a_n$}
\end{DGCpicture}
\ .
\end{align*}
%\JS{I think we should replace what is below}
%For $0\leq j\leq n$, $j\not=i$ and $0 \leq k\leq n$, let
%\begin{align*}
%\aleph_3({\bf a},b,c,j,\ell)
%&~=~
%\begin{cases}
%\prod_{k=1}^n x_k^{a_k}y^b %\psi_{\ell}\psi_{\ell-1}\cdots\psi_{j+1}\otimes x_i^c e_j&\text{if %} j\leq \ell\\
%\prod_{k=1}^n x_k^{a_k}y^b\psi_{\ell+1}\psi_{\ell+2}\cdots\psi_{j}%\otimes x_i^c e_j &\text{if } j> \ell
%\end{cases}
%\\
%&=\begin{DGCpicture}
%%\DGCcoupon*(-2.75,.5)(-2.25,1){$\cdots$}
%\DGCstrand[red](-2,0)(-2,1.5)[$^{1}$`{\ }]
%\DGCdot{E}[u]{$a_1$}
%\DGCcoupon*(-1.75,1)(-1.25,1.5){$\cdots$}
%\DGCstrand(-1,0)(-1,.75)(2,1.5)
%\DGCdot{E}[u]{$b$}
%\DGCstrand[red](0,0)(.5,.5)[$^{i}$`{\ }]
%\DGCdot{B}[l]{$c$}
%\DGCstrand[red](1,0)(.5,.5)[$^{i+1}$`{\ }]
%\DGCstrand[Red](.5,.5)(.5,.75)
%\DGCstrand[red](.5,.75)(0,1.25)(0,1.5)
%\DGCdot{E}[u]{$a_i$}
%\DGCstrand[red](.5,.75)(1,1.25)(1,1.5)
%\DGCdot{E}[u]{$a_{i+1}$}
%\DGCcoupon*(2.25,0)(2.75,.5){$\cdots$}
%\DGCstrand[red](3,0)(3,1.5)[$^{n}$`{\ }]
%\DGCdot{E}[u]{$a_n$}
%%\DGCcoupon*(3.25,.5)(3.75,1){$\cdots$}
%\end{DGCpicture}
%\end{align*}
%\JS{Replace with this}
For $0\leq j\leq n$ and $0 \leq \ell \leq n$ unless $j=\ell =i$, let
\begin{align*}
\aleph_3({\bf a},b,c,j,\ell)
&~=~
\begin{cases}
\prod_{k=1}^n x_k^{a_k}y^b \psi_{\ell}\psi_{\ell-1}\cdots\psi_{j+1}\otimes x_i^c e_j&\text{if } j < \ell \quad\text{and}\quad j\not=i\\
\prod_{k=1}^n x_k^{a_k}y^b\psi_{\ell+1}\psi_{\ell+2}\cdots\psi_{j}\otimes x_i^c e_j &\text{if } j> \ell \quad\text{and}\quad j\not=i\\
\prod_{k=1}^n x_k^{a_k}y^b \otimes x_i^c e_j&\text{if } j = \ell, j \neq i\\
\prod_{k=1}^n x_k^{a_k}y^b \otimes x_i^c \psi_{\ell}\psi_{\ell-1}\cdots\psi_{i+1}e_i&\text{if } j < \ell \quad\text{and}\quad j=i\\
\prod_{k=1}^n x_k^{a_k}y^b\otimes x_i^c\psi_{\ell+1}\psi_{\ell+2}\cdots\psi_{i} e_i &\text{if } j> \ell \quad\text{and}\quad j=i.
\end{cases}
\end{align*}
For $j< \ell$ and $ j \neq  i$, diagramatically $\aleph_3({\bf a}, b, c, j, \ell)$ is given by
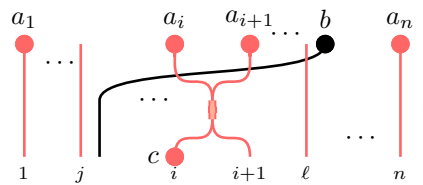
\begin{equation*}
\begin{DGCpicture}
%\DGCcoupon*(-2.75,.5)(-2.25,1){$\cdots$}
\DGCstrand[red](-2,0)(-2,1.5)[$^{1}$`{\ }]
\DGCdot{E}[u]{$a_1$}
\DGCcoupon*(-1.75,1)(-1.25,1.5){$\cdots$}
\DGCstrand(-1,0)(-1,.75)(2,1.5)
\DGCdot{E}[u]{$b$}
\DGCstrand[red](0,0)(.5,.5)[$^{i}$`{\ }]
\DGCdot{B}[l]{$c$}
\DGCstrand[red](1,0)(.5,.5)[$^{i+1}$`{\ }]
\DGCstrand[Red](.5,.5)(.5,.75)
\DGCstrand[red](.5,.75)(0,1.25)(0,1.5)
\DGCdot{E}[u]{$a_i$}
\DGCstrand[red](.5,.75)(1,1.25)(1,1.5)
\DGCdot{E}[u]{$a_{i+1}$}
\DGCcoupon*(2.25,0)(2.75,.5){$\cdots$}
\DGCstrand[red](3,0)(3,1.5)[$^{n}$`{\ }]
\DGCdot{E}[u]{$a_n$}
\DGCcoupon*(-.5,.5)(0,1){$\cdots$}
\DGCcoupon*(1.25,1.5)(1.75,1.75){$\cdots$}
\DGCstrand[red](-1.25,0)(-1.25,1.5)[$^{j}$`{\ }]
\DGCstrand[red](1.75,0)(1.75,1.5)[$^{\ell}$`{\ }]
\end{DGCpicture}
\ .
\end{equation*}
%where on the bottom, the endpoint of the black strand is directly to the right of the $j$th red strand, and on the top, the endpoint of the black strand is directly to the right of the $\ell$th red strand.

\begin{lem} \label{classical}
For any $b_1, b_2, b_3 \geq 0$, the elements
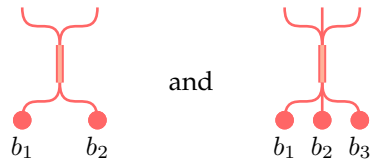
\begin{equation*}
\begin{DGCpicture}
\DGCstrand[red](0,0)(.5,.5)
\DGCdot{B}[d]{$b_1$}
\DGCstrand[red](1,0)(.5,.5)
\DGCdot{B}[d]{$b_2$}
\DGCstrand[Red](.5,.5)(.5,1)
\DGCstrand[red](.5,1)(0,1.5)
\DGCstrand[red](.5,1)(1,1.5)
\end{DGCpicture}
\quad \quad
\text{ and }
\quad \quad
\begin{DGCpicture}
\DGCstrand[red](0,0)(.5,.5)
\DGCdot{B}[d]{$b_1$}
\DGCstrand[red](1,0)(.5,.5)
\DGCdot{B}[d]{$b_3$}
\DGCstrand[Red](.5,.5)(.5,1)
\DGCstrand[red](.5,1)(0,1.5)
\DGCstrand[red](.5,1)(1,1.5)
\DGCstrand[red](.5,0)(.5,.5)
\DGCdot{B}[d]{$b_2$}
\DGCstrand[red](.5,1)(.5,1.5)
\end{DGCpicture}
\end{equation*}
could be written as a linear combination of elements of the form
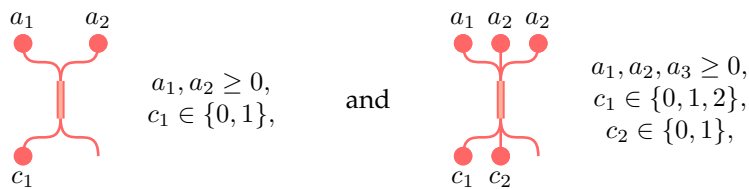
\begin{equation*}
\begin{DGCpicture}
\DGCstrand[red](0,0)(.5,.5)
\DGCdot{B}[d]{$c_1$}
\DGCstrand[red](1,0)(.5,.5)
\DGCstrand[Red](.5,.5)(.5,1)
\DGCstrand[red](.5,1)(0,1.5)
\DGCdot{E}[u]{$a_1$}
\DGCstrand[red](.5,1)(1,1.5)
\DGCdot{E}[u]{$a_2$}
%\DGCcoupon*(2,0)(4,1.5){$a_1, a_2 \geq 0$, $c_1 \in \{0,1\}$}
\DGCcoupon*(1.5,0)(3.5,1.5){$\begin{matrix} a_1, a_2 \geq 0, \\ c_1 \in \{0,1\}, \end{matrix}$} 
\end{DGCpicture}
\quad \quad
\text{ and }
\quad \quad
\begin{DGCpicture}
\DGCstrand[red](0,0)(.5,.5)
\DGCdot{B}[d]{$c_1$}
\DGCstrand[red](1,0)(.5,.5)
\DGCstrand[Red](.5,.5)(.5,1)
\DGCstrand[red](.5,1)(0,1.5)
\DGCdot{E}[u]{$a_1$}
\DGCstrand[red](.5,1)(1,1.5)
\DGCdot{E}[u]{$a_2$}
\DGCstrand[red](.5,0)(.5,.5)
\DGCdot{B}[d]{$c_2$}
\DGCstrand[red](.5,1)(.5,1.5)
\DGCdot{E}[u]{$a_2$}
\DGCcoupon*(1.5,0)(4,1.5){$\begin{matrix} a_1, a_2, a_3 \geq 0, \\ c_1 \in \{0,1,2 \}, \\ c_2 \in \{0,1 \}, \end{matrix}$} 
\end{DGCpicture}
\end{equation*}
respectively.
\end{lem}

\begin{proof}
This is a well-known result.  See for example \cite[Section 2.2]{EliasKh}.
\end{proof}

\begin{prop} \label{W_ispan}
The set
%\[
%\{ \aleph_1({\bf a},b,c), \aleph_2({\bf a}), \aleph_3({\bf %a},b,c,j,\ell)
%%, SE((a_i),b,c) 
%| a_1, \ldots , a_n, b \in \mathbb{Z}, c \in \{0,1 \}, 0\leq j\leq %n, j\not=i, 0\leq \ell\leq n
%\}
%\ .
%\]
\[
\{ \aleph_1({\bf a},b,c), \aleph_2({\bf a}), \aleph_3({\bf a},b,c,j,\ell)
%, SE((a_i),b,c) 
| a_1, \ldots , a_n, b \in \mathbb{Z}, c \in \{0,1 \}, 0\leq j\leq n, 0\leq \ell\leq n
\}
\]
is a spanning set of $W_i$.
\end{prop}

\begin{proof}
We will only highlight interesting features of the proof, which include cases where the black strand begins and ends in between the two red strands connected to the thick red strand.
First we will show that $\psi_{i+1}\otimes \psi_{i+1}x_i e_i$ is in the span of the proposed spanning set.
The diagram for this element is
\begin{equation}
\label{swringrightdotleft}
\begin{DGCpicture}
\DGCstrand(.5,0)(1,.75)(.5,1.5)
\DGCstrand[red](0,0)(.5,.5)[]
\DGCdot{B}[l]{$ $}
\DGCstrand[red](1,0)(.5,.5)[]
\DGCstrand[Red](.5,.5)(.5,1)
\DGCstrand[red](.5,1)(0,1.5)
\DGCstrand[red](.5,1)(1,1.5)
\end{DGCpicture}
\ .
\end{equation}
Using relations \eqref{R2rel}, we have the following equations
\begin{equation}
\label{spantrick1}
\begin{DGCpicture}
\DGCstrand(.5,0)(1,.75)(.5,1.5)
\DGCdot{B}[l]{$ $}
\DGCstrand[red](0,0)(.5,.5)[]
\DGCstrand[red](1,0)(.5,.5)[]
\DGCstrand[Red](.5,.5)(.5,1)
\DGCstrand[red](.5,1)(0,1.5)
\DGCstrand[red](.5,1)(1,1.5)
\end{DGCpicture}
~-~
\begin{DGCpicture}
\DGCstrand(.5,0)(1,.75)(.5,1.5)
\DGCstrand[red](0,0)(.5,.5)[]
\DGCdot{B}[l]{$ $}
\DGCstrand[red](1,0)(.5,.5)[]
\DGCstrand[Red](.5,.5)(.5,1)
\DGCstrand[red](.5,1)(0,1.5)
\DGCstrand[red](.5,1)(1,1.5)
\end{DGCpicture}
~=~
\begin{DGCpicture}
\DGCstrand(.5,0)(0,.75)(.5,1.5)
\DGCdot{E}[l]{$ $}
\DGCstrand[red](0,0)(.5,.5)[]
\DGCstrand[red](1,0)(.5,.5)[]
\DGCstrand[Red](.5,.5)(.5,1)
\DGCstrand[red](.5,1)(0,1.5)
\DGCstrand[red](.5,1)(1,1.5)
\end{DGCpicture}
~-~
\begin{DGCpicture}
\DGCstrand(.5,0)(0,.75)(.5,1.5)
\DGCstrand[red](0,0)(.5,.5)[]
\DGCstrand[red](1,0)(.5,.5)[]
\DGCstrand[Red](.5,.5)(.5,1)
\DGCstrand[red](.5,1)(0,1.5)
\DGCstrand[red](.5,1)(1,1.5)
\DGCdot{E}[l]{$ $}
\end{DGCpicture}
\ ,
\end{equation}
\begin{equation}
\label{spantrick2}
\begin{DGCpicture}
\DGCstrand(.5,0)(1,.75)(.5,1.5)
\DGCdot{E}[l]{$ $}
\DGCstrand[red](0,0)(.5,.5)[]
\DGCstrand[red](1,0)(.5,.5)[]
\DGCstrand[Red](.5,.5)(.5,1)
\DGCstrand[red](.5,1)(0,1.5)
\DGCstrand[red](.5,1)(1,1.5)
\end{DGCpicture}
~-~
\begin{DGCpicture}
\DGCstrand(.5,0)(1,.75)(.5,1.5)
\DGCstrand[red](0,0)(.5,.5)[]
\DGCstrand[red](1,0)(.5,.5)[]
\DGCstrand[Red](.5,.5)(.5,1)
\DGCstrand[red](.5,1)(0,1.5)
\DGCdot{E}[l]{$ $}
\DGCstrand[red](.5,1)(1,1.5)
\end{DGCpicture}
~=~
\begin{DGCpicture}
\DGCstrand(.5,0)(0,.75)(.5,1.5)
\DGCdot{B}[l]{$ $}
\DGCstrand[red](0,0)(.5,.5)[]
\DGCstrand[red](1,0)(.5,.5)[]
\DGCstrand[Red](.5,.5)(.5,1)
\DGCstrand[red](.5,1)(0,1.5)
\DGCstrand[red](.5,1)(1,1.5)
\end{DGCpicture}
~-~
\begin{DGCpicture}
\DGCstrand(.5,0)(0,.75)(.5,1.5)
\DGCstrand[red](0,0)(.5,.5)[]
\DGCstrand[red](1,0)(.5,.5)[]
\DGCdot{B}[l]{$ $}
\DGCstrand[Red](.5,.5)(.5,1)
\DGCstrand[red](.5,1)(0,1.5)
\DGCstrand[red](.5,1)(1,1.5)
\end{DGCpicture}
\ .
\end{equation}
Subtracting equation \eqref{spantrick2} from equation \eqref{spantrick1} yields
\begin{equation}
\label{spantrick3}
\begin{DGCpicture}
\DGCstrand(.5,0)(1,.75)(.5,1.5)
\DGCstrand[red](0,0)(.5,.5)[]
\DGCstrand[red](1,0)(.5,.5)[]
\DGCstrand[Red](.5,.5)(.5,1)
\DGCstrand[red](.5,1)(0,1.5)
\DGCdot{E}[l]{$ $}
\DGCstrand[red](.5,1)(1,1.5)
\end{DGCpicture}
~-~
\begin{DGCpicture}
\DGCstrand(.5,0)(1,.75)(.5,1.5)
\DGCstrand[red](0,0)(.5,.5)[]
\DGCdot{B}[l]{$ $}
\DGCstrand[red](1,0)(.5,.5)[]
\DGCstrand[Red](.5,.5)(.5,1)
\DGCstrand[red](.5,1)(0,1.5)
\DGCstrand[red](.5,1)(1,1.5)
\end{DGCpicture}
~=~
\begin{DGCpicture}
\DGCstrand(.5,0)(0,.75)(.5,1.5)
\DGCstrand[red](0,0)(.5,.5)[]
\DGCstrand[red](1,0)(.5,.5)[]
\DGCdot{B}[l]{$ $}
\DGCstrand[Red](.5,.5)(.5,1)
\DGCstrand[red](.5,1)(0,1.5)
\DGCstrand[red](.5,1)(1,1.5)
\end{DGCpicture}
~-~
\begin{DGCpicture}
\DGCstrand(.5,0)(0,.75)(.5,1.5)
\DGCstrand[red](0,0)(.5,.5)[]
\DGCstrand[red](1,0)(.5,.5)[]
\DGCstrand[Red](.5,.5)(.5,1)
\DGCstrand[red](.5,1)(0,1.5)
\DGCstrand[red](.5,1)(1,1.5)
\DGCdot{E}[l]{$ $}
\end{DGCpicture}
\ .
\end{equation}
Using symmetric function relations on \eqref{spantrick3} we get
\begin{equation*}
\begin{DGCpicture}
\DGCstrand(.5,0)(1,.75)(.5,1.5)
\DGCstrand[red](0,0)(.5,.5)[]
\DGCdot{B}[l]{$ $}
\DGCstrand[red](1,0)(.5,.5)[]
\DGCstrand[Red](.5,.5)(.5,1)
\DGCstrand[red](.5,1)(0,1.5)
\DGCstrand[red](.5,1)(1,1.5)
\end{DGCpicture}
~=~
\begin{DGCpicture}
\DGCstrand(.5,0)(1,.75)(.5,1.5)
\DGCstrand[red](0,0)(.5,.5)[]
\DGCstrand[red](1,0)(.5,.5)[]
\DGCstrand[Red](.5,.5)(.5,1)
\DGCstrand[red](.5,1)(0,1.5)
\DGCdot{E}[l]{$ $}
\DGCstrand[red](.5,1)(1,1.5)
\end{DGCpicture}
~+~
\begin{DGCpicture}
\DGCstrand(.5,0)(0,.75)(.5,1.5)
\DGCstrand[red](0,0)(.5,.5)[]
\DGCdot{B}[l]{$ $}
\DGCstrand[red](1,0)(.5,.5)[]
\DGCstrand[Red](.5,.5)(.5,1)
\DGCstrand[red](.5,1)(0,1.5)
\DGCstrand[red](.5,1)(1,1.5)
\end{DGCpicture}
~-~
\begin{DGCpicture}
\DGCstrand(.5,0)(0,.75)(.5,1.5)
\DGCstrand[red](0,0)(.5,.5)[]
\DGCstrand[red](1,0)(.5,.5)[]
\DGCstrand[Red](.5,.5)(.5,1)
\DGCstrand[red](.5,1)(0,1.5)
\DGCdot{E}[l]{$ $}
\DGCstrand[red](.5,1)(1,1.5)
\end{DGCpicture}
\ .
\end{equation*}
Thus a diagram with a portion as in \eqref{swringrightdotleft} is in the span.

Consider a portion of a diagram of the form
\begin{equation}
\begin{DGCpicture}
\DGCstrand(.5,0)(1,.75)(.5,1.5)
\DGCdot{B}[u]{$ $}
\DGCstrand[red](0,0)(.5,.5)[]
%\DGCdot{B}[l]{$c$}
\DGCstrand[red](1,0)(.5,.5)[]
\DGCstrand[Red](.5,.5)(.5,1)
\DGCstrand[red](.5,1)(0,1.5)
\DGCstrand[red](.5,1)(1,1.5)
\end{DGCpicture}
\ .
\end{equation}

Then using relations \eqref{R2rel}, we have
\begin{equation*}
\begin{DGCpicture}
\DGCstrand(.5,0)(1,.75)(.5,1.5)
\DGCdot{B}[u]{$ $}
\DGCstrand[red](0,0)(.5,.5)[]
%\DGCdot{B}[l]{$c$}
\DGCstrand[red](1,0)(.5,.5)[]
\DGCstrand[Red](.5,.5)(.5,1)
\DGCstrand[red](.5,1)(0,1.5)
\DGCstrand[red](.5,1)(1,1.5)
\end{DGCpicture}
~=~
\begin{DGCpicture}
\DGCstrand(.5,0)(1,.75)(.5,1.5)
\DGCstrand[red](0,0)(.5,.5)[]
\DGCdot{B}[l]{$ $}
\DGCstrand[red](1,0)(.5,.5)[]
\DGCstrand[Red](.5,.5)(.5,1)
\DGCstrand[red](.5,1)(0,1.5)
\DGCstrand[red](.5,1)(1,1.5)
\end{DGCpicture}
~+~
\begin{DGCpicture}
\DGCstrand(.5,0)(0,.75)(.5,1.5)
\DGCdot{E}[u]{$ $}
\DGCstrand[red](0,0)(.5,.5)[]
%\DGCdot{B}[l]{$c$}
\DGCstrand[red](1,0)(.5,.5)[]
\DGCstrand[Red](.5,.5)(.5,1)
\DGCstrand[red](.5,1)(0,1.5)
\DGCstrand[red](.5,1)(1,1.5)
%\DGCdot{E}[u]{$ $}
\end{DGCpicture}
~-~
\begin{DGCpicture}
\DGCstrand(.5,0)(0,.75)(.5,1.5)
%\DGCdot{E}[u]{$ $}
\DGCstrand[red](0,0)(.5,.5)[]
%\DGCdot{B}[l]{$c$}
\DGCstrand[red](1,0)(.5,.5)[]
\DGCstrand[Red](.5,.5)(.5,1)
\DGCstrand[red](.5,1)(0,1.5)
\DGCstrand[red](.5,1)(1,1.5)
\DGCdot{E}[u]{$ $}
\end{DGCpicture}
\end{equation*}
which shows that a picture (4.29) is in the span of the proposed spanning set since we already showed the first term above is already in the proposed spanning set.
\end{proof}

In order to prove that the elements in Proposition \ref{W_ispan} are actually a basis of $W_i$, we will construct a $(W(n,1),W(n,1))$-bimodule homomorphism 
$ \phi_i \colon W_i \rightarrow  \Hom_{\Bbbk}(V_n,V_n)$.

Consider the divided difference operator
\[
D_i \colon \Bbbk[y_j][x_1, \ldots , x_n] \rightarrow  \Bbbk[y_j][x_1, \ldots, x_n]
\quad \quad \quad
f \mapsto \frac{f - f^{s_i}}{x_i - x_{i+1}}
\]
where $f^{s_i}$ is the polynomial obtained from $f$ by exchanging the variables $x_i $ and $x_{i+1}$ and keeping all other variables the same.

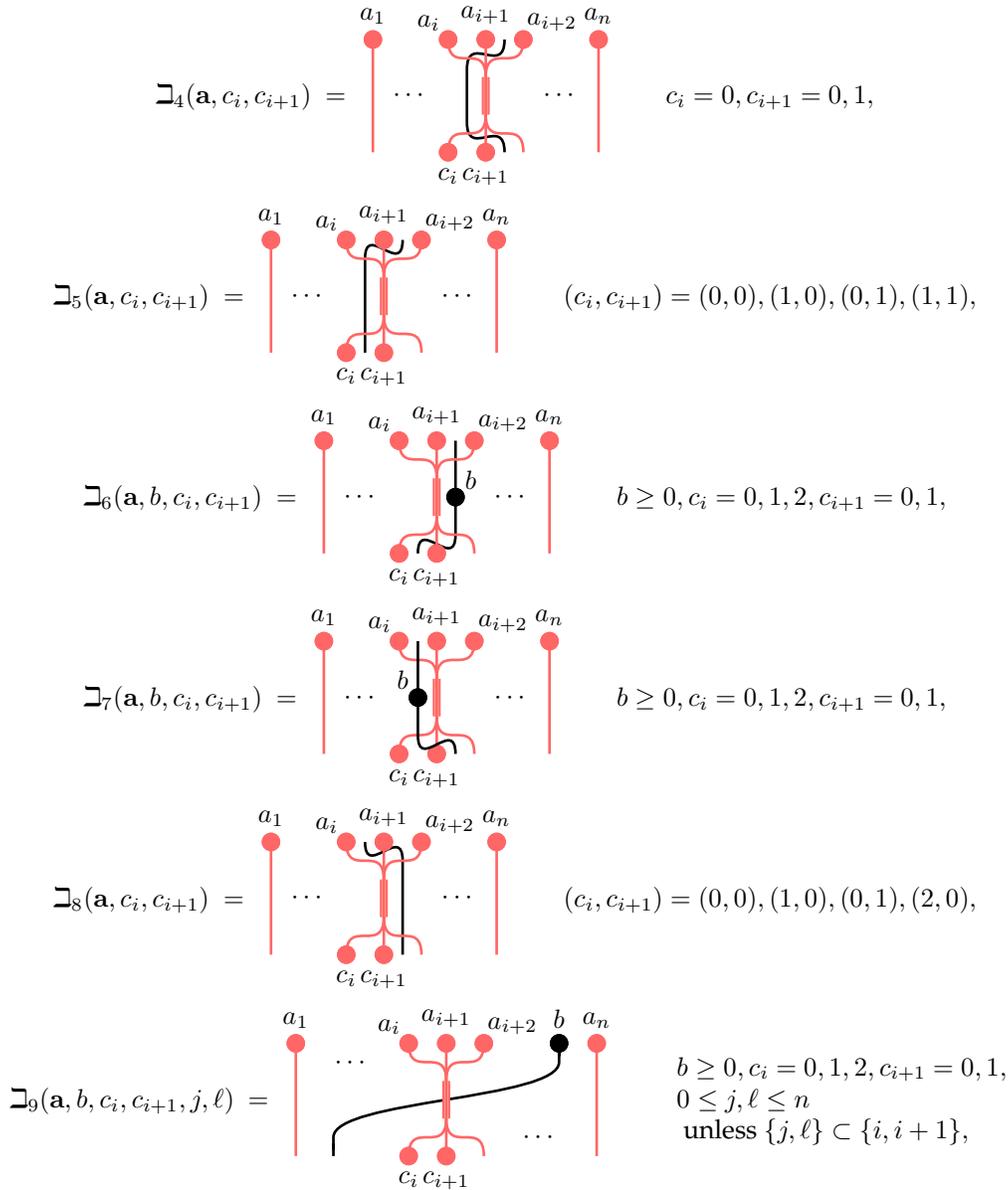
\begin{prop} \label{propphi}
There is a bimodule homomorphism $ \phi_i \colon W_i \rightarrow  \Hom_{\Bbbk}(V_n,V_n)$ inherited from the representation of $W(n,1)$ on $V_n$ determined by
\begin{equation*}
\begin{DGCpicture}
\DGCstrand[red](-1,0)(-1,1.5)[$^{1}$`{\ }]
%\DGCdot{E}[u]{$a_1$}
\DGCcoupon*(-.75,0.5)(-.25,1){$\cdots$}
%\DGCstrand(.5,0)(0,.75)(.5,1.5)
%\DGCdot{E}[u]{$b$}
\DGCstrand[red](0,0)(.5,.5)[$^{i}$`{\ }]
%\DGCdot{B}[l]{$c$}
\DGCstrand[red](1,0)(.5,.5)[$^{i+1}$`{\ }]
\DGCstrand[Red](.5,.5)(.5,1)
\DGCstrand[red](.5,1)(0,1.5)
%\DGCdot{E}[u]{$a_i$}
\DGCstrand[red](.5,1)(1,1.5)
%\DGCdot{E}[u]{$a_{i+1}$}
\DGCcoupon*(1.25,0.5)(1.75,1){$\cdots$}
\DGCstrand[red](2,0)(2,1.5)[$^{n}$`{\ }]
%\DGCdot{E}[u]{$a_n$}
\end{DGCpicture}
~\colon~
f \in V_{n} 
~\mapsto
 D_i(f) \in V_{n} 
\ .
\end{equation*}
\end{prop}

\begin{proof}
It is a routine check using the fact that symmetric functions are invariant under the action of a divided difference operator.
\end{proof}

\begin{prop}
\label{basis1}
The spanning set in Proposition \ref{W_ispan} is a basis of $W_i$.
\end{prop}

\begin{proof}
Clearly a collection of vectors in this bimodule could be linearly dependent only if the configuration of boundary points are the same.  We will focus on the most interesting case.  For simplicity, assume there is a dependence relation of the form
\begin{equation}
\label{depend1}
k^1_{a_1,a_2,b}
\begin{DGCpicture}
\DGCstrand(.5,0)(0,.75)(.5,1.5)
\DGCdot{E}[u]{$b$}
\DGCstrand[red](0,0)(.5,.5)[$^{}$`{\ }]
%\DGCdot{B}[l]{$c$}
\DGCstrand[red](1,0)(.5,.5)[$^{}$`{\ }]
\DGCstrand[Red](.5,.5)(.5,1)
\DGCstrand[red](.5,1)(0,1.5)
\DGCdot{E}[u]{$a_1$}
\DGCstrand[red](.5,1)(1,1.5)
\DGCdot{E}[u]{$a_{2}$}
\end{DGCpicture}
~+~
k^2_{a_1,a_2,b}
\begin{DGCpicture}
\DGCstrand(.5,0)(0,.75)(.5,1.5)
\DGCdot{E}[u]{$b$}
\DGCstrand[red](0,0)(.5,.5)[$^{}$`{\ }]
\DGCdot{B}[l]{$ $}
\DGCstrand[red](1,0)(.5,.5)[$^{}$`{\ }]
\DGCstrand[Red](.5,.5)(.5,1)
\DGCstrand[red](.5,1)(0,1.5)
\DGCdot{E}[u]{$a_1$}
\DGCstrand[red](.5,1)(1,1.5)
\DGCdot{E}[u]{$a_{2}$}
\end{DGCpicture}
~+~
k^3_{a_1,a_2}
\begin{DGCpicture}
\DGCstrand(.5,0)(1,.75)(.5,1.5)
%\DGCdot{E}[u]{$b$}
\DGCstrand[red](0,0)(.5,.5)[$^{}$`{\ }]
%\DGCdot{B}[l]{$ $}
\DGCstrand[red](1,0)(.5,.5)[$^{}$`{\ }]
\DGCstrand[Red](.5,.5)(.5,1)
\DGCstrand[red](.5,1)(0,1.5)
\DGCdot{E}[u]{$a_1$}
\DGCstrand[red](.5,1)(1,1.5)
\DGCdot{E}[u]{$a_{2}$}
\end{DGCpicture}
~=~
0
\ .
\end{equation}
Applying the bimodule homomorphism $\phi_1$ from above on the element $1 $ yields
\begin{equation*}
k^2_{a_1,a_2,b} (y_1-x_1) x_1^{a_1} x_2^{a_2} y_1^b
+
k^3_{a_1,a_2} x_1^{a_1} x_2^{a_2} 
=
0
\ .
\end{equation*}
This implies that $k^2_{a_1,a_2,b}=k^3_{a_1,a_2}=0$ and thus 
$k^1_{a_1,a_2,b}=0$.  Therefore the elements in \eqref{depend1} are in fact linearly independent.  

Checking linear independence for the other elements in the spanning set proceeds in a similar manner.
\end{proof}

\begin{prop}
\label{W_ii+1span}
The following elements span $W_{i,i+1}$
\begin{equation*}
\beth_1({\bf a},b,c_i, c_{i+1})
~=~
\begin{DGCpicture}
\DGCstrand[red](-1,0)(-1,1.5)[$^{}$`{\ }]
\DGCdot{E}[u]{$a_1$}
\DGCcoupon*(-.75,0.5)(-.25,1){$\cdots$}
\DGCstrand(.25,0)(0,.75)(.25,1.5)
\DGCdot{.75}[ul]{$b$}
\DGCstrand[red](0,0)(.5,.5)[$^{}$`{\ }]
\DGCdot{B}[d]{$c_i$}
\DGCstrand[red](1,0)(.5,.5)[$^{}$`{\ }]
\DGCstrand[Red](.5,.5)(.5,1)
\DGCstrand[red](.5,0)(.5,1.5)[$^{}$`{\ }]
\DGCdot{B}[d]{$c_{i+1}$}
\DGCdot{E}[u]{$a_{i+1}$}
\DGCstrand[red](.5,1)(0,1.5)
\DGCdot{E}[ul]{$a_i$}
\DGCstrand[red](.5,1)(1,1.5)
\DGCdot{E}[ur]{$a_{i+2}$}
\DGCcoupon*(1.25,0.5)(1.75,1){$\cdots$}
\DGCstrand[red](2,0)(2,1.5)[$^{}$`{\ }]
\DGCdot{E}[u]{$a_n$}
\end{DGCpicture}
\quad \quad 
b \geq 0, c_i=0,1,2, c_{i+1}=0,1 ,
\end{equation*}

\begin{equation*}
\beth_2({\bf a},c_i, c_{i+1})
~=~
\begin{DGCpicture}
\DGCstrand[red](-1,0)(-1,1.5)[$^{}$`{\ }]
\DGCdot{E}[u]{$a_1$}
\DGCcoupon*(-.75,0.5)(-.25,1){$\cdots$}
\DGCstrand(.25,0)(.75,.4)(.75,.75)(.75,1.1)(.25,1.5)
%\DGCdot{.75}[ul]{$b$}
\DGCstrand[red](0,0)(.5,.5)[$^{}$`{\ }]
\DGCdot{B}[d]{$c_i$}
\DGCstrand[red](1,0)(.5,.5)[$^{}$`{\ }]
\DGCstrand[Red](.5,.5)(.5,1)
\DGCstrand[red](.5,0)(.5,1.5)[$^{}$`{\ }]
\DGCdot{B}[d]{$c_{i+1}$}
\DGCdot{E}[u]{$a_{i+1}$}
\DGCstrand[red](.5,1)(0,1.5)
\DGCdot{E}[ul]{$a_i$}
\DGCstrand[red](.5,1)(1,1.5)
\DGCdot{E}[ur]{$a_{i+2}$}
\DGCcoupon*(1.25,0.5)(1.75,1){$\cdots$}
\DGCstrand[red](2,0)(2,1.5)[$^{}$`{\ }]
\DGCdot{E}[u]{$a_n$}
\end{DGCpicture}
\quad \quad 
c_i=0, c_{i+1}=0,1 ,
\end{equation*}

\begin{equation*}
\beth_3({\bf a},b,c_i, c_{i+1})
~=~
\begin{DGCpicture}
\DGCstrand[red](-1,0)(-1,1.5)[$^{}$`{\ }]
\DGCdot{E}[u]{$a_1$}
\DGCcoupon*(-.75,0.5)(-.25,1){$\cdots$}
\DGCstrand(.75,0)(1,.75)(.75,1.5)
\DGCdot{.75}[l]{$b$}
\DGCstrand[red](0,0)(.5,.5)[$^{}$`{\ }]
\DGCdot{B}[d]{$c_i$}
\DGCstrand[red](1,0)(.5,.5)[$^{}$`{\ }]
\DGCstrand[Red](.5,.5)(.5,1)
\DGCstrand[red](.5,0)(.5,1.5)[$^{}$`{\ }]
\DGCdot{B}[d]{$c_{i+1}$}
\DGCdot{E}[u]{$a_{i+1}$}
\DGCstrand[red](.5,1)(0,1.5)
\DGCdot{E}[ul]{$a_i$}
\DGCstrand[red](.5,1)(1,1.5)
\DGCdot{E}[ur]{$a_{i+2}$}
\DGCcoupon*(1.25,0.5)(1.75,1){$\cdots$}
\DGCstrand[red](2,0)(2,1.5)[$^{}$`{\ }]
\DGCdot{E}[u]{$a_n$}
\end{DGCpicture}
\quad \quad 
b \geq 0, c_i=0,1,2, c_{i+1}=0,1 ,
\end{equation*}

\begin{equation*}
\beth_4({\bf a},c_i, c_{i+1})
~=~
\begin{DGCpicture}
\DGCstrand[red](-1,0)(-1,1.5)[$^{}$`{\ }]
\DGCdot{E}[u]{$a_1$}
\DGCcoupon*(-.75,0.5)(-.25,1){$\cdots$}
\DGCstrand(.75,0)(.25,.4)(.25,.75)(.25,1.1)(.75,1.5)
%\DGCdot{.75}[ul]{$b$}
\DGCstrand[red](0,0)(.5,.5)[$^{}$`{\ }]
\DGCdot{B}[d]{$c_i$}
\DGCstrand[red](1,0)(.5,.5)[$^{}$`{\ }]
\DGCstrand[Red](.5,.5)(.5,1)
\DGCstrand[red](.5,0)(.5,1.5)[$^{}$`{\ }]
\DGCdot{B}[d]{$c_{i+1}$}
\DGCdot{E}[u]{$a_{i+1}$}
\DGCstrand[red](.5,1)(0,1.5)
\DGCdot{E}[ul]{$a_i$}
\DGCstrand[red](.5,1)(1,1.5)
\DGCdot{E}[ur]{$a_{i+2}$}
\DGCcoupon*(1.25,0.5)(1.75,1){$\cdots$}
\DGCstrand[red](2,0)(2,1.5)[$^{}$`{\ }]
\DGCdot{E}[u]{$a_n$}
\end{DGCpicture}
\quad \quad 
c_i=0, c_{i+1}=0,1 ,
\end{equation*}

\begin{equation*}
\beth_5({\bf a},c_i, c_{i+1})
~=~
\begin{DGCpicture}
\DGCstrand[red](-1,0)(-1,1.5)[$^{}$`{\ }]
\DGCdot{E}[u]{$a_1$}
\DGCcoupon*(-.75,0.5)(-.25,1){$\cdots$}
\DGCstrand(.25,0)(.25,1.25)(.75,1.5)
\DGCstrand[red](0,0)(.5,.5)[$^{}$`{\ }]
\DGCdot{B}[d]{$c_i$}
\DGCstrand[red](1,0)(.5,.5)[$^{}$`{\ }]
\DGCstrand[Red](.5,.5)(.5,1)
\DGCstrand[red](.5,0)(.5,1.5)[$^{}$`{\ }]
\DGCdot{B}[d]{$c_{i+1}$}
\DGCdot{E}[u]{$a_{i+1}$}
\DGCstrand[red](.5,1)(0,1.5)
\DGCdot{E}[ul]{$a_i$}
\DGCstrand[red](.5,1)(1,1.5)
\DGCdot{E}[ur]{$a_{i+2}$}
\DGCcoupon*(1.25,0.5)(1.75,1){$\cdots$}
\DGCstrand[red](2,0)(2,1.5)[$^{}$`{\ }]
\DGCdot{E}[u]{$a_n$}
\end{DGCpicture}
\quad \quad 
(c_i, c_{i+1})=(0,0), (1,0), (0,1), (1,1) ,
\end{equation*}

\begin{equation*}
\beth_6({\bf a},b,c_i, c_{i+1})
~=~
\begin{DGCpicture}
\DGCstrand[red](-1,0)(-1,1.5)[$^{}$`{\ }]
\DGCdot{E}[u]{$a_1$}
\DGCcoupon*(-.75,0.5)(-.25,1){$\cdots$}
\DGCstrand(.25,0)(.75,.25)(.75,1.5)
\DGCdot{.75}[ur]{$b$}
\DGCstrand[red](0,0)(.5,.5)[$^{}$`{\ }]
\DGCdot{B}[d]{$c_i$}
\DGCstrand[red](1,0)(.5,.5)[$^{}$`{\ }]
\DGCstrand[Red](.5,.5)(.5,1)
\DGCstrand[red](.5,0)(.5,1.5)[$^{}$`{\ }]
\DGCdot{B}[d]{$c_{i+1}$}
\DGCdot{E}[u]{$a_{i+1}$}
\DGCstrand[red](.5,1)(0,1.5)
\DGCdot{E}[ul]{$a_i$}
\DGCstrand[red](.5,1)(1,1.5)
\DGCdot{E}[ur]{$a_{i+2}$}
\DGCcoupon*(1.25,0.5)(1.75,1){$\cdots$}
\DGCstrand[red](2,0)(2,1.5)[$^{}$`{\ }]
\DGCdot{E}[u]{$a_n$}
\end{DGCpicture}
\quad \quad
b \geq 0, c_i=0,1,2, c_{i+1}=0,1 ,
\end{equation*}

\begin{equation*}
\beth_7({\bf a},b,c_i, c_{i+1})
~=~
\begin{DGCpicture}
\DGCstrand[red](-1,0)(-1,1.5)[$^{}$`{\ }]
\DGCdot{E}[u]{$a_1$}
\DGCcoupon*(-.75,0.5)(-.25,1){$\cdots$}
%\DGCstrand(.75,0)(.75,1.25)(.25,1.5)
%\DGCdot{.75}[ur]{$b$}
\DGCstrand[red](0,0)(.5,.5)[$^{}$`{\ }]
\DGCdot{B}[d]{$c_i$}
\DGCstrand[red](1,0)(.5,.5)[$^{}$`{\ }]
\DGCstrand[Red](.5,.5)(.5,1)
\DGCstrand[red](.5,0)(.5,1.5)[$^{}$`{\ }]
\DGCdot{B}[d]{$c_{i+1}$}
\DGCdot{E}[u]{$a_{i+1}$}
\DGCstrand[red](.5,1)(0,1.5)
\DGCdot{E}[ul]{$a_i$}
\DGCstrand[red](.5,1)(1,1.5)
\DGCdot{E}[ur]{$a_{i+2}$}
\DGCcoupon*(1.25,0.5)(1.75,1){$\cdots$}
\DGCstrand[red](2,0)(2,1.5)[$^{}$`{\ }]
\DGCdot{E}[u]{$a_n$}
\DGCstrand(.75,0)(.25,.25)(.25,1.5)
\DGCdot{.75}[ul]{$b$}
\end{DGCpicture}
\quad \quad 
b \geq 0, c_i=0,1,2, c_{i+1}=0,1 ,
\end{equation*}

\begin{equation*}
\beth_8({\bf a},c_i, c_{i+1})
~=~
\begin{DGCpicture}
\DGCstrand[red](-1,0)(-1,1.5)[$^{}$`{\ }]
\DGCdot{E}[u]{$a_1$}
\DGCcoupon*(-.75,0.5)(-.25,1){$\cdots$}
\DGCstrand(.75,0)(.75,1.25)(.25,1.5)
%\DGCdot{.75}[ur]{$b$}
\DGCstrand[red](0,0)(.5,.5)[$^{}$`{\ }]
\DGCdot{B}[d]{$c_i$}
\DGCstrand[red](1,0)(.5,.5)[$^{}$`{\ }]
\DGCstrand[Red](.5,.5)(.5,1)
\DGCstrand[red](.5,0)(.5,1.5)[$^{}$`{\ }]
\DGCdot{B}[d]{$c_{i+1}$}
\DGCdot{E}[u]{$a_{i+1}$}
\DGCstrand[red](.5,1)(0,1.5)
\DGCdot{E}[ul]{$a_i$}
\DGCstrand[red](.5,1)(1,1.5)
\DGCdot{E}[ur]{$a_{i+2}$}
\DGCcoupon*(1.25,0.5)(1.75,1){$\cdots$}
\DGCstrand[red](2,0)(2,1.5)[$^{}$`{\ }]
\DGCdot{E}[u]{$a_n$}
\end{DGCpicture}
\quad \quad 
(c_i, c_{i+1})=(0,0), (1,0), (0,1), (2,0) ,
\end{equation*}

\begin{equation*}
\beth_9({\bf a},b,c_i, c_{i+1},j,\ell)
~=~
\begin{DGCpicture}
%\DGCcoupon*(-2.75,.5)(-2.25,1){$\cdots$}
\DGCstrand[red](-1.5,0)(-1.5,1.5)[$^{}$`{\ }]
\DGCdot{E}[u]{$a_1$}
\DGCcoupon*(-1.5,1)(0,1.5){$\cdots$}
\DGCstrand(-1,0)(-1,.25)(2,1.25)(2,1.5)
\DGCdot{E}[u]{$b$}
\DGCstrand[red](0,0)(.5,.5)[$^{}$`{\ }]
\DGCdot{B}[d]{$c_i$}
\DGCstrand[red](1,0)(.5,.5)[$^{}$`{\ }]
\DGCstrand[Red](.5,.5)(.5,1)
\DGCstrand[red](.5,0)(.5,1.5)[$^{}$`{\ }]
\DGCdot{B}[d]{$c_{i+1}$}
\DGCdot{E}[u]{$a_{i+1}$}
\DGCstrand[red](.5,1)(0,1.5)
\DGCdot{E}[ul]{$a_i$}
\DGCstrand[red](.5,1)(1,1.5)
\DGCdot{E}[ur]{$a_{i+2}$}
\DGCcoupon*(1,0)(2.5,.5){$\cdots$}
\DGCstrand[red](2.5,0)(2.5,1.5)[$^{}$`{\ }]
\DGCdot{E}[u]{$a_n$}
%\DGCcoupon*(3.25,.5)(3.75,1){$\cdots$}
\end{DGCpicture}
\quad \quad 
\begin{array}{l}b \geq 0, c_i=0,1,2, c_{i+1}=0,1,\\ 0\leq j,\ell\leq n\\
%\text{ unless }(j,\ell)=(i,i),(i,i+1),(i+1,i),(i+1,i+1) ,
\text{ unless }\{j,\ell\} \subset \{i,i+1\} ,
\end{array}
\end{equation*}
%\begin{equation*}
%X_{10}((a_i),b,c_i, c_{i+1})
%~=~
%\begin{DGCpicture}
%\DGCcoupon*(-2.75,.5)(-2.25,1){$\cdots$}
%\DGCstrand[red](-2,0)(-2,1.5)[$^{}$`{\ }]
%\DGCdot{E}[u]{$a_k$}
%\DGCcoupon*(-1.75,1)(-1.25,1.5){$\cdots$}
%\DGCstrand(3,0)(1,.75)(-1,1.5)
%\DGCdot{E}[u]{$b$}
%\DGCstrand[red](0,0)(.5,.5)[$^{}$`{\ }]
%\DGCdot{B}[d]{$c_i$}
%\DGCstrand[red](1,0)(.5,.5)[$^{}$`{\ }]
%\DGCstrand[Red](.5,.5)(.5,1)
%\DGCstrand[red](.5,0)(.5,1.5)[$^{}$`{\ }]
%\DGCdot{B}[d]{$c_{i+1}$}
%\DGCdot{E}[u]{$a_{i+1}$}
%\DGCstrand[red](.5,1)(0,1.5)
%\DGCdot{E}[ul]{$a_i$}
%\DGCstrand[red](.5,1)(1,1.5)
%\DGCdot{E}[ur]{$a_{i+2}$}
%\DGCcoupon*(1.25,0)(1.75,.5){$\cdots$}
%\DGCstrand[red](2,0)(2,1.5)[$^{}$`{\ }]
%\DGCdot{E}[u]{$a_l$}
%\DGCcoupon*(3.25,.5)(3.75,1){$\cdots$}
%\end{DGCpicture}
%\quad \quad 
%b \geq 0, c_i=0,1,2, c_{i+1}=0,1
%\end{equation*}
where the diagram in $\beth_9({\bf a},b,c_i,c_{i+1},j,\ell)$ represents any picture where the black strand begins after the $j$th red strand and ends after the $\ell$th red strand 
%outside the segments connecting red points $i$, $i+1$, and $i+2$ 
and has a minimal number of intersections with the red strands.
%where the diagram in $\beth_9({\bf a},b,c_i,c_{i+1},j,\ell)$ represents any picture where the black strand begins from the $(j+1)$-th strand or ends to the $(\ell+1)$-th strand outside the segments connecting red points $i$, $i+1$, and $i+2$ and has a minimal number of intersections with the red strands.
\end{prop}

\begin{proof}
Once again, we highlight features of the proof that only involve cases where the black strand begins and ends in the segments connecting red boundary points connected to the thick red strands.
First note that by Lemma \ref{classical}, we could assume that the red dots appear in the following configuration
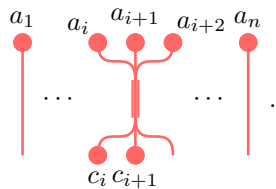
\begin{equation*}
\begin{DGCpicture}
\DGCstrand[red](-1,0)(-1,1.5)[$^{}$`{\ }]
\DGCdot{E}[u]{$a_1$}
\DGCcoupon*(-.75,0.5)(-.25,1){$\cdots$}
%\DGCstrand(.25,0)(0,.75)(.25,1.5)
%\DGCdot{.75}[ul]{$b$}
\DGCstrand[red](0,0)(.5,.5)[$^{}$`{\ }]
\DGCdot{B}[d]{$c_i$}
\DGCstrand[red](1,0)(.5,.5)[$^{}$`{\ }]
\DGCstrand[Red](.5,.5)(.5,1)
\DGCstrand[red](.5,0)(.5,1.5)[$^{}$`{\ }]
\DGCdot{B}[d]{$c_{i+1}$}
\DGCdot{E}[u]{$a_{i+1}$}
\DGCstrand[red](.5,1)(0,1.5)
\DGCdot{E}[ul]{$a_i$}
\DGCstrand[red](.5,1)(1,1.5)
\DGCdot{E}[ur]{$a_{i+2}$}
\DGCcoupon*(1.25,0.5)(1.75,1){$\cdots$}
\DGCstrand[red](2,0)(2,1.5)[$^{}$`{\ }]
\DGCdot{E}[u]{$a_n$}
\end{DGCpicture}
\ .
\end{equation*}
where all $a_i \in \mathbb{Z}_{\geq 0}$, $c_i \in \{0,1,2\}$, and $c_{i+1} \in \{0,1\}$.

Next we will explain why we could assume that in \eqref{dotreduction1}
that there are no black dots.
\begin{equation}
\label{dotreduction1}
\begin{DGCpicture}
\DGCstrand[red](-1,0)(-1,1.5)[$^{}$`{\ }]
\DGCdot{E}[u]{$a_1$}
\DGCcoupon*(-.75,0.5)(-.25,1){$\cdots$}
\DGCstrand(.25,0)(.75,.4)(.75,.75)(.75,1.1)(.25,1.5)
%\DGCdot{.75}[ul]{$b$}
\DGCstrand[red](0,0)(.5,.5)[$^{}$`{\ }]
\DGCdot{B}[d]{$c_i$}
\DGCstrand[red](1,0)(.5,.5)[$^{}$`{\ }]
\DGCstrand[Red](.5,.5)(.5,1)
\DGCstrand[red](.5,0)(.5,1.5)[$^{}$`{\ }]
\DGCdot{B}[d]{$c_{i+1}$}
\DGCdot{E}[u]{$a_{i+1}$}
\DGCstrand[red](.5,1)(0,1.5)
\DGCdot{E}[ul]{$a_i$}
\DGCstrand[red](.5,1)(1,1.5)
\DGCdot{E}[ur]{$a_{i+2}$}
\DGCcoupon*(1.25,0.5)(1.75,1){$\cdots$}
\DGCstrand[red](2,0)(2,1.5)[$^{}$`{\ }]
\DGCdot{E}[u]{$a_n$}
\end{DGCpicture}
\end{equation}
Using relations we have
\begin{equation}
\label{dotreduction2}
\begin{DGCpicture}
\DGCstrand(.25,0)(.75,.4)(.75,.75)(.75,1.1)(.25,1.5)
\DGCdot{B}[d]{$ $}
\DGCstrand[red](0,0)(.5,.5)[$^{}$`{\ }]
%\DGCdot{B}[d]{$ $}
\DGCstrand[red](1,0)(.5,.5)[$^{}$`{\ }]
\DGCstrand[Red](.5,.5)(.5,1)
\DGCstrand[red](.5,0)(.5,1.5)[$^{}$`{\ }]
%\DGCdot{B}[d]{$ $}
%\DGCdot{E}[u]{$ $}
\DGCstrand[red](.5,1)(0,1.5)
%\DGCdot{E}[ul]{$ $}
\DGCstrand[red](.5,1)(1,1.5)
%\DGCdot{E}[ur]{$ $}
\end{DGCpicture}
~=~
\begin{DGCpicture}
\DGCstrand(.25,0)(.75,.4)(.75,.75)(.75,1.1)(.25,1.5)
%\DGCdot{B}[d]{$ $}
\DGCstrand[red](0,0)(.5,.5)[$^{}$`{\ }]
\DGCdot{B}[d]{$ $}
\DGCstrand[red](1,0)(.5,.5)[$^{}$`{\ }]
\DGCstrand[Red](.5,.5)(.5,1)
\DGCstrand[red](.5,0)(.5,1.5)[$^{}$`{\ }]
%\DGCdot{B}[d]{$ $}
%\DGCdot{E}[u]{$ $}
\DGCstrand[red](.5,1)(0,1.5)
%\DGCdot{E}[ul]{$ $}
\DGCstrand[red](.5,1)(1,1.5)
%\DGCdot{E}[ur]{$ $}
\end{DGCpicture}
~+~
\begin{DGCpicture}
\DGCstrand(.25,0)(0,.75)(.25,1.5)
\DGCdot{E}[u]{$ 2 $}
\DGCstrand[red](0,0)(.5,.5)[$^{}$`{\ }]
%\DGCdot{B}[d]{$ $}
\DGCstrand[red](1,0)(.5,.5)[$^{}$`{\ }]
\DGCstrand[Red](.5,.5)(.5,1)
\DGCstrand[red](.5,0)(.5,1.5)[$^{}$`{\ }]
%\DGCdot{B}[d]{$ $}
%\DGCdot{E}[u]{$ $}
\DGCstrand[red](.5,1)(0,1.5)
%\DGCdot{E}[ul]{$ $}
\DGCstrand[red](.5,1)(1,1.5)
%\DGCdot{E}[ur]{$ $}
\end{DGCpicture}
~-~
\begin{DGCpicture}
\DGCstrand(.25,0)(0,.75)(.25,1.5)
\DGCdot{E}[u]{$ 1 $}
\DGCstrand[red](0,0)(.5,.5)[$^{}$`{\ }]
%\DGCdot{B}[d]{$ $}
\DGCstrand[red](1,0)(.5,.5)[$^{}$`{\ }]
\DGCstrand[Red](.5,.5)(.5,1)
\DGCstrand[red](.5,0)(.5,1.5)[$^{}$`{\ }]
\DGCdot{E}[u]{$ $}
\DGCstrand[red](.5,1)(0,1.5)
%\DGCdot{E}[u]{$ $}
\DGCstrand[red](.5,1)(1,1.5)
%\DGCdot{E}[ur]{$ $}
\end{DGCpicture}
~-~
\begin{DGCpicture}
\DGCstrand(.25,0)(0,.75)(.25,1.5)
\DGCdot{E}[u]{$ 1 $}
\DGCstrand[red](0,0)(.5,.5)[$^{}$`{\ }]
%\DGCdot{B}[d]{$ $}
\DGCstrand[red](1,0)(.5,.5)[$^{}$`{\ }]
\DGCstrand[Red](.5,.5)(.5,1)
\DGCstrand[red](.5,0)(.5,1.5)[$^{}$`{\ }]
%\DGCdot{E}[u]{$ $}
\DGCstrand[red](.5,1)(0,1.5)
%\DGCdot{E}[u]{$ $}
\DGCstrand[red](.5,1)(1,1.5)
\DGCdot{E}[u]{$ $}
\end{DGCpicture}
~+~
\begin{DGCpicture}
\DGCstrand(.25,0)(0,.75)(.25,1.5)
%\DGCdot{E}[u]{$ 1 $}
\DGCstrand[red](0,0)(.5,.5)[$^{}$`{\ }]
%\DGCdot{B}[d]{$ $}
\DGCstrand[red](1,0)(.5,.5)[$^{}$`{\ }]
\DGCstrand[Red](.5,.5)(.5,1)
\DGCstrand[red](.5,0)(.5,1.5)[$^{}$`{\ }]
\DGCdot{E}[u]{$ $}
\DGCstrand[red](.5,1)(0,1.5)
%\DGCdot{E}[u]{$ $}
\DGCstrand[red](.5,1)(1,1.5)
\DGCdot{E}[u]{$ $}
\end{DGCpicture}
\ ,
\end{equation}
so a black dot produces a linear combination of other elements in the proposed spanning set.

Next we will explain why we may assume that $c_i=0$ in \eqref{dotreduction1}.
As in \eqref{dotreduction2} we have a relation
\begin{equation}
\label{dotreduction3}
\begin{DGCpicture}
\DGCstrand(.25,0)(.75,.4)(.75,.75)(.75,1.1)(.25,1.5)
\DGCdot{E}[d]{$ $}
\DGCstrand[red](0,0)(.5,.5)[$^{}$`{\ }]
%\DGCdot{B}[d]{$ $}
\DGCstrand[red](1,0)(.5,.5)[$^{}$`{\ }]
\DGCstrand[Red](.5,.5)(.5,1)
\DGCstrand[red](.5,0)(.5,1.5)[$^{}$`{\ }]
%\DGCdot{B}[d]{$ $}
%\DGCdot{E}[u]{$ $}
\DGCstrand[red](.5,1)(0,1.5)
%\DGCdot{E}[ul]{$ $}
\DGCstrand[red](.5,1)(1,1.5)
%\DGCdot{E}[ur]{$ $}
\end{DGCpicture}
~=~
\begin{DGCpicture}
\DGCstrand(.25,0)(.75,.4)(.75,.75)(.75,1.1)(.25,1.5)
%\DGCdot{B}[d]{$ $}
\DGCstrand[red](0,0)(.5,.5)[$^{}$`{\ }]
%\DGCdot{B}[d]{$ $}
\DGCstrand[red](1,0)(.5,.5)[$^{}$`{\ }]
%\DGCdot{B}[u]{$ $}
\DGCstrand[Red](.5,.5)(.5,1)
\DGCstrand[red](.5,0)(.5,1.5)[$^{}$`{\ }]
%\DGCdot{B}[d]{$ $}
%\DGCdot{E}[u]{$ $}
\DGCstrand[red](.5,1)(0,1.5)
\DGCdot{E}[u]{$ $}
\DGCstrand[red](.5,1)(1,1.5)
%\DGCdot{E}[ur]{$ $}
\end{DGCpicture}
~+~
\begin{DGCpicture}
\DGCstrand(.25,0)(0,.75)(.25,1.5)
\DGCdot{B}[d]{$ 2 $}
\DGCstrand[red](0,0)(.5,.5)[$^{}$`{\ }]
%\DGCdot{B}[d]{$ $}
\DGCstrand[red](1,0)(.5,.5)[$^{}$`{\ }]
\DGCstrand[Red](.5,.5)(.5,1)
\DGCstrand[red](.5,0)(.5,1.5)[$^{}$`{\ }]
%\DGCdot{B}[d]{$ $}
%\DGCdot{E}[u]{$ $}
\DGCstrand[red](.5,1)(0,1.5)
%\DGCdot{E}[ul]{$ $}
\DGCstrand[red](.5,1)(1,1.5)
%\DGCdot{E}[ur]{$ $}
\end{DGCpicture}
~-~
\begin{DGCpicture}
\DGCstrand(.25,0)(0,.75)(.25,1.5)
\DGCdot{B}[d]{$ 1 $}
\DGCstrand[red](0,0)(.5,.5)[$^{}$`{\ }]
%\DGCdot{B}[d]{$ $}
\DGCstrand[red](1,0)(.5,.5)[$^{}$`{\ }]
\DGCstrand[Red](.5,.5)(.5,1)
\DGCstrand[red](.5,0)(.5,1.5)[$^{}$`{\ }]
\DGCdot{B}[d]{$ $}
\DGCstrand[red](.5,1)(0,1.5)
%\DGCdot{E}[u]{$ $}
\DGCstrand[red](.5,1)(1,1.5)
%\DGCdot{E}[ur]{$ $}
\end{DGCpicture}
~-~
\begin{DGCpicture}
\DGCstrand(.25,0)(0,.75)(.25,1.5)
\DGCdot{B}[d]{$ 1 $}
\DGCstrand[red](0,0)(.5,.5)[$^{}$`{\ }]
%\DGCdot{B}[d]{$ $}
\DGCstrand[red](1,0)(.5,.5)[$^{}$`{\ }]
\DGCdot{B}[d]{$ $}
\DGCstrand[Red](.5,.5)(.5,1)
\DGCstrand[red](.5,0)(.5,1.5)[$^{}$`{\ }]
%\DGCdot{B}[d]{$ $}
\DGCstrand[red](.5,1)(0,1.5)
%\DGCdot{E}[u]{$ $}
\DGCstrand[red](.5,1)(1,1.5)
%\DGCdot{E}[u]{$ $}
\end{DGCpicture}
~+~
\begin{DGCpicture}
\DGCstrand(.25,0)(0,.75)(.25,1.5)
%\DGCdot{E}[u]{$ 1 $}
\DGCstrand[red](0,0)(.5,.5)[$^{}$`{\ }]
%\DGCdot{B}[d]{$ $}
\DGCstrand[red](1,0)(.5,.5)[$^{}$`{\ }]
\DGCdot{B}[d]{$ $}
\DGCstrand[Red](.5,.5)(.5,1)
\DGCstrand[red](.5,0)(.5,1.5)[$^{}$`{\ }]
\DGCdot{B}[d]{$ $}
\DGCstrand[red](.5,1)(0,1.5)
%\DGCdot{E}[u]{$ $}
\DGCstrand[red](.5,1)(1,1.5)
%\DGCdot{E}[u]{$ $}
\end{DGCpicture}
\ .
\end{equation}
Subtracting \eqref{dotreduction3} from \eqref{dotreduction2} yields
\begin{equation*}
\begin{DGCpicture}
\DGCstrand(.25,0)(.75,.4)(.75,.75)(.75,1.1)(.25,1.5)
%\DGCdot{B}[d]{$ $}
\DGCstrand[red](0,0)(.5,.5)[$^{}$`{\ }]
\DGCdot{B}[d]{$ $}
\DGCstrand[red](1,0)(.5,.5)[$^{}$`{\ }]
\DGCstrand[Red](.5,.5)(.5,1)
\DGCstrand[red](.5,0)(.5,1.5)[$^{}$`{\ }]
%\DGCdot{B}[d]{$ $}
%\DGCdot{E}[u]{$ $}
\DGCstrand[red](.5,1)(0,1.5)
%\DGCdot{E}[ul]{$ $}
\DGCstrand[red](.5,1)(1,1.5)
%\DGCdot{E}[ur]{$ $}
\end{DGCpicture}
~=~
\begin{DGCpicture}
\DGCstrand(.25,0)(.75,.4)(.75,.75)(.75,1.1)(.25,1.5)
%\DGCdot{B}[d]{$ $}
\DGCstrand[red](0,0)(.5,.5)[$^{}$`{\ }]
%\DGCdot{B}[d]{$ $}
\DGCstrand[red](1,0)(.5,.5)[$^{}$`{\ }]
\DGCstrand[Red](.5,.5)(.5,1)
\DGCstrand[red](.5,0)(.5,1.5)[$^{}$`{\ }]
%\DGCdot{B}[d]{$ $}
%\DGCdot{E}[u]{$ $}
\DGCstrand[red](.5,1)(0,1.5)
\DGCdot{E}[u]{$ $}
\DGCstrand[red](.5,1)(1,1.5)
%\DGCdot{E}[ur]{$ $}
\end{DGCpicture}
%%%%%%%%%%%
~+~
\begin{DGCpicture}
\DGCstrand(.25,0)(0,.75)(.25,1.5)
\DGCdot{E}[u]{$ 1 $}
\DGCstrand[red](0,0)(.5,.5)[$^{}$`{\ }]
%\DGCdot{B}[d]{$ $}
\DGCstrand[red](1,0)(.5,.5)[$^{}$`{\ }]
\DGCstrand[Red](.5,.5)(.5,1)
\DGCstrand[red](.5,0)(.5,1.5)[$^{}$`{\ }]
\DGCdot{E}[u]{$ $}
\DGCstrand[red](.5,1)(0,1.5)
%\DGCdot{E}[u]{$ $}
\DGCstrand[red](.5,1)(1,1.5)
%\DGCdot{E}[ur]{$ $}
\end{DGCpicture}
~+~
\begin{DGCpicture}
\DGCstrand(.25,0)(0,.75)(.25,1.5)
\DGCdot{E}[u]{$ 1 $}
\DGCstrand[red](0,0)(.5,.5)[$^{}$`{\ }]
%\DGCdot{B}[d]{$ $}
\DGCstrand[red](1,0)(.5,.5)[$^{}$`{\ }]
\DGCstrand[Red](.5,.5)(.5,1)
\DGCstrand[red](.5,0)(.5,1.5)[$^{}$`{\ }]
%\DGCdot{E}[u]{$ $}
\DGCstrand[red](.5,1)(0,1.5)
%\DGCdot{E}[u]{$ $}
\DGCstrand[red](.5,1)(1,1.5)
\DGCdot{E}[u]{$ $}
\end{DGCpicture}
~-~
\begin{DGCpicture}
\DGCstrand(.25,0)(0,.75)(.25,1.5)
\DGCdot{B}[d]{$ 1 $}
\DGCstrand[red](0,0)(.5,.5)[$^{}$`{\ }]
%\DGCdot{B}[d]{$ $}
\DGCstrand[red](1,0)(.5,.5)[$^{}$`{\ }]
\DGCstrand[Red](.5,.5)(.5,1)
\DGCstrand[red](.5,0)(.5,1.5)[$^{}$`{\ }]
\DGCdot{B}[d]{$ $}
\DGCstrand[red](.5,1)(0,1.5)
%\DGCdot{E}[u]{$ $}
\DGCstrand[red](.5,1)(1,1.5)
%\DGCdot{E}[ur]{$ $}
\end{DGCpicture}
~-~
\begin{DGCpicture}
\DGCstrand(.25,0)(0,.75)(.25,1.5)
\DGCdot{B}[d]{$ 1 $}
\DGCstrand[red](0,0)(.5,.5)[$^{}$`{\ }]
%\DGCdot{B}[d]{$ $}
\DGCstrand[red](1,0)(.5,.5)[$^{}$`{\ }]
\DGCdot{B}[d]{$ $}
\DGCstrand[Red](.5,.5)(.5,1)
\DGCstrand[red](.5,0)(.5,1.5)[$^{}$`{\ }]
%\DGCdot{B}[d]{$ $}
\DGCstrand[red](.5,1)(0,1.5)
%\DGCdot{E}[u]{$ $}
\DGCstrand[red](.5,1)(1,1.5)
%\DGCdot{E}[u]{$ $}
\end{DGCpicture}
~-~
\begin{DGCpicture}
\DGCstrand(.25,0)(0,.75)(.25,1.5)
%\DGCdot{E}[u]{$ 1 $}
\DGCstrand[red](0,0)(.5,.5)[$^{}$`{\ }]
%\DGCdot{B}[d]{$ $}
\DGCstrand[red](1,0)(.5,.5)[$^{}$`{\ }]
\DGCstrand[Red](.5,.5)(.5,1)
\DGCstrand[red](.5,0)(.5,1.5)[$^{}$`{\ }]
\DGCdot{E}[u]{$ $}
\DGCstrand[red](.5,1)(0,1.5)
%\DGCdot{E}[u]{$ $}
\DGCstrand[red](.5,1)(1,1.5)
\DGCdot{E}[u]{$ $}
\end{DGCpicture}
~+~
\begin{DGCpicture}
\DGCstrand(.25,0)(0,.75)(.25,1.5)
%\DGCdot{E}[u]{$ 1 $}
\DGCstrand[red](0,0)(.5,.5)[$^{}$`{\ }]
%\DGCdot{B}[d]{$ $}
\DGCstrand[red](1,0)(.5,.5)[$^{}$`{\ }]
\DGCdot{B}[d]{$ $}
\DGCstrand[Red](.5,.5)(.5,1)
\DGCstrand[red](.5,0)(.5,1.5)[$^{}$`{\ }]
\DGCdot{B}[d]{$ $}
\DGCstrand[red](.5,1)(0,1.5)
%\DGCdot{E}[u]{$ $}
\DGCstrand[red](.5,1)(1,1.5)
%\DGCdot{E}[u]{$ $}
\end{DGCpicture}
\ .
\end{equation*}
By Lemma \ref{classical}, the third to last and last diagrams above are in the proposed span, and thus
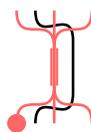
\begin{equation*}
\begin{DGCpicture}
\DGCstrand(.25,0)(.75,.4)(.75,.75)(.75,1.1)(.25,1.5)
%\DGCdot{B}[d]{$ $}
\DGCstrand[red](0,0)(.5,.5)[$^{}$`{\ }]
\DGCdot{B}[d]{$ $}
\DGCstrand[red](1,0)(.5,.5)[$^{}$`{\ }]
\DGCstrand[Red](.5,.5)(.5,1)
\DGCstrand[red](.5,0)(.5,1.5)[$^{}$`{\ }]
%\DGCdot{B}[d]{$ $}
%\DGCdot{E}[u]{$ $}
\DGCstrand[red](.5,1)(0,1.5)
%\DGCdot{E}[ul]{$ $}
\DGCstrand[red](.5,1)(1,1.5)
%\DGCdot{E}[ur]{$ $}
\end{DGCpicture}
\end{equation*}
is in the span of other elements of the proposed spanning set.

Next we will show that we could assume that there are no black dots on diagrams of the form
\begin{equation}
\label{dotreduction4}
\begin{DGCpicture}
\DGCstrand(.75,0)(.75,1.25)(.25,1.5)
%\DGCdot{.75}[ur]{$ $}
\DGCstrand[red](0,0)(.5,.5)[$^{}$`{\ }]
%\DGCdot{B}[d]{$ $}
\DGCstrand[red](1,0)(.5,.5)[$^{}$`{\ }]
\DGCstrand[Red](.5,.5)(.5,1)
\DGCstrand[red](.5,0)(.5,1.5)[$^{}$`{\ }]
%\DGCdot{B}[d]{$ $}
%\DGCdot{E}[u]{$ $}
\DGCstrand[red](.5,1)(0,1.5)
%\DGCdot{E}[ul]{$ $}
\DGCstrand[red](.5,1)(1,1.5)
%\DGCdot{E}[ur]{$ $}
\end{DGCpicture}
\ .
\end{equation}
Using bimodule relations \eqref{R2rel} and \eqref{thickblackthickred1} we have
\begin{equation}
\label{dotreduction6}
\begin{DGCpicture}
\DGCstrand(.75,0)(.75,1.25)(.25,1.5)
\DGCdot{.75}[ur]{$ $}
\DGCstrand[red](0,0)(.5,.5)[$^{}$`{\ }]
%\DGCdot{B}[d]{$ $}
\DGCstrand[red](1,0)(.5,.5)[$^{}$`{\ }]
\DGCstrand[Red](.5,.5)(.5,1)
\DGCstrand[red](.5,0)(.5,1.5)[$^{}$`{\ }]
%\DGCdot{B}[d]{$ $}
%\DGCdot{E}[u]{$ $}
\DGCstrand[red](.5,1)(0,1.5)
%\DGCdot{E}[ul]{$ $}
\DGCstrand[red](.5,1)(1,1.5)
%\DGCdot{E}[ur]{$ $}
\end{DGCpicture}
~=~
\begin{DGCpicture}
\DGCstrand(.75,0)(.75,1.25)(.25,1.5)
%\DGCdot{.75}[ur]{$ $}
\DGCstrand[red](0,0)(.5,.5)[$^{}$`{\ }]
%\DGCdot{B}[d]{$ $}
\DGCstrand[red](1,0)(.5,.5)[$^{}$`{\ }]
\DGCstrand[Red](.5,.5)(.5,1)
\DGCstrand[red](.5,0)(.5,1.5)[$^{}$`{\ }]
%\DGCdot{B}[d]{$ $}
%\DGCdot{E}[u]{$ $}
\DGCstrand[red](.5,1)(0,1.5)
\DGCdot{E}[u]{$ $}
\DGCstrand[red](.5,1)(1,1.5)
%\DGCdot{E}[ur]{$ $}
\end{DGCpicture}
~+~
\begin{DGCpicture}
\DGCstrand(.75,0)(.25,.25)(.25,1.5)
\DGCdot{.75}[ul]{$ $}
\DGCstrand[red](0,0)(.5,.5)[$^{}$`{\ }]
%\DGCdot{B}[d]{$ $}
\DGCstrand[red](1,0)(.5,.5)[$^{}$`{\ }]
\DGCstrand[Red](.5,.5)(.5,1)
\DGCstrand[red](.5,0)(.5,1.5)[$^{}$`{\ }]
%\DGCdot{B}[d]{$ $}
%\DGCdot{E}[u]{$ $}
\DGCstrand[red](.5,1)(0,1.5)
%\DGCdot{E}[ul]{$ $}
\DGCstrand[red](.5,1)(1,1.5)
%\DGCdot{E}[ur]{$ $}
\end{DGCpicture}
~-~
\begin{DGCpicture}
\DGCstrand(.75,0)(.25,.25)(.25,1.5)
%\DGCdot{.75}[ul]{$ $}
\DGCstrand[red](0,0)(.5,.5)[$^{}$`{\ }]
%\DGCdot{B}[d]{$ $}
\DGCstrand[red](1,0)(.5,.5)[$^{}$`{\ }]
\DGCdot{B}[d]{$ $}
\DGCstrand[Red](.5,.5)(.5,1)
\DGCstrand[red](.5,0)(.5,1.5)[$^{}$`{\ }]
%\DGCdot{B}[d]{$ $}
%\DGCdot{E}[u]{$ $}
\DGCstrand[red](.5,1)(0,1.5)
%\DGCdot{B}[d]{$ $}
\DGCstrand[red](.5,1)(1,1.5)
%\DGCdot{E}[ur]{$ $}
\end{DGCpicture}
\ .
\end{equation}
Thus we may assume that there are no black dots on diagrams of the form \eqref{dotreduction4}.

Next we will show that elements of the form
\begin{equation}
\label{dotreduction5}
\begin{DGCpicture}
\DGCstrand(.75,0)(.75,1.25)(.25,1.5)
%\DGCdot{.75}[ur]{$ $}
\DGCstrand[red](0,0)(.5,.5)[$^{}$`{\ }]
\DGCdot{B}[d]{$ $}
\DGCstrand[red](1,0)(.5,.5)[$^{}$`{\ }]
\DGCstrand[Red](.5,.5)(.5,1)
\DGCstrand[red](.5,0)(.5,1.5)[$^{}$`{\ }]
\DGCdot{B}[d]{$ $}
%\DGCdot{E}[u]{$ $}
\DGCstrand[red](.5,1)(0,1.5)
%\DGCdot{E}[ul]{$ $}
\DGCstrand[red](.5,1)(1,1.5)
%\DGCdot{E}[ur]{$ $}
\end{DGCpicture}
\end{equation}
are already in the span of the proposed spanning set.
Using relations \eqref{R2rel}, the left side of \eqref{dotreduction7} below could be rewritten as two double crossings of the black strand and the two bottom left red strands. Relations \eqref{thickblackthickred1}, then allow us to move the black strand up through the thick red strand.  Applying relations \eqref{R2rel} then implies \eqref{dotreduction7}.
\begin{equation}
\label{dotreduction7}
\begin{DGCpicture}
\DGCstrand(.75,0)(.75,1.25)(.25,1.5)
\DGCdot{B}[d]{$ 2 $}
\DGCstrand[red](0,0)(.5,.5)[$^{}$`{\ }]
%\DGCdot{B}[d]{$ $}
\DGCstrand[red](1,0)(.5,.5)[$^{}$`{\ }]
\DGCstrand[Red](.5,.5)(.5,1)
\DGCstrand[red](.5,0)(.5,1.5)[$^{}$`{\ }]
%\DGCdot{B}[d]{$ $}
%\DGCdot{E}[u]{$ $}
\DGCstrand[red](.5,1)(0,1.5)
%\DGCdot{E}[ul]{$ $}
\DGCstrand[red](.5,1)(1,1.5)
%\DGCdot{E}[ur]{$ $}
\end{DGCpicture}
~-~
\begin{DGCpicture}
\DGCstrand(.75,0)(.75,1.25)(.25,1.5)
\DGCdot{.75}[ur]{$ $}
\DGCstrand[red](0,0)(.5,.5)[$^{}$`{\ }]
\DGCdot{B}[d]{$ $}
\DGCstrand[red](1,0)(.5,.5)[$^{}$`{\ }]
\DGCstrand[Red](.5,.5)(.5,1)
\DGCstrand[red](.5,0)(.5,1.5)[$^{}$`{\ }]
%\DGCdot{B}[d]{$ $}
%\DGCdot{E}[u]{$ $}
\DGCstrand[red](.5,1)(0,1.5)
%\DGCdot{E}[ul]{$ $}
\DGCstrand[red](.5,1)(1,1.5)
%\DGCdot{E}[ur]{$ $}
\end{DGCpicture}
~-~
\begin{DGCpicture}
\DGCstrand(.75,0)(.75,1.25)(.25,1.5)
\DGCdot{.75}[ur]{$ $}
\DGCstrand[red](0,0)(.5,.5)[$^{}$`{\ }]
%\DGCdot{B}[d]{$ $}
\DGCstrand[red](1,0)(.5,.5)[$^{}$`{\ }]
\DGCstrand[Red](.5,.5)(.5,1)
\DGCstrand[red](.5,0)(.5,1.5)[$^{}$`{\ }]
\DGCdot{B}[d]{$ $}
%\DGCdot{E}[u]{$ $}
\DGCstrand[red](.5,1)(0,1.5)
%\DGCdot{E}[ul]{$ $}
\DGCstrand[red](.5,1)(1,1.5)
%\DGCdot{E}[ur]{$ $}
\end{DGCpicture}
~+~
\begin{DGCpicture}
\DGCstrand(.75,0)(.75,1.25)(.25,1.5)
%\DGCdot{.75}[ur]{$ $}
\DGCstrand[red](0,0)(.5,.5)[$^{}$`{\ }]
\DGCdot{B}[d]{$ $}
\DGCstrand[red](1,0)(.5,.5)[$^{}$`{\ }]
\DGCstrand[Red](.5,.5)(.5,1)
\DGCstrand[red](.5,0)(.5,1.5)[$^{}$`{\ }]
\DGCdot{B}[d]{$ $}
%\DGCdot{E}[u]{$ $}
\DGCstrand[red](.5,1)(0,1.5)
%\DGCdot{E}[ul]{$ $}
\DGCstrand[red](.5,1)(1,1.5)
%\DGCdot{E}[ur]{$ $}
\end{DGCpicture}
~=~
\begin{DGCpicture}
\DGCstrand(.75,0)(.25,.25)(.25,1.5)
\DGCdot{B}[d]{$2 $}
\DGCstrand[red](0,0)(.5,.5)[$^{}$`{\ }]
%\DGCdot{B}[d]{$ $}
\DGCstrand[red](1,0)(.5,.5)[$^{}$`{\ }]
%\DGCdot{B}[d]{$ $}
\DGCstrand[Red](.5,.5)(.5,1)
\DGCstrand[red](.5,0)(.5,1.5)[$^{}$`{\ }]
%\DGCdot{B}[d]{$ $}
%\DGCdot{E}[u]{$ $}
\DGCstrand[red](.5,1)(0,1.5)
%\DGCdot{B}[d]{$ $}
\DGCstrand[red](.5,1)(1,1.5)
%\DGCdot{E}[ur]{$ $}
\end{DGCpicture}
~-~
\begin{DGCpicture}
\DGCstrand(.75,0)(.25,.25)(.25,1.5)
\DGCdot{.75}[d]{$ $}
\DGCstrand[red](0,0)(.5,.5)[$^{}$`{\ }]
%\DGCdot{B}[d]{$ $}
\DGCstrand[red](1,0)(.5,.5)[$^{}$`{\ }]
%\DGCdot{B}[d]{$ $}
\DGCstrand[Red](.5,.5)(.5,1)
\DGCstrand[red](.5,0)(.5,1.5)[$^{}$`{\ }]
%\DGCdot{B}[d]{$ $}
\DGCdot{E}[u]{$ $}
\DGCstrand[red](.5,1)(0,1.5)
%\DGCdot{B}[d]{$ $}
\DGCstrand[red](.5,1)(1,1.5)
%\DGCdot{E}[ur]{$ $}
\end{DGCpicture}
~-~
\begin{DGCpicture}
\DGCstrand(.75,0)(.25,.25)(.25,1.5)
\DGCdot{.75}[d]{$ $}
\DGCstrand[red](0,0)(.5,.5)[$^{}$`{\ }]
%\DGCdot{B}[d]{$ $}
\DGCstrand[red](1,0)(.5,.5)[$^{}$`{\ }]
%\DGCdot{B}[d]{$ $}
\DGCstrand[Red](.5,.5)(.5,1)
\DGCstrand[red](.5,0)(.5,1.5)[$^{}$`{\ }]
%\DGCdot{B}[d]{$ $}
%\DGCdot{E}[u]{$ $}
\DGCstrand[red](.5,1)(0,1.5)
%\DGCdot{B}[d]{$ $}
\DGCstrand[red](.5,1)(1,1.5)
\DGCdot{E}[ur]{$ $}
\end{DGCpicture}
~+~
\begin{DGCpicture}
\DGCstrand(.75,0)(.25,.25)(.25,1.5)
%\DGCdot{.75}[d]{$ $}
\DGCstrand[red](0,0)(.5,.5)[$^{}$`{\ }]
%\DGCdot{B}[d]{$ $}
\DGCstrand[red](1,0)(.5,.5)[$^{}$`{\ }]
%\DGCdot{B}[d]{$ $}
\DGCstrand[Red](.5,.5)(.5,1)
\DGCstrand[red](.5,0)(.5,1.5)[$^{}$`{\ }]
%\DGCdot{B}[d]{$ $}
\DGCdot{E}[u]{$ $}
\DGCstrand[red](.5,1)(0,1.5)
%\DGCdot{B}[d]{$ $}
\DGCstrand[red](.5,1)(1,1.5)
\DGCdot{E}[ur]{$ $}
\end{DGCpicture}
\ .
\end{equation}
Using \eqref{dotreduction7} and \eqref{dotreduction6} repeatedly, we get
\begin{equation*}
\begin{DGCpicture}
\DGCstrand(.75,0)(.75,1.25)(.25,1.5)
%\DGCdot{.75}[ur]{$ $}
\DGCstrand[red](0,0)(.5,.5)[$^{}$`{\ }]
\DGCdot{B}[d]{$ $}
\DGCstrand[red](1,0)(.5,.5)[$^{}$`{\ }]
\DGCstrand[Red](.5,.5)(.5,1)
\DGCstrand[red](.5,0)(.5,1.5)[$^{}$`{\ }]
\DGCdot{B}[d]{$ $}
%\DGCdot{E}[u]{$ $}
\DGCstrand[red](.5,1)(0,1.5)
%\DGCdot{E}[ul]{$ $}
\DGCstrand[red](.5,1)(1,1.5)
%\DGCdot{E}[ur]{$ $}
\end{DGCpicture}
~=~
\begin{DGCpicture}
\DGCstrand(.75,0)(.75,1.25)(.25,1.5)
%\DGCdot{.75}[ur]{$ $}
\DGCstrand[red](0,0)(.5,.5)[$^{}$`{\ }]
\DGCdot{B}[d]{$ $}
\DGCstrand[red](1,0)(.5,.5)[$^{}$`{\ }]
%\DGCdot{B}[d]{$ $}
\DGCstrand[Red](.5,.5)(.5,1)
\DGCstrand[red](.5,0)(.5,1.5)[$^{}$`{\ }]
%\DGCdot{B}[d]{$ $}
%\DGCdot{E}[u]{$ $}
\DGCstrand[red](.5,1)(0,1.5)
\DGCdot{E}[u]{$ $}
\DGCstrand[red](.5,1)(1,1.5)
%\DGCdot{E}[ur]{$ $}
\end{DGCpicture}
~+~
\begin{DGCpicture}
\DGCstrand(.75,0)(.75,1.25)(.25,1.5)
%\DGCdot{.75}[ur]{$ $}
\DGCstrand[red](0,0)(.5,.5)[$^{}$`{\ }]
%\DGCdot{B}[d]{$ $}
\DGCstrand[red](1,0)(.5,.5)[$^{}$`{\ }]
%\DGCdot{B}[d]{$ $}
\DGCstrand[Red](.5,.5)(.5,1)
\DGCstrand[red](.5,0)(.5,1.5)[$^{}$`{\ }]
\DGCdot{B}[d]{$ $}
%\DGCdot{E}[u]{$ $}
\DGCstrand[red](.5,1)(0,1.5)
\DGCdot{E}[u]{$ $}
\DGCstrand[red](.5,1)(1,1.5)
%\DGCdot{E}[ur]{$ $}
\end{DGCpicture}
~-~
\begin{DGCpicture}
\DGCstrand(.75,0)(.75,1.25)(.25,1.5)
%\DGCdot{.75}[ur]{$ $}
\DGCstrand[red](0,0)(.5,.5)[$^{}$`{\ }]
%\DGCdot{B}[d]{$ $}
\DGCstrand[red](1,0)(.5,.5)[$^{}$`{\ }]
%\DGCdot{B}[d]{$ $}
\DGCstrand[Red](.5,.5)(.5,1)
\DGCstrand[red](.5,0)(.5,1.5)[$^{}$`{\ }]
%\DGCdot{B}[d]{$ $}
%\DGCdot{E}[u]{$ $}
\DGCstrand[red](.5,1)(0,1.5)
\DGCdot{E}[u]{$2 $}
\DGCstrand[red](.5,1)(1,1.5)
%\DGCdot{E}[ur]{$ $}
\end{DGCpicture}
~-~
\begin{DGCpicture}
\DGCstrand(.75,0)(.25,.25)(.25,1.5)
%\DGCdot{.75}[d]{$ $}
\DGCstrand[red](0,0)(.5,.5)[$^{}$`{\ }]
\DGCdot{B}[d]{$ $}
\DGCstrand[red](1,0)(.5,.5)[$^{}$`{\ }]
\DGCdot{B}[d]{$ $}
\DGCstrand[Red](.5,.5)(.5,1)
\DGCstrand[red](.5,0)(.5,1.5)[$^{}$`{\ }]
%\DGCdot{B}[d]{$ $}
%\DGCdot{E}[u]{$ $}
\DGCstrand[red](.5,1)(0,1.5)
%\DGCdot{B}[d]{$ $}
\DGCstrand[red](.5,1)(1,1.5)
%\DGCdot{E}[ur]{$ $}
\end{DGCpicture}
~-~
\begin{DGCpicture}
\DGCstrand(.75,0)(.25,.25)(.25,1.5)
%\DGCdot{.75}[d]{$ $}
\DGCstrand[red](0,0)(.5,.5)[$^{}$`{\ }]
%\DGCdot{B}[d]{$ $}
\DGCstrand[red](1,0)(.5,.5)[$^{}$`{\ }]
\DGCdot{B}[d]{$ $}
\DGCstrand[Red](.5,.5)(.5,1)
\DGCstrand[red](.5,0)(.5,1.5)[$^{}$`{\ }]
\DGCdot{B}[d]{$ $}
%\DGCdot{E}[u]{$ $}
\DGCstrand[red](.5,1)(0,1.5)
%\DGCdot{B}[d]{$ $}
\DGCstrand[red](.5,1)(1,1.5)
%\DGCdot{E}[ur]{$ $}
\end{DGCpicture}
~+~
\begin{DGCpicture}
\DGCstrand(.75,0)(.25,.25)(.25,1.5)
%\DGCdot{.75}[d]{$ $}
\DGCstrand[red](0,0)(.5,.5)[$^{}$`{\ }]
%\DGCdot{B}[d]{$ $}
\DGCstrand[red](1,0)(.5,.5)[$^{}$`{\ }]
\DGCdot{B}[d]{$ $}
\DGCstrand[Red](.5,.5)(.5,1)
\DGCstrand[red](.5,0)(.5,1.5)[$^{}$`{\ }]
%\DGCdot{B}[d]{$ $}
%\DGCdot{E}[u]{$ $}
\DGCstrand[red](.5,1)(0,1.5)
\DGCdot{E}[d]{$ $}
\DGCstrand[red](.5,1)(1,1.5)
%\DGCdot{E}[ur]{$ $}
\end{DGCpicture}
~+~
\begin{DGCpicture}
\DGCstrand(.75,0)(.25,.25)(.25,1.5)
%\DGCdot{.75}[d]{$ $}
\DGCstrand[red](0,0)(.5,.5)[$^{}$`{\ }]
%\DGCdot{B}[d]{$ $}
\DGCstrand[red](1,0)(.5,.5)[$^{}$`{\ }]
%\DGCdot{B}[d]{$ $}
\DGCstrand[Red](.5,.5)(.5,1)
\DGCstrand[red](.5,0)(.5,1.5)[$^{}$`{\ }]
%\DGCdot{B}[d]{$ $}
\DGCdot{E}[u]{$ $}
\DGCstrand[red](.5,1)(0,1.5)
%\DGCdot{E}[d]{$ $}
\DGCstrand[red](.5,1)(1,1.5)
\DGCdot{E}[ur]{$ $}
\end{DGCpicture}
\ .
\end{equation*}
Thus \eqref{dotreduction5} is in the span of the proposed spanning set.

Similar techniques show that the other elements stated in the proposition complete a spanning set.
\end{proof}

In order to prove that the elements in Proposition \ref{W_ii+1span} are actually a basis of $W_{i,i+1}$, we use a $(W(n,1),W(n,1))$-bimodule homomorphism 
$ \gamma_{i,i+1} \colon W_{i,i+1} \rightarrow  \Hom_{\Bbbk}(V_n,V_n)$.

\begin{prop} \label{gammaprop}
There is a bimodule homomorphism $ \gamma_{i,i+1} \colon W_{i,i+1} \rightarrow  \Hom_{\Bbbk}(V_n,V_n)$ inherited from the representation of $W(n,1)$ on $V_n$ determined by
\begin{equation*}
\begin{DGCpicture}
\DGCstrand[red](-1,0)(-1,1.5)[$^{1}$`{\ }]
%\DGCdot{E}[u]{$a_1$}
\DGCcoupon*(-.75,0.5)(-.25,1){$\cdots$}
%\DGCstrand(.5,0)(0,.75)(.5,1.5)
%\DGCdot{E}[u]{$b$}
\DGCstrand[red](0,0)(.5,.5)[$^{i}$`{\ }]
%\DGCdot{B}[l]{$c$}
\DGCstrand[red](1,0)(.5,.5)[$^{i+2}$`{\ }]
\DGCstrand[Red](.5,.5)(.5,1)
\DGCstrand[red](.5,1)(0,1.5)
%\DGCdot{E}[u]{$a_i$}
\DGCstrand[red](.5,1)(1,1.5)
%\DGCdot{E}[u]{$a_{i+1}$}
\DGCcoupon*(1.25,0.5)(1.75,1){$\cdots$}
\DGCstrand[red](2,0)(2,1.5)[$^{n}$`{\ }]
%\DGCdot{E}[u]{$a_n$}
\DGCstrand[red](.5,0)(.5,1.5)[$^{i+1}$`{\ }]
\end{DGCpicture}
~\colon~
f \in V_{n} 
~\mapsto
 D_i D_{i+1} D_i(f) \in V_{n} 
\ .
\end{equation*}
\end{prop}

\begin{proof}
This is a routine calculation.
\end{proof}

The next two results contain important facts used later on, but the proofs are just lengthy computations that the reader interested in a smooth flow of the text could safely ignore.

\begin{prop}
\label{basis2}
The spanning set in Proposition \ref{W_ii+1span} is a basis of $W_{i,i+1}$.
\end{prop}

\begin{proof}
If $\epsilon_1' v_1 \epsilon_1, \epsilon_2' v_2 \epsilon_2 \in W_{i,i+1}$, where $\epsilon_1, \epsilon_1', \epsilon_2, \epsilon_2'$ are idempotents, then there could only be a dependence relation if $\epsilon_1'=\epsilon_2'$ and $\epsilon_1=\epsilon_2$.  Using this fact, there are several cases to check.  We supply the details for two non-trivial cases.

First consider a dependence relation of the form
\begin{align} 
\label{Dep1}
&k^1_{a_1,a_2,a_3,b}
\begin{DGCpicture}
\DGCstrand(.25,0)(0,.75)(.25,1.5)
\DGCdot{.75}[ul]{$b$}
\DGCstrand[red](0,0)(.5,.5)[$^{}$`{\ }]
\DGCdot{B}[d]{$2$}
\DGCstrand[red](1,0)(.5,.5)[$^{}$`{\ }]
\DGCstrand[Red](.5,.5)(.5,1)
\DGCstrand[red](.5,0)(.5,1.5)[$^{}$`{\ }]
\DGCdot{B}[d]{${1}$}
\DGCdot{E}[u]{$a_{2}$}
\DGCstrand[red](.5,1)(0,1.5)
\DGCdot{E}[ul]{$a_1$}
\DGCstrand[red](.5,1)(1,1.5)
\DGCdot{E}[ur]{$a_{3}$}
\end{DGCpicture}
~+~
k^2_{a_1,a_2,a_3,b}
\begin{DGCpicture}
\DGCstrand(.25,0)(0,.75)(.25,1.5)
\DGCdot{.75}[ul]{$b$}
\DGCstrand[red](0,0)(.5,.5)[$^{}$`{\ }]
\DGCdot{B}[d]{$2$}
\DGCstrand[red](1,0)(.5,.5)[$^{}$`{\ }]
\DGCstrand[Red](.5,.5)(.5,1)
\DGCstrand[red](.5,0)(.5,1.5)[$^{}$`{\ }]
%\DGCdot{B}[d]{${1}$}
\DGCdot{E}[u]{$a_{2}$}
\DGCstrand[red](.5,1)(0,1.5)
\DGCdot{E}[ul]{$a_1$}
\DGCstrand[red](.5,1)(1,1.5)
\DGCdot{E}[ur]{$a_{3}$}
\end{DGCpicture}
~+~
k^3_{a_1,a_2,a_3,b}
\begin{DGCpicture}
\DGCstrand(.25,0)(0,.75)(.25,1.5)
\DGCdot{.75}[ul]{$b$}
\DGCstrand[red](0,0)(.5,.5)[$^{}$`{\ }]
\DGCdot{B}[d]{$1$}
\DGCstrand[red](1,0)(.5,.5)[$^{}$`{\ }]
\DGCstrand[Red](.5,.5)(.5,1)
\DGCstrand[red](.5,0)(.5,1.5)[$^{}$`{\ }]
\DGCdot{B}[d]{${1}$}
\DGCdot{E}[u]{$a_{2}$}
\DGCstrand[red](.5,1)(0,1.5)
\DGCdot{E}[ul]{$a_1$}
\DGCstrand[red](.5,1)(1,1.5)
\DGCdot{E}[ur]{$a_{3}$}
\end{DGCpicture} 
~+~
k^4_{a_1,a_2,a_3,b}
\begin{DGCpicture}
\DGCstrand(.25,0)(0,.75)(.25,1.5)
\DGCdot{.75}[ul]{$b$}
\DGCstrand[red](0,0)(.5,.5)[$^{}$`{\ }]
\DGCdot{B}[d]{$1$}
\DGCstrand[red](1,0)(.5,.5)[$^{}$`{\ }]
\DGCstrand[Red](.5,.5)(.5,1)
\DGCstrand[red](.5,0)(.5,1.5)[$^{}$`{\ }]
%\DGCdot{B}[d]{${1}$}
\DGCdot{E}[u]{$a_{2}$}
\DGCstrand[red](.5,1)(0,1.5)
\DGCdot{E}[ul]{$a_1$}
\DGCstrand[red](.5,1)(1,1.5)
\DGCdot{E}[ur]{$a_{3}$}
\end{DGCpicture} \\
~+~
&k^5_{a_1,a_2,a_3,b}
\begin{DGCpicture}
\DGCstrand(.25,0)(0,.75)(.25,1.5)
\DGCdot{.75}[ul]{$b$}
\DGCstrand[red](0,0)(.5,.5)[$^{}$`{\ }]
%\DGCdot{B}[d]{$1$}
\DGCstrand[red](1,0)(.5,.5)[$^{}$`{\ }]
\DGCstrand[Red](.5,.5)(.5,1)
\DGCstrand[red](.5,0)(.5,1.5)[$^{}$`{\ }]
\DGCdot{B}[d]{${1}$}
\DGCdot{E}[u]{$a_{2}$}
\DGCstrand[red](.5,1)(0,1.5)
\DGCdot{E}[ul]{$a_1$}
\DGCstrand[red](.5,1)(1,1.5)
\DGCdot{E}[ur]{$a_{3}$}
\end{DGCpicture}
~+~
k^6_{a_1,a_2,a_3,b}
\begin{DGCpicture}
\DGCstrand(.25,0)(0,.75)(.25,1.5)
\DGCdot{.75}[ul]{$b$}
\DGCstrand[red](0,0)(.5,.5)[$^{}$`{\ }]
%\DGCdot{B}[d]{$1$}
\DGCstrand[red](1,0)(.5,.5)[$^{}$`{\ }]
\DGCstrand[Red](.5,.5)(.5,1)
\DGCstrand[red](.5,0)(.5,1.5)[$^{}$`{\ }]
%\DGCdot{B}[d]{${1}$}
\DGCdot{E}[u]{$a_{2}$}
\DGCstrand[red](.5,1)(0,1.5)
\DGCdot{E}[ul]{$a_1$}
\DGCstrand[red](.5,1)(1,1.5)
\DGCdot{E}[ur]{$a_{3}$}
\end{DGCpicture}
~+~
k^7_{a_1,a_2,a_3}
\begin{DGCpicture}
\DGCstrand(.25,0)(.75,.4)(.75,.75)(.75,1.1)(.25,1.5)
%\DGCdot{.75}[ul]{$b$}
\DGCstrand[red](0,0)(.5,.5)[$^{}$`{\ }]
\DGCstrand[red](1,0)(.5,.5)[$^{}$`{\ }]
\DGCstrand[Red](.5,.5)(.5,1)
\DGCstrand[red](.5,0)(.5,1.5)[$^{}$`{\ }]
\DGCdot{B}[d]{$1$}
\DGCdot{E}[u]{$a_{2}$}
\DGCstrand[red](.5,1)(0,1.5)
\DGCdot{E}[ul]{$a_1$}
\DGCstrand[red](.5,1)(1,1.5)
\DGCdot{E}[ur]{$a_{3}$}
\end{DGCpicture}
~+~
k^8_{a_1,a_2,a_3}
\begin{DGCpicture}
\DGCstrand(.25,0)(.75,.4)(.75,.75)(.75,1.1)(.25,1.5)
%\DGCdot{.75}[ul]{$b$}
\DGCstrand[red](0,0)(.5,.5)[$^{}$`{\ }]
\DGCstrand[red](1,0)(.5,.5)[$^{}$`{\ }]
\DGCstrand[Red](.5,.5)(.5,1)
\DGCstrand[red](.5,0)(.5,1.5)[$^{}$`{\ }]
%\DGCdot{B}[d]{$1$}
\DGCdot{E}[u]{$a_{2}$}
\DGCstrand[red](.5,1)(0,1.5)
\DGCdot{E}[ul]{$a_1$}
\DGCstrand[red](.5,1)(1,1.5)
\DGCdot{E}[ur]{$a_{3}$}
\end{DGCpicture}
~=~
0 \ . \nonumber
\end{align}
Apply the homomorphism $\gamma_{i,i+1}$ to \eqref{Dep1} and evaluate on $1$ to get
\begin{equation*}
k^1_{a_1,a_2,a_3,b} (y_1-x_1) x_1^{a_1} x_2^{a_2} x_3^{a_3} y_1^b
+ 
k^7_{a_1,a_2,a_3} x_1^{a_1} x_2^{a_2} x_3^{a_3}
= 0 \ .
\end{equation*}
This forces
\begin{equation*}
k^1_{a_1,a_2,a_3,b}
= k^7_{a_1,a_2,a_3} = 0 \ .
\end{equation*}
Then the dependence relation becomes
\begin{align} 
\label{Dep2}
&k^2_{a_1,a_2,a_3,b}
\begin{DGCpicture}
\DGCstrand(.25,0)(0,.75)(.25,1.5)
\DGCdot{.75}[ul]{$b$}
\DGCstrand[red](0,0)(.5,.5)[$^{}$`{\ }]
\DGCdot{B}[d]{$2$}
\DGCstrand[red](1,0)(.5,.5)[$^{}$`{\ }]
\DGCstrand[Red](.5,.5)(.5,1)
\DGCstrand[red](.5,0)(.5,1.5)[$^{}$`{\ }]
%\DGCdot{B}[d]{${1}$}
\DGCdot{E}[u]{$a_{2}$}
\DGCstrand[red](.5,1)(0,1.5)
\DGCdot{E}[ul]{$a_1$}
\DGCstrand[red](.5,1)(1,1.5)
\DGCdot{E}[ur]{$a_{3}$}
\end{DGCpicture}
~+~
k^3_{a_1,a_2,a_3,b}
\begin{DGCpicture}
\DGCstrand(.25,0)(0,.75)(.25,1.5)
\DGCdot{.75}[ul]{$b$}
\DGCstrand[red](0,0)(.5,.5)[$^{}$`{\ }]
\DGCdot{B}[d]{$1$}
\DGCstrand[red](1,0)(.5,.5)[$^{}$`{\ }]
\DGCstrand[Red](.5,.5)(.5,1)
\DGCstrand[red](.5,0)(.5,1.5)[$^{}$`{\ }]
\DGCdot{B}[d]{${1}$}
\DGCdot{E}[u]{$a_{2}$}
\DGCstrand[red](.5,1)(0,1.5)
\DGCdot{E}[ul]{$a_1$}
\DGCstrand[red](.5,1)(1,1.5)
\DGCdot{E}[ur]{$a_{3}$}
\end{DGCpicture} 
~+~
k^4_{a_1,a_2,a_3,b}
\begin{DGCpicture}
\DGCstrand(.25,0)(0,.75)(.25,1.5)
\DGCdot{.75}[ul]{$b$}
\DGCstrand[red](0,0)(.5,.5)[$^{}$`{\ }]
\DGCdot{B}[d]{$1$}
\DGCstrand[red](1,0)(.5,.5)[$^{}$`{\ }]
\DGCstrand[Red](.5,.5)(.5,1)
\DGCstrand[red](.5,0)(.5,1.5)[$^{}$`{\ }]
%\DGCdot{B}[d]{${1}$}
\DGCdot{E}[u]{$a_{2}$}
\DGCstrand[red](.5,1)(0,1.5)
\DGCdot{E}[ul]{$a_1$}
\DGCstrand[red](.5,1)(1,1.5)
\DGCdot{E}[ur]{$a_{3}$}
\end{DGCpicture} \\
~+~
&k^5_{a_1,a_2,a_3,b}
\begin{DGCpicture}
\DGCstrand(.25,0)(0,.75)(.25,1.5)
\DGCdot{.75}[ul]{$b$}
\DGCstrand[red](0,0)(.5,.5)[$^{}$`{\ }]
%\DGCdot{B}[d]{$1$}
\DGCstrand[red](1,0)(.5,.5)[$^{}$`{\ }]
\DGCstrand[Red](.5,.5)(.5,1)
\DGCstrand[red](.5,0)(.5,1.5)[$^{}$`{\ }]
\DGCdot{B}[d]{${1}$}
\DGCdot{E}[u]{$a_{2}$}
\DGCstrand[red](.5,1)(0,1.5)
\DGCdot{E}[ul]{$a_1$}
\DGCstrand[red](.5,1)(1,1.5)
\DGCdot{E}[ur]{$a_{3}$}
\end{DGCpicture}
~+~
k^6_{a_1,a_2,a_3,b}
\begin{DGCpicture}
\DGCstrand(.25,0)(0,.75)(.25,1.5)
\DGCdot{.75}[ul]{$b$}
\DGCstrand[red](0,0)(.5,.5)[$^{}$`{\ }]
%\DGCdot{B}[d]{$1$}
\DGCstrand[red](1,0)(.5,.5)[$^{}$`{\ }]
\DGCstrand[Red](.5,.5)(.5,1)
\DGCstrand[red](.5,0)(.5,1.5)[$^{}$`{\ }]
%\DGCdot{B}[d]{${1}$}
\DGCdot{E}[u]{$a_{2}$}
\DGCstrand[red](.5,1)(0,1.5)
\DGCdot{E}[ul]{$a_1$}
\DGCstrand[red](.5,1)(1,1.5)
\DGCdot{E}[ur]{$a_{3}$}
\end{DGCpicture}
~+~
k^8_{a_1,a_2,a_3}
\begin{DGCpicture}
\DGCstrand(.25,0)(.75,.4)(.75,.75)(.75,1.1)(.25,1.5)
%\DGCdot{.75}[ul]{$b$}
\DGCstrand[red](0,0)(.5,.5)[$^{}$`{\ }]
\DGCstrand[red](1,0)(.5,.5)[$^{}$`{\ }]
\DGCstrand[Red](.5,.5)(.5,1)
\DGCstrand[red](.5,0)(.5,1.5)[$^{}$`{\ }]
%\DGCdot{B}[d]{$1$}
\DGCdot{E}[u]{$a_{2}$}
\DGCstrand[red](.5,1)(0,1.5)
\DGCdot{E}[ul]{$a_1$}
\DGCstrand[red](.5,1)(1,1.5)
\DGCdot{E}[ur]{$a_{3}$}
\end{DGCpicture}
~=~
0 \ . \nonumber
\end{align}
Now apply $\gamma_{i,i+1}$ to \eqref{Dep2} and evaluate on $x_1$ to get
\begin{equation*}
k^3_{a_1,a_2,a_3,b} (y_1-x_1) x_1^{a_1} x_2^{a_2} x_3^{a_3} y_1^b
%+
%k^8_{a_1,a_2,a_3}  x_1^{a_1} x_2^{a_2} x_3^{a_3}
=
0 \ .
\end{equation*}
This implies
\begin{equation*}
k^3_{a_1,a_2,a_3,b}
%=
%k^8_{a_1,a_2,a_3}
= 0
\end{equation*}
and the dependence relations becomes
\begin{align} 
\label{Dep3}
&k^2_{a_1,a_2,a_3,b}
\begin{DGCpicture}
\DGCstrand(.25,0)(0,.75)(.25,1.5)
\DGCdot{.75}[ul]{$b$}
\DGCstrand[red](0,0)(.5,.5)[$^{}$`{\ }]
\DGCdot{B}[d]{$2$}
\DGCstrand[red](1,0)(.5,.5)[$^{}$`{\ }]
\DGCstrand[Red](.5,.5)(.5,1)
\DGCstrand[red](.5,0)(.5,1.5)[$^{}$`{\ }]
%\DGCdot{B}[d]{${1}$}
\DGCdot{E}[u]{$a_{2}$}
\DGCstrand[red](.5,1)(0,1.5)
\DGCdot{E}[ul]{$a_1$}
\DGCstrand[red](.5,1)(1,1.5)
\DGCdot{E}[ur]{$a_{3}$}
\end{DGCpicture}
~+~
k^4_{a_1,a_2,a_3,b}
\begin{DGCpicture}
\DGCstrand(.25,0)(0,.75)(.25,1.5)
\DGCdot{.75}[ul]{$b$}
\DGCstrand[red](0,0)(.5,.5)[$^{}$`{\ }]
\DGCdot{B}[d]{$1$}
\DGCstrand[red](1,0)(.5,.5)[$^{}$`{\ }]
\DGCstrand[Red](.5,.5)(.5,1)
\DGCstrand[red](.5,0)(.5,1.5)[$^{}$`{\ }]
%\DGCdot{B}[d]{${1}$}
\DGCdot{E}[u]{$a_{2}$}
\DGCstrand[red](.5,1)(0,1.5)
\DGCdot{E}[ul]{$a_1$}
\DGCstrand[red](.5,1)(1,1.5)
\DGCdot{E}[ur]{$a_{3}$}
\end{DGCpicture} 
~+~
k^5_{a_1,a_2,a_3,b}
\begin{DGCpicture}
\DGCstrand(.25,0)(0,.75)(.25,1.5)
\DGCdot{.75}[ul]{$b$}
\DGCstrand[red](0,0)(.5,.5)[$^{}$`{\ }]
%\DGCdot{B}[d]{$1$}
\DGCstrand[red](1,0)(.5,.5)[$^{}$`{\ }]
\DGCstrand[Red](.5,.5)(.5,1)
\DGCstrand[red](.5,0)(.5,1.5)[$^{}$`{\ }]
\DGCdot{B}[d]{${1}$}
\DGCdot{E}[u]{$a_{2}$}
\DGCstrand[red](.5,1)(0,1.5)
\DGCdot{E}[ul]{$a_1$}
\DGCstrand[red](.5,1)(1,1.5)
\DGCdot{E}[ur]{$a_{3}$}
\end{DGCpicture} \\
~+~
&k^6_{a_1,a_2,a_3,b}
\begin{DGCpicture}
\DGCstrand(.25,0)(0,.75)(.25,1.5)
\DGCdot{.75}[ul]{$b$}
\DGCstrand[red](0,0)(.5,.5)[$^{}$`{\ }]
%\DGCdot{B}[d]{$1$}
\DGCstrand[red](1,0)(.5,.5)[$^{}$`{\ }]
\DGCstrand[Red](.5,.5)(.5,1)
\DGCstrand[red](.5,0)(.5,1.5)[$^{}$`{\ }]
%\DGCdot{B}[d]{${1}$}
\DGCdot{E}[u]{$a_{2}$}
\DGCstrand[red](.5,1)(0,1.5)
\DGCdot{E}[ul]{$a_1$}
\DGCstrand[red](.5,1)(1,1.5)
\DGCdot{E}[ur]{$a_{3}$}
\end{DGCpicture}
~+~
k^8_{a_1,a_2,a_3}
\begin{DGCpicture}
\DGCstrand(.25,0)(.75,.4)(.75,.75)(.75,1.1)(.25,1.5)
%\DGCdot{.75}[ul]{$b$}
\DGCstrand[red](0,0)(.5,.5)[$^{}$`{\ }]
\DGCstrand[red](1,0)(.5,.5)[$^{}$`{\ }]
\DGCstrand[Red](.5,.5)(.5,1)
\DGCstrand[red](.5,0)(.5,1.5)[$^{}$`{\ }]
%\DGCdot{B}[d]{$1$}
\DGCdot{E}[u]{$a_{2}$}
\DGCstrand[red](.5,1)(0,1.5)
\DGCdot{E}[ul]{$a_1$}
\DGCstrand[red](.5,1)(1,1.5)
\DGCdot{E}[ur]{$a_{3}$}
\end{DGCpicture}
~=~
0 \ .  \nonumber
\end{align}
Now apply $\gamma_{i,i+1}$ to \eqref{Dep3} and evaluate on $x_2$ to get
\begin{equation*}
k^2_{a_1,a_2,a_3,b} (y_1-x_1) x_1^{a_1} x_2^{a_2} x_3^{a_3} y_1^b 
+ k^8_{a_1,a_2,a_3}  x_1^{a_1} x_2^{a_2} x_3^{a_3}= 0
\ .
\end{equation*}
This implies that $k^2_{a_1,a_2,a_3,b}=k^8_{a_1,a_2,a_3}=0$ and the dependence relation becomes
\begin{equation} 
\label{Dep4}
k^4_{a_1,a_2,a_3,b}
\begin{DGCpicture}
\DGCstrand(.25,0)(0,.75)(.25,1.5)
\DGCdot{.75}[ul]{$b$}
\DGCstrand[red](0,0)(.5,.5)[$^{}$`{\ }]
\DGCdot{B}[d]{$1$}
\DGCstrand[red](1,0)(.5,.5)[$^{}$`{\ }]
\DGCstrand[Red](.5,.5)(.5,1)
\DGCstrand[red](.5,0)(.5,1.5)[$^{}$`{\ }]
%\DGCdot{B}[d]{${1}$}
\DGCdot{E}[u]{$a_{2}$}
\DGCstrand[red](.5,1)(0,1.5)
\DGCdot{E}[ul]{$a_1$}
\DGCstrand[red](.5,1)(1,1.5)
\DGCdot{E}[ur]{$a_{3}$}
\end{DGCpicture} 
~+~
k^5_{a_1,a_2,a_3,b}
\begin{DGCpicture}
\DGCstrand(.25,0)(0,.75)(.25,1.5)
\DGCdot{.75}[ul]{$b$}
\DGCstrand[red](0,0)(.5,.5)[$^{}$`{\ }]
%\DGCdot{B}[d]{$1$}
\DGCstrand[red](1,0)(.5,.5)[$^{}$`{\ }]
\DGCstrand[Red](.5,.5)(.5,1)
\DGCstrand[red](.5,0)(.5,1.5)[$^{}$`{\ }]
\DGCdot{B}[d]{${1}$}
\DGCdot{E}[u]{$a_{2}$}
\DGCstrand[red](.5,1)(0,1.5)
\DGCdot{E}[ul]{$a_1$}
\DGCstrand[red](.5,1)(1,1.5)
\DGCdot{E}[ur]{$a_{3}$}
\end{DGCpicture}
~+~
k^6_{a_1,a_2,a_3,b}
\begin{DGCpicture}
\DGCstrand(.25,0)(0,.75)(.25,1.5)
\DGCdot{.75}[ul]{$b$}
\DGCstrand[red](0,0)(.5,.5)[$^{}$`{\ }]
%\DGCdot{B}[d]{$1$}
\DGCstrand[red](1,0)(.5,.5)[$^{}$`{\ }]
\DGCstrand[Red](.5,.5)(.5,1)
\DGCstrand[red](.5,0)(.5,1.5)[$^{}$`{\ }]
%\DGCdot{B}[d]{${1}$}
\DGCdot{E}[u]{$a_{2}$}
\DGCstrand[red](.5,1)(0,1.5)
\DGCdot{E}[ul]{$a_1$}
\DGCstrand[red](.5,1)(1,1.5)
\DGCdot{E}[ur]{$a_{3}$}
\end{DGCpicture}
~=~
0 \ . 
\end{equation}
Now apply $\gamma_{i,i+1}$ to \eqref{Dep4} and evaluate on $x_2 x_3$ to get
\begin{equation*}
k^5_{a_1,a_2,a_3,b} (y_1-x_1) x_1^{a_1} x_2^{a_2} x_3^{a_3} y_1^b = 0
\ .
\end{equation*}
Thus $k_5 = 0$ and the dependence relation becomes
\begin{equation} 
\label{Dep5}
k^4_{a_1,a_2,a_3,b}
\begin{DGCpicture}
\DGCstrand(.25,0)(0,.75)(.25,1.5)
\DGCdot{.75}[ul]{$b$}
\DGCstrand[red](0,0)(.5,.5)[$^{}$`{\ }]
\DGCdot{B}[d]{$1$}
\DGCstrand[red](1,0)(.5,.5)[$^{}$`{\ }]
\DGCstrand[Red](.5,.5)(.5,1)
\DGCstrand[red](.5,0)(.5,1.5)[$^{}$`{\ }]
%\DGCdot{B}[d]{${1}$}
\DGCdot{E}[u]{$a_{2}$}
\DGCstrand[red](.5,1)(0,1.5)
\DGCdot{E}[ul]{$a_1$}
\DGCstrand[red](.5,1)(1,1.5)
\DGCdot{E}[ur]{$a_{3}$}
\end{DGCpicture} 
~+~
k^6_{a_1,a_2,a_3,b}
\begin{DGCpicture}
\DGCstrand(.25,0)(0,.75)(.25,1.5)
\DGCdot{.75}[ul]{$b$}
\DGCstrand[red](0,0)(.5,.5)[$^{}$`{\ }]
%\DGCdot{B}[d]{$1$}
\DGCstrand[red](1,0)(.5,.5)[$^{}$`{\ }]
\DGCstrand[Red](.5,.5)(.5,1)
\DGCstrand[red](.5,0)(.5,1.5)[$^{}$`{\ }]
%\DGCdot{B}[d]{${1}$}
\DGCdot{E}[u]{$a_{2}$}
\DGCstrand[red](.5,1)(0,1.5)
\DGCdot{E}[ul]{$a_1$}
\DGCstrand[red](.5,1)(1,1.5)
\DGCdot{E}[ur]{$a_{3}$}
\end{DGCpicture}
~=~
0 \ . 
\end{equation}
Apply $\gamma_{i,i+1}$ to \eqref{Dep5} and evaluate on $x_1 x_2$ to get 
\begin{equation*}
k^4_{a_1,a_2,a_3,b} (y_1-x_1) x_1^{a_1} x_2^{a_2} x_3^{a_3} y_1^b = 0
\ .
\end{equation*}
Thus $k^4_{a_1,a_2,a_3,b} = 0$ which implies $k^6_{a_1,a_2,a_3,b}=0$.
It is straightforward to show that each element in the spanning set gets sent to something non-zero under $\gamma_{i,i+1}$ by evaluating on an appropriate element. For instance,
\begin{equation}
    \gamma_{i, i+1} \left(~
    \begin{DGCpicture}
\DGCstrand(.25,0)(0,.75)(.25,1.5)
\DGCdot{.75}[ul]{$b$}
\DGCstrand[red](0,0)(.5,.5)[$^{}$`{\ }]
\DGCdot{B}[d]{$2$}
\DGCstrand[red](1,0)(.5,.5)[$^{}$`{\ }]
\DGCstrand[Red](.5,.5)(.5,1)
\DGCstrand[red](.5,0)(.5,1.5)[$^{}$`{\ }]
\DGCdot{B}[d]{${1}$}
\DGCdot{E}[u]{$a_{2}$}
\DGCstrand[red](.5,1)(0,1.5)
\DGCdot{E}[ul]{$a_1$}
\DGCstrand[red](.5,1)(1,1.5)
\DGCdot{E}[ur]{$a_{3}$}
\end{DGCpicture}
    ~\right) (1)
= (y_1-x_1) y_1^b x_1^{a_1} x_2^{a_2} x_3^{a_3 }.
\end{equation}
Thus these elements of the spanning set are non-zero and linearly independent.

Next we will consider a dependence relation of the form
\begin{align}
\label{Dep6}
&k^1_{a_1,a_2,a_3,b}
\begin{DGCpicture}
%\DGCstrand(.75,0)(.75,1.25)(.25,1.5)
%\DGCdot{.75}[ur]{$b$}
\DGCstrand[red](0,0)(.5,.5)[$^{}$`{\ }]
\DGCdot{B}[d]{$2$}
\DGCstrand[red](1,0)(.5,.5)[$^{}$`{\ }]
\DGCstrand[Red](.5,.5)(.5,1)
\DGCstrand[red](.5,0)(.5,1.5)[$^{}$`{\ }]
\DGCdot{B}[d]{$1$}
\DGCdot{E}[u]{$a_{2}$}
\DGCstrand[red](.5,1)(0,1.5)
\DGCdot{E}[ul]{$a_1$}
\DGCstrand[red](.5,1)(1,1.5)
\DGCdot{E}[ur]{$a_{3}$}
\DGCstrand(.75,0)(.25,.25)(.25,1.5)
\DGCdot{.75}[ul]{$b$}
\end{DGCpicture}
~+~
k^2_{a_1,a_2,a_3,b}
\begin{DGCpicture}
%\DGCstrand(.75,0)(.75,1.25)(.25,1.5)
%\DGCdot{.75}[ur]{$b$}
\DGCstrand[red](0,0)(.5,.5)[$^{}$`{\ }]
\DGCdot{B}[d]{$2$}
\DGCstrand[red](1,0)(.5,.5)[$^{}$`{\ }]
\DGCstrand[Red](.5,.5)(.5,1)
\DGCstrand[red](.5,0)(.5,1.5)[$^{}$`{\ }]
%\DGCdot{B}[d]{$1$}
\DGCdot{E}[u]{$a_{2}$}
\DGCstrand[red](.5,1)(0,1.5)
\DGCdot{E}[ul]{$a_1$}
\DGCstrand[red](.5,1)(1,1.5)
\DGCdot{E}[ur]{$a_{3}$}
\DGCstrand(.75,0)(.25,.25)(.25,1.5)
\DGCdot{.75}[ul]{$b$}
\end{DGCpicture}
~+~
k^3_{a_1,a_2,a_3,b}
\begin{DGCpicture}
%\DGCstrand(.75,0)(.75,1.25)(.25,1.5)
%\DGCdot{.75}[ur]{$b$}
\DGCstrand[red](0,0)(.5,.5)[$^{}$`{\ }]
\DGCdot{B}[d]{$1$}
\DGCstrand[red](1,0)(.5,.5)[$^{}$`{\ }]
\DGCstrand[Red](.5,.5)(.5,1)
\DGCstrand[red](.5,0)(.5,1.5)[$^{}$`{\ }]
\DGCdot{B}[d]{$1$}
\DGCdot{E}[u]{$a_{2}$}
\DGCstrand[red](.5,1)(0,1.5)
\DGCdot{E}[ul]{$a_1$}
\DGCstrand[red](.5,1)(1,1.5)
\DGCdot{E}[ur]{$a_{3}$}
\DGCstrand(.75,0)(.25,.25)(.25,1.5)
\DGCdot{.75}[ul]{$b$}
\end{DGCpicture}
~+~
k^4_{a_1,a_2,a_3,b}
\begin{DGCpicture}
%\DGCstrand(.75,0)(.75,1.25)(.25,1.5)
%\DGCdot{.75}[ur]{$b$}
\DGCstrand[red](0,0)(.5,.5)[$^{}$`{\ }]
\DGCdot{B}[d]{$1$}
\DGCstrand[red](1,0)(.5,.5)[$^{}$`{\ }]
\DGCstrand[Red](.5,.5)(.5,1)
\DGCstrand[red](.5,0)(.5,1.5)[$^{}$`{\ }]
%\DGCdot{B}[d]{$1$}
\DGCdot{E}[u]{$a_{2}$}
\DGCstrand[red](.5,1)(0,1.5)
\DGCdot{E}[ul]{$a_1$}
\DGCstrand[red](.5,1)(1,1.5)
\DGCdot{E}[ur]{$a_{3}$}
\DGCstrand(.75,0)(.25,.25)(.25,1.5)
\DGCdot{.75}[ul]{$b$}
\end{DGCpicture} \\
~+~
&k^5_{a_1,a_2,a_3,b}
\begin{DGCpicture}
%\DGCstrand(.75,0)(.75,1.25)(.25,1.5)
%\DGCdot{.75}[ur]{$b$}
\DGCstrand[red](0,0)(.5,.5)[$^{}$`{\ }]
%\DGCdot{B}[d]{$1$}
\DGCstrand[red](1,0)(.5,.5)[$^{}$`{\ }]
\DGCstrand[Red](.5,.5)(.5,1)
\DGCstrand[red](.5,0)(.5,1.5)[$^{}$`{\ }]
\DGCdot{B}[d]{$1$}
\DGCdot{E}[u]{$a_{2}$}
\DGCstrand[red](.5,1)(0,1.5)
\DGCdot{E}[ul]{$a_1$}
\DGCstrand[red](.5,1)(1,1.5)
\DGCdot{E}[ur]{$a_{3}$}
\DGCstrand(.75,0)(.25,.25)(.25,1.5)
\DGCdot{.75}[ul]{$b$}
\end{DGCpicture}
~+~
k^6_{a_1,a_2,a_3,b}
\begin{DGCpicture}
%\DGCstrand(.75,0)(.75,1.25)(.25,1.5)
%\DGCdot{.75}[ur]{$b$}
\DGCstrand[red](0,0)(.5,.5)[$^{}$`{\ }]
%\DGCdot{B}[d]{$1$}
\DGCstrand[red](1,0)(.5,.5)[$^{}$`{\ }]
\DGCstrand[Red](.5,.5)(.5,1)
\DGCstrand[red](.5,0)(.5,1.5)[$^{}$`{\ }]
%\DGCdot{B}[d]{$1$}
\DGCdot{E}[u]{$a_{2}$}
\DGCstrand[red](.5,1)(0,1.5)
\DGCdot{E}[ul]{$a_1$}
\DGCstrand[red](.5,1)(1,1.5)
\DGCdot{E}[ur]{$a_{3}$}
\DGCstrand(.75,0)(.25,.25)(.25,1.5)
\DGCdot{.75}[ul]{$b$}
\end{DGCpicture}
~+~
k^7_{a_1,a_2,a_3}
\begin{DGCpicture}
\DGCstrand(.75,0)(.75,1.25)(.25,1.5)
\DGCstrand[red](0,0)(.5,.5)[$^{}$`{\ }]
\DGCdot{B}[d]{$2$}
\DGCstrand[red](1,0)(.5,.5)[$^{}$`{\ }]
\DGCstrand[Red](.5,.5)(.5,1)
\DGCstrand[red](.5,0)(.5,1.5)[$^{}$`{\ }]
%\DGCdot{B}[d]{$1$}
\DGCdot{E}[u]{$a_{2}$}
\DGCstrand[red](.5,1)(0,1.5)
\DGCdot{E}[ul]{$a_1$}
\DGCstrand[red](.5,1)(1,1.5)
\DGCdot{E}[ur]{$a_{3}$}
\end{DGCpicture}
~+~
k^8_{a_1,a_2,a_3}
\begin{DGCpicture}
\DGCstrand(.75,0)(.75,1.25)(.25,1.5)
\DGCstrand[red](0,0)(.5,.5)[$^{}$`{\ }]
\DGCdot{B}[d]{$1$}
\DGCstrand[red](1,0)(.5,.5)[$^{}$`{\ }]
\DGCstrand[Red](.5,.5)(.5,1)
\DGCstrand[red](.5,0)(.5,1.5)[$^{}$`{\ }]
%\DGCdot{B}[d]{$1$}
\DGCdot{E}[u]{$a_{2}$}
\DGCstrand[red](.5,1)(0,1.5)
\DGCdot{E}[ul]{$a_1$}
\DGCstrand[red](.5,1)(1,1.5)
\DGCdot{E}[ur]{$a_{3}$}
\end{DGCpicture} \nonumber \\ 
~+~
&k^9_{a_1,a_2,a_3}
\begin{DGCpicture}
\DGCstrand(.75,0)(.75,1.25)(.25,1.5)
\DGCstrand[red](0,0)(.5,.5)[$^{}$`{\ }]
%\DGCdot{B}[d]{$1$}
\DGCstrand[red](1,0)(.5,.5)[$^{}$`{\ }]
\DGCstrand[Red](.5,.5)(.5,1)
\DGCstrand[red](.5,0)(.5,1.5)[$^{}$`{\ }]
\DGCdot{B}[d]{$1$}
\DGCdot{E}[u]{$a_{2}$}
\DGCstrand[red](.5,1)(0,1.5)
\DGCdot{E}[ul]{$a_1$}
\DGCstrand[red](.5,1)(1,1.5)
\DGCdot{E}[ur]{$a_{3}$}
\end{DGCpicture}
~+~
k^{10}_{a_1,a_2,a_3}
\begin{DGCpicture}
\DGCstrand(.75,0)(.75,1.25)(.25,1.5)
\DGCstrand[red](0,0)(.5,.5)[$^{}$`{\ }]
%\DGCdot{B}[d]{$1$}
\DGCstrand[red](1,0)(.5,.5)[$^{}$`{\ }]
\DGCstrand[Red](.5,.5)(.5,1)
\DGCstrand[red](.5,0)(.5,1.5)[$^{}$`{\ }]
%\DGCdot{B}[d]{$1$}
\DGCdot{E}[u]{$a_{2}$}
\DGCstrand[red](.5,1)(0,1.5)
\DGCdot{E}[ul]{$a_1$}
\DGCstrand[red](.5,1)(1,1.5)
\DGCdot{E}[ur]{$a_{3}$}
\end{DGCpicture}
~=~
0 \ . \nonumber
\end{align}
Applying the homomorphism $\gamma_{i,i+1}$ to \eqref{Dep6} and evaluating on $1$ yields the equation
\begin{equation*}
k^1_{a_1,a_2,a_3,b} (y_1-x_1) x_1^{a_1} x_2^{a_2} x_3^{a_3} y_1^b
+
k^7_{a_1,a_2,a_3} x_1^{a_1} x_2^{a_2} x_3^{a_3}
= 0 \ .
\end{equation*}
Thus $k_1 = k_7 = 0$ and the dependence relation becomes
\begin{align}
\label{Dep7}
&k^2_{a_1,a_2,a_3,b}
\begin{DGCpicture}
%\DGCstrand(.75,0)(.75,1.25)(.25,1.5)
%\DGCdot{.75}[ur]{$b$}
\DGCstrand[red](0,0)(.5,.5)[$^{}$`{\ }]
\DGCdot{B}[d]{$2$}
\DGCstrand[red](1,0)(.5,.5)[$^{}$`{\ }]
\DGCstrand[Red](.5,.5)(.5,1)
\DGCstrand[red](.5,0)(.5,1.5)[$^{}$`{\ }]
%\DGCdot{B}[d]{$1$}
\DGCdot{E}[u]{$a_{2}$}
\DGCstrand[red](.5,1)(0,1.5)
\DGCdot{E}[ul]{$a_1$}
\DGCstrand[red](.5,1)(1,1.5)
\DGCdot{E}[ur]{$a_{3}$}
\DGCstrand(.75,0)(.25,.25)(.25,1.5)
\DGCdot{.75}[ul]{$b$}
\end{DGCpicture}
~+~
k^3_{a_1,a_2,a_3,b}
\begin{DGCpicture}
%\DGCstrand(.75,0)(.75,1.25)(.25,1.5)
%\DGCdot{.75}[ur]{$b$}
\DGCstrand[red](0,0)(.5,.5)[$^{}$`{\ }]
\DGCdot{B}[d]{$1$}
\DGCstrand[red](1,0)(.5,.5)[$^{}$`{\ }]
\DGCstrand[Red](.5,.5)(.5,1)
\DGCstrand[red](.5,0)(.5,1.5)[$^{}$`{\ }]
\DGCdot{B}[d]{$1$}
\DGCdot{E}[u]{$a_{2}$}
\DGCstrand[red](.5,1)(0,1.5)
\DGCdot{E}[ul]{$a_1$}
\DGCstrand[red](.5,1)(1,1.5)
\DGCdot{E}[ur]{$a_{3}$}
\DGCstrand(.75,0)(.25,.25)(.25,1.5)
\DGCdot{.75}[ul]{$b$}
\end{DGCpicture}
~+~
k^4_{a_1,a_2,a_3,b}
\begin{DGCpicture}
%\DGCstrand(.75,0)(.75,1.25)(.25,1.5)
%\DGCdot{.75}[ur]{$b$}
\DGCstrand[red](0,0)(.5,.5)[$^{}$`{\ }]
\DGCdot{B}[d]{$1$}
\DGCstrand[red](1,0)(.5,.5)[$^{}$`{\ }]
\DGCstrand[Red](.5,.5)(.5,1)
\DGCstrand[red](.5,0)(.5,1.5)[$^{}$`{\ }]
%\DGCdot{B}[d]{$1$}
\DGCdot{E}[u]{$a_{2}$}
\DGCstrand[red](.5,1)(0,1.5)
\DGCdot{E}[ul]{$a_1$}
\DGCstrand[red](.5,1)(1,1.5)
\DGCdot{E}[ur]{$a_{3}$}
\DGCstrand(.75,0)(.25,.25)(.25,1.5)
\DGCdot{.75}[ul]{$b$}
\end{DGCpicture} 
~+~
k^5_{a_1,a_2,a_3,b}
\begin{DGCpicture}
%\DGCstrand(.75,0)(.75,1.25)(.25,1.5)
%\DGCdot{.75}[ur]{$b$}
\DGCstrand[red](0,0)(.5,.5)[$^{}$`{\ }]
%\DGCdot{B}[d]{$1$}
\DGCstrand[red](1,0)(.5,.5)[$^{}$`{\ }]
\DGCstrand[Red](.5,.5)(.5,1)
\DGCstrand[red](.5,0)(.5,1.5)[$^{}$`{\ }]
\DGCdot{B}[d]{$1$}
\DGCdot{E}[u]{$a_{2}$}
\DGCstrand[red](.5,1)(0,1.5)
\DGCdot{E}[ul]{$a_1$}
\DGCstrand[red](.5,1)(1,1.5)
\DGCdot{E}[ur]{$a_{3}$}
\DGCstrand(.75,0)(.25,.25)(.25,1.5)
\DGCdot{.75}[ul]{$b$}
\end{DGCpicture}  \\
~+~
&k^6_{a_1,a_2,a_3,b}
\begin{DGCpicture}
%\DGCstrand(.75,0)(.75,1.25)(.25,1.5)
%\DGCdot{.75}[ur]{$b$}
\DGCstrand[red](0,0)(.5,.5)[$^{}$`{\ }]
%\DGCdot{B}[d]{$1$}
\DGCstrand[red](1,0)(.5,.5)[$^{}$`{\ }]
\DGCstrand[Red](.5,.5)(.5,1)
\DGCstrand[red](.5,0)(.5,1.5)[$^{}$`{\ }]
%\DGCdot{B}[d]{$1$}
\DGCdot{E}[u]{$a_{2}$}
\DGCstrand[red](.5,1)(0,1.5)
\DGCdot{E}[ul]{$a_1$}
\DGCstrand[red](.5,1)(1,1.5)
\DGCdot{E}[ur]{$a_{3}$}
\DGCstrand(.75,0)(.25,.25)(.25,1.5)
\DGCdot{.75}[ul]{$b$}
\end{DGCpicture}
~+~
k^8_{a_1,a_2,a_3}
\begin{DGCpicture}
\DGCstrand(.75,0)(.75,1.25)(.25,1.5)
\DGCstrand[red](0,0)(.5,.5)[$^{}$`{\ }]
\DGCdot{B}[d]{$1$}
\DGCstrand[red](1,0)(.5,.5)[$^{}$`{\ }]
\DGCstrand[Red](.5,.5)(.5,1)
\DGCstrand[red](.5,0)(.5,1.5)[$^{}$`{\ }]
%\DGCdot{B}[d]{$1$}
\DGCdot{E}[u]{$a_{2}$}
\DGCstrand[red](.5,1)(0,1.5)
\DGCdot{E}[ul]{$a_1$}
\DGCstrand[red](.5,1)(1,1.5)
\DGCdot{E}[ur]{$a_{3}$}
\end{DGCpicture} 
~+~
k^9_{a_1,a_2,a_3}
\begin{DGCpicture}
\DGCstrand(.75,0)(.75,1.25)(.25,1.5)
\DGCstrand[red](0,0)(.5,.5)[$^{}$`{\ }]
%\DGCdot{B}[d]{$1$}
\DGCstrand[red](1,0)(.5,.5)[$^{}$`{\ }]
\DGCstrand[Red](.5,.5)(.5,1)
\DGCstrand[red](.5,0)(.5,1.5)[$^{}$`{\ }]
\DGCdot{B}[d]{$1$}
\DGCdot{E}[u]{$a_{2}$}
\DGCstrand[red](.5,1)(0,1.5)
\DGCdot{E}[ul]{$a_1$}
\DGCstrand[red](.5,1)(1,1.5)
\DGCdot{E}[ur]{$a_{3}$}
\end{DGCpicture}
~+~
k^{10}_{a_1,a_2,a_3}
\begin{DGCpicture}
\DGCstrand(.75,0)(.75,1.25)(.25,1.5)
\DGCstrand[red](0,0)(.5,.5)[$^{}$`{\ }]
%\DGCdot{B}[d]{$1$}
\DGCstrand[red](1,0)(.5,.5)[$^{}$`{\ }]
\DGCstrand[Red](.5,.5)(.5,1)
\DGCstrand[red](.5,0)(.5,1.5)[$^{}$`{\ }]
%\DGCdot{B}[d]{$1$}
\DGCdot{E}[u]{$a_{2}$}
\DGCstrand[red](.5,1)(0,1.5)
\DGCdot{E}[ul]{$a_1$}
\DGCstrand[red](.5,1)(1,1.5)
\DGCdot{E}[ur]{$a_{3}$}
\end{DGCpicture}
~=~
0 \ . \nonumber
\end{align}
Now apply $\gamma_{i,i+1}$ to \eqref{Dep7} and evaluate on $x_1$ to get
\begin{equation*}
k^3_{a_1,a_2,a_3,b} (y_1-x_1) x_1^{a_1} x_2^{a_2} x_3^{a_3} y_1^b
+
k^8_{a_1,a_2,a_3} x_1^{a_1} x_2^{a_2} x_3^{a_3}
= 0 \ .
\end{equation*}
Thus $k^3_{a_1,a_2,a_3,b} = k^8_{a_1,a_2,a_3} = 0 $ and the dependence relation becomes 
\begin{align}
\label{Dep8}
&k^2_{a_1,a_2,a_3,b}
\begin{DGCpicture}
%\DGCstrand(.75,0)(.75,1.25)(.25,1.5)
%\DGCdot{.75}[ur]{$b$}
\DGCstrand[red](0,0)(.5,.5)[$^{}$`{\ }]
\DGCdot{B}[d]{$2$}
\DGCstrand[red](1,0)(.5,.5)[$^{}$`{\ }]
\DGCstrand[Red](.5,.5)(.5,1)
\DGCstrand[red](.5,0)(.5,1.5)[$^{}$`{\ }]
%\DGCdot{B}[d]{$1$}
\DGCdot{E}[u]{$a_{2}$}
\DGCstrand[red](.5,1)(0,1.5)
\DGCdot{E}[ul]{$a_1$}
\DGCstrand[red](.5,1)(1,1.5)
\DGCdot{E}[ur]{$a_{3}$}
\DGCstrand(.75,0)(.25,.25)(.25,1.5)
\DGCdot{.75}[ul]{$b$}
\end{DGCpicture}
~+~
k^4_{a_1,a_2,a_3,b}
\begin{DGCpicture}
%\DGCstrand(.75,0)(.75,1.25)(.25,1.5)
%\DGCdot{.75}[ur]{$b$}
\DGCstrand[red](0,0)(.5,.5)[$^{}$`{\ }]
\DGCdot{B}[d]{$1$}
\DGCstrand[red](1,0)(.5,.5)[$^{}$`{\ }]
\DGCstrand[Red](.5,.5)(.5,1)
\DGCstrand[red](.5,0)(.5,1.5)[$^{}$`{\ }]
%\DGCdot{B}[d]{$1$}
\DGCdot{E}[u]{$a_{2}$}
\DGCstrand[red](.5,1)(0,1.5)
\DGCdot{E}[ul]{$a_1$}
\DGCstrand[red](.5,1)(1,1.5)
\DGCdot{E}[ur]{$a_{3}$}
\DGCstrand(.75,0)(.25,.25)(.25,1.5)
\DGCdot{.75}[ul]{$b$}
\end{DGCpicture} 
~+~
k^5_{a_1,a_2,a_3,b}
\begin{DGCpicture}
%\DGCstrand(.75,0)(.75,1.25)(.25,1.5)
%\DGCdot{.75}[ur]{$b$}
\DGCstrand[red](0,0)(.5,.5)[$^{}$`{\ }]
%\DGCdot{B}[d]{$1$}
\DGCstrand[red](1,0)(.5,.5)[$^{}$`{\ }]
\DGCstrand[Red](.5,.5)(.5,1)
\DGCstrand[red](.5,0)(.5,1.5)[$^{}$`{\ }]
\DGCdot{B}[d]{$1$}
\DGCdot{E}[u]{$a_{2}$}
\DGCstrand[red](.5,1)(0,1.5)
\DGCdot{E}[ul]{$a_1$}
\DGCstrand[red](.5,1)(1,1.5)
\DGCdot{E}[ur]{$a_{3}$}
\DGCstrand(.75,0)(.25,.25)(.25,1.5)
\DGCdot{.75}[ul]{$b$}
\end{DGCpicture}  \\
~+~
&k^6_{a_1,a_2,a_3,b}
\begin{DGCpicture}
%\DGCstrand(.75,0)(.75,1.25)(.25,1.5)
%\DGCdot{.75}[ur]{$b$}
\DGCstrand[red](0,0)(.5,.5)[$^{}$`{\ }]
%\DGCdot{B}[d]{$1$}
\DGCstrand[red](1,0)(.5,.5)[$^{}$`{\ }]
\DGCstrand[Red](.5,.5)(.5,1)
\DGCstrand[red](.5,0)(.5,1.5)[$^{}$`{\ }]
%\DGCdot{B}[d]{$1$}
\DGCdot{E}[u]{$a_{2}$}
\DGCstrand[red](.5,1)(0,1.5)
\DGCdot{E}[ul]{$a_1$}
\DGCstrand[red](.5,1)(1,1.5)
\DGCdot{E}[ur]{$a_{3}$}
\DGCstrand(.75,0)(.25,.25)(.25,1.5)
\DGCdot{.75}[ul]{$b$}
\end{DGCpicture}
~+~
k^9_{a_1,a_2,a_3}
\begin{DGCpicture}
\DGCstrand(.75,0)(.75,1.25)(.25,1.5)
\DGCstrand[red](0,0)(.5,.5)[$^{}$`{\ }]
%\DGCdot{B}[d]{$1$}
\DGCstrand[red](1,0)(.5,.5)[$^{}$`{\ }]
\DGCstrand[Red](.5,.5)(.5,1)
\DGCstrand[red](.5,0)(.5,1.5)[$^{}$`{\ }]
\DGCdot{B}[d]{$1$}
\DGCdot{E}[u]{$a_{2}$}
\DGCstrand[red](.5,1)(0,1.5)
\DGCdot{E}[ul]{$a_1$}
\DGCstrand[red](.5,1)(1,1.5)
\DGCdot{E}[ur]{$a_{3}$}
\end{DGCpicture}
~+~
k^{10}_{a_1,a_2,a_3}
\begin{DGCpicture}
\DGCstrand(.75,0)(.75,1.25)(.25,1.5)
\DGCstrand[red](0,0)(.5,.5)[$^{}$`{\ }]
%\DGCdot{B}[d]{$1$}
\DGCstrand[red](1,0)(.5,.5)[$^{}$`{\ }]
\DGCstrand[Red](.5,.5)(.5,1)
\DGCstrand[red](.5,0)(.5,1.5)[$^{}$`{\ }]
%\DGCdot{B}[d]{$1$}
\DGCdot{E}[u]{$a_{2}$}
\DGCstrand[red](.5,1)(0,1.5)
\DGCdot{E}[ul]{$a_1$}
\DGCstrand[red](.5,1)(1,1.5)
\DGCdot{E}[ur]{$a_{3}$}
\end{DGCpicture}
~=~
0 \ . \nonumber
\end{align}
Now apply $\gamma_{i,i+1}$ to \eqref{Dep8} and evaluate on $x_2$ to get
\begin{equation*}
k^2_{a_1,a_2,a_3,b} (y_1-x_1) x_1^{a_1} x_2^{a_2} x_3^{a_3} y_1^b
-
k^9_{a_1,a_2,a_3} x_1^{a_1} x_2^{a_2} x_3^{a_3}
= 0 \ .
\end{equation*}
Thus $k^2_{a_1,a_2,a_3,b}=k^9_{a_1,a_2,a_3}=0$ and the dependence relation becomes
\begin{equation}
\label{Dep9}
k^4_{a_1,a_2,a_3,b}
\begin{DGCpicture}
%\DGCstrand(.75,0)(.75,1.25)(.25,1.5)
%\DGCdot{.75}[ur]{$b$}
\DGCstrand[red](0,0)(.5,.5)[$^{}$`{\ }]
\DGCdot{B}[d]{$1$}
\DGCstrand[red](1,0)(.5,.5)[$^{}$`{\ }]
\DGCstrand[Red](.5,.5)(.5,1)
\DGCstrand[red](.5,0)(.5,1.5)[$^{}$`{\ }]
%\DGCdot{B}[d]{$1$}
\DGCdot{E}[u]{$a_{2}$}
\DGCstrand[red](.5,1)(0,1.5)
\DGCdot{E}[ul]{$a_1$}
\DGCstrand[red](.5,1)(1,1.5)
\DGCdot{E}[ur]{$a_{3}$}
\DGCstrand(.75,0)(.25,.25)(.25,1.5)
\DGCdot{.75}[ul]{$b$}
\end{DGCpicture} 
~+~
k^5_{a_1,a_2,a_3,b}
\begin{DGCpicture}
%\DGCstrand(.75,0)(.75,1.25)(.25,1.5)
%\DGCdot{.75}[ur]{$b$}
\DGCstrand[red](0,0)(.5,.5)[$^{}$`{\ }]
%\DGCdot{B}[d]{$1$}
\DGCstrand[red](1,0)(.5,.5)[$^{}$`{\ }]
\DGCstrand[Red](.5,.5)(.5,1)
\DGCstrand[red](.5,0)(.5,1.5)[$^{}$`{\ }]
\DGCdot{B}[d]{$1$}
\DGCdot{E}[u]{$a_{2}$}
\DGCstrand[red](.5,1)(0,1.5)
\DGCdot{E}[ul]{$a_1$}
\DGCstrand[red](.5,1)(1,1.5)
\DGCdot{E}[ur]{$a_{3}$}
\DGCstrand(.75,0)(.25,.25)(.25,1.5)
\DGCdot{.75}[ul]{$b$}
\end{DGCpicture}  
~+~
k^6_{a_1,a_2,a_3,b}
\begin{DGCpicture}
%\DGCstrand(.75,0)(.75,1.25)(.25,1.5)
%\DGCdot{.75}[ur]{$b$}
\DGCstrand[red](0,0)(.5,.5)[$^{}$`{\ }]
%\DGCdot{B}[d]{$1$}
\DGCstrand[red](1,0)(.5,.5)[$^{}$`{\ }]
\DGCstrand[Red](.5,.5)(.5,1)
\DGCstrand[red](.5,0)(.5,1.5)[$^{}$`{\ }]
%\DGCdot{B}[d]{$1$}
\DGCdot{E}[u]{$a_{2}$}
\DGCstrand[red](.5,1)(0,1.5)
\DGCdot{E}[ul]{$a_1$}
\DGCstrand[red](.5,1)(1,1.5)
\DGCdot{E}[ur]{$a_{3}$}
\DGCstrand(.75,0)(.25,.25)(.25,1.5)
\DGCdot{.75}[ul]{$b$}
\end{DGCpicture}
~+~
k^{10}_{a_1,a_2,a_3}
\begin{DGCpicture}
\DGCstrand(.75,0)(.75,1.25)(.25,1.5)
\DGCstrand[red](0,0)(.5,.5)[$^{}$`{\ }]
%\DGCdot{B}[d]{$1$}
\DGCstrand[red](1,0)(.5,.5)[$^{}$`{\ }]
\DGCstrand[Red](.5,.5)(.5,1)
\DGCstrand[red](.5,0)(.5,1.5)[$^{}$`{\ }]
%\DGCdot{B}[d]{$1$}
\DGCdot{E}[u]{$a_{2}$}
\DGCstrand[red](.5,1)(0,1.5)
\DGCdot{E}[ul]{$a_1$}
\DGCstrand[red](.5,1)(1,1.5)
\DGCdot{E}[ur]{$a_{3}$}
\end{DGCpicture}
~=~
0 \ . 
\end{equation}
Now apply $\gamma_{i,i+1}$ to \eqref{Dep9} and evaluate on $x_1^2$ to get
\begin{equation*}
k^5_{a_1,a_2,a_3,b} (y_1-x_1) x_1^{a_1} x_2^{a_2} x_3^{a_3} y_1^b
+
k^{10}_{a_1,a_2,a_3} x_1^{a_1} x_2^{a_2} x_3^{a_3}
= 0 \ .
\end{equation*}
Thus $ k^5_{a_1,a_2,a_3,b} = k^{10}_{a_1,a_2,a_3}=0$ and the dependence relation becomes 
\begin{equation}
\label{Dep10}
k^4_{a_1,a_2,a_3,b}
\begin{DGCpicture}
%\DGCstrand(.75,0)(.75,1.25)(.25,1.5)
%\DGCdot{.75}[ur]{$b$}
\DGCstrand[red](0,0)(.5,.5)[$^{}$`{\ }]
\DGCdot{B}[d]{$1$}
\DGCstrand[red](1,0)(.5,.5)[$^{}$`{\ }]
\DGCstrand[Red](.5,.5)(.5,1)
\DGCstrand[red](.5,0)(.5,1.5)[$^{}$`{\ }]
%\DGCdot{B}[d]{$1$}
\DGCdot{E}[u]{$a_{2}$}
\DGCstrand[red](.5,1)(0,1.5)
\DGCdot{E}[ul]{$a_1$}
\DGCstrand[red](.5,1)(1,1.5)
\DGCdot{E}[ur]{$a_{3}$}
\DGCstrand(.75,0)(.25,.25)(.25,1.5)
\DGCdot{.75}[ul]{$b$}
\end{DGCpicture}  
~+~
k^6_{a_1,a_2,a_3,b}
\begin{DGCpicture}
%\DGCstrand(.75,0)(.75,1.25)(.25,1.5)
%\DGCdot{.75}[ur]{$b$}
\DGCstrand[red](0,0)(.5,.5)[$^{}$`{\ }]
%\DGCdot{B}[d]{$1$}
\DGCstrand[red](1,0)(.5,.5)[$^{}$`{\ }]
\DGCstrand[Red](.5,.5)(.5,1)
\DGCstrand[red](.5,0)(.5,1.5)[$^{}$`{\ }]
%\DGCdot{B}[d]{$1$}
\DGCdot{E}[u]{$a_{2}$}
\DGCstrand[red](.5,1)(0,1.5)
\DGCdot{E}[ul]{$a_1$}
\DGCstrand[red](.5,1)(1,1.5)
\DGCdot{E}[ur]{$a_{3}$}
\DGCstrand(.75,0)(.25,.25)(.25,1.5)
\DGCdot{.75}[ul]{$b$}
\end{DGCpicture}
~=~
0 \ . 
\end{equation}
Finally apply $\gamma_{i,i+1}$ to \eqref{Dep10} and evaluate on $x_1 x_2$ to get
\begin{equation*}
k^4_{a_1,a_2,a_3,b} (y_1-x_1) x_1^{a_1} x_2^{a_2} x_3^{a_3} y_1^b
= 0 \ .
\end{equation*}
Thus $ k^4_{a_1,a_2,a_3,b}=0$ and consequently $k^6_{a_1,a_2,a_3,b}=0$.
Therefore these spanning elements are also linearly independent.

The other cases are checked in a similar fashion.
\end{proof}

\begin{prop}
\label{WiWi+1Wispan}
The following elements span the bimodule $W_i \otimes_{W} W_{i+1} \otimes_{W} W_i $.
%\JS{Below is part of a spanning set for $W_i \otimes W_{i+1} \otimes W_i $.
%I think these elements are linearly independent.  The last two pictures are shorthand for any picture where the black strand starts or ends outside the main region.}
\begin{equation*}
\gimel_1({\bf a},b,r,s,t)
~=~
\begin{DGCpicture}[scale={.7,.7}]
\DGCstrand[red](-2,0)(-2,4.5)[$^{1}$`{\ }]
\DGCdot{E}[u]{$a_{1}$}
\DGCcoupon*(-1.8,2)(-1.2,2.5){$\cdots$}
\DGCstrand[red](4,0)(4,4.5)[$^{l}$`{\ }]
\DGCdot{E}[u]{$a_{n}$}
\DGCcoupon*(3.2,2)(3.8,2.5){$\cdots$}
\DGCstrand[red](0,0)(.5,.5)[$^{i}$`{\ }]
\DGCdot{B}[l]{$t$}
\DGCstrand[red](1,0)(.5,.5)[$^{i+1}$`{\ }]
\DGCstrand[Red](.5,.5)(.5,1)
\DGCstrand[red](.5,1)(0,1.5)
\DGCstrand[red](.5,1)(1,1.5)
\DGCstrand[red](2,0)(2,1.5)[$^{i+2}$`{\ }]
\DGCstrand[red](1,1.5)(1.5,2)
\DGCdot{B}[l]{$s$}
\DGCstrand[red](2,1.5)(1.5,2)
\DGCstrand[Red](1.5,2)(1.5,2.5)
\DGCstrand[red](1.5,2.5)(1,3)
\DGCstrand[red](1.5,2.5)(2,3)
\DGCstrand[red](0,1.5)(0,3)
\DGCdot{E}[r]{$r$}
%%%%
\DGCstrand[red](0,3)(.5,3.5)
\DGCstrand[red](1,3)(.5,3.5)
\DGCstrand[Red](.5,3.5)(.5,4)
\DGCstrand[red](.5,4)(0,4.5)
\DGCdot{E}[ul]{$a_{i}$}
\DGCstrand[red](.5,4)(1,4.5)
\DGCdot{E}[ur]{$a_{i+1}$}
\DGCstrand[red](2,3)(2,4.5)
\DGCdot{E}[ur]{$a_{i+2}$}
%\DGCcoupon*(-2.75,2)(-2.25,2.5){$\cdots$}
%\DGCcoupon*(4.25,2)(4.75,2.5){$\cdots$}
\DGCstrand(.5,0)(-.5,.75)(-.5,3.75)(.5,4.5)
%\DGCstrand(.5,0)(-.5,2.25)(.5,4.5) replace above
\DGCdot{E}[u]{$b$}
\end{DGCpicture}
\quad \quad
a_i \in \mathbb{Z}_{\geq 0}, 
b \in \mathbb{Z}_{\geq 0}, 
r,s,t \in \{0, 1\} ,
\end{equation*}

\begin{equation*}
\gimel_2({\bf a},s)
~=~
\begin{DGCpicture}[scale={.7,.7}]
\DGCstrand[red](-2,0)(-2,4.5)[$^{1}$`{\ }]
\DGCdot{E}[u]{$a_{1}$}
\DGCcoupon*(-1.8,2)(-1.2,2.5){$\cdots$}
\DGCstrand[red](4,0)(4,4.5)[$^{n}$`{\ }]
\DGCdot{E}[u]{$a_{n}$}
\DGCcoupon*(3.2,2)(3.8,2.5){$\cdots$}
\DGCstrand[red](0,0)(.5,.5)[$^{i}$`{\ }]
%\DGCdot{B}[l]{$t$}
\DGCstrand[red](1,0)(.5,.5)[$^{i+1}$`{\ }]
\DGCstrand[Red](.5,.5)(.5,1)
\DGCstrand[red](.5,1)(0,1.5)
\DGCstrand[red](.5,1)(1,1.5)
\DGCstrand[red](2,0)(2,1.5)[$^{i+2}$`{\ }]
\DGCstrand[red](1,1.5)(1.5,2)
\DGCdot{B}[r]{$s$}
\DGCstrand[red](2,1.5)(1.5,2)
\DGCstrand[Red](1.5,2)(1.5,2.5)
\DGCstrand[red](1.5,2.5)(1,3)
\DGCstrand[red](1.5,2.5)(2,3)
\DGCstrand[red](0,1.5)(0,3)
%\DGCdot{E}[r]{$r$}
%%%%
\DGCstrand[red](0,3)(.5,3.5)
\DGCstrand[red](1,3)(.5,3.5)
\DGCstrand[Red](.5,3.5)(.5,4)
\DGCstrand[red](.5,4)(0,4.5)
\DGCdot{E}[u]{$a_{i}$}
\DGCstrand[red](.5,4)(1,4.5)
\DGCdot{E}[u]{$a_{i+1}$}
\DGCstrand[red](2,3)(2,4.5)
\DGCdot{E}[u]{$a_{i+2}$}
%\DGCcoupon*(-2.75,2)(-2.25,2.5){$\cdots$}
%\DGCcoupon*(4.25,2)(4.75,2.5){$\cdots$}
\DGCstrand(.5,0)(0.75,.75)(0.75,3.75)(.5,4.5)
%\DGCstrand(.5,0)(0.75,2.25)(.5,4.5) replace above
%\DGCdot{E}[u]{$b$}
\end{DGCpicture}
\quad \quad 
a_i \in \mathbb{Z}_{\geq 0}, 
%b \in \mathbb{Z}_{\geq 0}, 
s \in \{0, 1\} ,
\end{equation*}

\begin{equation*}
\gimel_3({\bf a})
~=~
\begin{DGCpicture}[scale={.7,.7}]
\DGCstrand[red](-2,0)(-2,4.5)[$^{1}$`{\ }]
\DGCdot{E}[u]{$a_{1}$}
\DGCcoupon*(-1.8,2)(-1.2,2.5){$\cdots$}
\DGCstrand[red](4,0)(4,4.5)[$^{n}$`{\ }]
\DGCdot{E}[u]{$a_{n}$}
\DGCcoupon*(3.2,2)(3.8,2.5){$\cdots$}
\DGCstrand[red](0,0)(.5,.5)[$^{i}$`{\ }]
%\DGCdot{B}[l]{$t$}
\DGCstrand[red](1,0)(.5,.5)[$^{i+1}$`{\ }]
\DGCstrand[Red](.5,.5)(.5,1)
\DGCstrand[red](.5,1)(0,1.5)
\DGCstrand[red](.5,1)(1,1.5)
\DGCstrand[red](2,0)(2,1.5)[$^{i+2}$`{\ }]
\DGCstrand[red](1,1.5)(1.5,2)
%\DGCdot{B}[r]{$s$}
\DGCstrand[red](2,1.5)(1.5,2)
\DGCstrand[Red](1.5,2)(1.5,2.5)
\DGCstrand[red](1.5,2.5)(1,3)
\DGCstrand[red](1.5,2.5)(2,3)
\DGCstrand[red](0,1.5)(0,3)
%\DGCdot{E}[r]{$r$}
%%%%
\DGCstrand[red](0,3)(.5,3.5)
\DGCstrand[red](1,3)(.5,3.5)
\DGCstrand[Red](.5,3.5)(.5,4)
\DGCstrand[red](.5,4)(0,4.5)
\DGCdot{E}[u]{$a_{i}$}
\DGCstrand[red](.5,4)(1,4.5)
\DGCdot{E}[u]{$a_{i+1}$}
\DGCstrand[red](2,3)(2,4.5)
\DGCdot{E}[u]{$a_{i+2}$}
%\DGCcoupon*(-2.75,2)(-2.25,2.5){$\cdots$}
%\DGCcoupon*(4.25,2)(4.75,2.5){$\cdots$}
\DGCstrand(.5,0)(2.5,.75)(2.5,3.75)(.5,4.5)
%\DGCstrand(.5,0)(2.5,2.25)(.5,4.5) replace above
%\DGCdot{E}[u]{$b$}
\end{DGCpicture}
\quad \quad
a_i \in \mathbb{Z}_{\geq 0},
%b \in \mathbb{Z}_{\geq 0}, 
\end{equation*}

\begin{equation*}
\gimel_4({\bf a},b,r,s,t)
~=~
\begin{DGCpicture}[scale={.7,.7}]
\DGCstrand[red](-2,0)(-2,4.5)[$^{1}$`{\ }]
\DGCdot{E}[u]{$a_{1}$}
\DGCcoupon*(-1.8,2)(-1.2,2.5){$\cdots$}
\DGCstrand[red](4,0)(4,4.5)[$^{n}$`{\ }]
\DGCdot{E}[u]{$a_{n}$}
\DGCcoupon*(3.2,2)(3.8,2.5){$\cdots$}
\DGCstrand[red](0,0)(.5,.5)[$^{i}$`{\ }]
\DGCdot{B}[l]{$t$}
\DGCstrand[red](1,0)(.5,.5)[$^{i+1}$`{\ }]
\DGCstrand[Red](.5,.5)(.5,1)
\DGCstrand[red](.5,1)(0,1.5)
\DGCstrand[red](.5,1)(1,1.5)
\DGCstrand[red](2,0)(2,1.5)[$^{i+2}$`{\ }]
\DGCstrand[red](1,1.5)(1.5,2)
\DGCdot{B}[l]{$s$}
\DGCstrand[red](2,1.5)(1.5,2)
\DGCstrand[Red](1.5,2)(1.5,2.5)
\DGCstrand[red](1.5,2.5)(1,3)
\DGCstrand[red](1.5,2.5)(2,3)
\DGCstrand[red](0,1.5)(0,3)
\DGCdot{E}[r]{$r$}
%%%%
\DGCstrand[red](0,3)(.5,3.5)
\DGCstrand[red](1,3)(.5,3.5)
\DGCstrand[Red](.5,3.5)(.5,4)
\DGCstrand[red](.5,4)(0,4.5)
\DGCdot{E}[ul]{$a_{i}$}
\DGCstrand[red](.5,4)(1,4.5)
\DGCdot{E}[ul]{$a_{i+1}$}
\DGCstrand[red](2,3)(2,4.5)
\DGCdot{E}[ur]{$a_{i+2}$}
%\DGCcoupon*(-2.75,2)(-2.25,2.5){$\cdots$}
%\DGCcoupon*(4.25,2)(4.75,2.5){$\cdots$}
\DGCstrand(1.5,0)(1.25,.75)(1.25,3.75)(1.5,4.5)
%\DGCstrand(1.5,0)(1,2.25)(1.5,4.5) replace above
\DGCdot{E}[u]{$b$}
\end{DGCpicture}
\quad \quad
a_i \in \mathbb{Z}_{\geq 0}, 
b \in \mathbb{Z}_{\geq 0}, 
r,s,t \in \{0, 1\} ,
\end{equation*}

\begin{equation*}
\gimel_5({\bf a},r,t)
~=~
\begin{DGCpicture}[scale={.7,.7}]
\DGCstrand[red](-2,0)(-2,4.5)[$^{1}$`{\ }]
\DGCdot{E}[u]{$a_{1}$}
\DGCcoupon*(-1.8,2)(-1.2,2.5){$\cdots$}
\DGCstrand[red](4,0)(4,4.5)[$^{n}$`{\ }]
\DGCdot{E}[u]{$a_{n}$}
\DGCcoupon*(3.2,2)(3.8,2.5){$\cdots$}
\DGCstrand[red](0,0)(.5,.5)[$^{i}$`{\ }]
\DGCdot{B}[l]{$t$}
\DGCstrand[red](1,0)(.5,.5)[$^{i+1}$`{\ }]
\DGCstrand[Red](.5,.5)(.5,1)
\DGCstrand[red](.5,1)(0,1.5)
\DGCstrand[red](.5,1)(1,1.5)
\DGCstrand[red](2,0)(2,1.5)[$^{i+2}$`{\ }]
\DGCstrand[red](1,1.5)(1.5,2)
%\DGCdot{B}[l]{$s$}
\DGCstrand[red](2,1.5)(1.5,2)
\DGCstrand[Red](1.5,2)(1.5,2.5)
\DGCstrand[red](1.5,2.5)(1,3)
\DGCstrand[red](1.5,2.5)(2,3)
\DGCstrand[red](0,1.5)(0,3)
\DGCdot{E}[r]{$r$}
%%%%
\DGCstrand[red](0,3)(.5,3.5)
\DGCstrand[red](1,3)(.5,3.5)
\DGCstrand[Red](.5,3.5)(.5,4)
\DGCstrand[red](.5,4)(0,4.5)
\DGCdot{E}[u]{$a_{i}$}
\DGCstrand[red](.5,4)(1,4.5)
\DGCdot{E}[u]{$a_{i+1}$}
\DGCstrand[red](2,3)(2,4.5)
\DGCdot{E}[u]{$a_{i+2}$}
%\DGCcoupon*(-2.75,2)(-2.25,2.5){$\cdots$}
%\DGCcoupon*(4.25,2)(4.75,2.5){$\cdots$}
\DGCstrand(1.5,0)(1.75,.75)(1.75,3.75)(1.5,4.5)
%\DGCstrand(1.5,0)(2,2.25)(1.5,4.5) replace above
%\DGCdot{E}[u]{$b$}
\end{DGCpicture}
\quad \quad
a_i \in \mathbb{Z}_{\geq 0}, 
%b \in \mathbb{Z}_{\geq 0}, 
r,t \in \{0, 1\} ,
\end{equation*}

\begin{equation*}
\gimel_6({\bf a},b,r,s,t)
~=~
\begin{DGCpicture}[scale={.7,.7}]
\DGCstrand[red](-2,0)(-2,4.5)[$^{1}$`{\ }]
\DGCdot{E}[u]{$a_{1}$}
\DGCcoupon*(-1.8,2)(-1.2,2.5){$\cdots$}
\DGCstrand[red](4,0)(4,4.5)[$^{n}$`{\ }]
\DGCdot{E}[u]{$a_{n}$}
\DGCcoupon*(3.2,2)(3.8,2.5){$\cdots$}
\DGCstrand[red](0,0)(.5,.5)[$^{i}$`{\ }]
\DGCdot{B}[l]{$t$}
\DGCstrand[red](1,0)(.5,.5)[$^{i+1}$`{\ }]
\DGCstrand[Red](.5,.5)(.5,1)
\DGCstrand[red](.5,1)(0,1.5)
\DGCstrand[red](.5,1)(1,1.5)
\DGCstrand[red](2,0)(2,1.5)[$^{i+2}$`{\ }]
\DGCstrand[red](1,1.5)(1.5,2)
\DGCdot{B}[l]{$s$}
\DGCstrand[red](2,1.5)(1.5,2)
\DGCstrand[Red](1.5,2)(1.5,2.5)
\DGCstrand[red](1.5,2.5)(1,3)
\DGCstrand[red](1.5,2.5)(2,3)
\DGCstrand[red](0,1.5)(0,3)
\DGCdot{E}[r]{$r$}
%%%%
\DGCstrand[red](0,3)(.5,3.5)
\DGCstrand[red](1,3)(.5,3.5)
\DGCstrand[Red](.5,3.5)(.5,4)
\DGCstrand[red](.5,4)(0,4.5)
\DGCdot{E}[ul]{$a_{i}$}
\DGCstrand[red](.5,4)(1,4.5)
\DGCdot{E}[ur]{$a_{i+1}$}
\DGCstrand[red](2,3)(2,4.5)
\DGCdot{E}[ur]{$a_{i+2}$}
%\DGCcoupon*(-2.75,2)(-2.25,2.5){$\cdots$}
%\DGCcoupon*(4.25,2)(4.75,2.5){$\cdots$}
\DGCstrand(1.5,0)(1.25,.75)(1.25,3.75)(.5,4.5)
%\DGCstrand(1.5,0)(1,4.25)(.5,4.5) replace above
\DGCdot{E}[u]{$b$}
\end{DGCpicture}
\quad \quad
a_i \in \mathbb{Z}_{\geq 0}, 
b \in \mathbb{Z}_{\geq 0}, 
r,s,t \in \{0, 1\} ,
\end{equation*}

\begin{equation*}
\gimel_7({\bf a},r,t)
~=~
\begin{DGCpicture}[scale={.7,.7}]
\DGCstrand[red](-2,0)(-2,4.5)[$^{1}$`{\ }]
\DGCdot{E}[u]{$a_{1}$}
\DGCcoupon*(-1.8,2)(-1.2,2.5){$\cdots$}
\DGCstrand[red](4,0)(4,4.5)[$^{n}$`{\ }]
\DGCdot{E}[u]{$a_{n}$}
\DGCcoupon*(3.2,2)(3.8,2.5){$\cdots$}
\DGCstrand[red](0,0)(.5,.5)[$^{i}$`{\ }]
\DGCdot{B}[l]{$t$}
\DGCstrand[red](1,0)(.5,.5)[$^{i+1}$`{\ }]
\DGCstrand[Red](.5,.5)(.5,1)
\DGCstrand[red](.5,1)(0,1.5)
\DGCstrand[red](.5,1)(1,1.5)
\DGCstrand[red](2,0)(2,1.5)[$^{i+2}$`{\ }]
\DGCstrand[red](1,1.5)(1.5,2)
%\DGCdot{B}[l]{$s$}
\DGCstrand[red](2,1.5)(1.5,2)
\DGCstrand[Red](1.5,2)(1.5,2.5)
\DGCstrand[red](1.5,2.5)(1,3)
\DGCstrand[red](1.5,2.5)(2,3)
\DGCstrand[red](0,1.5)(0,3)
\DGCdot{E}[r]{$r$}
%%%%
\DGCstrand[red](0,3)(.5,3.5)
\DGCstrand[red](1,3)(.5,3.5)
\DGCstrand[Red](.5,3.5)(.5,4)
\DGCstrand[red](.5,4)(0,4.5)
\DGCdot{E}[u]{$a_{i}$}
\DGCstrand[red](.5,4)(1,4.5)
\DGCdot{E}[u]{$a_{i+1}$}
\DGCstrand[red](2,3)(2,4.5)
\DGCdot{E}[u]{$a_{i+2}$}
%\DGCcoupon*(-2.75,2)(-2.25,2.5){$\cdots$}
%\DGCcoupon*(4.25,2)(4.75,2.5){$\cdots$}
\DGCstrand(1.5,0)(1.75,.75)(1.75,3.75)(.5,4.5)
%\DGCstrand(1.5,0)(2,2.25)(.5,4.5) replace above
%\DGCdot{E}[u]{$b$}
\end{DGCpicture}
\quad \quad
a_i \in \mathbb{Z}_{\geq 0}, 
%b \in \mathbb{Z}_{\geq 0}, 
r,t \in \{0, 1\} ,
\end{equation*}

\begin{equation*}
\gimel_8({\bf a},t)
~=~
\begin{DGCpicture}[scale={.7,.7}]
\DGCstrand[red](-2,0)(-2,4.5)[$^{1}$`{\ }]
\DGCdot{E}[u]{$a_{1}$}
\DGCcoupon*(-1.8,2)(-1.2,2.5){$\cdots$}
\DGCstrand[red](4,0)(4,4.5)[$^{n}$`{\ }]
\DGCdot{E}[u]{$a_{n}$}
\DGCcoupon*(3.2,2)(3.8,2.5){$\cdots$}
\DGCstrand[red](0,0)(.5,.5)[$^{i}$`{\ }]
\DGCdot{B}[l]{$t$}
\DGCstrand[red](1,0)(.5,.5)[$^{i+1}$`{\ }]
\DGCstrand[Red](.5,.5)(.5,1)
\DGCstrand[red](.5,1)(0,1.5)
\DGCstrand[red](.5,1)(1,1.5)
\DGCstrand[red](2,0)(2,1.5)[$^{i+2}$`{\ }]
\DGCstrand[red](1,1.5)(1.5,2)
%\DGCdot{B}[l]{$s$}
\DGCstrand[red](2,1.5)(1.5,2)
\DGCstrand[Red](1.5,2)(1.5,2.5)
\DGCstrand[red](1.5,2.5)(1,3)
\DGCstrand[red](1.5,2.5)(2,3)
\DGCstrand[red](0,1.5)(0,3)
%\DGCdot{E}[r]{$r$}
%%%%
\DGCstrand[red](0,3)(.5,3.5)
\DGCstrand[red](1,3)(.5,3.5)
\DGCstrand[Red](.5,3.5)(.5,4)
\DGCstrand[red](.5,4)(0,4.5)
\DGCdot{E}[u]{$a_{i}$}
\DGCstrand[red](.5,4)(1,4.5)
\DGCdot{E}[u]{$a_{i+1}$}
\DGCstrand[red](2,3)(2,4.5)
\DGCdot{E}[u]{$a_{i+2}$}
%\DGCcoupon*(-2.75,2)(-2.25,2.5){$\cdots$}
%\DGCcoupon*(4.25,2)(4.75,2.5){$\cdots$}
\DGCstrand(1.5,0)(1.5,1.25)(-.5,2)(-.5,3.75)(.5,4.5)
%\DGCstrand(1.5,0)(-.5,2.25)(.5,4.5) replace above
%\DGCdot{E}[u]{$b$}
\end{DGCpicture}
\quad \quad
a_i \in \mathbb{Z}_{\geq 0}, 
%b \in \mathbb{Z}_{\geq 0}, 
t \in \{0, 1\} ,
\end{equation*}

\begin{equation*}
\gimel_9({\bf a},b,r,s,t)
~=~
\begin{DGCpicture}[scale={.7,.7}]
\DGCstrand[red](-2,0)(-2,4.5)[$^{1}$`{\ }]
\DGCdot{E}[u]{$a_{1}$}
\DGCcoupon*(-1.8,2)(-1.2,2.5){$\cdots$}
\DGCstrand[red](4,0)(4,4.5)[$^{n}$`{\ }]
\DGCdot{E}[u]{$a_{n}$}
\DGCcoupon*(3.2,2)(3.8,2.5){$\cdots$}
\DGCstrand[red](0,0)(.5,.5)[$^{i}$`{\ }]
\DGCdot{B}[l]{$t$}
\DGCstrand[red](1,0)(.5,.5)[$^{i+1}$`{\ }]
\DGCstrand[Red](.5,.5)(.5,1)
\DGCstrand[red](.5,1)(0,1.5)
\DGCstrand[red](.5,1)(1,1.5)
\DGCstrand[red](2,0)(2,1.5)[$^{i+2}$`{\ }]
\DGCstrand[red](1,1.5)(1.5,2)
\DGCdot{B}[l]{$s$}
\DGCstrand[red](2,1.5)(1.5,2)
\DGCstrand[Red](1.5,2)(1.5,2.5)
\DGCstrand[red](1.5,2.5)(1,3)
\DGCstrand[red](1.5,2.5)(2,3)
\DGCstrand[red](0,1.5)(0,3)
\DGCdot{E}[r]{$r$}
%%%%
\DGCstrand[red](0,3)(.5,3.5)
\DGCstrand[red](1,3)(.5,3.5)
\DGCstrand[Red](.5,3.5)(.5,4)
\DGCstrand[red](.5,4)(0,4.5)
\DGCdot{E}[ul]{$a_{i}$}
\DGCstrand[red](.5,4)(1,4.5)
\DGCdot{E}[ul]{$a_{i+1}$}
\DGCstrand[red](2,3)(2,4.5)
\DGCdot{E}[ur]{$a_{i+2}$}
%\DGCcoupon*(-2.75,2)(-2.25,2.5){$\cdots$}
%\DGCcoupon*(4.25,2)(4.75,2.5){$\cdots$}
%\DGCstrand(1.5,0)(1,4.25)(.5,4.5)
\DGCstrand(.5,0)(1.25,.75)(1.25,3.75)(1.5,4.5)
%\DGCstrand(.5,0)(1,.25)(1.5,4.5) replace above
\DGCdot{E}[u]{$b$}
\end{DGCpicture}
\quad \quad
a_i \in \mathbb{Z}_{\geq 0}, 
b \in \mathbb{Z}_{\geq 0}, 
r,s,t \in \{0, 1\} ,
\end{equation*}

\begin{equation*}
\gimel_{10}({\bf a},r,t)
~=~
\begin{DGCpicture}[scale={.7,.7}]
\DGCstrand[red](-2,0)(-2,4.5)[$^{1}$`{\ }]
\DGCdot{E}[u]{$a_{1}$}
\DGCcoupon*(-1.8,2)(-1.2,2.5){$\cdots$}
\DGCstrand[red](4,0)(4,4.5)[$^{n}$`{\ }]
\DGCdot{E}[u]{$a_{n}$}
\DGCcoupon*(3.2,2)(3.8,2.5){$\cdots$}
\DGCstrand[red](0,0)(.5,.5)[$^{i}$`{\ }]
\DGCdot{B}[l]{$t$}
\DGCstrand[red](1,0)(.5,.5)[$^{i+1}$`{\ }]
\DGCstrand[Red](.5,.5)(.5,1)
\DGCstrand[red](.5,1)(0,1.5)
\DGCstrand[red](.5,1)(1,1.5)
\DGCstrand[red](2,0)(2,1.5)[$^{i+2}$`{\ }]
\DGCstrand[red](1,1.5)(1.5,2)
%\DGCdot{B}[l]{$s$}
\DGCstrand[red](2,1.5)(1.5,2)
\DGCstrand[Red](1.5,2)(1.5,2.5)
\DGCstrand[red](1.5,2.5)(1,3)
\DGCstrand[red](1.5,2.5)(2,3)
\DGCstrand[red](0,1.5)(0,3)
\DGCdot{E}[r]{$r$}
%%%%
\DGCstrand[red](0,3)(.5,3.5)
\DGCstrand[red](1,3)(.5,3.5)
\DGCstrand[Red](.5,3.5)(.5,4)
\DGCstrand[red](.5,4)(0,4.5)
\DGCdot{E}[u]{$a_{i}$}
\DGCstrand[red](.5,4)(1,4.5)
\DGCdot{E}[u]{$a_{i+1}$}
\DGCstrand[red](2,3)(2,4.5)
\DGCdot{E}[u]{$a_{i+2}$}
%\DGCcoupon*(-2.75,2)(-2.25,2.5){$\cdots$}
%\DGCcoupon*(4.25,2)(4.75,2.5){$\cdots$}
%\DGCstrand(1.5,0)(2,2.25)(.5,4.5)
\DGCstrand(.5,0)(1.75,.75)(1.75,3.75)(1.5,4.5)
%\DGCstrand(.5,0)(2,2.25)(1.5,4.5) replace above
%\DGCdot{E}[u]{$b$}
\end{DGCpicture}
\quad \quad
a_i \in \mathbb{Z}_{\geq 0}, 
%b \in \mathbb{Z}_{\geq 0}, 
r,t \in \{0, 1\} ,
\end{equation*}

\begin{equation*}
\gimel_{11}({\bf a},r)
~=~
\begin{DGCpicture}[scale={.7,.7}]
\DGCstrand[red](-2,0)(-2,4.5)[$^{1}$`{\ }]
\DGCdot{E}[u]{$a_{1}$}
\DGCcoupon*(-1.8,2)(-1.2,2.5){$\cdots$}
\DGCstrand[red](4,0)(4,4.5)[$^{n}$`{\ }]
\DGCdot{E}[u]{$a_{n}$}
\DGCcoupon*(3.2,2)(3.8,2.5){$\cdots$}
\DGCstrand[red](0,0)(.5,.5)[$^{i}$`{\ }]
%\DGCdot{B}[l]{$t$}
\DGCstrand[red](1,0)(.5,.5)[$^{i+1}$`{\ }]
\DGCstrand[Red](.5,.5)(.5,1)
\DGCstrand[red](.5,1)(0,1.5)
\DGCstrand[red](.5,1)(1,1.5)
\DGCstrand[red](2,0)(2,1.5)[$^{i+2}$`{\ }]
\DGCstrand[red](1,1.5)(1.5,2)
%\DGCdot{B}[l]{$s$}
\DGCstrand[red](2,1.5)(1.5,2)
\DGCstrand[Red](1.5,2)(1.5,2.5)
\DGCstrand[red](1.5,2.5)(1,3)
\DGCstrand[red](1.5,2.5)(2,3)
\DGCstrand[red](0,1.5)(0,3)
%\DGCdot{E}[r]{$r$}
\DGCdot{2.25}[r]{$r$}
%%%%
\DGCstrand[red](0,3)(.5,3.5)
\DGCstrand[red](1,3)(.5,3.5)
\DGCstrand[Red](.5,3.5)(.5,4)
\DGCstrand[red](.5,4)(0,4.5)
\DGCdot{E}[u]{$a_{i}$}
\DGCstrand[red](.5,4)(1,4.5)
\DGCdot{E}[u]{$a_{i+1}$}
\DGCstrand[red](2,3)(2,4.5)
\DGCdot{E}[u]{$a_{i+2}$}
%\DGCcoupon*(-2.75,2)(-2.25,2.5){$\cdots$}
%\DGCcoupon*(4.25,2)(4.75,2.5){$\cdots$}
\DGCstrand(.5,0)(-.5,.75)(-.5,2.5)(1.5,3.25)(1.5,4.5)
%\DGCstrand(.5,0)(-.5,2.25)(1.5,4.5) replace above
%\DGCdot{E}[u]{$b$}
\end{DGCpicture}
\quad \quad
a_i \in \mathbb{Z}_{\geq 0}, 
%b \in \mathbb{Z}_{\geq 0}, 
r \in \{0, 1\} ,
\end{equation*}

\begin{equation*}
\gimel_{12}({\bf a},b,r,s,t,j,\ell)
~=~
\begin{DGCpicture}[scale={.7,.7}]
\DGCstrand[red](-2,0)(-2,4.5)[$^{1}$`{\ }]
\DGCdot{E}[u]{$a_{1}$}
\DGCcoupon*(-1.8,2)(-1.2,2.5){$\cdots$}
\DGCstrand[red](4,0)(4,4.5)[$^{n}$`{\ }]
\DGCdot{E}[u]{$a_{n}$}
\DGCcoupon*(3.2,2)(3.8,2.5){$\cdots$}
\DGCstrand[red](0,0)(.5,.5)[$^{i}$`{\ }]
\DGCdot{B}[l]{$t$}
\DGCstrand[red](1,0)(.5,.5)[$^{i+1}$`{\ }]
\DGCstrand[Red](.5,.5)(.5,1)
\DGCstrand[red](.5,1)(0,1.5)
\DGCstrand[red](.5,1)(1,1.5)
\DGCstrand[red](2,0)(2,1.5)[$^{i+2}$`{\ }]
\DGCstrand[red](1,1.5)(1.5,2)
\DGCdot{B}[l]{$s$}
\DGCstrand[red](2,1.5)(1.5,2)
\DGCstrand[Red](1.5,2)(1.5,2.5)
\DGCstrand[red](1.5,2.5)(1,3)
\DGCstrand[red](1.5,2.5)(2,3)
\DGCstrand[red](0,1.5)(0,3)
\DGCdot{E}[r]{$r$}
%%%%
\DGCstrand[red](0,3)(.5,3.5)
\DGCstrand[red](1,3)(.5,3.5)
\DGCstrand[Red](.5,3.5)(.5,4)
\DGCstrand[red](.5,4)(0,4.5)
\DGCdot{E}[ul]{$a_{i}$}
\DGCstrand[red](.5,4)(1,4.5)
\DGCdot{E}[ul]{$a_{i+1}$}
\DGCstrand[red](2,3)(2,4.5)
\DGCdot{E}[u]{$a_{i+2}$}
%\DGCcoupon*(-2.75,2)(-2.25,2.5){$\cdots$}
%\DGCcoupon*(4.25,2)(4.75,2.5){$\cdots$}
%\DGCstrand(1.5,0)(1,4.25)(.5,4.5)
\DGCstrand(-1,0)(-1,2.45)(3,3.25)(3,4.5)
%\DGCstrand(-1,0)(1,.75)(3,4.5) replace above
\DGCdot{E}[u]{$b$}
\end{DGCpicture}
\quad \quad
\begin{array}{l}a_i \in \mathbb{Z}_{\geq 0}, 
b \in \mathbb{Z}_{\geq 0}, 
r,s,t \in \{0, 1\}\\ 0\leq j,\ell\leq n\\
\text{ unless }(j,\ell)=(i,i),(i,i+1),(i+1,i),(i+1,i+1) ,
\end{array}
\end{equation*}
%\begin{equation*}
%\begin{DGCpicture}[scale={.7,.7}]
%\DGCstrand[red](-2,0)(-2,4.5)[$^{k}$`{\ }]
%\DGCdot{E}[u]{$a_{k}$}
%\DGCcoupon*(-1.75,2)(-1.25,2.5){$\cdots$}
%\DGCstrand[red](4,0)(4,4.5)[$^{l}$`{\ }]
%\DGCdot{E}[u]{$a_{l}$}
%\DGCcoupon*(3.25,2)(3.75,2.5){$\cdots$}
%\DGCstrand[red](0,0)(.5,.5)[$^{i}$`{\ }]
%\DGCdot{B}[l]{$t$}
%\DGCstrand[red](1,0)(.5,.5)[$^{i+1}$`{\ }]
%\DGCstrand[Red](.5,.5)(.5,1)
%\DGCstrand[red](.5,1)(0,1.5)
%\DGCstrand[red](.5,1)(1,1.5)
%\DGCstrand[red](2,0)(2,1.5)[$^{i+2}$`{\ }]
%%
%\DGCstrand[red](1,1.5)(1.5,2)
%\DGCdot{B}[r]{$s$}
%\DGCstrand[red](2,1.5)(1.5,2)
%\DGCstrand[Red](1.5,2)(1.5,2.5)
%\DGCstrand[red](1.5,2.5)(1,3)
%\DGCstrand[red](1.5,2.5)(2,3)
%\DGCstrand[red](0,1.5)(0,3)
%\DGCdot{E}[r]{$r$}
%%%%
%\DGCstrand[red](0,3)(.5,3.5)
%\DGCstrand[red](1,3)(.5,3.5)
%\DGCstrand[Red](.5,3.5)(.5,4)
%\DGCstrand[red](.5,4)(0,4.5)
%\DGCdot{E}[ul]{$a_{i}$}
%\DGCstrand[red](.5,4)(1,4.5)
%\DGCdot{E}[ul]{$a_{i+1}$}
%\DGCstrand[red](2,3)(2,4.5)
%\DGCdot{E}[u]{$a_{i+2}$}
%\DGCcoupon*(-2.75,2)(-2.25,2.5){$\cdots$}
%\DGCcoupon*(4.25,2)(4.75,2.5){$\cdots$}
%\DGCstrand(1.5,0)(1,4.25)(.5,4.5)
%\DGCstrand(-1,4.5)(1,.75)(3,0)
%\DGCdot{E}[u]{$b$}
%\end{DGCpicture}
%\quad \quad
%a_i \in \mathbb{Z}_{\geq 0}, 
%b \in \mathbb{Z}_{\geq 0}, 
%r,s,t \in \{0, 1\}
%\end{equation*}
%where the diagram in $\gimel_{12}({\bf a},b,r,s,t,j,\ell)$ represents any picture where the black strand begins from the $j+1$-th strand or ends to the $\ell+1$-th strand outside the segments connecting red points $i$, $i+1$, and $i+2$. Note that it could equally well be positioned in a northwest-southeast configuration.
%%%
where the diagram in $\gimel_{12}({\bf a},b,r,s,t,j,\ell)$ represents any picture where the black strand begins after the $j$th red strand and ends after the $\ell$th red strand 
%outside the segments connecting red points $i$, $i+1$, and $i+2$ 
and has a minimal number of intersections with the red strands. Note that it could equally well be positioned in a northwest-southeast configuration.
\end{prop}

\begin{proof}
By Lemma \ref{classical}, we may assume that any red dot configuration in the part of the picture connecting the $i$, $i+1$, and $i+2$ red boundary points is a linear combination of elements
\begin{equation}
\label{rbrr1}
\begin{DGCpicture}[scale={.7,.7}]
\DGCstrand[red](0,0)(.5,.5)[$ $`{\ }]
\DGCdot{B}[l]{$t$}
\DGCstrand[red](1,0)(.5,.5)[$ $`{\ }]
\DGCstrand[Red](.5,.5)(.5,1)
\DGCstrand[red](.5,1)(0,1.5)
\DGCstrand[red](.5,1)(1,1.5)
\DGCstrand[red](2,0)(2,1.5)[$ $`{\ }]
\DGCstrand[red](1,1.5)(1.5,2)
\DGCdot{B}[l]{$s$}
\DGCstrand[red](2,1.5)(1.5,2)
\DGCstrand[Red](1.5,2)(1.5,2.5)
\DGCstrand[red](1.5,2.5)(1,3)
\DGCstrand[red](1.5,2.5)(2,3)
\DGCstrand[red](0,1.5)(0,3)
\DGCdot{E}[r]{$r$}
%%%%
\DGCstrand[red](0,3)(.5,3.5)
\DGCstrand[red](1,3)(.5,3.5)
\DGCstrand[Red](.5,3.5)(.5,4)
\DGCstrand[red](.5,4)(0,4.5)
\DGCdot{E}[u]{$a_{i}$}
\DGCstrand[red](.5,4)(1,4.5)
\DGCdot{E}[ur]{$a_{i+1}$}
\DGCstrand[red](2,3)(2,4.5)
\DGCdot{E}[ur]{$a_{i+2}$}
%\DGCstrand(.5,0)(-.5,2.25)(.5,4.5)
%\DGCdot{E}[u]{$b$}
\end{DGCpicture}
\end{equation}
where $r,s,t \in \{0,1\}$.  
Next we will show that we do not need any black dots on a diagram of the form
\begin{equation}
\label{rbrr2}
\begin{DGCpicture}[scale={.7,.7}]
\DGCstrand[red](0,0)(.5,.5)[$ $`{\ }]
%\DGCdot{B}[l]{$t$}
\DGCstrand[red](1,0)(.5,.5)[$ $`{\ }]
\DGCstrand[Red](.5,.5)(.5,1)
\DGCstrand[red](.5,1)(0,1.5)
\DGCstrand[red](.5,1)(1,1.5)
\DGCstrand[red](2,0)(2,1.5)[$ $`{\ }]
\DGCstrand[red](1,1.5)(1.5,2)
%\DGCdot{B}[l]{$s$}
\DGCstrand[red](2,1.5)(1.5,2)
\DGCstrand[Red](1.5,2)(1.5,2.5)
\DGCstrand[red](1.5,2.5)(1,3)
\DGCstrand[red](1.5,2.5)(2,3)
\DGCstrand[red](0,1.5)(0,3)
%\DGCdot{E}[r]{$r$}
%%%%
\DGCstrand[red](0,3)(.5,3.5)
\DGCstrand[red](1,3)(.5,3.5)
\DGCstrand[Red](.5,3.5)(.5,4)
\DGCstrand[red](.5,4)(0,4.5)
\DGCdot{E}[u]{$a_{i}$}
\DGCstrand[red](.5,4)(1,4.5)
\DGCdot{E}[ur]{$a_{i+1}$}
\DGCstrand[red](2,3)(2,4.5)
\DGCdot{E}[ur]{$a_{i+2}$}
\DGCstrand(.5,0)(1.25,.75)(1.25,3.75)(.5,4.5)
%\DGCstrand(.5,0)(1,2.25)(.5,4.5) replace above
\DGCdot{E}[u]{$b$}
\end{DGCpicture}
\ .
\end{equation}
This follows from the calculation
\begin{equation*}
\label{rbrr3}
\begin{DGCpicture}[scale={.7,.7}]
\DGCstrand[red](0,0)(.5,.5)[$ $`{\ }]
%\DGCdot{B}[l]{$t$}
\DGCstrand[red](1,0)(.5,.5)[$ $`{\ }]
\DGCstrand[Red](.5,.5)(.5,1)
\DGCstrand[red](.5,1)(0,1.5)
\DGCstrand[red](.5,1)(1,1.5)
\DGCstrand[red](2,0)(2,1.5)[$ $`{\ }]
\DGCstrand[red](1,1.5)(1.5,2)
%\DGCdot{B}[l]{$s$}
\DGCstrand[red](2,1.5)(1.5,2)
\DGCstrand[Red](1.5,2)(1.5,2.5)
\DGCstrand[red](1.5,2.5)(1,3)
\DGCstrand[red](1.5,2.5)(2,3)
\DGCstrand[red](0,1.5)(0,3)
%\DGCdot{E}[r]{$r$}
%%%%
\DGCstrand[red](0,3)(.5,3.5)
\DGCstrand[red](1,3)(.5,3.5)
\DGCstrand[Red](.5,3.5)(.5,4)
\DGCstrand[red](.5,4)(0,4.5)
%\DGCdot{E}[u]{$a_{i}$}
\DGCstrand[red](.5,4)(1,4.5)
%\DGCdot{E}[u]{$a_{i+1}$}
\DGCstrand[red](2,3)(2,4.5)
%\DGCdot{E}[u]{$a_{i+2}$}
\DGCstrand(.5,0)(1.25,.75)(1.25,3.75)(.5,4.5)
%\DGCstrand(.5,0)(1,2.25)(.5,4.5) replace above
\DGCdot{E}[u]{$ $}
\end{DGCpicture}
~-~
\begin{DGCpicture}[scale={.7,.7}]
\DGCstrand[red](0,0)(.5,.5)[$ $`{\ }]
%\DGCdot{B}[l]{$t$}
\DGCstrand[red](1,0)(.5,.5)[$ $`{\ }]
\DGCstrand[Red](.5,.5)(.5,1)
\DGCstrand[red](.5,1)(0,1.5)
\DGCstrand[red](.5,1)(1,1.5)
\DGCstrand[red](2,0)(2,1.5)[$ $`{\ }]
\DGCstrand[red](1,1.5)(1.5,2)
%\DGCdot{B}[l]{$s$}
\DGCstrand[red](2,1.5)(1.5,2)
\DGCstrand[Red](1.5,2)(1.5,2.5)
\DGCstrand[red](1.5,2.5)(1,3)
\DGCstrand[red](1.5,2.5)(2,3)
\DGCstrand[red](0,1.5)(0,3)
%\DGCdot{E}[r]{$r$}
%%%%
\DGCstrand[red](0,3)(.5,3.5)
\DGCstrand[red](1,3)(.5,3.5)
\DGCstrand[Red](.5,3.5)(.5,4)
\DGCstrand[red](.5,4)(0,4.5)
\DGCdot{E}[u]{$ $}
\DGCstrand[red](.5,4)(1,4.5)
%\DGCdot{E}[u]{$a_{i+1}$}
\DGCstrand[red](2,3)(2,4.5)
%\DGCdot{E}[u]{$a_{i+2}$}
\DGCstrand(.5,0)(1.25,.75)(1.25,3.75)(.5,4.5)
%\DGCstrand(.5,0)(1,2.25)(.5,4.5) replace above
%\DGCdot{E}[u]{$ $}
\end{DGCpicture}
~=~
\begin{DGCpicture}[scale={.7,.7}]
\DGCstrand[red](0,0)(.5,.5)[$ $`{\ }]
%\DGCdot{B}[l]{$t$}
\DGCstrand[red](1,0)(.5,.5)[$ $`{\ }]
\DGCstrand[Red](.5,.5)(.5,1)
\DGCstrand[red](.5,1)(0,1.5)
\DGCstrand[red](.5,1)(1,1.5)
\DGCstrand[red](2,0)(2,1.5)[$ $`{\ }]
\DGCstrand[red](1,1.5)(1.5,2)
%\DGCdot{B}[l]{$s$}
\DGCstrand[red](2,1.5)(1.5,2)
\DGCstrand[Red](1.5,2)(1.5,2.5)
\DGCstrand[red](1.5,2.5)(1,3)
\DGCstrand[red](1.5,2.5)(2,3)
\DGCstrand[red](0,1.5)(0,3)
%\DGCdot{E}[r]{$r$}
%%%%
\DGCstrand[red](0,3)(.5,3.5)
\DGCstrand[red](1,3)(.5,3.5)
\DGCstrand[Red](.5,3.5)(.5,4)
\DGCstrand[red](.5,4)(0,4.5)
%\DGCdot{E}[u]{$a_{i}$}
\DGCstrand[red](.5,4)(1,4.5)
%\DGCdot{E}[u]{$a_{i+1}$}
\DGCstrand[red](2,3)(2,4.5)
%\DGCdot{E}[u]{$a_{i+2}$}
\DGCstrand(.5,0)(-.5,.75)(-.5,3.75)(.5,4.5)
%\DGCstrand(.5,0)(-.5,2.25)(.5,4.5) replace above
\DGCdot{E}[u]{$2 $}
\end{DGCpicture}
~-~
\begin{DGCpicture}[scale={.7,.7}]
\DGCstrand[red](0,0)(.5,.5)[$ $`{\ }]
%\DGCdot{B}[l]{$t$}
\DGCstrand[red](1,0)(.5,.5)[$ $`{\ }]
\DGCdot{B}[l]{$ $}
\DGCstrand[Red](.5,.5)(.5,1)
\DGCstrand[red](.5,1)(0,1.5)
\DGCstrand[red](.5,1)(1,1.5)
\DGCstrand[red](2,0)(2,1.5)[$ $`{\ }]
\DGCstrand[red](1,1.5)(1.5,2)
%\DGCdot{B}[l]{$s$}
\DGCstrand[red](2,1.5)(1.5,2)
\DGCstrand[Red](1.5,2)(1.5,2.5)
\DGCstrand[red](1.5,2.5)(1,3)
\DGCstrand[red](1.5,2.5)(2,3)
\DGCstrand[red](0,1.5)(0,3)
%\DGCdot{E}[r]{$r$}
%%%%
\DGCstrand[red](0,3)(.5,3.5)
\DGCstrand[red](1,3)(.5,3.5)
\DGCstrand[Red](.5,3.5)(.5,4)
\DGCstrand[red](.5,4)(0,4.5)
%\DGCdot{E}[u]{$a_{i}$}
\DGCstrand[red](.5,4)(1,4.5)
%\DGCdot{E}[u]{$a_{i+1}$}
\DGCstrand[red](2,3)(2,4.5)
%\DGCdot{E}[u]{$a_{i+2}$}
\DGCstrand(.5,0)(-.5,.75)(-.5,3.75)(.5,4.5)
%\DGCstrand(.5,0)(-.5,2.25)(.5,4.5) replace above
\DGCdot{E}[u]{$1 $}
\end{DGCpicture} 
~-~
\begin{DGCpicture}[scale={.7,.7}]
\DGCstrand[red](0,0)(.5,.5)[$ $`{\ }]
%\DGCdot{B}[l]{$t$}
\DGCstrand[red](1,0)(.5,.5)[$ $`{\ }]
%\DGCdot{B}[l]{$ $}
\DGCstrand[Red](.5,.5)(.5,1)
\DGCstrand[red](.5,1)(0,1.5)
\DGCstrand[red](.5,1)(1,1.5)
\DGCstrand[red](2,0)(2,1.5)[$ $`{\ }]
\DGCstrand[red](1,1.5)(1.5,2)
%\DGCdot{B}[l]{$s$}
\DGCstrand[red](2,1.5)(1.5,2)
\DGCstrand[Red](1.5,2)(1.5,2.5)
\DGCstrand[red](1.5,2.5)(1,3)
\DGCdot{E}[r]{$ $}
\DGCstrand[red](1.5,2.5)(2,3)
\DGCstrand[red](0,1.5)(0,3)
%\DGCdot{E}[r]{$r$}
%%%%
\DGCstrand[red](0,3)(.5,3.5)
\DGCstrand[red](1,3)(.5,3.5)
\DGCstrand[Red](.5,3.5)(.5,4)
\DGCstrand[red](.5,4)(0,4.5)
%\DGCdot{E}[u]{$a_{i}$}
\DGCstrand[red](.5,4)(1,4.5)
%\DGCdot{E}[u]{$a_{i+1}$}
\DGCstrand[red](2,3)(2,4.5)
%\DGCdot{E}[u]{$a_{i+2}$}
\DGCstrand(.5,0)(-.5,.75)(-.5,3.75)(.5,4.5)
%\DGCstrand(.5,0)(-.5,2.25)(.5,4.5) replace above
\DGCdot{E}[u]{$1 $}
\end{DGCpicture}
~+~
\begin{DGCpicture}[scale={.7,.7}]
\DGCstrand[red](0,0)(.5,.5)[$ $`{\ }]
%\DGCdot{B}[l]{$t$}
\DGCstrand[red](1,0)(.5,.5)[$ $`{\ }]
\DGCdot{B}[l]{$ $}
\DGCstrand[Red](.5,.5)(.5,1)
\DGCstrand[red](.5,1)(0,1.5)
\DGCstrand[red](.5,1)(1,1.5)
\DGCstrand[red](2,0)(2,1.5)[$ $`{\ }]
\DGCstrand[red](1,1.5)(1.5,2)
% \DGCdot{B}[l]{$ $}
\DGCstrand[red](2,1.5)(1.5,2)
\DGCstrand[Red](1.5,2)(1.5,2.5)
\DGCstrand[red](1.5,2.5)(1,3)
\DGCdot{E}[r]{$ $}
\DGCstrand[red](1.5,2.5)(2,3)
\DGCstrand[red](0,1.5)(0,3)
%\DGCdot{E}[r]{$r$}
%%%%
\DGCstrand[red](0,3)(.5,3.5)
\DGCstrand[red](1,3)(.5,3.5)
\DGCstrand[Red](.5,3.5)(.5,4)
\DGCstrand[red](.5,4)(0,4.5)
%\DGCdot{E}[u]{$a_{i}$}
\DGCstrand[red](.5,4)(1,4.5)
%\DGCdot{E}[u]{$a_{i+1}$}
\DGCstrand[red](2,3)(2,4.5)
%\DGCdot{E}[u]{$a_{i+2}$}
\DGCstrand(.5,0)(-.5,.75)(-.5,3.75)(.5,4.5)
%\DGCstrand(.5,0)(-.5,2.25)(.5,4.5) replace above
%\DGCdot{E}[u]{$1 $}
\end{DGCpicture}
\ .
\end{equation*}
Next note that
\begin{equation*}
\label{rbrr4}
\begin{DGCpicture}[scale={.7,.7}]
\DGCstrand[red](0,0)(.5,.5)[$ $`{\ }]
%\DGCdot{B}[l]{$t$}
\DGCstrand[red](1,0)(.5,.5)[$ $`{\ }]
\DGCstrand[Red](.5,.5)(.5,1)
\DGCstrand[red](.5,1)(0,1.5)
\DGCstrand[red](.5,1)(1,1.5)
\DGCstrand[red](2,0)(2,1.5)[$ $`{\ }]
\DGCstrand[red](1,1.5)(1.5,2)
%\DGCdot{B}[l]{$s$}
\DGCstrand[red](2,1.5)(1.5,2)
\DGCstrand[Red](1.5,2)(1.5,2.5)
\DGCstrand[red](1.5,2.5)(1,3)
\DGCstrand[red](1.5,2.5)(2,3)
\DGCstrand[red](0,1.5)(0,3)
%\DGCdot{E}[r]{$r$}
%%%%
\DGCstrand[red](0,3)(.5,3.5)
\DGCstrand[red](1,3)(.5,3.5)
\DGCstrand[Red](.5,3.5)(.5,4)
\DGCstrand[red](.5,4)(0,4.5)
%\DGCdot{E}[u]{$a_{i}$}
\DGCstrand[red](.5,4)(1,4.5)
%\DGCdot{E}[u]{$a_{i+1}$}
\DGCstrand[red](2,3)(2,4.5)
%\DGCdot{E}[u]{$a_{i+2}$}
\DGCstrand(.5,0)(1.25,.75)(1.25,3.75)(.5,4.5)
%\DGCstrand(.5,0)(1,2.25)(.5,4.5) replace above
\DGCdot{E}[u]{$ $}
\end{DGCpicture}
~-~
\begin{DGCpicture}[scale={.7,.7}]
\DGCstrand[red](0,0)(.5,.5)[$ $`{\ }]
%\DGCdot{B}[l]{$t$}
\DGCstrand[red](1,0)(.5,.5)[$ $`{\ }]
\DGCstrand[Red](.5,.5)(.5,1)
\DGCstrand[red](.5,1)(0,1.5)
\DGCstrand[red](.5,1)(1,1.5)
\DGCstrand[red](2,0)(2,1.5)[$ $`{\ }]
\DGCstrand[red](1,1.5)(1.5,2)
%\DGCdot{B}[l]{$s$}
\DGCstrand[red](2,1.5)(1.5,2)
\DGCstrand[Red](1.5,2)(1.5,2.5)
\DGCstrand[red](1.5,2.5)(1,3)
\DGCstrand[red](1.5,2.5)(2,3)
\DGCstrand[red](0,1.5)(0,3)
\DGCdot{E}[r]{$ $}
%%%%
\DGCstrand[red](0,3)(.5,3.5)
\DGCstrand[red](1,3)(.5,3.5)
\DGCstrand[Red](.5,3.5)(.5,4)
\DGCstrand[red](.5,4)(0,4.5)
%\DGCdot{E}[u]{$ $}
\DGCstrand[red](.5,4)(1,4.5)
%\DGCdot{E}[u]{$a_{i+1}$}
\DGCstrand[red](2,3)(2,4.5)
%\DGCdot{E}[u]{$a_{i+2}$}
\DGCstrand(.5,0)(1.25,.75)(1.25,3.75)(.5,4.5)
%\DGCstrand(.5,0)(1,2.25)(.5,4.5) replace above
%\DGCdot{E}[u]{$ $}
\end{DGCpicture}
~=~
\begin{DGCpicture}[scale={.7,.7}]
\DGCstrand[red](0,0)(.5,.5)[$ $`{\ }]
%\DGCdot{B}[l]{$t$}
\DGCstrand[red](1,0)(.5,.5)[$ $`{\ }]
\DGCstrand[Red](.5,.5)(.5,1)
\DGCstrand[red](.5,1)(0,1.5)
\DGCstrand[red](.5,1)(1,1.5)
\DGCstrand[red](2,0)(2,1.5)[$ $`{\ }]
\DGCstrand[red](1,1.5)(1.5,2)
%\DGCdot{B}[l]{$s$}
\DGCstrand[red](2,1.5)(1.5,2)
\DGCstrand[Red](1.5,2)(1.5,2.5)
\DGCstrand[red](1.5,2.5)(1,3)
\DGCstrand[red](1.5,2.5)(2,3)
\DGCstrand[red](0,1.5)(0,3)
%\DGCdot{E}[r]{$r$}
%%%%
\DGCstrand[red](0,3)(.5,3.5)
\DGCstrand[red](1,3)(.5,3.5)
\DGCstrand[Red](.5,3.5)(.5,4)
\DGCstrand[red](.5,4)(0,4.5)
%\DGCdot{E}[u]{$a_{i}$}
\DGCstrand[red](.5,4)(1,4.5)
%\DGCdot{E}[u]{$a_{i+1}$}
\DGCstrand[red](2,3)(2,4.5)
%\DGCdot{E}[u]{$a_{i+2}$}
\DGCstrand(.5,0)(-.5,.75)(-.5,3.75)(.5,4.5)
%\DGCstrand(.5,0)(-.5,2.25)(.5,4.5) replace above
\DGCdot{E}[u]{$2 $}
\end{DGCpicture}
~-~
\begin{DGCpicture}[scale={.7,.7}]
\DGCstrand[red](0,0)(.5,.5)[$ $`{\ }]
%\DGCdot{B}[l]{$t$}
\DGCstrand[red](1,0)(.5,.5)[$ $`{\ }]
\DGCdot{B}[l]{$ $}
\DGCstrand[Red](.5,.5)(.5,1)
\DGCstrand[red](.5,1)(0,1.5)
\DGCstrand[red](.5,1)(1,1.5)
\DGCstrand[red](2,0)(2,1.5)[$ $`{\ }]
\DGCstrand[red](1,1.5)(1.5,2)
%\DGCdot{B}[l]{$s$}
\DGCstrand[red](2,1.5)(1.5,2)
\DGCstrand[Red](1.5,2)(1.5,2.5)
\DGCstrand[red](1.5,2.5)(1,3)
\DGCstrand[red](1.5,2.5)(2,3)
\DGCstrand[red](0,1.5)(0,3)
%\DGCdot{E}[r]{$r$}
%%%%
\DGCstrand[red](0,3)(.5,3.5)
\DGCstrand[red](1,3)(.5,3.5)
\DGCstrand[Red](.5,3.5)(.5,4)
\DGCstrand[red](.5,4)(0,4.5)
%\DGCdot{E}[u]{$a_{i}$}
\DGCstrand[red](.5,4)(1,4.5)
%\DGCdot{E}[u]{$a_{i+1}$}
\DGCstrand[red](2,3)(2,4.5)
%\DGCdot{E}[u]{$a_{i+2}$}
\DGCstrand(.5,0)(-.5,.75)(-.5,3.75)(.5,4.5)
%\DGCstrand(.5,0)(-.5,2.25)(.5,4.5) replace above
\DGCdot{E}[u]{$1 $}
\end{DGCpicture} 
~-~
\begin{DGCpicture}[scale={.7,.7}]
\DGCstrand[red](0,0)(.5,.5)[$ $`{\ }]
%\DGCdot{B}[l]{$t$}
\DGCstrand[red](1,0)(.5,.5)[$ $`{\ }]
%\DGCdot{B}[l]{$ $}
\DGCstrand[Red](.5,.5)(.5,1)
\DGCstrand[red](.5,1)(0,1.5)
\DGCstrand[red](.5,1)(1,1.5)
\DGCstrand[red](2,0)(2,1.5)[$ $`{\ }]
\DGCstrand[red](1,1.5)(1.5,2)
%\DGCdot{B}[l]{$s$}
\DGCstrand[red](2,1.5)(1.5,2)
\DGCstrand[Red](1.5,2)(1.5,2.5)
\DGCstrand[red](1.5,2.5)(1,3)
%\DGCdot{E}[r]{$ $}
\DGCstrand[red](1.5,2.5)(2,3)
\DGCstrand[red](0,1.5)(0,3)
%\DGCdot{E}[r]{$r$}
%%%%
\DGCstrand[red](0,3)(.5,3.5)
\DGCstrand[red](1,3)(.5,3.5)
\DGCstrand[Red](.5,3.5)(.5,4)
\DGCstrand[red](.5,4)(0,4.5)
%\DGCdot{E}[u]{$a_{i}$}
\DGCstrand[red](.5,4)(1,4.5)
\DGCdot{E}[u]{$ $}
\DGCstrand[red](2,3)(2,4.5)
%\DGCdot{E}[u]{$a_{i+2}$}
\DGCstrand(.5,0)(-.5,.75)(-.5,3.75)(.5,4.5)
%\DGCstrand(.5,0)(-.5,2.25)(.5,4.5) replace above
\DGCdot{E}[u]{$1 $}
\end{DGCpicture}
~+~
\begin{DGCpicture}[scale={.7,.7}]
\DGCstrand[red](0,0)(.5,.5)[$ $`{\ }]
%\DGCdot{B}[l]{$t$}
\DGCstrand[red](1,0)(.5,.5)[$ $`{\ }]
\DGCdot{B}[l]{$ $}
\DGCstrand[Red](.5,.5)(.5,1)
\DGCstrand[red](.5,1)(0,1.5)
\DGCstrand[red](.5,1)(1,1.5)
\DGCstrand[red](2,0)(2,1.5)[$ $`{\ }]
\DGCstrand[red](1,1.5)(1.5,2)
%\DGCdot{B}[l]{$ $}
\DGCstrand[red](2,1.5)(1.5,2)
\DGCstrand[Red](1.5,2)(1.5,2.5)
\DGCstrand[red](1.5,2.5)(1,3)
%\DGCdot{E}[r]{$ $}
\DGCstrand[red](1.5,2.5)(2,3)
\DGCstrand[red](0,1.5)(0,3)
%\DGCdot{E}[r]{$r$}
%%%%
\DGCstrand[red](0,3)(.5,3.5)
\DGCstrand[red](1,3)(.5,3.5)
\DGCstrand[Red](.5,3.5)(.5,4)
\DGCstrand[red](.5,4)(0,4.5)
%\DGCdot{E}[u]{$a_{i}$}
\DGCstrand[red](.5,4)(1,4.5)
\DGCdot{E}[u]{$ $}
\DGCstrand[red](2,3)(2,4.5)
%\DGCdot{E}[u]{$a_{i+2}$}
\DGCstrand(.5,0)(-.5,.75)(-.5,3.75)(.5,4.5)
%\DGCstrand(.5,0)(-.5,2.25)(.5,4.5) replace above
%\DGCdot{E}[u]{$1 $}
\end{DGCpicture}
\ .
\end{equation*}
Thus 
\begin{equation*}
\begin{DGCpicture}[scale={.7,.7}]
\DGCstrand[red](0,0)(.5,.5)[$ $`{\ }]
%\DGCdot{B}[l]{$t$}
\DGCstrand[red](1,0)(.5,.5)[$ $`{\ }]
\DGCstrand[Red](.5,.5)(.5,1)
\DGCstrand[red](.5,1)(0,1.5)
\DGCstrand[red](.5,1)(1,1.5)
\DGCstrand[red](2,0)(2,1.5)[$ $`{\ }]
\DGCstrand[red](1,1.5)(1.5,2)
%\DGCdot{B}[l]{$s$}
\DGCstrand[red](2,1.5)(1.5,2)
\DGCstrand[Red](1.5,2)(1.5,2.5)
\DGCstrand[red](1.5,2.5)(1,3)
\DGCstrand[red](1.5,2.5)(2,3)
\DGCstrand[red](0,1.5)(0,3)
\DGCdot{E}[r]{$ $}
%%%%
\DGCstrand[red](0,3)(.5,3.5)
\DGCstrand[red](1,3)(.5,3.5)
\DGCstrand[Red](.5,3.5)(.5,4)
\DGCstrand[red](.5,4)(0,4.5)
%\DGCdot{E}[u]{$ $}
\DGCstrand[red](.5,4)(1,4.5)
%\DGCdot{E}[u]{$a_{i+1}$}
\DGCstrand[red](2,3)(2,4.5)
%\DGCdot{E}[u]{$a_{i+2}$}
\DGCstrand(.5,0)(1.25,.75)(1.25,3.75)(.5,4.5)
%\DGCstrand(.5,0)(1,2.25)(.5,4.5) replace above
%\DGCdot{E}[u]{$ $}
\end{DGCpicture}
\end{equation*}
is already in the span of the proposed spanning set.  Similarly,
\begin{equation*}
\begin{DGCpicture}[scale={.7,.7}]
\DGCstrand[red](0,0)(.5,.5)[$ $`{\ }]
\DGCdot{B}[l]{$ $}
\DGCstrand[red](1,0)(.5,.5)[$ $`{\ }]
\DGCstrand[Red](.5,.5)(.5,1)
\DGCstrand[red](.5,1)(0,1.5)
\DGCstrand[red](.5,1)(1,1.5)
\DGCstrand[red](2,0)(2,1.5)[$ $`{\ }]
\DGCstrand[red](1,1.5)(1.5,2)
%\DGCdot{B}[l]{$s$}
\DGCstrand[red](2,1.5)(1.5,2)
\DGCstrand[Red](1.5,2)(1.5,2.5)
\DGCstrand[red](1.5,2.5)(1,3)
\DGCstrand[red](1.5,2.5)(2,3)
\DGCstrand[red](0,1.5)(0,3)
%\DGCdot{E}[r]{$ $}
%%%%
\DGCstrand[red](0,3)(.5,3.5)
\DGCstrand[red](1,3)(.5,3.5)
\DGCstrand[Red](.5,3.5)(.5,4)
\DGCstrand[red](.5,4)(0,4.5)
%\DGCdot{E}[u]{$ $}
\DGCstrand[red](.5,4)(1,4.5)
%\DGCdot{E}[u]{$a_{i+1}$}
\DGCstrand[red](2,3)(2,4.5)
%\DGCdot{E}[u]{$a_{i+2}$}
\DGCstrand(.5,0)(1.25,.75)(1.25,3.75)(.5,4.5)
%\DGCstrand(.5,0)(1,2.25)(.5,4.5) replace above
%\DGCdot{E}[u]{$ $}
\end{DGCpicture}
\end{equation*}
is also already in the span of the proposed spanning set.

Next we will show that we do not need black dots on a diagram of the form
\begin{equation*}
\begin{DGCpicture}[scale={.7,.7}]
\DGCstrand[red](0,0)(.5,.5)[$^{i}$`{\ }]
%\DGCdot{B}[l]{$t$}
\DGCstrand[red](1,0)(.5,.5)[$^{i+1}$`{\ }]
\DGCstrand[Red](.5,.5)(.5,1)
\DGCstrand[red](.5,1)(0,1.5)
\DGCstrand[red](.5,1)(1,1.5)
\DGCstrand[red](2,0)(2,1.5)[$^{i+2}$`{\ }]
\DGCstrand[red](1,1.5)(1.5,2)
%\DGCdot{B}[r]{$s$}
\DGCstrand[red](2,1.5)(1.5,2)
\DGCstrand[Red](1.5,2)(1.5,2.5)
\DGCstrand[red](1.5,2.5)(1,3)
\DGCstrand[red](1.5,2.5)(2,3)
\DGCstrand[red](0,1.5)(0,3)
%\DGCdot{E}[r]{$r$}
%%%%
\DGCstrand[red](0,3)(.5,3.5)
\DGCstrand[red](1,3)(.5,3.5)
\DGCstrand[Red](.5,3.5)(.5,4)
\DGCstrand[red](.5,4)(0,4.5)
\DGCdot{E}[u]{$a_{i}$}
\DGCstrand[red](.5,4)(1,4.5)
\DGCdot{E}[u]{$a_{i+1}$}
\DGCstrand[red](2,3)(2,4.5)
\DGCdot{E}[u]{$a_{i+2}$}
\DGCstrand(.5,0)(2.5,.75)(2.5,3.75)(.5,4.5)
%\DGCstrand(.5,0)(2.5,2.25)(.5,4.5) replace above
%\DGCdot{E}[u]{$b$}
\end{DGCpicture}
\ .
\end{equation*}
This follows from the calculation
\begin{equation*}
\begin{DGCpicture}[scale={.7,.7}]
\DGCstrand[red](0,0)(.5,.5)[$^{}$`{\ }]
%\DGCdot{B}[l]{$t$}
\DGCstrand[red](1,0)(.5,.5)[$^{}$`{\ }]
\DGCstrand[Red](.5,.5)(.5,1)
\DGCstrand[red](.5,1)(0,1.5)
\DGCstrand[red](.5,1)(1,1.5)
\DGCstrand[red](2,0)(2,1.5)[$^{}$`{\ }]
\DGCstrand[red](1,1.5)(1.5,2)
%\DGCdot{B}[r]{$s$}
\DGCstrand[red](2,1.5)(1.5,2)
\DGCstrand[Red](1.5,2)(1.5,2.5)
\DGCstrand[red](1.5,2.5)(1,3)
\DGCstrand[red](1.5,2.5)(2,3)
\DGCstrand[red](0,1.5)(0,3)
%\DGCdot{E}[r]{$r$}
%%%%
\DGCstrand[red](0,3)(.5,3.5)
\DGCstrand[red](1,3)(.5,3.5)
\DGCstrand[Red](.5,3.5)(.5,4)
\DGCstrand[red](.5,4)(0,4.5)
%\DGCdot{E}[u]{$a_{i}$}
\DGCstrand[red](.5,4)(1,4.5)
%\DGCdot{E}[u]{$a_{i+1}$}
\DGCstrand[red](2,3)(2,4.5)
%\DGCdot{E}[u]{$a_{i+2}$}
\DGCstrand(.5,0)(2.5,.75)(2.5,3.75)(.5,4.5)
%\DGCstrand(.5,0)(2.5,2.25)(.5,4.5) replace above
\DGCdot{E}[u]{$ $}
\end{DGCpicture}
~-~
\begin{DGCpicture}[scale={.7,.7}]
\DGCstrand[red](0,0)(.5,.5)[$^{}$`{\ }]
%\DGCdot{B}[l]{$t$}
\DGCstrand[red](1,0)(.5,.5)[$^{}$`{\ }]
\DGCstrand[Red](.5,.5)(.5,1)
\DGCstrand[red](.5,1)(0,1.5)
\DGCstrand[red](.5,1)(1,1.5)
\DGCstrand[red](2,0)(2,1.5)[$^{}$`{\ }]
\DGCstrand[red](1,1.5)(1.5,2)
%\DGCdot{B}[r]{$s$}
\DGCstrand[red](2,1.5)(1.5,2)
\DGCstrand[Red](1.5,2)(1.5,2.5)
\DGCstrand[red](1.5,2.5)(1,3)
\DGCstrand[red](1.5,2.5)(2,3)
\DGCstrand[red](0,1.5)(0,3)
%\DGCdot{E}[r]{$r$}
%%%%
\DGCstrand[red](0,3)(.5,3.5)
\DGCstrand[red](1,3)(.5,3.5)
\DGCstrand[Red](.5,3.5)(.5,4)
\DGCstrand[red](.5,4)(0,4.5)
\DGCdot{E}[u]{$ $}
\DGCstrand[red](.5,4)(1,4.5)
%\DGCdot{E}[u]{$a_{i+1}$}
\DGCstrand[red](2,3)(2,4.5)
%\DGCdot{E}[u]{$a_{i+2}$}
\DGCstrand(.5,0)(2.5,.75)(2.5,3.75)(.5,4.5)
%\DGCstrand(.5,0)(2.5,2.25)(.5,4.5) replace above
%\DGCdot{E}[u]{$ $}
\end{DGCpicture}
~=~
\begin{DGCpicture}[scale={.7,.7}]
\DGCstrand[red](0,0)(.5,.5)[$ $`{\ }]
%\DGCdot{B}[l]{$t$}
\DGCstrand[red](1,0)(.5,.5)[$ $`{\ }]
\DGCstrand[Red](.5,.5)(.5,1)
\DGCstrand[red](.5,1)(0,1.5)
\DGCstrand[red](.5,1)(1,1.5)
\DGCstrand[red](2,0)(2,1.5)[$ $`{\ }]
\DGCstrand[red](1,1.5)(1.5,2)
%\DGCdot{B}[l]{$s$}
\DGCstrand[red](2,1.5)(1.5,2)
\DGCstrand[Red](1.5,2)(1.5,2.5)
\DGCstrand[red](1.5,2.5)(1,3)
\DGCstrand[red](1.5,2.5)(2,3)
\DGCstrand[red](0,1.5)(0,3)
%\DGCdot{E}[r]{$r$}
%%%%
\DGCstrand[red](0,3)(.5,3.5)
\DGCstrand[red](1,3)(.5,3.5)
\DGCstrand[Red](.5,3.5)(.5,4)
\DGCstrand[red](.5,4)(0,4.5)
%\DGCdot{E}[u]{$a_{i}$}
\DGCstrand[red](.5,4)(1,4.5)
%\DGCdot{E}[u]{$a_{i+1}$}
\DGCstrand[red](2,3)(2,4.5)
%\DGCdot{E}[u]{$a_{i+2}$}
\DGCstrand(.5,0)(-.5,.75)(-.5,3.75)(.5,4.5)
%\DGCstrand(.5,0)(-.5,2.25)(.5,4.5) replace above
\DGCdot{E}[u]{$2 $}
\end{DGCpicture}
~-~
\begin{DGCpicture}[scale={.7,.7}]
\DGCstrand[red](0,0)(.5,.5)[$ $`{\ }]
%\DGCdot{B}[l]{$t$}
\DGCstrand[red](1,0)(.5,.5)[$ $`{\ }]
\DGCdot{B}[l]{$ $}
\DGCstrand[Red](.5,.5)(.5,1)
\DGCstrand[red](.5,1)(0,1.5)
\DGCstrand[red](.5,1)(1,1.5)
\DGCstrand[red](2,0)(2,1.5)[$ $`{\ }]
\DGCstrand[red](1,1.5)(1.5,2)
%\DGCdot{B}[l]{$s$}
\DGCstrand[red](2,1.5)(1.5,2)
\DGCstrand[Red](1.5,2)(1.5,2.5)
\DGCstrand[red](1.5,2.5)(1,3)
\DGCstrand[red](1.5,2.5)(2,3)
\DGCstrand[red](0,1.5)(0,3)
%\DGCdot{E}[r]{$r$}
%%%%
\DGCstrand[red](0,3)(.5,3.5)
\DGCstrand[red](1,3)(.5,3.5)
\DGCstrand[Red](.5,3.5)(.5,4)
\DGCstrand[red](.5,4)(0,4.5)
%\DGCdot{E}[u]{$a_{i}$}
\DGCstrand[red](.5,4)(1,4.5)
%\DGCdot{E}[u]{$a_{i+1}$}
\DGCstrand[red](2,3)(2,4.5)
%\DGCdot{E}[u]{$a_{i+2}$}
\DGCstrand(.5,0)(-.5,.75)(-.5,3.75)(.5,4.5)
%\DGCstrand(.5,0)(-.5,2.25)(.5,4.5) replace above
\DGCdot{E}[u]{$1 $}
\end{DGCpicture} 
~-~
\begin{DGCpicture}[scale={.7,.7}]
\DGCstrand[red](0,0)(.5,.5)[$ $`{\ }]
%\DGCdot{B}[l]{$t$}
\DGCstrand[red](1,0)(.5,.5)[$ $`{\ }]
%\DGCdot{B}[l]{$ $}
\DGCstrand[Red](.5,.5)(.5,1)
\DGCstrand[red](.5,1)(0,1.5)
\DGCstrand[red](.5,1)(1,1.5)
\DGCstrand[red](2,0)(2,1.5)[$ $`{\ }]
\DGCdot{B}[l]{$ $}
\DGCstrand[red](1,1.5)(1.5,2)
%\DGCdot{B}[l]{$s$}
\DGCstrand[red](2,1.5)(1.5,2)
\DGCstrand[Red](1.5,2)(1.5,2.5)
\DGCstrand[red](1.5,2.5)(1,3)
%\DGCdot{E}[r]{$ $}
\DGCstrand[red](1.5,2.5)(2,3)
\DGCstrand[red](0,1.5)(0,3)
%\DGCdot{E}[r]{$r$}
%%%%
\DGCstrand[red](0,3)(.5,3.5)
\DGCstrand[red](1,3)(.5,3.5)
\DGCstrand[Red](.5,3.5)(.5,4)
\DGCstrand[red](.5,4)(0,4.5)
%\DGCdot{E}[u]{$a_{i}$}
\DGCstrand[red](.5,4)(1,4.5)
%\DGCdot{E}[u]{$ $}
\DGCstrand[red](2,3)(2,4.5)
%\DGCdot{E}[u]{$a_{i+2}$}
\DGCstrand(.5,0)(-.5,.75)(-.5,3.75)(.5,4.5)
%\DGCstrand(.5,0)(-.5,2.25)(.5,4.5) replace above
\DGCdot{E}[u]{$1 $}
\end{DGCpicture}
~+~
\begin{DGCpicture}[scale={.7,.7}]
\DGCstrand[red](0,0)(.5,.5)[$ $`{\ }]
%\DGCdot{B}[l]{$t$}
\DGCstrand[red](1,0)(.5,.5)[$ $`{\ }]
\DGCdot{B}[l]{$ $}
\DGCstrand[Red](.5,.5)(.5,1)
\DGCstrand[red](.5,1)(0,1.5)
\DGCstrand[red](.5,1)(1,1.5)
\DGCstrand[red](2,0)(2,1.5)[$ $`{\ }]
\DGCdot{B}[l]{$ $}
\DGCstrand[red](1,1.5)(1.5,2)
%\DGCdot{B}[l]{$ $}
\DGCstrand[red](2,1.5)(1.5,2)
\DGCstrand[Red](1.5,2)(1.5,2.5)
\DGCstrand[red](1.5,2.5)(1,3)
%\DGCdot{E}[r]{$ $}
\DGCstrand[red](1.5,2.5)(2,3)
\DGCstrand[red](0,1.5)(0,3)
%\DGCdot{E}[r]{$r$}
%%%%
\DGCstrand[red](0,3)(.5,3.5)
\DGCstrand[red](1,3)(.5,3.5)
\DGCstrand[Red](.5,3.5)(.5,4)
\DGCstrand[red](.5,4)(0,4.5)
%\DGCdot{E}[u]{$a_{i}$}
\DGCstrand[red](.5,4)(1,4.5)
%\DGCdot{E}[u]{$ $}
\DGCstrand[red](2,3)(2,4.5)
%\DGCdot{E}[u]{$a_{i+2}$}
\DGCstrand(.5,0)(-.5,.75)(-.5,3.75)(.5,4.5)
%\DGCstrand(.5,0)(-.5,2.25)(.5,4.5) replace above
%\DGCdot{E}[u]{$1 $}
\end{DGCpicture}
\ .
\end{equation*}
Next note that
\begin{equation*}
\begin{DGCpicture}[scale={.7,.7}]
\DGCstrand[red](0,0)(.5,.5)[$^{}$`{\ }]
%\DGCdot{B}[l]{$t$}
\DGCstrand[red](1,0)(.5,.5)[$^{}$`{\ }]
\DGCstrand[Red](.5,.5)(.5,1)
\DGCstrand[red](.5,1)(0,1.5)
\DGCstrand[red](.5,1)(1,1.5)
\DGCstrand[red](2,0)(2,1.5)[$^{}$`{\ }]
\DGCstrand[red](1,1.5)(1.5,2)
%\DGCdot{B}[r]{$s$}
\DGCstrand[red](2,1.5)(1.5,2)
\DGCstrand[Red](1.5,2)(1.5,2.5)
\DGCstrand[red](1.5,2.5)(1,3)
\DGCstrand[red](1.5,2.5)(2,3)
\DGCstrand[red](0,1.5)(0,3)
%\DGCdot{E}[r]{$r$}
%%%%
\DGCstrand[red](0,3)(.5,3.5)
\DGCstrand[red](1,3)(.5,3.5)
\DGCstrand[Red](.5,3.5)(.5,4)
\DGCstrand[red](.5,4)(0,4.5)
%\DGCdot{E}[u]{$a_{i}$}
\DGCstrand[red](.5,4)(1,4.5)
%\DGCdot{E}[u]{$a_{i+1}$}
\DGCstrand[red](2,3)(2,4.5)
%\DGCdot{E}[u]{$a_{i+2}$}
\DGCstrand(.5,0)(2.5,.75)(2.5,3.75)(.5,4.5)
%\DGCstrand(.5,0)(2.5,2.25)(.5,4.5) replace above
\DGCdot{E}[u]{$ $}
\end{DGCpicture}
~-~
\begin{DGCpicture}[scale={.7,.7}]
\DGCstrand[red](0,0)(.5,.5)[$^{}$`{\ }]
%\DGCdot{B}[l]{$t$}
\DGCstrand[red](1,0)(.5,.5)[$^{}$`{\ }]
\DGCstrand[Red](.5,.5)(.5,1)
\DGCstrand[red](.5,1)(0,1.5)
\DGCstrand[red](.5,1)(1,1.5)
\DGCstrand[red](2,0)(2,1.5)[$^{}$`{\ }]
\DGCstrand[red](1,1.5)(1.5,2)
\DGCdot{B}[r]{$ $}
\DGCstrand[red](2,1.5)(1.5,2)
\DGCstrand[Red](1.5,2)(1.5,2.5)
\DGCstrand[red](1.5,2.5)(1,3)
\DGCstrand[red](1.5,2.5)(2,3)
\DGCstrand[red](0,1.5)(0,3)
%\DGCdot{E}[r]{$r$}
%%%%
\DGCstrand[red](0,3)(.5,3.5)
\DGCstrand[red](1,3)(.5,3.5)
\DGCstrand[Red](.5,3.5)(.5,4)
\DGCstrand[red](.5,4)(0,4.5)
%\DGCdot{E}[u]{$ $}
\DGCstrand[red](.5,4)(1,4.5)
%\DGCdot{E}[u]{$a_{i+1}$}
\DGCstrand[red](2,3)(2,4.5)
%\DGCdot{E}[u]{$a_{i+2}$}
\DGCstrand(.5,0)(2.5,.75)(2.5,3.75)(.5,4.5)
%\DGCstrand(.5,0)(2.5,2.25)(.5,4.5) replace above
%\DGCdot{E}[u]{$ $}
\end{DGCpicture}
~=~
\begin{DGCpicture}[scale={.7,.7}]
\DGCstrand[red](0,0)(.5,.5)[$ $`{\ }]
%\DGCdot{B}[l]{$t$}
\DGCstrand[red](1,0)(.5,.5)[$ $`{\ }]
\DGCstrand[Red](.5,.5)(.5,1)
\DGCstrand[red](.5,1)(0,1.5)
\DGCstrand[red](.5,1)(1,1.5)
\DGCstrand[red](2,0)(2,1.5)[$ $`{\ }]
\DGCstrand[red](1,1.5)(1.5,2)
%\DGCdot{B}[l]{$ $}
\DGCstrand[red](2,1.5)(1.5,2)
\DGCstrand[Red](1.5,2)(1.5,2.5)
\DGCstrand[red](1.5,2.5)(1,3)
\DGCstrand[red](1.5,2.5)(2,3)
\DGCstrand[red](0,1.5)(0,3)
%\DGCdot{E}[r]{$r$}
%%%%
\DGCstrand[red](0,3)(.5,3.5)
\DGCstrand[red](1,3)(.5,3.5)
\DGCstrand[Red](.5,3.5)(.5,4)
\DGCstrand[red](.5,4)(0,4.5)
%\DGCdot{E}[u]{$a_{i}$}
\DGCstrand[red](.5,4)(1,4.5)
%\DGCdot{E}[u]{$a_{i+1}$}
\DGCstrand[red](2,3)(2,4.5)
%\DGCdot{E}[u]{$a_{i+2}$}
\DGCstrand(.5,0)(1.25,.75)(1.25,3.75)(.5,4.5)
%\DGCstrand(.5,0)(1,2.25)(.5,4.5) replace above
\DGCdot{E}[u]{$ $}
\end{DGCpicture}
~-~
\begin{DGCpicture}[scale={.7,.7}]
\DGCstrand[red](0,0)(.5,.5)[$ $`{\ }]
%\DGCdot{B}[l]{$t$}
\DGCstrand[red](1,0)(.5,.5)[$ $`{\ }]
%\DGCdot{B}[l]{$ $}
\DGCstrand[Red](.5,.5)(.5,1)
\DGCstrand[red](.5,1)(0,1.5)
\DGCstrand[red](.5,1)(1,1.5)
\DGCstrand[red](2,0)(2,1.5)[$ $`{\ }]
\DGCstrand[red](1,1.5)(1.5,2)
%\DGCdot{B}[l]{$s$}
\DGCstrand[red](2,1.5)(1.5,2)
\DGCstrand[Red](1.5,2)(1.5,2.5)
\DGCstrand[red](1.5,2.5)(1,3)
\DGCstrand[red](1.5,2.5)(2,3)
\DGCstrand[red](0,1.5)(0,3)
%\DGCdot{E}[r]{$r$}
%%%%
\DGCstrand[red](0,3)(.5,3.5)
\DGCstrand[red](1,3)(.5,3.5)
\DGCstrand[Red](.5,3.5)(.5,4)
\DGCstrand[red](.5,4)(0,4.5)
%\DGCdot{E}[u]{$a_{i}$}
\DGCstrand[red](.5,4)(1,4.5)
%\DGCdot{E}[u]{$a_{i+1}$}
\DGCstrand[red](2,3)(2,4.5)
\DGCdot{E}[u]{$ $}
\DGCstrand(.5,0)(1.25,.75)(1.25,3.75)(.5,4.5)
%\DGCstrand(.5,0)(1,2.25)(.5,4.5) replace above
%\DGCdot{E}[u]{$1 $}
\end{DGCpicture} 
\end{equation*}
which shows that 
\begin{equation*}
\begin{DGCpicture}[scale={.7,.7}]
\DGCstrand[red](0,0)(.5,.5)[$^{}$`{\ }]
%\DGCdot{B}[l]{$t$}
\DGCstrand[red](1,0)(.5,.5)[$^{}$`{\ }]
\DGCstrand[Red](.5,.5)(.5,1)
\DGCstrand[red](.5,1)(0,1.5)
\DGCstrand[red](.5,1)(1,1.5)
\DGCstrand[red](2,0)(2,1.5)[$^{}$`{\ }]
\DGCstrand[red](1,1.5)(1.5,2)
\DGCdot{B}[r]{$ $}
\DGCstrand[red](2,1.5)(1.5,2)
\DGCstrand[Red](1.5,2)(1.5,2.5)
\DGCstrand[red](1.5,2.5)(1,3)
\DGCstrand[red](1.5,2.5)(2,3)
\DGCstrand[red](0,1.5)(0,3)
%\DGCdot{E}[r]{$r$}
%%%%
\DGCstrand[red](0,3)(.5,3.5)
\DGCstrand[red](1,3)(.5,3.5)
\DGCstrand[Red](.5,3.5)(.5,4)
\DGCstrand[red](.5,4)(0,4.5)
%\DGCdot{E}[u]{$ $}
\DGCstrand[red](.5,4)(1,4.5)
%\DGCdot{E}[u]{$a_{i+1}$}
\DGCstrand[red](2,3)(2,4.5)
%\DGCdot{E}[u]{$a_{i+2}$}
\DGCstrand(.5,0)(2.5,.75)(2.5,3.75)(.5,4.5)
%\DGCstrand(.5,0)(2.5,2.25)(.5,4.5) replace above
%\DGCdot{E}[u]{$ $}
\end{DGCpicture}
\end{equation*}
is in the span of the proposed spanning set.  
The equality
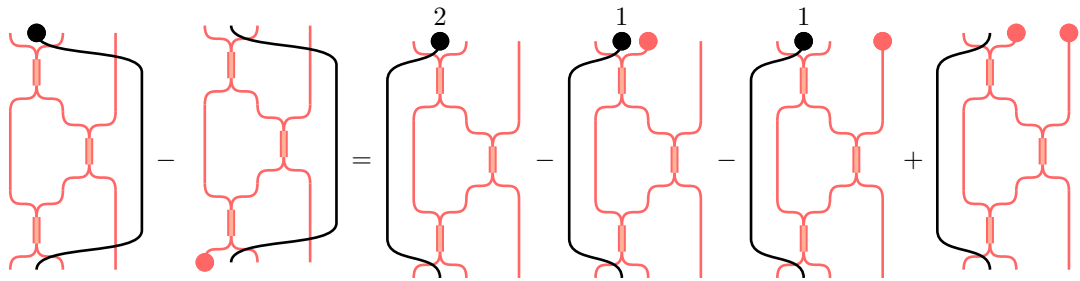
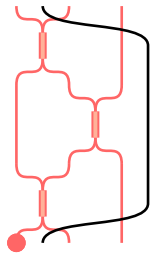
\begin{equation*}
\begin{DGCpicture}[scale={.7,.7}]
\DGCstrand[red](0,0)(.5,.5)[$^{}$`{\ }]
%\DGCdot{B}[l]{$t$}
\DGCstrand[red](1,0)(.5,.5)[$^{}$`{\ }]
\DGCstrand[Red](.5,.5)(.5,1)
\DGCstrand[red](.5,1)(0,1.5)
\DGCstrand[red](.5,1)(1,1.5)
\DGCstrand[red](2,0)(2,1.5)[$^{}$`{\ }]
\DGCstrand[red](1,1.5)(1.5,2)
%\DGCdot{B}[r]{$s$}
\DGCstrand[red](2,1.5)(1.5,2)
\DGCstrand[Red](1.5,2)(1.5,2.5)
\DGCstrand[red](1.5,2.5)(1,3)
\DGCstrand[red](1.5,2.5)(2,3)
\DGCstrand[red](0,1.5)(0,3)
%\DGCdot{E}[r]{$r$}
%%%%
\DGCstrand[red](0,3)(.5,3.5)
\DGCstrand[red](1,3)(.5,3.5)
\DGCstrand[Red](.5,3.5)(.5,4)
\DGCstrand[red](.5,4)(0,4.5)
%\DGCdot{E}[u]{$a_{i}$}
\DGCstrand[red](.5,4)(1,4.5)
%\DGCdot{E}[u]{$a_{i+1}$}
\DGCstrand[red](2,3)(2,4.5)
%\DGCdot{E}[u]{$a_{i+2}$}
\DGCstrand(.5,0)(2.5,.75)(2.5,3.75)(.5,4.5)
%\DGCstrand(.5,0)(2.5,2.25)(.5,4.5) replace above
\DGCdot{E}[u]{$ $}
\end{DGCpicture}
~-~
\begin{DGCpicture}[scale={.7,.7}]
\DGCstrand[red](0,0)(.5,.5)[$^{}$`{\ }]
\DGCdot{B}[l]{$ $}
\DGCstrand[red](1,0)(.5,.5)[$^{}$`{\ }]
\DGCstrand[Red](.5,.5)(.5,1)
\DGCstrand[red](.5,1)(0,1.5)
\DGCstrand[red](.5,1)(1,1.5)
\DGCstrand[red](2,0)(2,1.5)[$^{}$`{\ }]
\DGCstrand[red](1,1.5)(1.5,2)
%\DGCdot{B}[r]{$s$}
\DGCstrand[red](2,1.5)(1.5,2)
\DGCstrand[Red](1.5,2)(1.5,2.5)
\DGCstrand[red](1.5,2.5)(1,3)
\DGCstrand[red](1.5,2.5)(2,3)
\DGCstrand[red](0,1.5)(0,3)
%\DGCdot{E}[r]{$r$}
%%%%
\DGCstrand[red](0,3)(.5,3.5)
\DGCstrand[red](1,3)(.5,3.5)
\DGCstrand[Red](.5,3.5)(.5,4)
\DGCstrand[red](.5,4)(0,4.5)
%\DGCdot{E}[u]{$ $}
\DGCstrand[red](.5,4)(1,4.5)
%\DGCdot{E}[u]{$a_{i+1}$}
\DGCstrand[red](2,3)(2,4.5)
%\DGCdot{E}[u]{$a_{i+2}$}
\DGCstrand(.5,0)(2.5,.75)(2.5,3.75)(.5,4.5)
%\DGCstrand(.5,0)(2.5,2.25)(.5,4.5) replace above
%\DGCdot{E}[u]{$ $}
\end{DGCpicture}
~=~
\begin{DGCpicture}[scale={.7,.7}]
\DGCstrand[red](0,0)(.5,.5)[$ $`{\ }]
%\DGCdot{B}[l]{$t$}
\DGCstrand[red](1,0)(.5,.5)[$ $`{\ }]
\DGCstrand[Red](.5,.5)(.5,1)
\DGCstrand[red](.5,1)(0,1.5)
\DGCstrand[red](.5,1)(1,1.5)
\DGCstrand[red](2,0)(2,1.5)[$ $`{\ }]
\DGCstrand[red](1,1.5)(1.5,2)
%\DGCdot{B}[l]{$s$}
\DGCstrand[red](2,1.5)(1.5,2)
\DGCstrand[Red](1.5,2)(1.5,2.5)
\DGCstrand[red](1.5,2.5)(1,3)
\DGCstrand[red](1.5,2.5)(2,3)
\DGCstrand[red](0,1.5)(0,3)
%\DGCdot{E}[r]{$r$}
%%%%
\DGCstrand[red](0,3)(.5,3.5)
\DGCstrand[red](1,3)(.5,3.5)
\DGCstrand[Red](.5,3.5)(.5,4)
\DGCstrand[red](.5,4)(0,4.5)
%\DGCdot{E}[u]{$a_{i}$}
\DGCstrand[red](.5,4)(1,4.5)
%\DGCdot{E}[u]{$ $}
\DGCstrand[red](2,3)(2,4.5)
%\DGCdot{E}[u]{$a_{i+2}$}
\DGCstrand(.5,0)(-.5,.75)(-.5,3.75)(.5,4.5)
%\DGCstrand(.5,0)(-.5,2.25)(.5,4.5) replace above
\DGCdot{E}[u]{$2 $}
\end{DGCpicture}
~-~
\begin{DGCpicture}[scale={.7,.7}]
\DGCstrand[red](0,0)(.5,.5)[$ $`{\ }]
%\DGCdot{B}[l]{$t$}
\DGCstrand[red](1,0)(.5,.5)[$ $`{\ }]
%\DGCdot{B}[l]{$ $}
\DGCstrand[Red](.5,.5)(.5,1)
\DGCstrand[red](.5,1)(0,1.5)
\DGCstrand[red](.5,1)(1,1.5)
\DGCstrand[red](2,0)(2,1.5)[$ $`{\ }]
\DGCstrand[red](1,1.5)(1.5,2)
%\DGCdot{B}[l]{$s$}
\DGCstrand[red](2,1.5)(1.5,2)
\DGCstrand[Red](1.5,2)(1.5,2.5)
\DGCstrand[red](1.5,2.5)(1,3)
\DGCstrand[red](1.5,2.5)(2,3)
\DGCstrand[red](0,1.5)(0,3)
%\DGCdot{E}[r]{$r$}
%%%%
\DGCstrand[red](0,3)(.5,3.5)
\DGCstrand[red](1,3)(.5,3.5)
\DGCstrand[Red](.5,3.5)(.5,4)
\DGCstrand[red](.5,4)(0,4.5)
%\DGCdot{E}[u]{$a_{i}$}
\DGCstrand[red](.5,4)(1,4.5)
\DGCdot{E}[u]{$ $}
\DGCstrand[red](2,3)(2,4.5)
%\DGCdot{E}[u]{$a_{i+2}$}
\DGCstrand(.5,0)(-.5,.75)(-.5,3.75)(.5,4.5)
%\DGCstrand(.5,0)(-.5,2.25)(.5,4.5) replace above
\DGCdot{E}[u]{$1 $}
\end{DGCpicture} 
~-~
\begin{DGCpicture}[scale={.7,.7}]
\DGCstrand[red](0,0)(.5,.5)[$ $`{\ }]
%\DGCdot{B}[l]{$t$}
\DGCstrand[red](1,0)(.5,.5)[$ $`{\ }]
%\DGCdot{B}[l]{$ $}
\DGCstrand[Red](.5,.5)(.5,1)
\DGCstrand[red](.5,1)(0,1.5)
\DGCstrand[red](.5,1)(1,1.5)
\DGCstrand[red](2,0)(2,1.5)[$ $`{\ }]
%\DGCdot{B}[l]{$ $}
%%
\DGCstrand[red](1,1.5)(1.5,2)
%\DGCdot{B}[l]{$s$}
\DGCstrand[red](2,1.5)(1.5,2)
\DGCstrand[Red](1.5,2)(1.5,2.5)
\DGCstrand[red](1.5,2.5)(1,3)
%\DGCdot{E}[r]{$ $}
\DGCstrand[red](1.5,2.5)(2,3)
\DGCstrand[red](0,1.5)(0,3)
%\DGCdot{E}[r]{$r$}
%%%%
\DGCstrand[red](0,3)(.5,3.5)
\DGCstrand[red](1,3)(.5,3.5)
\DGCstrand[Red](.5,3.5)(.5,4)
\DGCstrand[red](.5,4)(0,4.5)
%\DGCdot{E}[u]{$a_{i}$}
\DGCstrand[red](.5,4)(1,4.5)
%\DGCdot{E}[u]{$ $}
\DGCstrand[red](2,3)(2,4.5)
\DGCdot{E}[u]{$ $}
\DGCstrand(.5,0)(-.5,.75)(-.5,3.75)(.5,4.5)
%\DGCstrand(.5,0)(-.5,2.25)(.5,4.5) replace above
\DGCdot{E}[u]{$1 $}
\end{DGCpicture}
~+~
\begin{DGCpicture}[scale={.7,.7}]
\DGCstrand[red](0,0)(.5,.5)[$ $`{\ }]
%\DGCdot{B}[l]{$t$}
\DGCstrand[red](1,0)(.5,.5)[$ $`{\ }]
%\DGCdot{B}[l]{$ $}
\DGCstrand[Red](.5,.5)(.5,1)
\DGCstrand[red](.5,1)(0,1.5)
\DGCstrand[red](.5,1)(1,1.5)
\DGCstrand[red](2,0)(2,1.5)[$ $`{\ }]
%\DGCdot{B}[l]{$ $}
%%
\DGCstrand[red](1,1.5)(1.5,2)
%\DGCdot{B}[l]{$ $}
\DGCstrand[red](2,1.5)(1.5,2)
\DGCstrand[Red](1.5,2)(1.5,2.5)
\DGCstrand[red](1.5,2.5)(1,3)
%\DGCdot{E}[r]{$ $}
\DGCstrand[red](1.5,2.5)(2,3)
\DGCstrand[red](0,1.5)(0,3)
%\DGCdot{E}[r]{$r$}
%%%%
\DGCstrand[red](0,3)(.5,3.5)
\DGCstrand[red](1,3)(.5,3.5)
\DGCstrand[Red](.5,3.5)(.5,4)
\DGCstrand[red](.5,4)(0,4.5)
%\DGCdot{E}[u]{$ $}
\DGCstrand[red](.5,4)(1,4.5)
\DGCdot{E}[u]{$ $}
\DGCstrand[red](2,3)(2,4.5)
\DGCdot{E}[u]{$ $}
\DGCstrand(.5,0)(-.5,.75)(-.5,3.75)(.5,4.5)
%\DGCstrand(.5,0)(-.5,2.25)(.5,4.5) replace above
%\DGCdot{E}[u]{$1 $}
\end{DGCpicture}
\end{equation*}
implies that 
\begin{equation*}
\begin{DGCpicture}[scale={.7,.7}]
\DGCstrand[red](0,0)(.5,.5)[$^{}$`{\ }]
\DGCdot{B}[l]{$ $}
\DGCstrand[red](1,0)(.5,.5)[$^{}$`{\ }]
\DGCstrand[Red](.5,.5)(.5,1)
\DGCstrand[red](.5,1)(0,1.5)
\DGCstrand[red](.5,1)(1,1.5)
\DGCstrand[red](2,0)(2,1.5)[$^{}$`{\ }]
\DGCstrand[red](1,1.5)(1.5,2)
%\DGCdot{B}[r]{$s$}
\DGCstrand[red](2,1.5)(1.5,2)
\DGCstrand[Red](1.5,2)(1.5,2.5)
\DGCstrand[red](1.5,2.5)(1,3)
\DGCstrand[red](1.5,2.5)(2,3)
\DGCstrand[red](0,1.5)(0,3)
%\DGCdot{E}[r]{$r$}
%%%%
\DGCstrand[red](0,3)(.5,3.5)
\DGCstrand[red](1,3)(.5,3.5)
\DGCstrand[Red](.5,3.5)(.5,4)
\DGCstrand[red](.5,4)(0,4.5)
%\DGCdot{E}[u]{$ $}
\DGCstrand[red](.5,4)(1,4.5)
%\DGCdot{E}[u]{$a_{i+1}$}
\DGCstrand[red](2,3)(2,4.5)
%\DGCdot{E}[u]{$a_{i+2}$}
\DGCstrand(.5,0)(2.5,.75)(2.5,3.75)(.5,4.5)
%\DGCstrand(.5,0)(2.5,2.25)(.5,4.5) replace above
%\DGCdot{E}[u]{$ $}
\end{DGCpicture}
\end{equation*}
is in the span of the proposed spanning set.
A longer computation using the equality
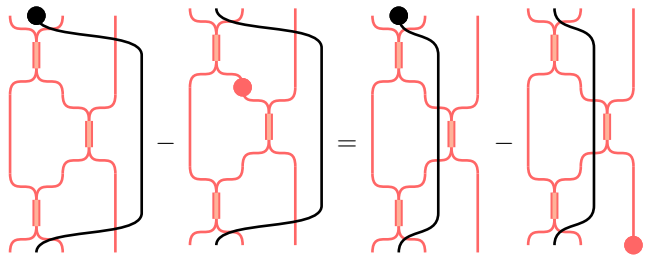
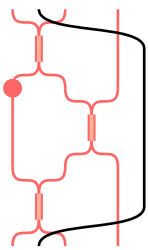
\begin{equation*}
\begin{DGCpicture}[scale={.7,.7}]
\DGCstrand[red](0,0)(.5,.5)[$^{}$`{\ }]
%\DGCdot{B}[l]{$t$}
\DGCstrand[red](1,0)(.5,.5)[$^{}$`{\ }]
\DGCstrand[Red](.5,.5)(.5,1)
\DGCstrand[red](.5,1)(0,1.5)
\DGCstrand[red](.5,1)(1,1.5)
\DGCstrand[red](2,0)(2,1.5)[$^{}$`{\ }]
\DGCstrand[red](1,1.5)(1.5,2)
%\DGCdot{B}[r]{$s$}
\DGCstrand[red](2,1.5)(1.5,2)
\DGCstrand[Red](1.5,2)(1.5,2.5)
\DGCstrand[red](1.5,2.5)(1,3)
\DGCstrand[red](1.5,2.5)(2,3)
\DGCstrand[red](0,1.5)(0,3)
%\DGCdot{E}[r]{$r$}
%%%%
\DGCstrand[red](0,3)(.5,3.5)
\DGCstrand[red](1,3)(.5,3.5)
\DGCstrand[Red](.5,3.5)(.5,4)
\DGCstrand[red](.5,4)(0,4.5)
%\DGCdot{E}[u]{$a_{i}$}
\DGCstrand[red](.5,4)(1,4.5)
%\DGCdot{E}[u]{$a_{i+1}$}
\DGCstrand[red](2,3)(2,4.5)
%\DGCdot{E}[u]{$a_{i+2}$}
\DGCstrand(.5,0)(2.5,.75)(2.5,3.75)(.5,4.5)
%\DGCstrand(.5,0)(2.5,2.25)(.5,4.5) replace above
\DGCdot{E}[u]{$ $}
\end{DGCpicture}
~-~
\begin{DGCpicture}[scale={.7,.7}]
\DGCstrand[red](0,0)(.5,.5)[$^{}$`{\ }]
%\DGCdot{B}[l]{$t$}
\DGCstrand[red](1,0)(.5,.5)[$^{}$`{\ }]
\DGCstrand[Red](.5,.5)(.5,1)
\DGCstrand[red](.5,1)(0,1.5)
\DGCstrand[red](.5,1)(1,1.5)
\DGCstrand[red](2,0)(2,1.5)[$^{}$`{\ }]
\DGCstrand[red](1,1.5)(1.5,2)
%\DGCdot{B}[r]{$ $}
\DGCstrand[red](2,1.5)(1.5,2)
\DGCstrand[Red](1.5,2)(1.5,2.5)
\DGCstrand[red](1.5,2.5)(1,3)
\DGCstrand[red](1.5,2.5)(2,3)
\DGCstrand[red](0,1.5)(0,3)
%\DGCdot{E}[r]{$r$}
%%%%
\DGCstrand[red](0,3)(.5,3.5)
\DGCstrand[red](1,3)(.5,3.5)
\DGCdot{B}[r]{$ $}
\DGCstrand[Red](.5,3.5)(.5,4)
\DGCstrand[red](.5,4)(0,4.5)
%\DGCdot{E}[u]{$ $}
\DGCstrand[red](.5,4)(1,4.5)
%\DGCdot{E}[u]{$a_{i+1}$}
\DGCstrand[red](2,3)(2,4.5)
%\DGCdot{E}[u]{$a_{i+2}$}
\DGCstrand(.5,0)(2.5,.75)(2.5,3.75)(.5,4.5)
%\DGCstrand(.5,0)(2.5,2.25)(.5,4.5) replace above
%\DGCdot{E}[u]{$ $}
\end{DGCpicture}
~=~
\begin{DGCpicture}[scale={.7,.7}]
\DGCstrand[red](0,0)(.5,.5)[$ $`{\ }]
%\DGCdot{B}[l]{$t$}
\DGCstrand[red](1,0)(.5,.5)[$ $`{\ }]
\DGCstrand[Red](.5,.5)(.5,1)
\DGCstrand[red](.5,1)(0,1.5)
\DGCstrand[red](.5,1)(1,1.5)
\DGCstrand[red](2,0)(2,1.5)[$ $`{\ }]
\DGCstrand[red](1,1.5)(1.5,2)
%\DGCdot{B}[l]{$ $}
\DGCstrand[red](2,1.5)(1.5,2)
\DGCstrand[Red](1.5,2)(1.5,2.5)
\DGCstrand[red](1.5,2.5)(1,3)
\DGCstrand[red](1.5,2.5)(2,3)
\DGCstrand[red](0,1.5)(0,3)
%\DGCdot{E}[r]{$r$}
%%%%
\DGCstrand[red](0,3)(.5,3.5)
\DGCstrand[red](1,3)(.5,3.5)
\DGCstrand[Red](.5,3.5)(.5,4)
\DGCstrand[red](.5,4)(0,4.5)
%\DGCdot{E}[u]{$a_{i}$}
\DGCstrand[red](.5,4)(1,4.5)
%\DGCdot{E}[u]{$a_{i+1}$}
\DGCstrand[red](2,3)(2,4.5)
%\DGCdot{E}[u]{$a_{i+2}$}
\DGCstrand(.5,0)(1.25,.75)(1.25,3.75)(.5,4.5)
%\DGCstrand(.5,0)(1,2.25)(.5,4.5) replace above
\DGCdot{E}[u]{$ $}
\end{DGCpicture}
~-~
\begin{DGCpicture}[scale={.7,.7}]
\DGCstrand[red](0,0)(.5,.5)[$ $`{\ }]
%\DGCdot{B}[l]{$t$}
\DGCstrand[red](1,0)(.5,.5)[$ $`{\ }]
%\DGCdot{B}[l]{$ $}
\DGCstrand[Red](.5,.5)(.5,1)
\DGCstrand[red](.5,1)(0,1.5)
\DGCstrand[red](.5,1)(1,1.5)
\DGCstrand[red](2,0)(2,1.5)[$ $`{\ }]
\DGCdot{B}[u]{$ $}
\DGCstrand[red](1,1.5)(1.5,2)
%\DGCdot{B}[l]{$s$}
\DGCstrand[red](2,1.5)(1.5,2)
\DGCstrand[Red](1.5,2)(1.5,2.5)
\DGCstrand[red](1.5,2.5)(1,3)
\DGCstrand[red](1.5,2.5)(2,3)
\DGCstrand[red](0,1.5)(0,3)
%\DGCdot{E}[r]{$r$}
%%%%
\DGCstrand[red](0,3)(.5,3.5)
\DGCstrand[red](1,3)(.5,3.5)
\DGCstrand[Red](.5,3.5)(.5,4)
\DGCstrand[red](.5,4)(0,4.5)
%\DGCdot{E}[u]{$a_{i}$}
\DGCstrand[red](.5,4)(1,4.5)
%\DGCdot{E}[u]{$a_{i+1}$}
\DGCstrand[red](2,3)(2,4.5)
%\DGCdot{E}[u]{$ $}
\DGCstrand(.5,0)(1.25,.75)(1.25,3.75)(.5,4.5)
%\DGCstrand(.5,0)(1,2.25)(.5,4.5) replace above
%\DGCdot{E}[u]{$1 $}
\end{DGCpicture} 
\end{equation*}
implies that
\begin{equation*}
\begin{DGCpicture}[scale={.7,.7}]
\DGCstrand[red](0,0)(.5,.5)[$^{}$`{\ }]
%\DGCdot{B}[l]{$t$}
\DGCstrand[red](1,0)(.5,.5)[$^{}$`{\ }]
\DGCstrand[Red](.5,.5)(.5,1)
\DGCstrand[red](.5,1)(0,1.5)
\DGCstrand[red](.5,1)(1,1.5)
\DGCstrand[red](2,0)(2,1.5)[$^{}$`{\ }]
\DGCstrand[red](1,1.5)(1.5,2)
%\DGCdot{B}[r]{$ $}
\DGCstrand[red](2,1.5)(1.5,2)
\DGCstrand[Red](1.5,2)(1.5,2.5)
\DGCstrand[red](1.5,2.5)(1,3)
\DGCstrand[red](1.5,2.5)(2,3)
\DGCstrand[red](0,1.5)(0,3)
\DGCdot{E}[r]{$ $}
%%%%
\DGCstrand[red](0,3)(.5,3.5)
\DGCstrand[red](1,3)(.5,3.5)
\DGCstrand[Red](.5,3.5)(.5,4)
\DGCstrand[red](.5,4)(0,4.5)
%\DGCdot{E}[u]{$ $}
\DGCstrand[red](.5,4)(1,4.5)
%\DGCdot{E}[u]{$a_{i+1}$}
\DGCstrand[red](2,3)(2,4.5)
%\DGCdot{E}[u]{$a_{i+2}$}
\DGCstrand(.5,0)(2.5,.75)(2.5,3.75)(.5,4.5)
%\DGCstrand(.5,0)(2.5,2.25)(.5,4.5) replace above
%\DGCdot{E}[u]{$ $}
\end{DGCpicture}
\end{equation*}
is in the proposed span.

Now consider elements of the form
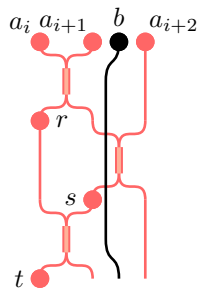
\begin{equation*}
\begin{DGCpicture}[scale={.7,.7}]
\DGCstrand[red](0,0)(.5,.5)[$ $`{\ }]
\DGCdot{B}[l]{$t$}
\DGCstrand[red](1,0)(.5,.5)[$ $`{\ }]
\DGCstrand[Red](.5,.5)(.5,1)
\DGCstrand[red](.5,1)(0,1.5)
\DGCstrand[red](.5,1)(1,1.5)
\DGCstrand[red](2,0)(2,1.5)[$ $`{\ }]
\DGCstrand[red](1,1.5)(1.5,2)
\DGCdot{B}[l]{$s$}
\DGCstrand[red](2,1.5)(1.5,2)
\DGCstrand[Red](1.5,2)(1.5,2.5)
\DGCstrand[red](1.5,2.5)(1,3)
\DGCstrand[red](1.5,2.5)(2,3)
\DGCstrand[red](0,1.5)(0,3)
\DGCdot{E}[r]{$r$}
%%%%
\DGCstrand[red](0,3)(.5,3.5)
\DGCstrand[red](1,3)(.5,3.5)
\DGCstrand[Red](.5,3.5)(.5,4)
\DGCstrand[red](.5,4)(0,4.5)
\DGCdot{E}[ul]{$a_{i}$}
\DGCstrand[red](.5,4)(1,4.5)
\DGCdot{E}[ul]{$a_{i+1}$}
\DGCstrand[red](2,3)(2,4.5)
\DGCdot{E}[ur]{$a_{i+2}$}
\DGCstrand(1.5,0)(1.25,.75)(1.25,3.75)(1.5,4.5)
%\DGCstrand(1.5,0)(1,2.25)(1.5,4.5) replace above
\DGCdot{E}[u]{$b$}
\end{DGCpicture}
\ .
\end{equation*}
By Lemma \ref{classical}, one may assume that $r,s,t \in \{0,1\}$ and $a_i, a_{i+1}, a_{i+2}, b \in \mathbb{Z}_{\geq 0}$.
Next note that the equation
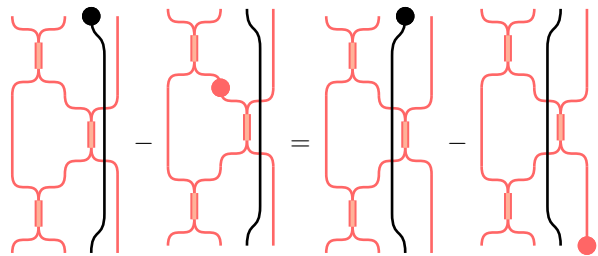
\begin{equation*}
\begin{DGCpicture}[scale={.7,.7}]
\DGCstrand[red](0,0)(.5,.5)[$ $`{\ }]
%\DGCdot{B}[l]{$t$}
\DGCstrand[red](1,0)(.5,.5)[$ $`{\ }]
\DGCstrand[Red](.5,.5)(.5,1)
\DGCstrand[red](.5,1)(0,1.5)
\DGCstrand[red](.5,1)(1,1.5)
\DGCstrand[red](2,0)(2,1.5)[$ $`{\ }]
\DGCstrand[red](1,1.5)(1.5,2)
%\DGCdot{B}[l]{$s$}
\DGCstrand[red](2,1.5)(1.5,2)
\DGCstrand[Red](1.5,2)(1.5,2.5)
\DGCstrand[red](1.5,2.5)(1,3)
\DGCstrand[red](1.5,2.5)(2,3)
\DGCstrand[red](0,1.5)(0,3)
%\DGCdot{E}[r]{$r$}
%%%%
\DGCstrand[red](0,3)(.5,3.5)
\DGCstrand[red](1,3)(.5,3.5)
\DGCstrand[Red](.5,3.5)(.5,4)
\DGCstrand[red](.5,4)(0,4.5)
%\DGCdot{E}[u]{$a_{i}$}
\DGCstrand[red](.5,4)(1,4.5)
%\DGCdot{E}[u]{$a_{i+1}$}
\DGCstrand[red](2,3)(2,4.5)
%\DGCdot{E}[u]{$a_{i+2}$}
\DGCstrand(1.5,0)(1.75,.75)(1.75,3.75)(1.5,4.5)
%\DGCstrand(1.5,0)(2,2.25)(1.5,4.5) replace above
\DGCdot{E}[u]{$ $}
\end{DGCpicture}
~-~
\begin{DGCpicture}[scale={.7,.7}]
\DGCstrand[red](0,0)(.5,.5)[$ $`{\ }]
%\DGCdot{B}[l]{$t$}
\DGCstrand[red](1,0)(.5,.5)[$ $`{\ }]
\DGCstrand[Red](.5,.5)(.5,1)
\DGCstrand[red](.5,1)(0,1.5)
\DGCstrand[red](.5,1)(1,1.5)
\DGCstrand[red](2,0)(2,1.5)[$ $`{\ }]
\DGCstrand[red](1,1.5)(1.5,2)
%\DGCdot{B}[l]{$s$}
\DGCstrand[red](2,1.5)(1.5,2)
\DGCstrand[Red](1.5,2)(1.5,2.5)
\DGCstrand[red](1.5,2.5)(1,3)
\DGCdot{E}[r]{$ $}
\DGCstrand[red](1.5,2.5)(2,3)
\DGCstrand[red](0,1.5)(0,3)
%\DGCdot{E}[r]{$r$}
%%%%
\DGCstrand[red](0,3)(.5,3.5)
\DGCstrand[red](1,3)(.5,3.5)
\DGCstrand[Red](.5,3.5)(.5,4)
\DGCstrand[red](.5,4)(0,4.5)
%\DGCdot{E}[u]{$a_{i}$}
\DGCstrand[red](.5,4)(1,4.5)
%\DGCdot{E}[u]{$a_{i+1}$}
\DGCstrand[red](2,3)(2,4.5)
%\DGCdot{E}[u]{$a_{i+2}$}
\DGCstrand(1.5,0)(1.75,.75)(1.75,3.75)(1.5,4.5)
%\DGCstrand(1.5,0)(2,2.25)(1.5,4.5) replace above
%\DGCdot{E}[u]{$b$}
\end{DGCpicture}
~=~
\begin{DGCpicture}[scale={.7,.7}]
\DGCstrand[red](0,0)(.5,.5)[$ $`{\ }]
%\DGCdot{B}[l]{$t$}
\DGCstrand[red](1,0)(.5,.5)[$ $`{\ }]
\DGCstrand[Red](.5,.5)(.5,1)
\DGCstrand[red](.5,1)(0,1.5)
\DGCstrand[red](.5,1)(1,1.5)
\DGCstrand[red](2,0)(2,1.5)[$ $`{\ }]
\DGCstrand[red](1,1.5)(1.5,2)
%\DGCdot{B}[l]{$s$}
\DGCstrand[red](2,1.5)(1.5,2)
\DGCstrand[Red](1.5,2)(1.5,2.5)
\DGCstrand[red](1.5,2.5)(1,3)
\DGCstrand[red](1.5,2.5)(2,3)
\DGCstrand[red](0,1.5)(0,3)
%\DGCdot{E}[r]{$r$}
%%%%
\DGCstrand[red](0,3)(.5,3.5)
\DGCstrand[red](1,3)(.5,3.5)
\DGCstrand[Red](.5,3.5)(.5,4)
\DGCstrand[red](.5,4)(0,4.5)
%\DGCdot{E}[u]{$a_{i}$}
\DGCstrand[red](.5,4)(1,4.5)
%\DGCdot{E}[u]{$a_{i+1}$}
\DGCstrand[red](2,3)(2,4.5)
%\DGCdot{E}[u]{$a_{i+2}$}
\DGCstrand(1.5,0)(1.25,.75)(1.25,3.75)(1.5,4.5)
%\DGCstrand(1.5,0)(1,2.25)(1.5,4.5) replace above
\DGCdot{E}[u]{$ $}
\end{DGCpicture}
~-~
\begin{DGCpicture}[scale={.7,.7}]
\DGCstrand[red](0,0)(.5,.5)[$ $`{\ }]
%\DGCdot{B}[l]{$t$}
\DGCstrand[red](1,0)(.5,.5)[$ $`{\ }]
\DGCstrand[Red](.5,.5)(.5,1)
\DGCstrand[red](.5,1)(0,1.5)
\DGCstrand[red](.5,1)(1,1.5)
\DGCstrand[red](2,0)(2,1.5)[$ $`{\ }]
\DGCdot{B}[l]{$ $}
\DGCstrand[red](1,1.5)(1.5,2)
%\DGCdot{B}[l]{$s$}
\DGCstrand[red](2,1.5)(1.5,2)
\DGCstrand[Red](1.5,2)(1.5,2.5)
\DGCstrand[red](1.5,2.5)(1,3)
\DGCstrand[red](1.5,2.5)(2,3)
\DGCstrand[red](0,1.5)(0,3)
%\DGCdot{E}[r]{$r$}
%%%%
\DGCstrand[red](0,3)(.5,3.5)
\DGCstrand[red](1,3)(.5,3.5)
\DGCstrand[Red](.5,3.5)(.5,4)
\DGCstrand[red](.5,4)(0,4.5)
%\DGCdot{E}[u]{$a_{i}$}
\DGCstrand[red](.5,4)(1,4.5)
%\DGCdot{E}[u]{$a_{i+1}$}
\DGCstrand[red](2,3)(2,4.5)
%\DGCdot{E}[u]{$a_{i+2}$}
\DGCstrand(1.5,0)(1.25,.75)(1.25,3.75)(1.5,4.5)
%\DGCstrand(1.5,0)(1,2.25)(1.5,4.5) replace above
%\DGCdot{E}[u]{$ $}
\end{DGCpicture}
\end{equation*}
implies that a black dot on a diagram 
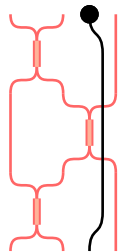
\begin{equation*}
\begin{DGCpicture}[scale={.7,.7}]
\DGCstrand[red](0,0)(.5,.5)[$ $`{\ }]
%\DGCdot{B}[l]{$t$}
\DGCstrand[red](1,0)(.5,.5)[$ $`{\ }]
\DGCstrand[Red](.5,.5)(.5,1)
\DGCstrand[red](.5,1)(0,1.5)
\DGCstrand[red](.5,1)(1,1.5)
\DGCstrand[red](2,0)(2,1.5)[$ $`{\ }]
\DGCstrand[red](1,1.5)(1.5,2)
%\DGCdot{B}[l]{$s$}
\DGCstrand[red](2,1.5)(1.5,2)
\DGCstrand[Red](1.5,2)(1.5,2.5)
\DGCstrand[red](1.5,2.5)(1,3)
\DGCstrand[red](1.5,2.5)(2,3)
\DGCstrand[red](0,1.5)(0,3)
%\DGCdot{E}[r]{$r$}
%%%%
\DGCstrand[red](0,3)(.5,3.5)
\DGCstrand[red](1,3)(.5,3.5)
\DGCstrand[Red](.5,3.5)(.5,4)
\DGCstrand[red](.5,4)(0,4.5)
%\DGCdot{E}[u]{$a_{i}$}
\DGCstrand[red](.5,4)(1,4.5)
%\DGCdot{E}[u]{$a_{i+1}$}
\DGCstrand[red](2,3)(2,4.5)
%\DGCdot{E}[u]{$a_{i+2}$}
\DGCstrand(1.5,0)(1.75,.75)(1.75,3.75)(1.5,4.5)
%\DGCstrand(1.5,0)(2,2.25)(1.5,4.5) replace above
\DGCdot{E}[u]{$ $}
\end{DGCpicture}
\end{equation*}
is already in the span of the proposed spanning set.
The equation
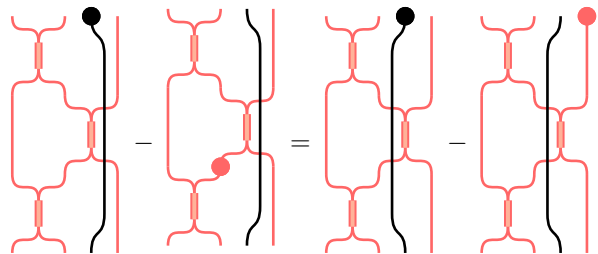
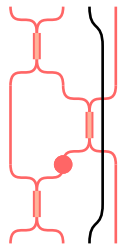
\begin{equation*}
\begin{DGCpicture}[scale={.7,.7}]
\DGCstrand[red](0,0)(.5,.5)[$ $`{\ }]
%\DGCdot{B}[l]{$t$}
\DGCstrand[red](1,0)(.5,.5)[$ $`{\ }]
\DGCstrand[Red](.5,.5)(.5,1)
\DGCstrand[red](.5,1)(0,1.5)
\DGCstrand[red](.5,1)(1,1.5)
\DGCstrand[red](2,0)(2,1.5)[$ $`{\ }]
\DGCstrand[red](1,1.5)(1.5,2)
%\DGCdot{B}[l]{$s$}
\DGCstrand[red](2,1.5)(1.5,2)
\DGCstrand[Red](1.5,2)(1.5,2.5)
\DGCstrand[red](1.5,2.5)(1,3)
\DGCstrand[red](1.5,2.5)(2,3)
\DGCstrand[red](0,1.5)(0,3)
%\DGCdot{E}[r]{$r$}
%%%%
\DGCstrand[red](0,3)(.5,3.5)
\DGCstrand[red](1,3)(.5,3.5)
\DGCstrand[Red](.5,3.5)(.5,4)
\DGCstrand[red](.5,4)(0,4.5)
%\DGCdot{E}[u]{$a_{i}$}
\DGCstrand[red](.5,4)(1,4.5)
%\DGCdot{E}[u]{$a_{i+1}$}
\DGCstrand[red](2,3)(2,4.5)
%\DGCdot{E}[u]{$a_{i+2}$}
\DGCstrand(1.5,0)(1.75,.75)(1.75,3.75)(1.5,4.5)
%\DGCstrand(1.5,0)(2,2.25)(1.5,4.5) replace above
\DGCdot{E}[u]{$ $}
\end{DGCpicture}
~-~
\begin{DGCpicture}[scale={.7,.7}]
\DGCstrand[red](0,0)(.5,.5)[$ $`{\ }]
%\DGCdot{B}[l]{$t$}
\DGCstrand[red](1,0)(.5,.5)[$ $`{\ }]
\DGCstrand[Red](.5,.5)(.5,1)
\DGCstrand[red](.5,1)(0,1.5)
\DGCstrand[red](.5,1)(1,1.5)
\DGCstrand[red](2,0)(2,1.5)[$ $`{\ }]
\DGCstrand[red](1,1.5)(1.5,2)
\DGCdot{B}[l]{$ $}
\DGCstrand[red](2,1.5)(1.5,2)
\DGCstrand[Red](1.5,2)(1.5,2.5)
\DGCstrand[red](1.5,2.5)(1,3)
%\DGCdot{E}[r]{$ $}
\DGCstrand[red](1.5,2.5)(2,3)
\DGCstrand[red](0,1.5)(0,3)
%\DGCdot{E}[r]{$r$}
%%%%
\DGCstrand[red](0,3)(.5,3.5)
\DGCstrand[red](1,3)(.5,3.5)
\DGCstrand[Red](.5,3.5)(.5,4)
\DGCstrand[red](.5,4)(0,4.5)
%\DGCdot{E}[u]{$a_{i}$}
\DGCstrand[red](.5,4)(1,4.5)
%\DGCdot{E}[u]{$a_{i+1}$}
\DGCstrand[red](2,3)(2,4.5)
%\DGCdot{E}[u]{$a_{i+2}$}
\DGCstrand(1.5,0)(1.75,.75)(1.75,3.75)(1.5,4.5)
%\DGCstrand(1.5,0)(2,2.25)(1.5,4.5) replace above
%\DGCdot{E}[u]{$b$}
\end{DGCpicture}
~=~
\begin{DGCpicture}[scale={.7,.7}]
\DGCstrand[red](0,0)(.5,.5)[$ $`{\ }]
%\DGCdot{B}[l]{$t$}
\DGCstrand[red](1,0)(.5,.5)[$ $`{\ }]
\DGCstrand[Red](.5,.5)(.5,1)
\DGCstrand[red](.5,1)(0,1.5)
\DGCstrand[red](.5,1)(1,1.5)
\DGCstrand[red](2,0)(2,1.5)[$ $`{\ }]
\DGCstrand[red](1,1.5)(1.5,2)
%\DGCdot{B}[l]{$s$}
\DGCstrand[red](2,1.5)(1.5,2)
\DGCstrand[Red](1.5,2)(1.5,2.5)
\DGCstrand[red](1.5,2.5)(1,3)
\DGCstrand[red](1.5,2.5)(2,3)
\DGCstrand[red](0,1.5)(0,3)
%\DGCdot{E}[r]{$r$}
%%%%
\DGCstrand[red](0,3)(.5,3.5)
\DGCstrand[red](1,3)(.5,3.5)
\DGCstrand[Red](.5,3.5)(.5,4)
\DGCstrand[red](.5,4)(0,4.5)
%\DGCdot{E}[u]{$a_{i}$}
\DGCstrand[red](.5,4)(1,4.5)
%\DGCdot{E}[u]{$a_{i+1}$}
\DGCstrand[red](2,3)(2,4.5)
%\DGCdot{E}[u]{$a_{i+2}$}
\DGCstrand(1.5,0)(1.25,.75)(1.25,3.75)(1.5,4.5)
%\DGCstrand(1.5,0)(1,2.25)(1.5,4.5) replace above
\DGCdot{E}[u]{$ $}
\end{DGCpicture}
~-~
\begin{DGCpicture}[scale={.7,.7}]
\DGCstrand[red](0,0)(.5,.5)[$ $`{\ }]
%\DGCdot{B}[l]{$t$}
\DGCstrand[red](1,0)(.5,.5)[$ $`{\ }]
\DGCstrand[Red](.5,.5)(.5,1)
\DGCstrand[red](.5,1)(0,1.5)
\DGCstrand[red](.5,1)(1,1.5)
\DGCstrand[red](2,0)(2,1.5)[$ $`{\ }]
%\DGCdot{B}[l]{$ $}
%%
\DGCstrand[red](1,1.5)(1.5,2)
%\DGCdot{B}[l]{$s$}
\DGCstrand[red](2,1.5)(1.5,2)
\DGCstrand[Red](1.5,2)(1.5,2.5)
\DGCstrand[red](1.5,2.5)(1,3)
\DGCstrand[red](1.5,2.5)(2,3)
\DGCstrand[red](0,1.5)(0,3)
%\DGCdot{E}[r]{$r$}
%%%%
\DGCstrand[red](0,3)(.5,3.5)
\DGCstrand[red](1,3)(.5,3.5)
\DGCstrand[Red](.5,3.5)(.5,4)
\DGCstrand[red](.5,4)(0,4.5)
%\DGCdot{E}[u]{$a_{i}$}
\DGCstrand[red](.5,4)(1,4.5)
%\DGCdot{E}[u]{$a_{i+1}$}
\DGCstrand[red](2,3)(2,4.5)
\DGCdot{E}[u]{$ $}
\DGCstrand(1.5,0)(1.25,.75)(1.25,3.75)(1.5,4.5)
%\DGCstrand(1.5,0)(1,2.25)(1.5,4.5) replace above
%\DGCdot{E}[u]{$ $}
\end{DGCpicture}
\end{equation*}
implies that 
\begin{equation*}
\begin{DGCpicture}[scale={.7,.7}]
\DGCstrand[red](0,0)(.5,.5)[$ $`{\ }]
%\DGCdot{B}[l]{$t$}
\DGCstrand[red](1,0)(.5,.5)[$ $`{\ }]
\DGCstrand[Red](.5,.5)(.5,1)
\DGCstrand[red](.5,1)(0,1.5)
\DGCstrand[red](.5,1)(1,1.5)
\DGCstrand[red](2,0)(2,1.5)[$ $`{\ }]
%\DGCdot{B}[l]{$ $}
%%
\DGCstrand[red](1,1.5)(1.5,2)
\DGCdot{B}[l]{$ $}
\DGCstrand[red](2,1.5)(1.5,2)
\DGCstrand[Red](1.5,2)(1.5,2.5)
\DGCstrand[red](1.5,2.5)(1,3)
\DGCstrand[red](1.5,2.5)(2,3)
\DGCstrand[red](0,1.5)(0,3)
%\DGCdot{E}[r]{$r$}
%%%%
\DGCstrand[red](0,3)(.5,3.5)
\DGCstrand[red](1,3)(.5,3.5)
\DGCstrand[Red](.5,3.5)(.5,4)
\DGCstrand[red](.5,4)(0,4.5)
%\DGCdot{E}[u]{$a_{i}$}
\DGCstrand[red](.5,4)(1,4.5)
%\DGCdot{E}[u]{$a_{i+1}$}
\DGCstrand[red](2,3)(2,4.5)
%\DGCdot{E}[u]{$ $}
\DGCstrand(1.5,0)(1.75,.75)(1.75,3.75)(1.5,4.5)
%\DGCstrand(1.5,0)(2,2.25)(1.5,4.5) replace above
%\DGCdot{E}[u]{$ $}
\end{DGCpicture}
\end{equation*}
is already in the span of the proposed spanning set.

By Lemma \ref{classical}, we may assume that dots on diagrams of the following form have the configuration
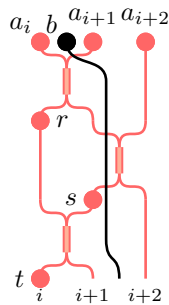
\begin{equation*}
\begin{DGCpicture}[scale={.7,.7}]
\DGCstrand[red](0,0)(.5,.5)[$^{i}$`{\ }]
\DGCdot{B}[l]{$t$}
\DGCstrand[red](1,0)(.5,.5)[$^{i+1}$`{\ }]
\DGCstrand[Red](.5,.5)(.5,1)
\DGCstrand[red](.5,1)(0,1.5)
\DGCstrand[red](.5,1)(1,1.5)
\DGCstrand[red](2,0)(2,1.5)[$^{i+2}$`{\ }]
\DGCstrand[red](1,1.5)(1.5,2)
\DGCdot{B}[l]{$s$}
\DGCstrand[red](2,1.5)(1.5,2)
\DGCstrand[Red](1.5,2)(1.5,2.5)
\DGCstrand[red](1.5,2.5)(1,3)
\DGCstrand[red](1.5,2.5)(2,3)
\DGCstrand[red](0,1.5)(0,3)
\DGCdot{E}[r]{$r$}
%%%%
\DGCstrand[red](0,3)(.5,3.5)
\DGCstrand[red](1,3)(.5,3.5)
\DGCstrand[Red](.5,3.5)(.5,4)
\DGCstrand[red](.5,4)(0,4.5)
\DGCdot{E}[ul]{$a_{i}$}
\DGCstrand[red](.5,4)(1,4.5)
\DGCdot{E}[u]{$a_{i+1}$}
\DGCstrand[red](2,3)(2,4.5)
\DGCdot{E}[u]{$a_{i+2}$}
\DGCstrand(1.5,0)(1.25,.75)(1.25,3.75)(.5,4.5)
%\DGCstrand(1.5,0)(1,4.25)(.5,4.5)
\DGCdot{E}[ul]{$b$}
\end{DGCpicture}
\end{equation*}
where $r,s,t \in \{0,1\}$ and $b, a_i, a_{i+1}, a_{i+2} \in \mathbb{Z}_{\geq 0}$.
Next note that
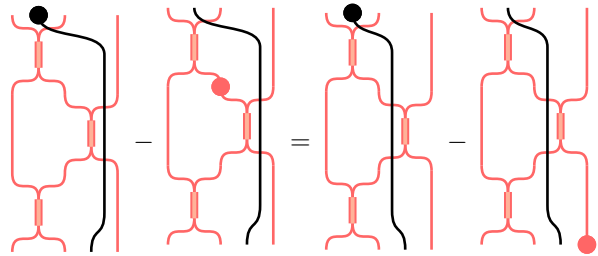
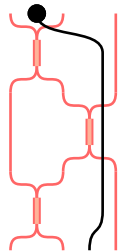
\begin{equation*}
\begin{DGCpicture}[scale={.7,.7}]
\DGCstrand[red](0,0)(.5,.5)[$ $`{\ }]
%\DGCdot{B}[l]{$t$}
\DGCstrand[red](1,0)(.5,.5)[$ $`{\ }]
\DGCstrand[Red](.5,.5)(.5,1)
\DGCstrand[red](.5,1)(0,1.5)
\DGCstrand[red](.5,1)(1,1.5)
\DGCstrand[red](2,0)(2,1.5)[$ $`{\ }]
\DGCstrand[red](1,1.5)(1.5,2)
%\DGCdot{B}[l]{$s$}
\DGCstrand[red](2,1.5)(1.5,2)
\DGCstrand[Red](1.5,2)(1.5,2.5)
\DGCstrand[red](1.5,2.5)(1,3)
\DGCstrand[red](1.5,2.5)(2,3)
\DGCstrand[red](0,1.5)(0,3)
%\DGCdot{E}[r]{$r$}
%%%%
\DGCstrand[red](0,3)(.5,3.5)
\DGCstrand[red](1,3)(.5,3.5)
\DGCstrand[Red](.5,3.5)(.5,4)
\DGCstrand[red](.5,4)(0,4.5)
%\DGCdot{E}[ul]{$a_{i}$}
\DGCstrand[red](.5,4)(1,4.5)
%\DGCdot{E}[u]{$a_{i+1}$}
\DGCstrand[red](2,3)(2,4.5)
%\DGCdot{E}[u]{$a_{i+2}$}
\DGCstrand(1.5,0)(1.75,.75)(1.75,3.75)(.5,4.5)
%\DGCstrand(1.5,0)(2,2.25)(.5,4.5) replace above
\DGCdot{E}[u]{$ $}
\end{DGCpicture}
~-~
\begin{DGCpicture}[scale={.7,.7}]
\DGCstrand[red](0,0)(.5,.5)[$ $`{\ }]
%\DGCdot{B}[l]{$t$}
\DGCstrand[red](1,0)(.5,.5)[$ $`{\ }]
\DGCstrand[Red](.5,.5)(.5,1)
\DGCstrand[red](.5,1)(0,1.5)
\DGCstrand[red](.5,1)(1,1.5)
\DGCstrand[red](2,0)(2,1.5)[$ $`{\ }]
\DGCstrand[red](1,1.5)(1.5,2)
%\DGCdot{B}[l]{$s$}
\DGCstrand[red](2,1.5)(1.5,2)
\DGCstrand[Red](1.5,2)(1.5,2.5)
\DGCstrand[red](1.5,2.5)(1,3)
\DGCdot{E}[u]{$ $}
\DGCstrand[red](1.5,2.5)(2,3)
\DGCstrand[red](0,1.5)(0,3)
%\DGCdot{E}[r]{$r$}
%%%%
\DGCstrand[red](0,3)(.5,3.5)
\DGCstrand[red](1,3)(.5,3.5)
\DGCstrand[Red](.5,3.5)(.5,4)
\DGCstrand[red](.5,4)(0,4.5)
%\DGCdot{E}[ul]{$a_{i}$}
\DGCstrand[red](.5,4)(1,4.5)
%\DGCdot{E}[u]{$a_{i+1}$}
\DGCstrand[red](2,3)(2,4.5)
%\DGCdot{E}[u]{$a_{i+2}$}
\DGCstrand(1.5,0)(1.75,.75)(1.75,3.75)(.5,4.5)
%\DGCstrand(1.5,0)(2,2.25)(.5,4.5) replace above
%\DGCdot{E}[u]{$ $}
\end{DGCpicture}
~=~
\begin{DGCpicture}[scale={.7,.7}]
\DGCstrand[red](0,0)(.5,.5)[$ $`{\ }]
%\DGCdot{B}[l]{$t$}
\DGCstrand[red](1,0)(.5,.5)[$ $`{\ }]
\DGCstrand[Red](.5,.5)(.5,1)
\DGCstrand[red](.5,1)(0,1.5)
\DGCstrand[red](.5,1)(1,1.5)
\DGCstrand[red](2,0)(2,1.5)[$ $`{\ }]
\DGCstrand[red](1,1.5)(1.5,2)
%\DGCdot{B}[l]{$s$}
\DGCstrand[red](2,1.5)(1.5,2)
\DGCstrand[Red](1.5,2)(1.5,2.5)
\DGCstrand[red](1.5,2.5)(1,3)
\DGCstrand[red](1.5,2.5)(2,3)
\DGCstrand[red](0,1.5)(0,3)
%\DGCdot{E}[r]{$r$}
%%%%
\DGCstrand[red](0,3)(.5,3.5)
\DGCstrand[red](1,3)(.5,3.5)
\DGCstrand[Red](.5,3.5)(.5,4)
\DGCstrand[red](.5,4)(0,4.5)
%\DGCdot{E}[ul]{$a_{i}$}
\DGCstrand[red](.5,4)(1,4.5)
%\DGCdot{E}[u]{$a_{i+1}$}
\DGCstrand[red](2,3)(2,4.5)
%\DGCdot{E}[u]{$a_{i+2}$}
\DGCstrand(1.5,0)(1.25,.75)(1.25,3.75)(.5,4.5)
%\DGCstrand(1.5,0)(1,4.25)(.5,4.5) replace above
\DGCdot{E}[ul]{$ $}
\end{DGCpicture}
~-~
\begin{DGCpicture}[scale={.7,.7}]
\DGCstrand[red](0,0)(.5,.5)[$ $`{\ }]
%\DGCdot{B}[l]{$t$}
\DGCstrand[red](1,0)(.5,.5)[$ $`{\ }]
\DGCstrand[Red](.5,.5)(.5,1)
\DGCstrand[red](.5,1)(0,1.5)
\DGCstrand[red](.5,1)(1,1.5)
\DGCstrand[red](2,0)(2,1.5)[$ $`{\ }]
\DGCdot{B}[u]{$ $}
\DGCstrand[red](1,1.5)(1.5,2)
%\DGCdot{B}[l]{$s$}
\DGCstrand[red](2,1.5)(1.5,2)
\DGCstrand[Red](1.5,2)(1.5,2.5)
\DGCstrand[red](1.5,2.5)(1,3)
\DGCstrand[red](1.5,2.5)(2,3)
\DGCstrand[red](0,1.5)(0,3)
%\DGCdot{E}[r]{$r$}
%%%%
\DGCstrand[red](0,3)(.5,3.5)
\DGCstrand[red](1,3)(.5,3.5)
\DGCstrand[Red](.5,3.5)(.5,4)
\DGCstrand[red](.5,4)(0,4.5)
%\DGCdot{E}[ul]{$a_{i}$}
\DGCstrand[red](.5,4)(1,4.5)
%\DGCdot{E}[u]{$a_{i+1}$}
\DGCstrand[red](2,3)(2,4.5)
%\DGCdot{E}[u]{$ $}
\DGCstrand(1.5,0)(1.25,.75)(1.25,3.75)(.5,4.5)
%\DGCstrand(1.5,0)(1,4.25)(.5,4.5) replace above
%\DGCdot{E}[ul]{$ $}
\end{DGCpicture}
\end{equation*}
implies that
\begin{equation*}
\begin{DGCpicture}[scale={.7,.7}]
\DGCstrand[red](0,0)(.5,.5)[$ $`{\ }]
%\DGCdot{B}[l]{$t$}
\DGCstrand[red](1,0)(.5,.5)[$ $`{\ }]
\DGCstrand[Red](.5,.5)(.5,1)
\DGCstrand[red](.5,1)(0,1.5)
\DGCstrand[red](.5,1)(1,1.5)
\DGCstrand[red](2,0)(2,1.5)[$ $`{\ }]
\DGCstrand[red](1,1.5)(1.5,2)
%\DGCdot{B}[l]{$s$}
\DGCstrand[red](2,1.5)(1.5,2)
\DGCstrand[Red](1.5,2)(1.5,2.5)
\DGCstrand[red](1.5,2.5)(1,3)
\DGCstrand[red](1.5,2.5)(2,3)
\DGCstrand[red](0,1.5)(0,3)
%\DGCdot{E}[r]{$r$}
%%%%
\DGCstrand[red](0,3)(.5,3.5)
\DGCstrand[red](1,3)(.5,3.5)
\DGCstrand[Red](.5,3.5)(.5,4)
\DGCstrand[red](.5,4)(0,4.5)
%\DGCdot{E}[ul]{$a_{i}$}
\DGCstrand[red](.5,4)(1,4.5)
%\DGCdot{E}[u]{$a_{i+1}$}
\DGCstrand[red](2,3)(2,4.5)
%\DGCdot{E}[u]{$a_{i+2}$}
\DGCstrand(1.5,0)(1.75,.75)(1.75,3.75)(.5,4.5)
%\DGCstrand(1.5,0)(2,2.25)(.5,4.5) replace above
\DGCdot{E}[u]{$ $}
\end{DGCpicture}
\end{equation*}
is in the span of the proposed spanning set.
The equation
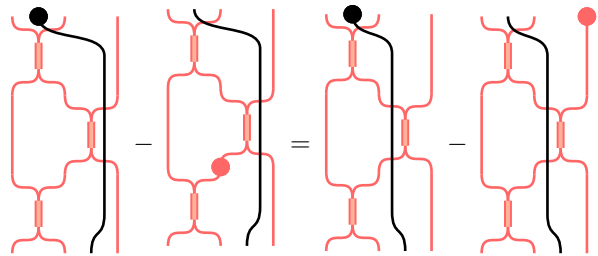
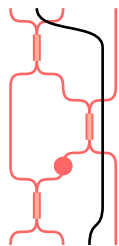
\begin{equation*}
\begin{DGCpicture}[scale={.7,.7}]
\DGCstrand[red](0,0)(.5,.5)[$ $`{\ }]
%\DGCdot{B}[l]{$t$}
\DGCstrand[red](1,0)(.5,.5)[$ $`{\ }]
\DGCstrand[Red](.5,.5)(.5,1)
\DGCstrand[red](.5,1)(0,1.5)
\DGCstrand[red](.5,1)(1,1.5)
\DGCstrand[red](2,0)(2,1.5)[$ $`{\ }]
\DGCstrand[red](1,1.5)(1.5,2)
%\DGCdot{B}[l]{$s$}
\DGCstrand[red](2,1.5)(1.5,2)
\DGCstrand[Red](1.5,2)(1.5,2.5)
\DGCstrand[red](1.5,2.5)(1,3)
\DGCstrand[red](1.5,2.5)(2,3)
\DGCstrand[red](0,1.5)(0,3)
%\DGCdot{E}[r]{$r$}
%%%%
\DGCstrand[red](0,3)(.5,3.5)
\DGCstrand[red](1,3)(.5,3.5)
\DGCstrand[Red](.5,3.5)(.5,4)
\DGCstrand[red](.5,4)(0,4.5)
%\DGCdot{E}[ul]{$a_{i}$}
\DGCstrand[red](.5,4)(1,4.5)
%\DGCdot{E}[u]{$a_{i+1}$}
\DGCstrand[red](2,3)(2,4.5)
%\DGCdot{E}[u]{$a_{i+2}$}
\DGCstrand(1.5,0)(1.75,.75)(1.75,3.75)(.5,4.5)
%\DGCstrand(1.5,0)(2,2.25)(.5,4.5) replace above
\DGCdot{E}[u]{$ $}
\end{DGCpicture}
~-~
\begin{DGCpicture}[scale={.7,.7}]
\DGCstrand[red](0,0)(.5,.5)[$ $`{\ }]
%\DGCdot{B}[l]{$t$}
\DGCstrand[red](1,0)(.5,.5)[$ $`{\ }]
\DGCstrand[Red](.5,.5)(.5,1)
\DGCstrand[red](.5,1)(0,1.5)
\DGCstrand[red](.5,1)(1,1.5)
\DGCstrand[red](2,0)(2,1.5)[$ $`{\ }]
\DGCstrand[red](1,1.5)(1.5,2)
\DGCdot{B}[l]{$ $}
\DGCstrand[red](2,1.5)(1.5,2)
\DGCstrand[Red](1.5,2)(1.5,2.5)
\DGCstrand[red](1.5,2.5)(1,3)
%\DGCdot{E}[u]{$ $}
\DGCstrand[red](1.5,2.5)(2,3)
\DGCstrand[red](0,1.5)(0,3)
%\DGCdot{E}[r]{$r$}
%%%%
\DGCstrand[red](0,3)(.5,3.5)
\DGCstrand[red](1,3)(.5,3.5)
\DGCstrand[Red](.5,3.5)(.5,4)
\DGCstrand[red](.5,4)(0,4.5)
%\DGCdot{E}[ul]{$a_{i}$}
\DGCstrand[red](.5,4)(1,4.5)
%\DGCdot{E}[u]{$a_{i+1}$}
\DGCstrand[red](2,3)(2,4.5)
%\DGCdot{E}[u]{$a_{i+2}$}
\DGCstrand(1.5,0)(1.75,.75)(1.75,3.75)(.5,4.5)
%\DGCstrand(1.5,0)(2,2.25)(.5,4.5) replace above
%\DGCdot{E}[u]{$ $}
\end{DGCpicture}
~=~
\begin{DGCpicture}[scale={.7,.7}]
\DGCstrand[red](0,0)(.5,.5)[$ $`{\ }]
%\DGCdot{B}[l]{$t$}
\DGCstrand[red](1,0)(.5,.5)[$ $`{\ }]
\DGCstrand[Red](.5,.5)(.5,1)
\DGCstrand[red](.5,1)(0,1.5)
\DGCstrand[red](.5,1)(1,1.5)
\DGCstrand[red](2,0)(2,1.5)[$ $`{\ }]
\DGCstrand[red](1,1.5)(1.5,2)
%\DGCdot{B}[l]{$s$}
\DGCstrand[red](2,1.5)(1.5,2)
\DGCstrand[Red](1.5,2)(1.5,2.5)
\DGCstrand[red](1.5,2.5)(1,3)
\DGCstrand[red](1.5,2.5)(2,3)
\DGCstrand[red](0,1.5)(0,3)
%\DGCdot{E}[r]{$r$}
%%%%
\DGCstrand[red](0,3)(.5,3.5)
\DGCstrand[red](1,3)(.5,3.5)
\DGCstrand[Red](.5,3.5)(.5,4)
\DGCstrand[red](.5,4)(0,4.5)
%\DGCdot{E}[ul]{$a_{i}$}
\DGCstrand[red](.5,4)(1,4.5)
%\DGCdot{E}[u]{$a_{i+1}$}
\DGCstrand[red](2,3)(2,4.5)
%\DGCdot{E}[u]{$a_{i+2}$}
\DGCstrand(1.5,0)(1.25,.75)(1.25,3.75)(.5,4.5)
%\DGCstrand(1.5,0)(1,4.25)(.5,4.5) replace above
\DGCdot{E}[ul]{$ $}
\end{DGCpicture}
~-~
\begin{DGCpicture}[scale={.7,.7}]
\DGCstrand[red](0,0)(.5,.5)[$ $`{\ }]
%\DGCdot{B}[l]{$t$}
\DGCstrand[red](1,0)(.5,.5)[$ $`{\ }]
\DGCstrand[Red](.5,.5)(.5,1)
\DGCstrand[red](.5,1)(0,1.5)
\DGCstrand[red](.5,1)(1,1.5)
\DGCstrand[red](2,0)(2,1.5)[$ $`{\ }]
\DGCstrand[red](1,1.5)(1.5,2)
%\DGCdot{B}[l]{$s$}
\DGCstrand[red](2,1.5)(1.5,2)
\DGCstrand[Red](1.5,2)(1.5,2.5)
\DGCstrand[red](1.5,2.5)(1,3)
\DGCstrand[red](1.5,2.5)(2,3)
\DGCstrand[red](0,1.5)(0,3)
%\DGCdot{E}[r]{$r$}
%%%%
\DGCstrand[red](0,3)(.5,3.5)
\DGCstrand[red](1,3)(.5,3.5)
\DGCstrand[Red](.5,3.5)(.5,4)
\DGCstrand[red](.5,4)(0,4.5)
%\DGCdot{E}[ul]{$a_{i}$}
\DGCstrand[red](.5,4)(1,4.5)
%\DGCdot{E}[u]{$a_{i+1}$}
\DGCstrand[red](2,3)(2,4.5)
\DGCdot{E}[u]{$ $}
\DGCstrand(1.5,0)(1.25,.75)(1.25,3.75)(.5,4.5)
%\DGCstrand(1.5,0)(1,4.25)(.5,4.5) replace above
%\DGCdot{E}[ul]{$ $}
\end{DGCpicture}
\end{equation*}
implies that 
\begin{equation*}
\begin{DGCpicture}[scale={.7,.7}]
\DGCstrand[red](0,0)(.5,.5)[$ $`{\ }]
%\DGCdot{B}[l]{$t$}
\DGCstrand[red](1,0)(.5,.5)[$ $`{\ }]
\DGCstrand[Red](.5,.5)(.5,1)
\DGCstrand[red](.5,1)(0,1.5)
\DGCstrand[red](.5,1)(1,1.5)
\DGCstrand[red](2,0)(2,1.5)[$ $`{\ }]
\DGCstrand[red](1,1.5)(1.5,2)
\DGCdot{B}[l]{$ $}
\DGCstrand[red](2,1.5)(1.5,2)
\DGCstrand[Red](1.5,2)(1.5,2.5)
\DGCstrand[red](1.5,2.5)(1,3)
%\DGCdot{E}[u]{$ $}
\DGCstrand[red](1.5,2.5)(2,3)
\DGCstrand[red](0,1.5)(0,3)
%\DGCdot{E}[r]{$r$}
%%%%
\DGCstrand[red](0,3)(.5,3.5)
\DGCstrand[red](1,3)(.5,3.5)
\DGCstrand[Red](.5,3.5)(.5,4)
\DGCstrand[red](.5,4)(0,4.5)
%\DGCdot{E}[ul]{$a_{i}$}
\DGCstrand[red](.5,4)(1,4.5)
%\DGCdot{E}[u]{$a_{i+1}$}
\DGCstrand[red](2,3)(2,4.5)
%\DGCdot{E}[u]{$a_{i+2}$}
\DGCstrand(1.5,0)(1.75,.75)(1.75,3.75)(.5,4.5)
%\DGCstrand(1.5,0)(2,2.25)(.5,4.5) replace above
%\DGCdot{E}[u]{$ $}
\end{DGCpicture}
\end{equation*}
is in the span.

Next note that
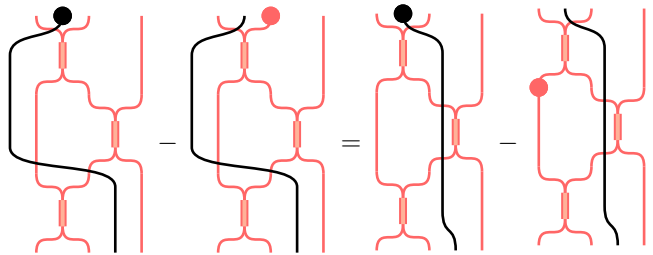
\begin{equation*}
\begin{DGCpicture}[scale={.7,.7}]
\DGCstrand[red](0,0)(.5,.5)[$ $`{\ }]
%\DGCdot{B}[l]{$t$}
\DGCstrand[red](1,0)(.5,.5)[$ $`{\ }]
\DGCstrand[Red](.5,.5)(.5,1)
\DGCstrand[red](.5,1)(0,1.5)
\DGCstrand[red](.5,1)(1,1.5)
\DGCstrand[red](2,0)(2,1.5)[$ $`{\ }]
\DGCstrand[red](1,1.5)(1.5,2)
%\DGCdot{B}[l]{$s$}
\DGCstrand[red](2,1.5)(1.5,2)
\DGCstrand[Red](1.5,2)(1.5,2.5)
\DGCstrand[red](1.5,2.5)(1,3)
\DGCstrand[red](1.5,2.5)(2,3)
\DGCstrand[red](0,1.5)(0,3)
%\DGCdot{E}[r]{$r$}
%%%%
\DGCstrand[red](0,3)(.5,3.5)
\DGCstrand[red](1,3)(.5,3.5)
\DGCstrand[Red](.5,3.5)(.5,4)
\DGCstrand[red](.5,4)(0,4.5)
%\DGCdot{E}[u]{$a_{i}$}
\DGCstrand[red](.5,4)(1,4.5)
%\DGCdot{E}[u]{$a_{i+1}$}
\DGCstrand[red](2,3)(2,4.5)
%\DGCdot{E}[u]{$a_{i+2}$}
\DGCstrand(1.5,0)(1.5,1.25)(-.5,2)(-.5,3.75)(.5,4.5)
%\DGCstrand(1.5,0)(-.5,2.25)(.5,4.5) replace above
\DGCdot{E}[u]{$ $}
\end{DGCpicture}
~-~
\begin{DGCpicture}[scale={.7,.7}]
\DGCstrand[red](0,0)(.5,.5)[$ $`{\ }]
%\DGCdot{B}[l]{$t$}
\DGCstrand[red](1,0)(.5,.5)[$ $`{\ }]
\DGCstrand[Red](.5,.5)(.5,1)
\DGCstrand[red](.5,1)(0,1.5)
\DGCstrand[red](.5,1)(1,1.5)
\DGCstrand[red](2,0)(2,1.5)[$ $`{\ }]
\DGCstrand[red](1,1.5)(1.5,2)
%\DGCdot{B}[l]{$s$}
\DGCstrand[red](2,1.5)(1.5,2)
\DGCstrand[Red](1.5,2)(1.5,2.5)
\DGCstrand[red](1.5,2.5)(1,3)
\DGCstrand[red](1.5,2.5)(2,3)
\DGCstrand[red](0,1.5)(0,3)
%\DGCdot{E}[r]{$r$}
%%%%
\DGCstrand[red](0,3)(.5,3.5)
\DGCstrand[red](1,3)(.5,3.5)
\DGCstrand[Red](.5,3.5)(.5,4)
\DGCstrand[red](.5,4)(0,4.5)
%\DGCdot{E}[u]{$a_{i}$}
\DGCstrand[red](.5,4)(1,4.5)
\DGCdot{E}[u]{$ $}
\DGCstrand[red](2,3)(2,4.5)
%\DGCdot{E}[u]{$a_{i+2}$}
\DGCstrand(1.5,0)(1.5,1.25)(-.5,2)(-.5,3.75)(.5,4.5)
%\DGCstrand(1.5,0)(-.5,2.25)(.5,4.5) replace above
%\DGCdot{E}[u]{$ $}
\end{DGCpicture}
~=~
\begin{DGCpicture}[scale={.7,.7}]
\DGCstrand[red](0,0)(.5,.5)[$ $`{\ }]
%\DGCdot{B}[l]{$t$}
\DGCstrand[red](1,0)(.5,.5)[$ $`{\ }]
\DGCstrand[Red](.5,.5)(.5,1)
\DGCstrand[red](.5,1)(0,1.5)
\DGCstrand[red](.5,1)(1,1.5)
\DGCstrand[red](2,0)(2,1.5)[$ $`{\ }]
\DGCstrand[red](1,1.5)(1.5,2)
%\DGCdot{B}[l]{$s$}
\DGCstrand[red](2,1.5)(1.5,2)
\DGCstrand[Red](1.5,2)(1.5,2.5)
\DGCstrand[red](1.5,2.5)(1,3)
\DGCstrand[red](1.5,2.5)(2,3)
\DGCstrand[red](0,1.5)(0,3)
%\DGCdot{E}[r]{$r$}
%%%%
\DGCstrand[red](0,3)(.5,3.5)
\DGCstrand[red](1,3)(.5,3.5)
\DGCstrand[Red](.5,3.5)(.5,4)
\DGCstrand[red](.5,4)(0,4.5)
%\DGCdot{E}[ul]{$a_{i}$}
\DGCstrand[red](.5,4)(1,4.5)
%\DGCdot{E}[u]{$a_{i+1}$}
\DGCstrand[red](2,3)(2,4.5)
%\DGCdot{E}[u]{$a_{i+2}$}
\DGCstrand(1.5,0)(1.25,.75)(1.25,3.75)(.5,4.5)
%\DGCstrand(1.5,0)(1,4.25)(.5,4.5) replace above
\DGCdot{E}[ul]{$ $}
\end{DGCpicture}
~-~
\begin{DGCpicture}[scale={.7,.7}]
\DGCstrand[red](0,0)(.5,.5)[$ $`{\ }]
%\DGCdot{B}[l]{$t$}
\DGCstrand[red](1,0)(.5,.5)[$ $`{\ }]
\DGCstrand[Red](.5,.5)(.5,1)
\DGCstrand[red](.5,1)(0,1.5)
\DGCstrand[red](.5,1)(1,1.5)
\DGCstrand[red](2,0)(2,1.5)[$ $`{\ }]
\DGCstrand[red](1,1.5)(1.5,2)
%\DGCdot{B}[l]{$s$}
\DGCstrand[red](2,1.5)(1.5,2)
\DGCstrand[Red](1.5,2)(1.5,2.5)
\DGCstrand[red](1.5,2.5)(1,3)
\DGCstrand[red](1.5,2.5)(2,3)
\DGCstrand[red](0,1.5)(0,3)
\DGCdot{E}[r]{$ $}
%%%%
\DGCstrand[red](0,3)(.5,3.5)
\DGCstrand[red](1,3)(.5,3.5)
\DGCstrand[Red](.5,3.5)(.5,4)
\DGCstrand[red](.5,4)(0,4.5)
%\DGCdot{E}[ul]{$a_{i}$}
\DGCstrand[red](.5,4)(1,4.5)
%\DGCdot{E}[u]{$a_{i+1}$}
\DGCstrand[red](2,3)(2,4.5)
%\DGCdot{E}[u]{$ $}
\DGCstrand(1.5,0)(1.25,.75)(1.25,3.75)(.5,4.5)
%\DGCstrand(1.5,0)(1,4.25)(.5,4.5) replace above
%\DGCdot{E}[ul]{$ $}
\end{DGCpicture}
\ .
\end{equation*}
Thus a diagram with a black dot on
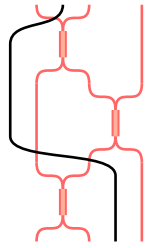
\begin{equation*}
\begin{DGCpicture}[scale={.7,.7}]
\DGCstrand[red](0,0)(.5,.5)[$ $`{\ }]
%\DGCdot{B}[l]{$t$}
\DGCstrand[red](1,0)(.5,.5)[$ $`{\ }]
\DGCstrand[Red](.5,.5)(.5,1)
\DGCstrand[red](.5,1)(0,1.5)
\DGCstrand[red](.5,1)(1,1.5)
\DGCstrand[red](2,0)(2,1.5)[$ $`{\ }]
\DGCstrand[red](1,1.5)(1.5,2)
%\DGCdot{B}[l]{$s$}
\DGCstrand[red](2,1.5)(1.5,2)
\DGCstrand[Red](1.5,2)(1.5,2.5)
\DGCstrand[red](1.5,2.5)(1,3)
\DGCstrand[red](1.5,2.5)(2,3)
\DGCstrand[red](0,1.5)(0,3)
%\DGCdot{E}[r]{$r$}
%%%%
\DGCstrand[red](0,3)(.5,3.5)
\DGCstrand[red](1,3)(.5,3.5)
\DGCstrand[Red](.5,3.5)(.5,4)
\DGCstrand[red](.5,4)(0,4.5)
%\DGCdot{E}[u]{$a_{i}$}
\DGCstrand[red](.5,4)(1,4.5)
%\DGCdot{E}[u]{$a_{i+1}$}
\DGCstrand[red](2,3)(2,4.5)
%\DGCdot{E}[u]{$a_{i+2}$}
\DGCstrand(1.5,0)(1.5,1.25)(-.5,2)(-.5,3.75)(.5,4.5)
%\DGCstrand(1.5,0)(-.5,2.25)(.5,4.5) replace above
%\DGCdot{E}[u]{$ $}
\end{DGCpicture}
\end{equation*}
is already in the span of the proposed spanning set.
The equation
\begin{equation*}
\begin{DGCpicture}[scale={.7,.7}]
\DGCstrand[red](0,0)(.5,.5)[$ $`{\ }]
%\DGCdot{B}[l]{$t$}
\DGCstrand[red](1,0)(.5,.5)[$ $`{\ }]
\DGCstrand[Red](.5,.5)(.5,1)
\DGCstrand[red](.5,1)(0,1.5)
\DGCstrand[red](.5,1)(1,1.5)
\DGCstrand[red](2,0)(2,1.5)[$ $`{\ }]
\DGCstrand[red](1,1.5)(1.5,2)
%\DGCdot{B}[l]{$s$}
\DGCstrand[red](2,1.5)(1.5,2)
\DGCstrand[Red](1.5,2)(1.5,2.5)
\DGCstrand[red](1.5,2.5)(1,3)
\DGCstrand[red](1.5,2.5)(2,3)
\DGCstrand[red](0,1.5)(0,3)
%\DGCdot{E}[r]{$r$}
%%%%
\DGCstrand[red](0,3)(.5,3.5)
\DGCstrand[red](1,3)(.5,3.5)
\DGCstrand[Red](.5,3.5)(.5,4)
\DGCstrand[red](.5,4)(0,4.5)
%\DGCdot{E}[u]{$a_{i}$}
\DGCstrand[red](.5,4)(1,4.5)
%\DGCdot{E}[u]{$a_{i+1}$}
\DGCstrand[red](2,3)(2,4.5)
%\DGCdot{E}[u]{$a_{i+2}$}
\DGCstrand(1.5,0)(1.5,1.25)(-.5,2)(-.5,3.75)(.5,4.5)
%\DGCstrand(1.5,0)(-.5,2.25)(.5,4.5) replace above
\DGCdot{E}[u]{$ $}
\end{DGCpicture}
~-~
\begin{DGCpicture}[scale={.7,.7}]
\DGCstrand[red](0,0)(.5,.5)[$ $`{\ }]
%\DGCdot{B}[l]{$t$}
\DGCstrand[red](1,0)(.5,.5)[$ $`{\ }]
\DGCstrand[Red](.5,.5)(.5,1)
\DGCstrand[red](.5,1)(0,1.5)
\DGCstrand[red](.5,1)(1,1.5)
\DGCstrand[red](2,0)(2,1.5)[$ $`{\ }]
\DGCstrand[red](1,1.5)(1.5,2)
%\DGCdot{B}[l]{$s$}
\DGCstrand[red](2,1.5)(1.5,2)
\DGCstrand[Red](1.5,2)(1.5,2.5)
\DGCstrand[red](1.5,2.5)(1,3)
\DGCdot{E}[l]{$ $}
\DGCstrand[red](1.5,2.5)(2,3)
\DGCstrand[red](0,1.5)(0,3)
%\DGCdot{E}[r]{$r$}
%%%%
\DGCstrand[red](0,3)(.5,3.5)
\DGCstrand[red](1,3)(.5,3.5)
\DGCstrand[Red](.5,3.5)(.5,4)
\DGCstrand[red](.5,4)(0,4.5)
%\DGCdot{E}[u]{$a_{i}$}
\DGCstrand[red](.5,4)(1,4.5)
%\DGCdot{E}[u]{$ $}
\DGCstrand[red](2,3)(2,4.5)
%\DGCdot{E}[u]{$a_{i+2}$}
\DGCstrand(1.5,0)(1.5,1.25)(-.5,2)(-.5,3.75)(.5,4.5)
%\DGCstrand(1.5,0)(-.5,2.25)(.5,4.5) replace above
%\DGCdot{E}[u]{$ $}
\end{DGCpicture}
~=~
\begin{DGCpicture}[scale={.7,.7}]
\DGCstrand[red](0,0)(.5,.5)[$ $`{\ }]
%\DGCdot{B}[l]{$t$}
\DGCstrand[red](1,0)(.5,.5)[$ $`{\ }]
\DGCstrand[Red](.5,.5)(.5,1)
\DGCstrand[red](.5,1)(0,1.5)
\DGCstrand[red](.5,1)(1,1.5)
\DGCstrand[red](2,0)(2,1.5)[$ $`{\ }]
\DGCstrand[red](1,1.5)(1.5,2)
%\DGCdot{B}[l]{$s$}
\DGCstrand[red](2,1.5)(1.5,2)
\DGCstrand[Red](1.5,2)(1.5,2.5)
\DGCstrand[red](1.5,2.5)(1,3)
\DGCstrand[red](1.5,2.5)(2,3)
\DGCstrand[red](0,1.5)(0,3)
%\DGCdot{E}[r]{$r$}
%%%%
\DGCstrand[red](0,3)(.5,3.5)
\DGCstrand[red](1,3)(.5,3.5)
\DGCstrand[Red](.5,3.5)(.5,4)
\DGCstrand[red](.5,4)(0,4.5)
%\DGCdot{E}[ul]{$a_{i}$}
\DGCstrand[red](.5,4)(1,4.5)
%\DGCdot{E}[u]{$a_{i+1}$}
\DGCstrand[red](2,3)(2,4.5)
%\DGCdot{E}[u]{$a_{i+2}$}
\DGCstrand(1.5,0)(1.25,.75)(1.25,3.75)(.5,4.5)
%\DGCstrand(1.5,0)(1,4.25)(.5,4.5) replace above
\DGCdot{E}[ul]{$ $}
\end{DGCpicture}
~-~
\begin{DGCpicture}[scale={.7,.7}]
\DGCstrand[red](0,0)(.5,.5)[$ $`{\ }]
%\DGCdot{B}[l]{$t$}
\DGCstrand[red](1,0)(.5,.5)[$ $`{\ }]
\DGCstrand[Red](.5,.5)(.5,1)
\DGCstrand[red](.5,1)(0,1.5)
\DGCstrand[red](.5,1)(1,1.5)
\DGCstrand[red](2,0)(2,1.5)[$ $`{\ }]
\DGCstrand[red](1,1.5)(1.5,2)
%\DGCdot{B}[l]{$s$}
\DGCstrand[red](2,1.5)(1.5,2)
\DGCstrand[Red](1.5,2)(1.5,2.5)
\DGCstrand[red](1.5,2.5)(1,3)
\DGCstrand[red](1.5,2.5)(2,3)
\DGCstrand[red](0,1.5)(0,3)
%\DGCdot{E}[r]{$ $}
%%%%
\DGCstrand[red](0,3)(.5,3.5)
\DGCstrand[red](1,3)(.5,3.5)
\DGCstrand[Red](.5,3.5)(.5,4)
\DGCstrand[red](.5,4)(0,4.5)
\DGCdot{E}[ul]{$ $}
\DGCstrand[red](.5,4)(1,4.5)
%\DGCdot{E}[u]{$a_{i+1}$}
\DGCstrand[red](2,3)(2,4.5)
%\DGCdot{E}[u]{$ $}
\DGCstrand(1.5,0)(1.25,.75)(1.25,3.75)(.5,4.5)
%\DGCstrand(1.5,0)(1,4.25)(.5,4.5) replace above
%\DGCdot{E}[ul]{$ $}
\end{DGCpicture}
\end{equation*}
implies that
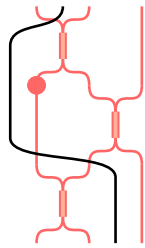
\begin{equation*}
\begin{DGCpicture}[scale={.7,.7}]
\DGCstrand[red](0,0)(.5,.5)[$ $`{\ }]
%\DGCdot{B}[l]{$t$}
\DGCstrand[red](1,0)(.5,.5)[$ $`{\ }]
\DGCstrand[Red](.5,.5)(.5,1)
\DGCstrand[red](.5,1)(0,1.5)
\DGCstrand[red](.5,1)(1,1.5)
\DGCstrand[red](2,0)(2,1.5)[$ $`{\ }]
\DGCstrand[red](1,1.5)(1.5,2)
%\DGCdot{B}[l]{$s$}
\DGCstrand[red](2,1.5)(1.5,2)
\DGCstrand[Red](1.5,2)(1.5,2.5)
\DGCstrand[red](1.5,2.5)(1,3)
%\DGCdot{E}[l]{$ $}
\DGCstrand[red](1.5,2.5)(2,3)
\DGCstrand[red](0,1.5)(0,3)
\DGCdot{E}[r]{$ $}
%%%%
\DGCstrand[red](0,3)(.5,3.5)
\DGCstrand[red](1,3)(.5,3.5)
\DGCstrand[Red](.5,3.5)(.5,4)
\DGCstrand[red](.5,4)(0,4.5)
%\DGCdot{E}[u]{$a_{i}$}
\DGCstrand[red](.5,4)(1,4.5)
%\DGCdot{E}[u]{$ $}
\DGCstrand[red](2,3)(2,4.5)
%\DGCdot{E}[u]{$a_{i+2}$}
\DGCstrand(1.5,0)(1.5,1.25)(-.5,2)(-.5,3.75)(.5,4.5)
%\DGCstrand(1.5,0)(-.5,2.25)(.5,4.5) replace above
%\DGCdot{E}[u]{$ $}
\end{DGCpicture}
\end{equation*}
is in the span.
The equation
\begin{equation*}
\begin{DGCpicture}[scale={.7,.7}]
\DGCstrand[red](0,0)(.5,.5)[$ $`{\ }]
%\DGCdot{B}[l]{$t$}
\DGCstrand[red](1,0)(.5,.5)[$ $`{\ }]
\DGCstrand[Red](.5,.5)(.5,1)
\DGCstrand[red](.5,1)(0,1.5)
\DGCstrand[red](.5,1)(1,1.5)
\DGCstrand[red](2,0)(2,1.5)[$ $`{\ }]
\DGCstrand[red](1,1.5)(1.5,2)
%\DGCdot{B}[l]{$s$}
\DGCstrand[red](2,1.5)(1.5,2)
\DGCstrand[Red](1.5,2)(1.5,2.5)
\DGCstrand[red](1.5,2.5)(1,3)
\DGCstrand[red](1.5,2.5)(2,3)
\DGCstrand[red](0,1.5)(0,3)
%\DGCdot{E}[r]{$r$}
%%%%
\DGCstrand[red](0,3)(.5,3.5)
\DGCstrand[red](1,3)(.5,3.5)
\DGCstrand[Red](.5,3.5)(.5,4)
\DGCstrand[red](.5,4)(0,4.5)
%\DGCdot{E}[u]{$a_{i}$}
\DGCstrand[red](.5,4)(1,4.5)
%\DGCdot{E}[u]{$a_{i+1}$}
\DGCstrand[red](2,3)(2,4.5)
%\DGCdot{E}[u]{$a_{i+2}$}
\DGCstrand(1.5,0)(1.5,1.25)(-.5,2)(-.5,3.75)(.5,4.5)
%\DGCstrand(1.5,0)(-.5,2.25)(.5,4.5) replace above
\DGCdot{E}[u]{$ $}
\end{DGCpicture}
~-~
\begin{DGCpicture}[scale={.7,.7}]
\DGCstrand[red](0,0)(.5,.5)[$ $`{\ }]
%\DGCdot{B}[l]{$t$}
\DGCstrand[red](1,0)(.5,.5)[$ $`{\ }]
\DGCstrand[Red](.5,.5)(.5,1)
\DGCstrand[red](.5,1)(0,1.5)
\DGCstrand[red](.5,1)(1,1.5)
\DGCstrand[red](2,0)(2,1.5)[$ $`{\ }]
\DGCdot{B}[l]{$ $}
\DGCstrand[red](1,1.5)(1.5,2)
%\DGCdot{B}[l]{$s$}
\DGCstrand[red](2,1.5)(1.5,2)
\DGCstrand[Red](1.5,2)(1.5,2.5)
\DGCstrand[red](1.5,2.5)(1,3)
%\DGCdot{E}[l]{$ $}
\DGCstrand[red](1.5,2.5)(2,3)
\DGCstrand[red](0,1.5)(0,3)
%\DGCdot{E}[r]{$r$}
%%%%
\DGCstrand[red](0,3)(.5,3.5)
\DGCstrand[red](1,3)(.5,3.5)
\DGCstrand[Red](.5,3.5)(.5,4)
\DGCstrand[red](.5,4)(0,4.5)
%\DGCdot{E}[u]{$a_{i}$}
\DGCstrand[red](.5,4)(1,4.5)
%\DGCdot{E}[u]{$ $}
\DGCstrand[red](2,3)(2,4.5)
%\DGCdot{E}[u]{$a_{i+2}$}
\DGCstrand(1.5,0)(1.5,1.25)(-.5,2)(-.5,3.75)(.5,4.5)
%\DGCstrand(1.5,0)(-.5,2.25)(.5,4.5) replace above
%\DGCdot{E}[u]{$ $}
\end{DGCpicture}
~=~
\begin{DGCpicture}[scale={.7,.7}]
\DGCstrand[red](0,0)(.5,.5)[$ $`{\ }]
%\DGCdot{B}[l]{$t$}
\DGCstrand[red](1,0)(.5,.5)[$ $`{\ }]
\DGCstrand[Red](.5,.5)(.5,1)
\DGCstrand[red](.5,1)(0,1.5)
\DGCstrand[red](.5,1)(1,1.5)
\DGCstrand[red](2,0)(2,1.5)[$ $`{\ }]
\DGCstrand[red](1,1.5)(1.5,2)
%\DGCdot{B}[l]{$s$}
\DGCstrand[red](2,1.5)(1.5,2)
\DGCstrand[Red](1.5,2)(1.5,2.5)
\DGCstrand[red](1.5,2.5)(1,3)
\DGCstrand[red](1.5,2.5)(2,3)
\DGCstrand[red](0,1.5)(0,3)
%\DGCdot{E}[r]{$r$}
%%%%
\DGCstrand[red](0,3)(.5,3.5)
\DGCstrand[red](1,3)(.5,3.5)
\DGCstrand[Red](.5,3.5)(.5,4)
\DGCstrand[red](.5,4)(0,4.5)
%\DGCdot{E}[ul]{$a_{i}$}
\DGCstrand[red](.5,4)(1,4.5)
%\DGCdot{E}[u]{$a_{i+1}$}
\DGCstrand[red](2,3)(2,4.5)
%\DGCdot{E}[u]{$a_{i+2}$}
\DGCstrand(1.5,0)(1.75,.75)(1.75,3.75)(.5,4.5)
%\DGCstrand(1.5,0)(2.5,2.25)(.5,4.5) replace above
\DGCdot{E}[ul]{$ $}
\end{DGCpicture}
~-~
\begin{DGCpicture}[scale={.7,.7}]
\DGCstrand[red](0,0)(.5,.5)[$ $`{\ }]
%\DGCdot{B}[l]{$t$}
\DGCstrand[red](1,0)(.5,.5)[$ $`{\ }]
\DGCstrand[Red](.5,.5)(.5,1)
\DGCstrand[red](.5,1)(0,1.5)
\DGCstrand[red](.5,1)(1,1.5)
\DGCstrand[red](2,0)(2,1.5)[$ $`{\ }]
\DGCstrand[red](1,1.5)(1.5,2)
%\DGCdot{B}[l]{$s$}
\DGCstrand[red](2,1.5)(1.5,2)
\DGCstrand[Red](1.5,2)(1.5,2.5)
\DGCstrand[red](1.5,2.5)(1,3)
\DGCstrand[red](1.5,2.5)(2,3)
\DGCstrand[red](0,1.5)(0,3)
%\DGCdot{E}[r]{$ $}
%%%%
\DGCstrand[red](0,3)(.5,3.5)
\DGCstrand[red](1,3)(.5,3.5)
\DGCstrand[Red](.5,3.5)(.5,4)
\DGCstrand[red](.5,4)(0,4.5)
\DGCdot{E}[ul]{$ $}
\DGCstrand[red](.5,4)(1,4.5)
%\DGCdot{E}[u]{$a_{i+1}$}
\DGCstrand[red](2,3)(2,4.5)
%\DGCdot{E}[u]{$ $}
\DGCstrand(1.5,0)(1.75,.75)(1.75,3.75)(.5,4.5)
%\DGCstrand(1.5,0)(2.5,2.25)(.5,4.5) replace above
%\DGCdot{E}[ul]{$ $}
\end{DGCpicture}
\end{equation*}
implies that
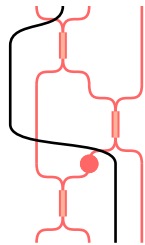
\begin{equation*}
\begin{DGCpicture}[scale={.7,.7}]
\DGCstrand[red](0,0)(.5,.5)[$ $`{\ }]
%\DGCdot{B}[l]{$t$}
\DGCstrand[red](1,0)(.5,.5)[$ $`{\ }]
\DGCstrand[Red](.5,.5)(.5,1)
\DGCstrand[red](.5,1)(0,1.5)
\DGCstrand[red](.5,1)(1,1.5)
\DGCstrand[red](2,0)(2,1.5)[$ $`{\ }]
%\DGCdot{B}[l]{$ $}
%%
\DGCstrand[red](1,1.5)(1.5,2)
\DGCdot{B}[l]{$ $}
\DGCstrand[red](2,1.5)(1.5,2)
\DGCstrand[Red](1.5,2)(1.5,2.5)
\DGCstrand[red](1.5,2.5)(1,3)
%\DGCdot{E}[l]{$ $}
\DGCstrand[red](1.5,2.5)(2,3)
\DGCstrand[red](0,1.5)(0,3)
%\DGCdot{E}[r]{$r$}
%%%%
\DGCstrand[red](0,3)(.5,3.5)
\DGCstrand[red](1,3)(.5,3.5)
\DGCstrand[Red](.5,3.5)(.5,4)
\DGCstrand[red](.5,4)(0,4.5)
%\DGCdot{E}[u]{$a_{i}$}
\DGCstrand[red](.5,4)(1,4.5)
%\DGCdot{E}[u]{$ $}
\DGCstrand[red](2,3)(2,4.5)
%\DGCdot{E}[u]{$a_{i+2}$}
\DGCstrand(1.5,0)(1.5,1.5)(-.5,2.25)(-.5,3.75)(.5,4.5)
%\DGCstrand(1.5,0)(-.5,2.25)(.5,4.5) replace above
%\DGCdot{E}[u]{$ $}
\end{DGCpicture}
\end{equation*}
is in the span.

Similar manipulations show that the set of elements in the statement of the proposition, do in fact provide a spanning set for the bimodule.
\end{proof}
We will prove that the spanning set above is indeed a basis.  This will utilize a bimodule homomorphism introduced in the next section.

\subsection{Bimodule homomorphisms} \label{secbihoms}
There is a bimodule homomorphism
\begin{equation*}
\epsilon_i \colon W_i \rightarrow W
\end{equation*}
determined by
\begin{equation}
\label{defofepsilon}
\begin{DGCpicture}
\DGCstrand[red](0,0)(.5,.5)
\DGCstrand[red](1,0)(.5,.5)
\DGCstrand[Red](.5,.5)(.5,1)
\DGCstrand[red](.5,1)(0,1.5)
\DGCstrand[red](.5,1)(1,1.5)
\end{DGCpicture}
~\mapsto~
\begin{DGCpicture}
\DGCstrand[red](0,0)(0,1.5)
\DGCstrand[red](1,0)(1,1.5)
\end{DGCpicture}
\ .
\end{equation}
It is clear that $\epsilon_i$ commutes with $\partial$.

There is a bimodule homomorphism
\begin{equation*}
\iota_i \colon W \rightarrow W_i^{-e_1} \{-2 \}
\end{equation*}
determined by 
\begin{equation*}
\begin{DGCpicture}
\DGCstrand[red](0,0)(0,1.5)
\DGCstrand[red](1,0)(1,1.5)
\end{DGCpicture}
~\mapsto~
\begin{DGCpicture}
\DGCstrand[red](0,0)(.5,.5)
\DGCstrand[red](1,0)(.5,.5)
\DGCstrand[Red](.5,.5)(.5,1)
\DGCstrand[red](.5,1)(0,1.5)
\DGCdot{E}
\DGCstrand[red](.5,1)(1,1.5)
\end{DGCpicture}
~-~
\begin{DGCpicture}
\DGCstrand[red](0,0)(.5,.5)
\DGCstrand[red](1,0)(.5,.5)
\DGCdot{B}
\DGCstrand[Red](.5,.5)(.5,1)
\DGCstrand[red](.5,1)(0,1.5)
\DGCstrand[red](.5,1)(1,1.5)
\end{DGCpicture}
, \\
\quad \quad \quad \quad
\begin{DGCpicture}
\DGCstrand[red](0,0)(0,1.5)
\DGCstrand(.5,0)(.5,1.5)
\DGCstrand[red](1,0)(1,1.5)
\end{DGCpicture}
~\mapsto~
\begin{DGCpicture}
\DGCstrand[red](0,0)(.5,.5)
\DGCstrand[red](1,0)(.5,.5)
\DGCstrand[Red](.5,.5)(.5,1)
\DGCstrand[red](.5,1)(0,1.5)
\DGCstrand[red](.5,1)(1,1.5)
\DGCstrand(.5,0)(1,.75)(.5,1.5)
\end{DGCpicture}
~-~
\begin{DGCpicture}
\DGCstrand[red](0,0)(.5,.5)
\DGCstrand[red](1,0)(.5,.5)
\DGCstrand[Red](.5,.5)(.5,1)
\DGCstrand[red](.5,1)(0,1.5)
\DGCstrand[red](.5,1)(1,1.5)
\DGCstrand(.5,0)(0,.75)(.5,1.5)
\end{DGCpicture}
\ .
\end{equation*}
The fact that $\iota_i$ is a well-defined bimodule homomorphism is shown in \cite[Proposition 17]{KhovSussan}.  

It is clear that $\iota_i \colon W \rightarrow W_i \{-2 \}$ does not commute with $\partial$.  However, we do have the following variation.

\begin{prop}
The bimodule homomorphism $\iota_i \colon W \rightarrow W_i^{-e_1} \{-2 \}$ commutes with the action of $\partial$.
\end{prop}

\begin{proof}
We first check the proposition in the case that there is no black strand in between the $i$th and $(i+1)$st red strands.  By definition
\begin{equation*}
\iota_i \circ  \partial
\left(~
\begin{DGCpicture}
\DGCstrand[red](0,0)(0,1.5)
\DGCstrand[red](1,0)(1,1.5)
\end{DGCpicture}
~\right)
~=~
0
\ .
\end{equation*}
On the other hand,
\begin{align*}
\partial \circ \iota_i
\left(~
\begin{DGCpicture}
\DGCstrand[red](0,0)(0,1.5)
\DGCstrand[red](1,0)(1,1.5)
\end{DGCpicture}
~\right)
&~=~
\partial
\left(~
\begin{DGCpicture}
\DGCstrand[red](0,0)(.5,.5)
\DGCstrand[red](1,0)(.5,.5)
\DGCstrand[Red](.5,.5)(.5,1)
\DGCstrand[red](.5,1)(0,1.5)
\DGCdot{E}
\DGCstrand[red](.5,1)(1,1.5)
\end{DGCpicture}
~-~
\begin{DGCpicture}
\DGCstrand[red](0,0)(.5,.5)
\DGCstrand[red](1,0)(.5,.5)
\DGCdot{B}
\DGCstrand[Red](.5,.5)(.5,1)
\DGCstrand[red](.5,1)(0,1.5)
\DGCstrand[red](.5,1)(1,1.5)
\end{DGCpicture}
~\right) \\
&~=~
%%%%
\begin{DGCpicture}
\DGCstrand[red](0,0)(.5,.5)
\DGCstrand[red](1,0)(.5,.5)
\DGCstrand[Red](.5,.5)(.5,1)
\DGCstrand[red](.5,1)(0,1.5)
\DGCdot{E}[l]{$2$}
\DGCstrand[red](.5,1)(1,1.5)
\end{DGCpicture}
~-~
\begin{DGCpicture}
\DGCstrand[red](0,0)(.5,.5)
\DGCstrand[red](1,0)(.5,.5)
\DGCdot{B}[r]{$2$}
\DGCstrand[Red](.5,.5)(.5,1)
\DGCstrand[red](.5,1)(0,1.5)
\DGCstrand[red](.5,1)(1,1.5)
\end{DGCpicture}
~-~
\begin{DGCpicture}
\DGCstrand[red](0,0)(.5,.5)
\DGCstrand[red](1,0)(.5,.5)
\DGCstrand[Red](.5,.5)(.5,1)
\DGCstrand[red](.5,1)(0,1.5)
\DGCdot{E}[l]{$2$}
\DGCstrand[red](.5,1)(1,1.5)
\end{DGCpicture}
~-~
\begin{DGCpicture}
\DGCstrand[red](0,0)(.5,.5)
\DGCstrand[red](1,0)(.5,.5)
\DGCstrand[Red](.5,.5)(.5,1)
\DGCstrand[red](.5,1)(0,1.5)
\DGCdot{E}
\DGCstrand[red](.5,1)(1,1.5)
\DGCdot{E}
\end{DGCpicture}
~+~
\begin{DGCpicture}
\DGCstrand[red](0,0)(.5,.5)
\DGCdot{B}
\DGCstrand[red](1,0)(.5,.5)
\DGCdot{B}
\DGCstrand[Red](.5,.5)(.5,1)
\DGCstrand[red](.5,1)(0,1.5)
\DGCstrand[red](.5,1)(1,1.5)
\end{DGCpicture}
~+~
\begin{DGCpicture}
\DGCstrand[red](0,0)(.5,.5)
\DGCstrand[red](1,0)(.5,.5)
\DGCdot{B}[r]{$2$}
\DGCstrand[Red](.5,.5)(.5,1)
\DGCstrand[red](.5,1)(0,1.5)
\DGCstrand[red](.5,1)(1,1.5)
\end{DGCpicture} \\
&~=~ 0
\end{align*}
where the four last terms come from the twisted $p$-DG structure on $W_i^{-e_1}$ (see \eqref{defofWitwisted}).

The second case we must consider is when the black strand lies between the $i$th and $(i+1)$st red strands.  Once again
\begin{equation*}
\iota_i \circ  \partial
\left(~
\begin{DGCpicture}
\DGCstrand[red](0,0)(0,1.5)
\DGCstrand(.5,0)(.5,1.5)
\DGCstrand[red](1,0)(1,1.5)
\end{DGCpicture}
~\right)
~=~
0
\ .
\end{equation*}
Let
\begin{equation*}
X~=~
\begin{DGCpicture}
\DGCstrand[red](0,0)(.5,.5)
\DGCstrand[red](1,0)(.5,.5)
\DGCstrand[Red](.5,.5)(.5,1)
\DGCstrand[red](.5,1)(0,1.5)
\DGCstrand[red](.5,1)(1,1.5)
\DGCstrand(.5,0)(1,.75)(.5,1.5)
\end{DGCpicture}
~-~
\begin{DGCpicture}
\DGCstrand[red](0,0)(.5,.5)
\DGCstrand[red](1,0)(.5,.5)
\DGCstrand[Red](.5,.5)(.5,1)
\DGCstrand[red](.5,1)(0,1.5)
\DGCstrand[red](.5,1)(1,1.5)
\DGCstrand(.5,0)(0,.75)(.5,1.5)
\end{DGCpicture}
~\in~ W_i^{-e_1}
\ .
\end{equation*}
We must show that $\partial(X)=0$.  By definition
\begin{equation}
\label{dX1}
\partial(X)
~=~
\begin{DGCpicture}
\DGCstrand[red](0,0)(.5,.5)
\DGCstrand[red](1,0)(.5,.5)
\DGCstrand[Red](.5,.5)(.5,1)
\DGCstrand[red](.5,1)(0,1.5)
\DGCstrand[red](.5,1)(1,1.5)
\DGCstrand(.5,0)(1,.75)(.5,1.5)
\DGCdot{}
%\DGCdot{B}[r]{$2$}
\end{DGCpicture}
~+~
\begin{DGCpicture}
\DGCstrand[red](0,0)(.5,.5)
\DGCstrand[red](1,0)(.5,.5)
\DGCdot{B}[r]{}
\DGCstrand[Red](.5,.5)(.5,1)
\DGCstrand[red](.5,1)(0,1.5)
\DGCstrand[red](.5,1)(1,1.5)
\DGCstrand(.5,0)(1,.75)(.5,1.5)
\end{DGCpicture}
~-~
\begin{DGCpicture}
\DGCstrand[red](0,0)(.5,.5)
\DGCstrand[red](1,0)(.5,.5)
\DGCstrand[Red](.5,.5)(.5,1)
\DGCstrand[red](.5,1)(0,1.5)
\DGCstrand[red](.5,1)(1,1.5)
\DGCstrand(.5,0)(0,.75)(.5,1.5)
\DGCdot{}
\end{DGCpicture}
~-~
\begin{DGCpicture}
\DGCstrand[red](0,0)(.5,.5)
\DGCstrand[red](1,0)(.5,.5)
\DGCstrand[Red](.5,.5)(.5,1)
\DGCstrand[red](.5,1)(0,1.5)
\DGCdot{E}[r]{}
\DGCstrand[red](.5,1)(1,1.5)
\DGCstrand(.5,0)(0,.75)(.5,1.5)
\end{DGCpicture}
~-~
\begin{DGCpicture}
\DGCstrand[red](0,0)(.5,.5)
\DGCstrand[red](1,0)(.5,.5)
\DGCstrand[Red](.5,.5)(.5,1)
\DGCstrand[red](.5,1)(0,1.5)
\DGCdot{E}[r]{}
\DGCstrand[red](.5,1)(1,1.5)
\DGCstrand(.5,0)(1,.75)(.5,1.5)
\end{DGCpicture}
~-~
\begin{DGCpicture}
\DGCstrand[red](0,0)(.5,.5)
\DGCstrand[red](1,0)(.5,.5)
\DGCstrand[Red](.5,.5)(.5,1)
\DGCstrand[red](.5,1)(0,1.5)
\DGCstrand[red](.5,1)(1,1.5)
\DGCdot{E}[r]{}
\DGCstrand(.5,0)(1,.75)(.5,1.5)
\end{DGCpicture}
~+~
\begin{DGCpicture}
\DGCstrand[red](0,0)(.5,.5)
\DGCstrand[red](1,0)(.5,.5)
\DGCstrand[Red](.5,.5)(.5,1)
\DGCstrand[red](.5,1)(0,1.5)
\DGCdot{E}[r]{}
\DGCstrand[red](.5,1)(1,1.5)
\DGCstrand(.5,0)(0,.75)(.5,1.5)
\end{DGCpicture}
~+~
\begin{DGCpicture}
\DGCstrand[red](0,0)(.5,.5)
\DGCstrand[red](1,0)(.5,.5)
\DGCstrand[Red](.5,.5)(.5,1)
\DGCstrand[red](.5,1)(0,1.5)
\DGCstrand[red](.5,1)(1,1.5)
\DGCdot{E}[r]{}
\DGCstrand(.5,0)(0,.75)(.5,1.5)
\end{DGCpicture}
\end{equation}
where the last four terms comes from the twisting on $W_i^{-e_1}$.  It follows that
\begin{equation}
\label{dX2}
\partial(X)
~=~
\begin{DGCpicture}
\DGCstrand[red](0,0)(.5,.5)
\DGCstrand[red](1,0)(.5,.5)
\DGCstrand[Red](.5,.5)(.5,1)
\DGCstrand[red](.5,1)(0,1.5)
\DGCstrand[red](.5,1)(1,1.5)
\DGCstrand(.5,0)(1,.75)(.5,1.5)
\DGCdot{}
%\DGCdot{B}[r]{$2$}
\end{DGCpicture}
~-~
\begin{DGCpicture}
\DGCstrand[red](0,0)(.5,.5)
\DGCdot{B}[r]{}
\DGCstrand[red](1,0)(.5,.5)
\DGCstrand[Red](.5,.5)(.5,1)
\DGCstrand[red](.5,1)(0,1.5)
\DGCstrand[red](.5,1)(1,1.5)
\DGCstrand(.5,0)(1,.75)(.5,1.5)
\end{DGCpicture}
~-~
\begin{DGCpicture}
\DGCstrand[red](0,0)(.5,.5)
\DGCstrand[red](1,0)(.5,.5)
\DGCstrand[Red](.5,.5)(.5,1)
\DGCstrand[red](.5,1)(0,1.5)
\DGCstrand[red](.5,1)(1,1.5)
\DGCstrand(.5,0)(0,.75)(.5,1.5)
\DGCdot{E}
\end{DGCpicture}
~+~
\begin{DGCpicture}
\DGCstrand[red](0,0)(.5,.5)
\DGCstrand[red](1,0)(.5,.5)
\DGCstrand[Red](.5,.5)(.5,1)
\DGCstrand[red](.5,1)(0,1.5)
\DGCstrand[red](.5,1)(1,1.5)
\DGCdot{E}[r]{}
\DGCstrand(.5,0)(0,.75)(.5,1.5)
\end{DGCpicture}
\end{equation}
where the second, fifth, and sixth terms in \eqref{dX1} combine to become the second term in \eqref{dX2}.  Using the first set of relations of \eqref{Wrelations} and \eqref{blackthickred1}, we get that the above simplifies to 
\begin{equation*}
\partial(X)
~=~
\begin{DGCpicture}
\DGCstrand[red](0,0)(.5,.5)
\DGCstrand[red](1,0)(.5,.5)
\DGCstrand[Red](.5,.5)(.5,1)
\DGCstrand[red](.5,1)(0,1.5)
\DGCstrand[red](.5,1)(1,1.5)
\DGCstrand(.5,0)(.2,.5)(.8,1)(.5,1.5)
\end{DGCpicture}
~-~
\begin{DGCpicture}
\DGCstrand[red](0,0)(.5,.5)
\DGCstrand[red](1,0)(.5,.5)
\DGCstrand[Red](.5,.5)(.5,1)
\DGCstrand[red](.5,1)(0,1.5)
\DGCstrand[red](.5,1)(1,1.5)
\DGCstrand(.5,0)(.2,.5)(.8,1)(.5,1.5)
\end{DGCpicture}
~=~
0
. \\
\end{equation*}
\end{proof}

\begin{prop}
There are bimodule homomorphisms
$\alpha_{i,i+1} \colon W_{i,i+1} \rightarrow W_i \otimes_W W_{i+1} \otimes_W W_i $ 
and
$\alpha_{i+1,i} \colon W_{i,i+1} \rightarrow W_{i+1} \otimes_W W_{i} \otimes_W W_{i+1} $ defined on the bimodule generator \eqref{thickthickgenerator} by
\begin{equation}
\alpha_{i,i+1} \colon 
\label{defofalphaii+1}
\begin{DGCpicture}[scale={.7,.7}]
\DGCstrand[red](0,0)(1,1)
\DGCstrand[red](1,0)(1,1)
\DGCstrand[red](2,0)(1,1)
\DGCstrand[Red](1,1)(1,3)
\DGCstrand[red](1,3)(0,4)
\DGCstrand[red](1,3)(1,4)
\DGCstrand[red](1,3)(2,4)
\end{DGCpicture}
~\mapsto~
\begin{DGCpicture}[scale={.7,.7}]
\DGCstrand[red](0,0)(.5,.5)
\DGCstrand[red](1,0)(.5,.5)
\DGCstrand[Red](.5,.5)(.5,1)
\DGCstrand[red](.5,1)(0,1.5)
\DGCstrand[red](.5,1)(1,1.5)
\DGCstrand[red](2,0)(2,1.5)
\DGCstrand[red](1,1.5)(1.5,2)
\DGCstrand[red](2,1.5)(1.5,2)
\DGCstrand[Red](1.5,2)(1.5,2.5)
\DGCstrand[red](1.5,2.5)(1,3)
\DGCstrand[red](1.5,2.5)(2,3)
\DGCstrand[red](0,1.5)(0,3)
%%%%
\DGCstrand[red](0,3)(.5,3.5)
\DGCstrand[red](1,3)(.5,3.5)
\DGCstrand[Red](.5,3.5)(.5,4)
\DGCstrand[red](.5,4)(0,4.5)
\DGCstrand[red](.5,4)(1,4.5)
\DGCstrand[red](2,3)(2,4.5)
\end{DGCpicture}
\end{equation}

\begin{equation}
\alpha_{i+1,i} \colon 
\begin{DGCpicture}[scale={.7,.7}]
\DGCstrand[red](0,0)(1,1)
\DGCstrand[red](1,0)(1,1)
\DGCstrand[red](2,0)(1,1)
\DGCstrand[Red](1,1)(1,3)
\DGCstrand[red](1,3)(0,4)
\DGCstrand[red](1,3)(1,4)
\DGCstrand[red](1,3)(2,4)
\end{DGCpicture}
~\mapsto~
\begin{DGCpicture}[scale={.7,.7}]
\DGCstrand[red](0,1.5)(.5,2)
\DGCstrand[red](1,1.5)(.5,2)
\DGCstrand[Red](.5,2)(.5,2.5)
\DGCstrand[red](.5,2.5)(0,3)
\DGCstrand[red](.5,2.5)(1,3)
\DGCstrand[red](2,1.5)(2,3)
\DGCstrand[red](1,0)(1.5,.5)
\DGCstrand[red](2,0)(1.5,.5)
\DGCstrand[Red](1.5,.5)(1.5,1)
\DGCstrand[red](1.5,1)(1,1.5)
\DGCstrand[red](1.5,1)(2,1.5)
\DGCstrand[red](0,0)(0,1.5)
%%%%
\DGCstrand[red](1,3)(1.5,3.5)
\DGCstrand[red](2,3)(1.5,3.5)
\DGCstrand[Red](1.5,3.5)(1.5,4)
\DGCstrand[red](1.5,4)(1,4.5)
\DGCstrand[red](1.5,4)(2,4.5)
\DGCstrand[red](0,3)(0,4.5)
\end{DGCpicture} \ .
\end{equation}
These homomorphisms commute with the action of $\partial$.
\end{prop}

\begin{proof}
It is straightforward to check that these are bimodule maps.
The differential $\partial$ annihilates both the generator \eqref{thickthickgenerator} and its image under the homomorphisms.  
\end{proof}

There is a bimodule homomorphism $W_i \otimes_W W_{i+1} \otimes_W W_i  \rightarrow W_i \lbrace 2 \rbrace$ defined as a composition of the maps $\epsilon$ and $\blacktriangle$ constructed in \cite{KhovSussan}.  It is defined on generators in \eqref{surj1} and \eqref{surj2}.

\begin{equation}
\label{surj1}
\begin{DGCpicture}[scale={.7,.7}]
\DGCstrand[red](0,0)(.5,.5)
\DGCstrand[red](1,0)(.5,.5)
\DGCstrand[Red](.5,.5)(.5,1)
\DGCstrand[red](.5,1)(0,1.5)
\DGCstrand[red](.5,1)(1,1.5)
\DGCstrand[red](2,0)(2,1.5)
\DGCstrand[red](1,1.5)(1.5,2)
\DGCstrand[red](2,1.5)(1.5,2)
\DGCstrand[Red](1.5,2)(1.5,2.5)
\DGCstrand[red](1.5,2.5)(1,3)
\DGCstrand[red](1.5,2.5)(2,3)
\DGCstrand[red](0,1.5)(0,3)
%%%%
\DGCstrand[red](0,3)(.5,3.5)
\DGCstrand[red](1,3)(.5,3.5)
\DGCstrand[Red](.5,3.5)(.5,4)
\DGCstrand[red](.5,4)(0,4.5)
\DGCstrand[red](.5,4)(1,4.5)
\DGCstrand[red](2,3)(2,4.5)
\end{DGCpicture}
%%%%%%%%%%
\mapsto
0
\  ,
\quad \quad \quad
\begin{DGCpicture}[scale={.7,.7}]
\DGCstrand[red](0,0)(.5,.5)
\DGCstrand[red](1,0)(.5,.5)
\DGCstrand[Red](.5,.5)(.5,1)
\DGCstrand[red](.5,1)(0,1.5)
\DGCstrand[red](.5,1)(1,1.5)
\DGCstrand[red](2,0)(2,1.5)
\DGCstrand[red](1,1.5)(1.5,2)
\DGCstrand[red](2,1.5)(1.5,2)
\DGCstrand[Red](1.5,2)(1.5,2.5)
\DGCstrand[red](1.5,2.5)(1,3)
\DGCstrand[red](1.5,2.5)(2,3)
\DGCstrand[red](0,1.5)(0,3)
\DGCdot{2.25}[r]{$ $}
%%%%
\DGCstrand[red](0,3)(.5,3.5)
\DGCstrand[red](1,3)(.5,3.5)
\DGCstrand[Red](.5,3.5)(.5,4)
\DGCstrand[red](.5,4)(0,4.5)
\DGCstrand[red](.5,4)(1,4.5)
\DGCstrand[red](2,3)(2,4.5)
\end{DGCpicture}
%%%%%%%%%%
\mapsto
\begin{DGCpicture}[scale={.7,.7}]
\DGCstrand[red](0,0)(.5,1.5)
\DGCstrand[red](1,0)(.5,1.5)
\DGCstrand[Red](.5,1.5)(.5,3)
\DGCstrand[red](.5,3)(0,4.5)
\DGCstrand[red](.5,3)(1,4.5)
\DGCstrand[red](2,0)(2,4.5)
\end{DGCpicture}
\end{equation}

\begin{equation}
\label{surj2}
\begin{DGCpicture}[scale={.7,.7}]
\DGCstrand[black](1.5,0)(.5,2.25)(1.5,4.5)
%%%%%
\DGCstrand[red](0,0)(.5,.5)
\DGCstrand[red](1,0)(.5,.5)
\DGCstrand[Red](.5,.5)(.5,1)
\DGCstrand[red](.5,1)(0,1.5)
\DGCstrand[red](.5,1)(1,1.5)
\DGCstrand[red](2,0)(2,1.5)
\DGCstrand[red](1,1.5)(1.5,2)
\DGCstrand[red](2,1.5)(1.5,2)
\DGCstrand[Red](1.5,2)(1.5,2.5)
\DGCstrand[red](1.5,2.5)(1,3)
\DGCstrand[red](1.5,2.5)(2,3)
\DGCstrand[red](0,1.5)(0,3)
%%%%
\DGCstrand[red](0,3)(.5,3.5)
\DGCstrand[red](1,3)(.5,3.5)
\DGCstrand[Red](.5,3.5)(.5,4)
\DGCstrand[red](.5,4)(0,4.5)
\DGCstrand[red](.5,4)(1,4.5)
\DGCstrand[red](2,3)(2,4.5)
\end{DGCpicture}
%%%%%%%%%%
\mapsto
\begin{DGCpicture}[scale={.7,.7}]
\DGCstrand[black](1.5,0)(1.5,4.5)
\DGCstrand[red](0,0)(.5,1.5)
\DGCstrand[red](1,0)(.5,1.5)
\DGCstrand[Red](.5,1.5)(.5,3)
\DGCstrand[red](.5,3)(0,4.5)
\DGCstrand[red](.5,3)(1,4.5)
\DGCstrand[red](2,0)(2,4.5)
\end{DGCpicture}
\  ,
\quad \quad \quad
\begin{DGCpicture}[scale={.7,.7}]
\DGCstrand[black](1.5,0)(.5,2.25)(1.5,4.5) %%%%%
\DGCstrand[red](0,0)(.5,.5)
\DGCstrand[red](1,0)(.5,.5)
\DGCstrand[Red](.5,.5)(.5,1)
\DGCstrand[red](.5,1)(0,1.5)
\DGCstrand[red](.5,1)(1,1.5)
\DGCstrand[red](2,0)(2,1.5)
\DGCstrand[red](1,1.5)(1.5,2)
\DGCstrand[red](2,1.5)(1.5,2)
\DGCstrand[Red](1.5,2)(1.5,2.5)
\DGCstrand[red](1.5,2.5)(1,3)
\DGCstrand[red](1.5,2.5)(2,3)
\DGCstrand[red](0,1.5)(0,3)
\DGCdot{2.25}[r]{$ $}
%%%%
\DGCstrand[red](0,3)(.5,3.5)
\DGCstrand[red](1,3)(.5,3.5)
\DGCstrand[Red](.5,3.5)(.5,4)
\DGCstrand[red](.5,4)(0,4.5)
\DGCstrand[red](.5,4)(1,4.5)
\DGCstrand[red](2,3)(2,4.5)
\end{DGCpicture}
%%%%%%%%%%
\mapsto
\begin{DGCpicture}[scale={.7,.7}]
\DGCstrand[black](1.5,0)(1.5,4.5)
\DGCdot{2.25}[r]{$ $}
\DGCstrand[red](0,0)(.5,1.5)
\DGCstrand[red](1,0)(.5,1.5)
\DGCstrand[Red](.5,1.5)(.5,3)
\DGCstrand[red](.5,3)(0,4.5)
\DGCstrand[red](.5,3)(1,4.5)
\DGCstrand[red](2,0)(2,4.5)
\end{DGCpicture}
\end{equation}

The proof of the next result is similar to the proofs of Propositions \ref{basis1} and \ref{basis2}, so we just supply the general strategy.
\begin{prop}
\label{basis3}
The spanning set in Proposition \ref{WiWi+1Wispan} is a basis of $W_i \otimes_W W_{i+1} \otimes_W W_i$.
\end{prop}

\begin{proof}
Let $B_{j}^i $ be elements in the spanning set which are in the image of multiplication on the left by $e_i$ and on the right by $e_{j}$.  

In order to prove  that the elements of $B_{i+1}^i$ are linearly independent, one writes down a dependence relation and uses the bimodule homomorphisms $\phi_i$ from Proposition \ref{propphi}, and the bimodule homomorphism defined in \eqref{surj1} and \eqref{surj2} to conclude that all the coefficients in the dependence relation are zero.
One proves in a similar way the linear independence of elements in $B_i^i$ and $B_i^{i+1}$.

By applying the element
\[
\begin{DGCpicture}
\DGCstrand[red](0,0)(0,1)[$^{1}$`{\ }]
\DGCcoupon*(.25,0.25)(.75,0.75){$\cdots$}
\DGCstrand[red](1,0)(2,1)[$^{i+1}$`{\ }]
\DGCstrand(2,0)(1,1)
\DGCstrand[red](3,0)(3,1)[$^{i+2}$`{\ }]
\DGCcoupon*(3.25,0.25)(3.75,0.75){$\cdots$}
\DGCstrand[red](4,0)(4,1)[$^{n}$`{\ }]
\end{DGCpicture}
\ ,
\]
on top, the linear independence of elements in $B_{i+1}^{i+1} $ follows from the linear independence of $B_{i+1}^i$.

Showing that all other sets $B^i_j$ are linearly independent is routine.
\end{proof}

\subsection{Braid group action}
As a consequence of the results of Section \ref{secbihoms}, there are complexes of $(W,W)\# H$-modules
\begin{equation*}
\Sigma_i=
\xymatrix{
W_i \ar[r]^{\epsilon_i} & W
}
, \\
\quad \quad \quad 
\Sigma_i'=
\xymatrix{
W \ar[r]^{\iota_i \quad \quad} & W_i^{-e_1} \{-2 \}
}
. \\
\end{equation*} 

\begin{lem} \label{BB=B+B}
There exists an isomorphism of $(W,W)\# H$-modules
\[
W_i \otimes_W W_i \cong W_i \oplus W_i^{e_1} \{ 2  \}
\]
\end{lem}

\begin{proof}
First note that using isotopies, black strands can be moved out of the way in these bimodules as in \cite[Lemma 5]{KhovSussan}.  Thus the proof reduces to the proof of \cite[Lemma 4.3]{KRWitt}
or \cite[Lemma 3.5]{QiSussan4}.
\end{proof}

\begin{lem} \label{sesWii+1}
There is a short exact sequence of $(W,W)\# H$-modules
which splits as $(W,W)$-bimodules
\[
0 \rightarrow W_{i,i+1} \rightarrow W_i \otimes_W W_{i+1} \otimes_W W_i \rightarrow 
W_i^{e_1} \rightarrow 0.
\]
\end{lem}

\begin{proof}
The map $W_{i,i+1} \rightarrow W_i \otimes_W W_{i+1} \otimes_W W_i $
is just given by $\alpha_{i,i+1}$ defined in \eqref{defofalphaii+1}.  

Recall there is a map $W_i \otimes_W W_{i+1} \otimes_W W_i  \rightarrow W_i$ defined  on generators in \eqref{surj1} and \eqref{surj2}.  This is clearly a surjection and 
a straightforward calculation shows that this surjection is a $p$-DG map.

Note that
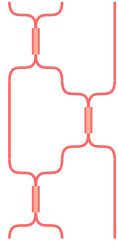
\begin{equation*}
\begin{DGCpicture}[scale={.7,.7}]
\DGCstrand[red](0,0)(.5,.5)
\DGCstrand[red](1,0)(.5,.5)
\DGCstrand[Red](.5,.5)(.5,1)
\DGCstrand[red](.5,1)(0,1.5)
\DGCstrand[red](.5,1)(1,1.5)
\DGCstrand[red](2,0)(2,1.5)
\DGCstrand[red](1,1.5)(1.5,2)
\DGCstrand[red](2,1.5)(1.5,2)
\DGCstrand[Red](1.5,2)(1.5,2.5)
\DGCstrand[red](1.5,2.5)(1,3)
\DGCstrand[red](1.5,2.5)(2,3)
\DGCstrand[red](0,1.5)(0,3)
%%%%
\DGCstrand[red](0,3)(.5,3.5)
\DGCstrand[red](1,3)(.5,3.5)
\DGCstrand[Red](.5,3.5)(.5,4)
\DGCstrand[red](.5,4)(0,4.5)
\DGCstrand[red](.5,4)(1,4.5)
\DGCstrand[red](2,3)(2,4.5)
\end{DGCpicture}
\end{equation*}
is in the kernel of $W_i \otimes_W W_{i+1} \otimes_W W_i \rightarrow W_i^{e_1}$.
By a graded dimension count utilizing the bases of the bimodules from Propositions \ref{basis1}, \ref{basis2}, \ref{basis3}, we get the exactness of the sequence of the lemma.

Now we define a splitting map
$ W_i \otimes_W W_{i+1} \otimes_W W_i \rightarrow W_{i,i+1}$ by
\begin{equation}
\label{split1}
\begin{DGCpicture}[scale={.7,.7}]
\DGCstrand[black](1.5,0)(.5,2.25)(1.5,4.5)
%%%%%
\DGCstrand[red](0,0)(.5,.5)
\DGCstrand[red](1,0)(.5,.5)
\DGCstrand[Red](.5,.5)(.5,1)
\DGCstrand[red](.5,1)(0,1.5)
\DGCstrand[red](.5,1)(1,1.5)
\DGCstrand[red](2,0)(2,1.5)
\DGCstrand[red](1,1.5)(1.5,2)
\DGCstrand[red](2,1.5)(1.5,2)
\DGCstrand[Red](1.5,2)(1.5,2.5)
\DGCstrand[red](1.5,2.5)(1,3)
\DGCstrand[red](1.5,2.5)(2,3)
\DGCstrand[red](0,1.5)(0,3)
%%%%
\DGCstrand[red](0,3)(.5,3.5)
\DGCstrand[red](1,3)(.5,3.5)
\DGCstrand[Red](.5,3.5)(.5,4)
\DGCstrand[red](.5,4)(0,4.5)
\DGCstrand[red](.5,4)(1,4.5)
\DGCstrand[red](2,3)(2,4.5)
\end{DGCpicture}
%%%%%%%%%%
\mapsto
0
\  ,
\quad \quad \quad \quad
\begin{DGCpicture}[scale={.7,.7}]
\DGCstrand[black](1.5,0)(.5,2.25)(1.5,4.5)
%%%%%
\DGCstrand[red](0,0)(.5,.5)
\DGCstrand[red](1,0)(.5,.5)
\DGCstrand[Red](.5,.5)(.5,1)
\DGCstrand[red](.5,1)(0,1.5)
\DGCstrand[red](.5,1)(1,1.5)
\DGCstrand[red](2,0)(2,1.5)
\DGCstrand[red](1,1.5)(1.5,2)
\DGCstrand[red](2,1.5)(1.5,2)
\DGCstrand[Red](1.5,2)(1.5,2.5)
\DGCstrand[red](1.5,2.5)(1,3)
\DGCstrand[red](1.5,2.5)(2,3)
\DGCstrand[red](0,1.5)(0,3)
\DGCdot{2.25}[r]{$ $}
%%%%
\DGCstrand[red](0,3)(.5,3.5)
\DGCstrand[red](1,3)(.5,3.5)
\DGCstrand[Red](.5,3.5)(.5,4)
\DGCstrand[red](.5,4)(0,4.5)
\DGCstrand[red](.5,4)(1,4.5)
\DGCstrand[red](2,3)(2,4.5)
\end{DGCpicture}
%%%%%%%%%%
\mapsto
0
\ ,
\end{equation}

\begin{equation}
\label{split2}
\begin{DGCpicture}[scale={.7,.7}]
\DGCstrand[red](0,0)(.5,.5)
\DGCstrand[red](1,0)(.5,.5)
\DGCstrand[Red](.5,.5)(.5,1)
\DGCstrand[red](.5,1)(0,1.5)
\DGCstrand[red](.5,1)(1,1.5)
\DGCstrand[red](2,0)(2,1.5)
\DGCstrand[red](1,1.5)(1.5,2)
\DGCstrand[red](2,1.5)(1.5,2)
\DGCstrand[Red](1.5,2)(1.5,2.5)
\DGCstrand[red](1.5,2.5)(1,3)
\DGCstrand[red](1.5,2.5)(2,3)
\DGCstrand[red](0,1.5)(0,3)
\DGCdot{2.25}[r]{$ $}
%%%%
\DGCstrand[red](0,3)(.5,3.5)
\DGCstrand[red](1,3)(.5,3.5)
\DGCstrand[Red](.5,3.5)(.5,4)
\DGCstrand[red](.5,4)(0,4.5)
\DGCstrand[red](.5,4)(1,4.5)
\DGCstrand[red](2,3)(2,4.5)
\end{DGCpicture}
\mapsto 
0
\ ,
\quad \quad \quad \quad
\begin{DGCpicture}[scale={.7,.7}]
\DGCstrand[red](0,0)(.5,.5)
\DGCstrand[red](1,0)(.5,.5)
\DGCstrand[Red](.5,.5)(.5,1)
\DGCstrand[red](.5,1)(0,1.5)
\DGCstrand[red](.5,1)(1,1.5)
\DGCstrand[red](2,0)(2,1.5)
\DGCstrand[red](1,1.5)(1.5,2)
\DGCstrand[red](2,1.5)(1.5,2)
\DGCstrand[Red](1.5,2)(1.5,2.5)
\DGCstrand[red](1.5,2.5)(1,3)
\DGCstrand[red](1.5,2.5)(2,3)
\DGCstrand[red](0,1.5)(0,3)
%\DGCdot{2.25}[r]{$ $}
%%%%
\DGCstrand[red](0,3)(.5,3.5)
\DGCstrand[red](1,3)(.5,3.5)
\DGCstrand[Red](.5,3.5)(.5,4)
\DGCstrand[red](.5,4)(0,4.5)
\DGCstrand[red](.5,4)(1,4.5)
\DGCstrand[red](2,3)(2,4.5)
\end{DGCpicture}
\mapsto 
\begin{DGCpicture}[scale={.7,.7}]
\DGCstrand[red](0,0)(1,1)
\DGCstrand[red](1,0)(1,1)
\DGCstrand[red](2,0)(1,1)
\DGCstrand[Red](1,1)(1,3)
\DGCstrand[red](1,3)(0,4)
\DGCstrand[red](1,3)(1,4)
\DGCstrand[red](1,3)(2,4)
\end{DGCpicture}
\ .
\end{equation}
Note that this splitting does not respect the $p$-DG structure.

Next we define a splitting map
$ W_i^{e_1} \rightarrow W_i \otimes_W W_{i+1} \otimes_W W_i$ by
\begin{equation}
\label{split3}
\begin{DGCpicture}[scale={.7,.7}]
\DGCstrand[red](0,0)(.5,1.5)
\DGCstrand[red](1,0)(.5,1.5)
\DGCstrand[Red](.5,1.5)(.5,3)
\DGCstrand[red](.5,3)(0,4.5)
\DGCstrand[red](.5,3)(1,4.5)
\DGCstrand[red](2,0)(2,4.5)
\end{DGCpicture}
\mapsto
\begin{DGCpicture}[scale={.7,.7}]
\DGCstrand[red](0,0)(.5,.5)
\DGCstrand[red](1,0)(.5,.5)
\DGCstrand[Red](.5,.5)(.5,1)
\DGCstrand[red](.5,1)(0,1.5)
\DGCstrand[red](.5,1)(1,1.5)
\DGCstrand[red](2,0)(2,1.5)
\DGCstrand[red](1,1.5)(1.5,2)
\DGCstrand[red](2,1.5)(1.5,2)
\DGCstrand[Red](1.5,2)(1.5,2.5)
\DGCstrand[red](1.5,2.5)(1,3)
\DGCstrand[red](1.5,2.5)(2,3)
\DGCstrand[red](0,1.5)(0,3)
\DGCdot{2.25}[r]{$ $}
%%%%
\DGCstrand[red](0,3)(.5,3.5)
\DGCstrand[red](1,3)(.5,3.5)
\DGCstrand[Red](.5,3.5)(.5,4)
\DGCstrand[red](.5,4)(0,4.5)
\DGCstrand[red](.5,4)(1,4.5)
\DGCstrand[red](2,3)(2,4.5)
\end{DGCpicture} \ ,
%%%%%%%%%%
\quad \quad \quad \quad
\begin{DGCpicture}[scale={.7,.7}]
\DGCstrand[black](1.5,0)(1.5,4.5)
\DGCstrand[red](0,0)(.5,1.5)
\DGCstrand[red](1,0)(.5,1.5)
\DGCstrand[Red](.5,1.5)(.5,3)
\DGCstrand[red](.5,3)(0,4.5)
\DGCstrand[red](.5,3)(1,4.5)
\DGCstrand[red](2,0)(2,4.5)
\end{DGCpicture}
\mapsto
\begin{DGCpicture}[scale={.7,.7}]
\DGCstrand[black](1.5,0)(.5,2.25)(1.5,4.5)
%%%%%
\DGCstrand[red](0,0)(.5,.5)
\DGCstrand[red](1,0)(.5,.5)
\DGCstrand[Red](.5,.5)(.5,1)
\DGCstrand[red](.5,1)(0,1.5)
\DGCstrand[red](.5,1)(1,1.5)
\DGCstrand[red](2,0)(2,1.5)
\DGCstrand[red](1,1.5)(1.5,2)
\DGCstrand[red](2,1.5)(1.5,2)
\DGCstrand[Red](1.5,2)(1.5,2.5)
\DGCstrand[red](1.5,2.5)(1,3)
\DGCstrand[red](1.5,2.5)(2,3)
\DGCstrand[red](0,1.5)(0,3)
%%%%
\DGCstrand[red](0,3)(.5,3.5)
\DGCstrand[red](1,3)(.5,3.5)
\DGCstrand[Red](.5,3.5)(.5,4)
\DGCstrand[red](.5,4)(0,4.5)
\DGCstrand[red](.5,4)(1,4.5)
\DGCstrand[red](2,3)(2,4.5)
\end{DGCpicture}
%%%%%%%%%%
\ .
\end{equation}
Note that this splitting map does not respect the $p$-DG structure either.  All of these splitting maps are bimodule homomorphisms since they are compositions of bimodules maps described in detail in \cite{KhovSussan}.

Thus there is a short exact sequence of $(W,W) \# H$-bimodules
which splits as $(W,W)$-bimodules.
\end{proof}

The next crucial proposition is proved in a similar manner as in \cite[Theorems 4.2, 4.4]{KRWitt}.

\begin{prop} \label{braidclassical}
The complexes of $(W,W)\# H$-modules satisfy the following relations in the relative homotopy category.
\begin{enumerate}
\item $ \Sigma_i \circ \Sigma_i' \cong \Id \cong \Sigma_i' \circ \Sigma_i$,
\item $ \Sigma_i \circ \Sigma_j \cong \Sigma_j \circ \Sigma_i$ for $|i-j|>1$,
\item $ \Sigma_i \circ \Sigma_j \circ \Sigma_i \cong \Sigma_j \circ \Sigma_i \circ \Sigma_j$ for $|i-j|=1$.
\end{enumerate}
\end{prop}

We point out that, as in the usual homotopy category case, relative homotopy classes of $p$-DG bimodules over two $p$-DG algebras $A$ and $B$ give rise to functors from $\mc{C}^\dif(B,d_0)$ to $\mc{C}^\dif(A,d_0)$. Therefore, Proposition \ref{braidclassical}
can be regarded as an isomorphism of functors on the relative homotopy category of $p$-DG $W$-modules.

\begin{proof}
The first isomorphism follows as in \cite[Theorem 4.2]{KRWitt} or \cite[Proposition 3.7]{QiSussan4} which use versions of Lemma \ref{BB=B+B}.

The second item is clear.  

The third isomorphism follows as in \cite[Theorem 4.4]{KRWitt} or \cite[Proposition 3.9]{QiSussan4} which use versions of Lemma \ref{sesWii+1}.  
\end{proof}

Applying the $p$-extension functor $\mc{P}$, we obtain $p$-complexes of $p$-DG $(W,W)$-bimodules
$T_i:=  \mc{P} (\Sigma_i)$ and
$T_i':= \mc{P} (\Sigma_i')$ in the relative $p$-homotopy category $\mc{C}^\dif(W,\dif_0)$.  Explicitly, these complexes look like: 

\begin{align}
 T_i:=  & \left( W_i  \xrightarrow{=}\cdots 
 \xrightarrow{=}   W_i   \xrightarrow{\epsilon_i}   W \right)\\
T_i':= & \left( W \xrightarrow{\iota_i} W_i^{-e_1}\{-2\} \xrightarrow{=} \cdots \xrightarrow{=} W_i^{-e_1} \{-2\}
\right)
, 
\end{align}
where the repeated terms appear $p-1$ times. It should be pointed out that $T_i$ and $T_i^\prime$ are just cones of the $p$-DG bimodule maps $\epsilon_i$ and $\iota_i$ in the relative homotopy category of $p$-DG $(W,W)$-bimodules. This is because $\mc{P}$ is exact (Proposition \ref{relextot}), and thus it sends the cones $\Sigma_i$ and $\Sigma_i^\prime$ to cones in the relative $p$-homotopy category.

\begin{thm} \label{mainthm}
The $p$-complexes of $p$-DG $(W,W)$-bimodules satisfy the following isomorphisms of functors on the relative $p$-homotopy category $\mathcal{C}^{\dif}(W,\dif_0)$.
\begin{enumerate}
\item $ T_i \circ T_i' \cong \Id \cong T_i' \circ T_i$,
\item $ T_i \circ T_j \cong T_j \circ T_i$ for $|i-j|>1$,
\item $ T_i \circ T_j \circ T_i \cong T_j \circ T_i \circ T_j$ for $|i-j|=1$.
\end{enumerate}
\end{thm}

\begin{proof}
This follows from Proposition \ref{braidclassical} by applying the functor $\mc{P}$ from Section \ref{secpdgtheory}.
%This follows from Proposition \ref{braidclassical} by applying the functor $\mc{T} \circ \mc{P}$ from Section \ref{secpdgtheory}.
\end{proof}

%\begin{rem}
%Theorem \ref{mainthm} could be stated for the algebra %$\overline{W}$ instead of $W$.
%\end{rem}

\addcontentsline{toc}{section}{References}

% ====================================================================
% REFERENCES

\bibliographystyle{alpha}
\bibliography{qy-bib}

\begin{thebibliography}{KLSY18}

\bibitem[CK12]{CKbraid}
S.~Cautis and Kamnitzer.
\newblock Braiding via geometric {L}ie algebra actions.
\newblock {\em Compos. Math.}, 148(2):464--506, 2012.
\newblock \href{http://arxiv.org/abs/1001.0619}{arXiv:1001.0619}.

\bibitem[EK10]{EliasKh}
B.~Elias and M.~Khovanov.
\newblock Diagrammatics for {S}oergel categories.
\newblock {\em Int. J. Math. Math. Sci.}, pages Art. ID 978635, 58, 2010.
\newblock \href{http://arxiv.org/abs/0902.4700}{arXiv:0902.4700}.

\bibitem[EQ16a]{EQ2}
B.~Elias and Y.~Qi.
\newblock A categorification of quantum sl(2) at prime roots of unity.
\newblock {\em Adv. Math.}, 299:863--930, 2016.
\newblock \href{http://arxiv.org/abs/1503.05114}{arXiv:1503.05114}.

\bibitem[EQ16b]{EQ1}
B.~Elias and Y.~Qi.
\newblock A categorification of some small quantum groups {II}.
\newblock {\em Adv. Math.}, 288:81--151, 2016.
\newblock \href{http://arxiv.org/abs/1302.5478}{arXiv:1302.5478}.

\bibitem[Kho16]{Hopforoots}
M.~Khovanov.
\newblock {H}opfological algebra and categorification at a root of unity: {T}he
  first steps.
\newblock {\em J. Knot Theory Ramifications}, 25(3):359--426, 2016.
\newblock
  \href{http://front.math.ucdavis.edu/math.QA/0509083}{arXiv:math/0509083}.

\bibitem[KL10]{KL3}
M.~Khovanov and A.~D. Lauda.
\newblock A categorification of quantum sl(n).
\newblock {\em Quantum Topol.}, 2(1):1--92, 2010.
\newblock \href{http://arxiv.org/abs/0807.3250}{arXiv:0807.3250}.

\bibitem[KLSY18]{KLSY}
M.~Khovanov, A.~Lauda, J.~Sussan, and Y.~Yonezawa.
\newblock Braid group actions from categorical symmetric {H}owe duality on
  deformed {W}ebster algebras.
\newblock {\em Trans. Groups}, 2018.
\newblock \href{http://arxiv.org/abs/1802.05358}{arXiv:1802.05358}.

\bibitem[KQ15]{KQ}
M.~Khovanov and Y.~Qi.
\newblock An approach to categorification of some small quantum groups.
\newblock {\em Quantum Topol.}, 6(2):185--311, 2015.
\newblock \href{http://arxiv.org/abs/1208.0616}{arXiv:1208.0616}.

\bibitem[KQS17]{KQS}
M.~Khovanov, Y.~Qi, and J.~Sussan.
\newblock $p$-{DG} cyclotomic nilhecke algebras.
\newblock 2017.
\newblock \href{https://arxiv.org/abs/1711.07159}{arXiv:1711.07159}.

\bibitem[KR16]{KRWitt}
M.~Khovanov and L.~Rozansky.
\newblock Positive half of the {W}itt algebra acts on triply graded link
  homology.
\newblock {\em Quantum Topol.}, 7(4):737--795, 2016.
\newblock \href{https://arxiv.org/abs/1305.1642}{arXiv:1305.1642}.

\bibitem[KS02]{KS}
M.~Khovanov and P.~Seidel.
\newblock Quivers, {F}loer cohomology, and braid group actions.
\newblock {\em J. Amer. Math. Soc.}, 15:203--271, 2002.
\newblock \href{http://arxiv.org/abs/math/0006056}{arXiv:math/0006056}.

\bibitem[KS18]{KhovSussan}
M.~Khovanov and J.~Sussan.
\newblock The {S}oergel category and the redotted {W}ebster algebra.
\newblock {\em J. of Pure and Applied Algebra}, 222:1957--2000, 2018.
\newblock \href{http://arxiv.org/abs/1605.02678}{arXiv:1605.02678}.

\bibitem[Qi14]{QYHopf}
Y.~Qi.
\newblock Hopfological algebra.
\newblock {\em Compos. Math.}, 150(01):1--45, 2014.
\newblock \href{http://arxiv.org/abs/1205.1814}{arXiv:1205.1814}.

\bibitem[QS16]{QiSussan}
Y.~Qi and J.~Sussan.
\newblock A categorification of the {B}urau representation at prime roots of
  unity.
\newblock {\em Selecta Math. (N.S.)}, 22(3):1157--1193, 2016.
\newblock \href{http://arxiv.org/abs/1312.7692}{arXiv:1312.7692}.

\bibitem[QS17]{QiSussan2}
Y.~Qi and J.~Sussan.
\newblock Categorification at prime roots of unity and hopfological finiteness.
\newblock In {\em Categorification and higher representation theory}, volume
  683 of {\em Contemp. Math.}, pages 261--286. Amer. Math. Soc., Providence,
  RI, 2017.
\newblock \href{https://arxiv.org/abs/1509.00438}{arXiv:1509.00438}.

\bibitem[QS18]{QiSussan3}
Y.~Qi and J.~Sussan.
\newblock $p$-{DG} cyclotomic nil{H}ecke algebras {II}.
\newblock 2018.
\newblock \href{https://arxiv.org/abs/1811.04372}{arXiv:1811.04372}.

\bibitem[QS20]{QiSussan4}
Y.~Qi and J.~Sussan.
\newblock On some $p$-differential graded link homologies.
\newblock 2020.
\newblock \href{https://arxiv.org/abs/2009.06498}{arXiv:2009.06498}.

\bibitem[Rou08]{Rou2}
R.~Rouquier.
\newblock 2-{K}ac-{M}oody algebras.
\newblock 2008.
\newblock \href{http://arxiv.org/abs/0812.5023}{arXiv:0812.5023}.

\bibitem[SW11]{SWSchur}
C.~Stroppel and B.~Webster.
\newblock Quiver {S}chur algebras and q-{F}ock space.
\newblock 2011.
\newblock \href{https://arxiv.org/abs/1110.1115}{arXiv:1110.1115}.

\bibitem[Web17]{Webcombined}
B.~Webster.
\newblock Knot invariants and higher representation theory.
\newblock {\em Mem. Amer. Math. Soc.}, 250(1191):pp.~133, 2017.
\newblock \href{http://arxiv.org/abs/1309.3796}{arXiv:1309.3796}.

\bibitem[Web20]{Webdeformed}
B.~Webster.
\newblock Three perspectives on categorical symmetric {H}owe duality.
\newblock 2020.
\newblock \href{http://arxiv.org/abs/2001.07584}{arXiv:2001.07584}.

\bibitem[Yon19]{Y}
Y.~Yonezawa.
\newblock A $p$-{DG} deformed {W}ebster algebra of type {$A_1$}.
\newblock 2019.
\newblock \href{http://arxiv.org/abs/1910.04897}{arXiv:1910.04897}.

\end{thebibliography}

%
% ====================================================================

\vspace{0.1in}

%\noindent M.~K.: { \sl \small Department of Mathematics, Columbia University, New York, NY 10027, USA} \newline \noindent {\tt \small email: khovanovi@math.columbia.edu}

\vspace{0.1in}

\noindent Y.~Q.: { \sl \small Department of Mathematics, University of Virginia, Charlottesville, VA 22904, USA} \newline \noindent {\tt \small email: yq2dw@virginia.edu}

\vspace{0.1in}

\noindent J.~S.:  {\sl \small Department of Mathematics, CUNY Medgar Evers, Brooklyn, NY, 11225, USA}\newline \noindent {\tt \small email: jsussan@mec.cuny.edu \newline 
\sl \small Mathematics Program, The Graduate Center, CUNY, New York, NY, 10016, USA}\newline \noindent {\tt \small email: jsussan@gc.cuny.edu}

\vspace{0.1in} 

\noindent Y.~Y. {\sl \small Advanced Mathematical Institute, Osaka City University, Osaka, Japan} \newline \noindent {\tt \small email: yasuyoshi.yonezawa@gmail.com}

% ==============================================================================
%
\end{document}